 \documentclass[11pt]{article}
  \usepackage[textsize=tiny,textwidth=1.7cm]{todonotes} 
 \usepackage{fullpage}

\usepackage{amsmath,amsthm} 
\usepackage{amsfonts,mathabx} 
\usepackage{amssymb}

\usepackage{trsym,txfonts} % This gives symbols for triangulation generators

\usepackage{mathdots}
\usepackage{caption}
\usepackage{subcaption}

\usepackage{graphicx} 
\usepackage{tikz,tikz-cd}
\usepackage{stackrel} 
\usetikzlibrary{decorations.pathreplacing,angles,quotes}

\usepackage{mathtools}
\usepackage{color}

\usepackage{setspace} % to icrease spacing in tocA
\usepackage{tocloft} %toc package

\usepackage{hyperref}

\makeatletter
\newcommand\tableofcontentsA{%
    \@starttoc{toca}%
}
\makeatother

\newtheorem{prop}{Proposition}

\newtheorem{cor}{Corollary}

\newtheorem{lem}{Lemma}
\newtheorem{theorem}{Theorem}
\newtheorem{definition}{Definition}

\theoremstyle{remark}

\newcommand{\secref}[1]{Section \ref{#1}}
\newcommand{\thmref}[1]{Theorem \ref{#1}}
\newcommand{\propref}[1]{Proposition \ref{#1}}
\newcommand{\lemref}[1]{Lemma \ref{#1}}
\newcommand{\defref}[1]{Definition \ref{#1}}
\newcommand{\figref}[1]{Figure \ref{#1}}

\newcommand{\nat}{\mathbb{N}}

\newcommand{\set}[1]{\left\{#1\right\}}

\newcommand{\eps}{\varepsilon}
\newcommand{\aeps}{\epsilon}
\newcommand{\st}{\star}
\renewcommand{\ast}{\blackdiamond}

\newcommand{\magma}{\mathcal{M}}
\newcommand{\magman}{\mathcal{N}}

\newcommand{\canmag}{\mathcal{W}}

\newcommand{\magp}{\magma^+}
\newcommand{\mprime}{\magma^0}

\DeclarePairedDelimiter\norm{\lVert}{\rVert}%
\newcommand{\val}{norm\ }

\newcommand{\vals}{norms\ }
\newcommand{\vald}{normed\ }
\newcommand{\valns}{norm}
\newcommand{\Valns}{Norm}

\newcommand{\pre}{\varphi}
\newcommand{\pdni}{\hat\pi}

\newcommand{\opp}[1]{{#1}^\text{\tiny op}}
\newcommand{\refp}[1]{{#1}^\text{\tiny ref}}
\newcommand{\revp}[1]{{#1}^\text{\tiny rev}}

\newcommand{\bst}{\st}
\newcommand{\lst}{\st_L}
\newcommand{\rst}{\st_R}

\DeclareMathOperator{\id}{id}  
\DeclareMathOperator{\revm}{rev}

\DeclareMathOperator{\opm}{op} 
\DeclareMathOperator{\refm}{ref} 
\newcommand{\ffm}{\mathbb{F}}
\newcommand{\fref}[1]{$\ffm_{\ref{#1}}$}

\newcommand{\dpn}{{\tiny\mbox{\ref{fam_dp}}}}
\newcommand{\ptn}{{\tiny\mbox{\ref{fam_pt}}}}

\newcounter{family} 
\newenvironment{family}[1][]{
    \refstepcounter{family} \par\medskip
   \textbf{$\ffm_{\thefamily}\,:\,$ #1}
    }{\medskip}

\newcommand{\unimap}{\Upsilon} % Universal map

\newcommand{\leftpoint}{\overset{ {\scriptstyle L} }{\leftarrow}}
\newcommand{\rightpoint}{\overset{ \scriptstyle  R }{\rightarrow}}

\newcommand{\struct}{\mathbb{S}}
\newcommand{\usp}{ \tikz{\draw (0,0)--(0.2,0.2);}}
\newcommand{\dsp}{ \tikz{\draw (0,0.20)--(0.2,0);}}

\newcommand{\triple}{(\leftpoint,P,\rightpoint)}

\colorlet{productColor}{orange}
\colorlet{tripleColor}{blue!90}
\tikzset{color_product/.style={opacity=0.4,color=blue} }

\tikzset{color_left/.style={opacity=0.5,orange!90!black} }
\tikzset{color_right/.style={opacity=0.5,cyan!75!black} }

\tikzset{color_cbt/.style={blue!75!black} }
\tikzset{style_size1/.style={scale=1,line width=1pt,radius=5pt} }
\tikzset{style_size2/.style={scale=0.75,line width=0.75pt,radius=2.5pt,node font=\small} }
\tikzset{style_size3/.style={scale=0.5,line width=0.75pt,radius=2pt,node font=\small} }
\tikzset{style_size4/.style={scale=0.4,line width=0.75pt,radius=2pt,node font=\small} }
\tikzset{size_prod/.style={line width=1.2pt } }
\tikzset{color_triple/.style={color=blue!50} }
\tikzset{color_prefix/.style={color=reg} }
\tikzset{color_infix/.style={color=green} }
\tikzset{color_post/.style={color=blue} }
\tikzset{circle_nocolor/.pic={\fill  circle ; }}
\tikzset{circle_black/.pic={\fill[black] circle ; }}
\tikzset{circle_blue/.pic={\fill[blue] circle ; }}
\tikzset{circle_product/.pic={\fill[productColor,draw=productColor] circle ; }}
\tikzset{circle_white/.pic={\fill[white,draw=black] circle ; }}
\tikzset{cbt_style/.style={line width=1pt,radius=3.5pt} } % scale = 1.0
\tikzset{cbt_style2/.style={line width=0.75pt,radius=2pt} } % scale = 0.5
\tikzset{cbt_root_black/.pic={\fill[white,draw=black] circle ;\fill   circle (1pt) ; }}
\tikzset{cbt_internal_black/.pic={\fill[black] circle ; }}
\tikzset{cbt_leaf_black/.pic={\fill[fill=white, draw=black]  circle ; }}
\tikzset{cbt_root_blue/.pic={\fill[white,draw=blue] circle ;\fill   circle (1pt) ; }}
\tikzset{cbt_internal_blue/.pic={\fill[blue] circle ; }}
\tikzset{cbt_leaf_blue/.pic={\fill[fill=white, draw=blue]  circle ; }}
\tikzset{cbt_root_orange/.pic={\fill[white,draw=orange] circle ; }}
\tikzset{cbt_internal_orange/.pic={\fill[orange] circle ; }}
\tikzset{cbt_leaf_orange/.pic={\fill[fill=white, draw=orange]  circle ; }}

\tikzset{gentrt/.pic=
{
\draw (0,0.75) -- (0,0);
\path(0,0.75) pic{circle_white};
}
}

\tikzset{gentri/.pic=
    {
    \draw (90:0.25) -- (270:0.25);
    \draw[fill=white] (270:0.25) circle (2pt);
    \draw[fill=black] (90:0.25) circle (2pt);
    }
} 
\newcommand{\gentri}{
\tikz{\path pic{gentri};}
}
\tikzset{gencbt/.pic={
\draw circle (2pt);  
} }

\tikzset{gennm/.pic={
\draw[fill] (0,0) circle (2pt) -- (0.5,0);  
} }
\newcommand{\gennm}{
\tikz{\path pic{gennm};}
}
\tikzset{gensp/.pic={
 \fill[white,draw=black ]  (0.4,0) arc[start angle=180,end angle=0,radius=0.1cm];  
\draw (0,0) -- (1,0)  ;
} }
\newcommand{\gensp}{
\tikz{\path pic[scale=0.75]{gensp};}
}

\tikzset{tick/.pic={
\draw  (0,0)  --  (-0.15cm,0);
\draw  (0,0)  --  (0,0.15cm);
}
}
\tikzset{tick2/.pic={
\draw  (0,0)  --  (45: 0.15cm );
\draw  (0,0)  --  (135: 0.15cm );
}
}
\tikzset{genncc/.pic={\def\R{0.2cm}
\path (135:0.2cm) pic{tick};
    \draw  (0,0) circle (\R);
} }

\tikzset{genncp/.pic={\def\R{0.2cm}
\path (135:0.2cm) pic{tick};
    \draw  (0,0) circle (\R);
} }
\newcommand{\genncp}{
\tikz{\path pic{genncp};}
}
\tikzset{genfp/.pic={
 \fill[white,draw=black ]  (0.4,0) arc[start angle=180,end angle=360,radius=0.1cm];  
\draw (0,0) -- (1,0)  ;
} }
\newcommand{\genfp}{
\tikz{\path pic[scale=0.75]{genfp};}
}
\newcommand{\floorgen}{\genfp}
\newcommand{\trigen}{
\gentri
}
\newcommand{\matchgen}{
\gennm
 }

 \newcommand{\stairgen}{
 \gensp
}
\tikzset{ 
up/.pic = { \path [pic actions]    -- +(1,1); }  
}
\tikzset{ 
down/.pic = { \path [pic actions]  -- +(1,-1); } 
}
\newcommand{\nmgenshape}{
\begin{tikzpicture}
 \fill circle [radius=0.1] (0,0);
  \draw [thick] (0,0)   -- (1,0);
\end{tikzpicture}
}

\newcommand{\nmprodshape}{
\begin{tikzpicture}
   \draw [thick] (0,0) arc [radius=0.5, start angle=0, end angle = 180] ;
   \fill circle [radius=0.1] (1,0);
  \draw [thick] (0,0)   -- (1,0);
\end{tikzpicture}
}

\tikzset{ 
rleaf/.pic = { \draw  (0,0) -- (0.5,-1); \fill[draw=black] ++(0.5,-1) circle (5pt); } 
}
\tikzset{ 
lleaf/.pic = {  \draw  (0,0) -- (-0.5,-1); \fill[draw=black] ++(-0.5,-1) circle (5pt); }
}
\tikzset{ 
root/.pic = {  \fill (150:1cm) -- (-90:1cm) -- (30:1cm) -- cycle ; }
}

\tikzset{ 
cleaf/.pic = {  \draw  (0,0) -- (0,-1); \fill[ draw=black] ++(0,-1) circle (5pt); }
}

\begin{document}
%=========================================%
%=========================================%

% \input{input/appendix.tex}
% \end{document}
% %=========================================%

\setcounter{page}{1}

\title{A  Universal Bijection for Catalan Magmas}

\author{R. Brak\thanks{rb1@unimelb.edu.au} 
    \vspace{0.15 in} \\
    School of Mathematics and Statistics,\\
    The University of Melbourne\\
    Parkville,  Victoria 3052,\\
    Australia\\
 }

\maketitle

\begin{abstract} 
% Every Catalan family is a free magma generated by a single element. 

Free magmas can be defined using the standard universal mapping principle for free structures. We provide an alternative characterisation of free magmas which is simpler to use when proving some combinatorial structure  is a  free magma: the magma has to have a norm and unique factorisation. This characterisation can be used to show conjectured Catalan families are (once a suitable product is defined) free magmas. 
The \val (a certain `size' function) partitions the magma base set into subsets enumerated by generalised Catalan numbers.
The universal isomorphism between free magmas gives a  universal bijection -- essentially one bijection algorithm for any pair of Catalan families.  The morphism property ensures the bijection is recursive. 
The universal bijection allows us to give some rigour to the idea of an ``embedding'' bijection between Catalan objects which, in many cases, show how to  embed an element of one Catalan family into one of a different family. 
Multiplication on the right (respectively left) by the free magma generator gives rise to the right (respectively left) Narayana statistic. 
We discuss the relation between the symbolic method for Catalan families and the magma structure on the base set defined by the symbolic method. This shows which ``atomic'' element is the magma generator. 
The appendix gives the magma structure for 14 Catalan families. 
 \end{abstract}
 
\vfill
\paragraph{Keywords:}  

\newpage

\tableofcontents
\newpage 

\section*{List of Catalan Families}
\doublespacing
\tableofcontentsA
\singlespacing
\newpage

%============================================
\section{Introduction}
%============================================
\label{sec1} 

The sequence of  Catalan numbers \cite{A000108:aa},   are   defined by 
\begin{equation}\label{eq_catnumbers}
  C_n=\frac{1}{n+1}\binom{2n}{n} 
\end{equation}
for $n\in\nat=\set{1,2,3,\dots}$ with $C_0=1$, and are well known to count  a large number of different families of objects. 
Sixty-six such families are given in problem set of chapter 6 of \cite{stanley:1999vw2,Koshy:2009aa}. 
A list of  214 families has been published by R.\ Stanley \cite{Stanley:2015aa}. 
A history of Catalan numbers by I.\ Pak can be found in the appendix of \cite{Stanley:2015aa} and a very extensive biography of Catalan (and Bell numbers ) was compiled by H. Gould \cite{Gould:aa}.  
New Catalan families are regularly being discovered, for example the Catalan floor plans \cite{Beaton:2015aa}.  

Each time a new family of objects are discovered and direct enumeration suggests they form a Catalan family, then proof of such a conjecture can be direct (eg.\ via generating function or sign reversing involution) or by bijection to a known Catalan family.
There is a long history of Catalan bijections. We make no attempt  to reference them all, but   defer to Chapter 3 (Bijective Solutions) of Stanley's ``Catalan Numbers'' book \cite{Stanley:2015aa}. In this paper we focus on such bijections. 

We will generalise all Catalan families to structures which are counted by  $p$-Catalan numbers, in particular the numbers
\begin{equation}\label{eq_gencatnumbers}
  C_n(p)=p^{n+1}\,\frac{1}{n+1}\binom{2n}{n} 
\end{equation}
where $p\in\nat$ is a constant which turns out to be the number of generators of the free magma.

The primary motivation of this paper  is to try and bring some structure to the large number of Catalan bijections. 
Once a new family has been bijected to any one of the known Catalan families, then of course, that proves it is a Catalan family, and a composition of the bijections will define a bijection between it and   other Catalan families. 
Whilst mathematically we thus then have bijections  between   pairs of families, it is not entirely satisfactory when bijecting between any pair to have to go via several other families. 
Is it possible to define some `direct' way of bijecting between any pair?

Another motivation is to provide some  sort of answer to the following question: given any two sets with the same cardinality, say $n$, then there exists $n!$ possible bijections between them -- however most of the known Catalan bijections pick out a few of these bijections as somehow more ``natural'' -- what is natural about these bijections?

The central idea presented in this paper is to replace ``bijection'' with ``isomorphism'', in particular, an isomorphism  between   free magmas. 
(For a definition of a magma -- see   \defref{def_magma}.) 

All free magmas with the same number of generators are isomorphic. It will be shown that Catalan families are single generator free  magmas (see \defref{def_free}) and thus all Catalan families are isomorphic. It is this isomorphism that gives us a bijection between any pair of families. Several additional results also follow from this formulation:

\begin{itemize}
    \item 
        The  algorithm defining the map arising from the isomorphism is essentially the same for any pair of free magmas which  leads to the notion  of a   `universal' \footnote{The bijection is called `universal' as it arises from a universal mapping principle.} bijection and hence a universal bijection between any pair of  Catalan families.

 \item 
        The universal bijection is a morphism, thus it respects the product structure of the magmas which turns out to be the ``natural'' recursive structure of Catalan objects \cite{Segner:1761aa}.  Thus the universal bijection is  natural  in that it respects the  recursive structure of Catalan objects.
        
    \item 
        For the `geometrical' families, the isomorphism    gives rise to the idea of an ``embedding'' bijection. 
        This is usually a  geometrical  way of embedding an object from one Catalan family into an object from another Catalan family -- in a sense it  shows how one object is ``hidden'' in another and thus shows manifestly why the two families have the same enumeration. For families which are sequences adding a bijection to  a geometrical family achieves a similar result.

    \item 
        Right (or left) multiplication by the  irreducible element  of a Catalan magma  defines the Narayana statistic which is preserved by the isomorphism. 
        Since the irreducible element usually has a natural interpretation in  any Catalan family the magma formulation   gives a direct identification of the part of the object  in any family carrying the Narayana statistic. 
\end{itemize}

In the last section we discuss the relationship between the Symbolic Method and the free magma representations.

The above ideas can be generalised \cite{Brak:2018aa}  to arbitrary (finite) choices of various $n$-arity maps. In \cite{Brak:2018aa}  the universal bijection is given for  Fibonacci,  Motzkin  \cite{motzkin:1948lr}, Schr\"oder  \cite{sulanke:1998ix} and ternary tree families.

%=============================================
\section{Free magmas}
%=============================================
\label{sec_magma}

In this section we summarise the standard definitions associated with  magmas (see Bourbaki \cite{Bourbaki:1981qy} Chapter 1, \S 1 and \S 7) and state some theorems useful for proving when a magma is associated with Catalan numbers. We are only interested in countable sets and hence all magmas discussed are assumed to be defined on such sets.
\begin{definition}[Magma] \label{def_magma}
    Let $\magma$ be a non-empty countable set called the \textbf{base set}. A \textbf{magma} defined on  $\magma$ is a pair $(\magma,\star)$ where $\st$ is a \textbf{product map}, 
    $$
    \st:\magma\times \magma\to  \magma.
    $$ 
    A \textbf{sub-magma}, $(\magman,\st)$, is a subset $\magman\subseteq \magma$ closed under $\st$. 
    If $(\magman,\bullet)$ is a  magma  then a \textbf{magma morphism} $\theta$ from $\magma$ to $\magman$  is a map $\theta: \magma\to \magman$ such that for all $m,m'\in\magma$, $\theta(m\st m')=\theta(m)\bullet\theta(m')$. 
\end{definition}
The above definition places no constraints on the map $\star$ (ie.\ associative, surjective etc.) however we will only be interested in product maps that satisfy the conditions necessary to be associated with Catalan (or similar)  numbers, in particular as will be shown,  the magma must be free.

A free magma is defined  using the usual universal mapping principle (UMP) for free structures. 
\begin{definition}[Free magma UMP]\label{def_free}
     The magma $(\magma ,\st)$ is free if the following is true: 
    There exists a   set $\mathcal{Y}$ and a map $i:\mathcal{Y}\to \magma $, such that for all magmas $(\magman,\bullet)$ and for all maps $f:\mathcal{Y}\to \magman$ there exists a unique magma morphism $\theta:\magma \to \magman$ such that $f=\theta\circ i$, that is, the   diagram:
\begin{equation}\label{eq_commdiag}
 \begin{tikzcd}
  \magma   \arrow[r,dotted, "\theta"]   & \magman \\
   \mathcal{Y} \arrow[u,"i"] \arrow[ru,"f"'] &                           
\end{tikzcd}   
\end{equation}
is commutative for all $f$ and $(\magman,\bullet)$.
The  image of the set $\mathcal{Y}$ in $\magma $, $X=i(\mathcal{Y})$, will be called the \textbf{set of generators} of $\magma$. 
\end{definition}

Clearly this definition can be used to prove  a magma is free (or not) however, we will provide an equivalent characterisation that is more  convenient  in this combinatorial context. 
In particular, we will prove (\thmref{thm_free}) that  a magma is free if it has a certain unique factorisation property   and possesses a norm (a positive, additive `size' function).

Free structures satisfying the  same universal mapping principle are known to be isomorphic.
We state the proof for the case of free magmas primarily because the proof is constructive and defines, in conjunction with \defref{def_free}, the isomorphism map.
It is this universal isomorphism that is used to define the universal bijection between any pair of free magmas and hence between any pair of Catalan families.

As an example of a free magma we will define a magma  which has arguably the  simplest product definition  and for which it is  straightforward to prove it is free. 
This magma, defined\footnote{In Bourbaki this is \emph{defined} as a free magma. 
Here we take the approach of defining free (\defref{def_free}) and then proving if a given magma is free or not.} in  Bourbaki \cite{Bourbaki:1981qy}, uses Cartesian product constructions  and hence we will call it a Cartesian magma.  
Note, a magma which also has a simple,  geometrically defined, product is that of binary trees \cite{Loday:2012aa}. 
However, to allow for more than one generator the trees  have to be generalised to   binary trees with multi-coloured leaves.

\begin{definition}[Cartesian Magma] \label{def_fm}
    Let $X$ be a non-empty finite set. Define the sequence $W_n(X)$  of sets of nested 2-tuples recursively by:
    \begin{subequations}\label{eq_freemagrec}
    \begin{align}
        W_1(X)&=X\label{eq_fm1}\\
        W_n(X)&=\bigcup_{p=1}^{n-1} W_p(X)\times W_{n-p}(X)\,,\qquad n>1\,\label{eq_fm2}
        \intertext{and}
        \canmag_X &=\bigcup_{n\ge1}W_n(X)\,.
    \end{align}
    \end{subequations}
    Define the product map 
    $
        \ast \,:\, \canmag_X \times \canmag_X \to \canmag_X 
    $
       by 
    \begin{equation}\label{eq_freeprod}
        m_1\ast m_2\mapsto (m_1,m_2)\,.
    \end{equation}
     The pair   $(\canmag_X,\ast)$ will be  called the \textbf{Cartesian magma} generated by $X$.
\end{definition} 

Thus, if  $X=\set{\aeps}$  the base set of the Cartesian magma, $(\canmag_\aeps,\ast)$ begins (sorting by number of generators used):
\begin{equation}\label{eq_cancatexpcar} 
    \aeps,\quad 
    (\aeps,\aeps), 
    \quad 
    (\aeps,(\aeps,\aeps)),
    \quad 
    ((\aeps, \aeps ),\aeps),
\end{equation}
\begin{equation*}
     (\aeps,(\aeps,(\aeps,\aeps))), 
     \quad
     ((\aeps,(\aeps,\aeps) ),\aeps),
    \quad
    (\aeps, ((\aeps,\aeps),\aeps)),  
    \quad
    (((\aeps,\aeps),\aeps ), \aeps), 
    \quad
    ((\aeps,\aeps),(\aeps,\aeps)) \,,\quad \dots
\end{equation*}
If we wish to make the product explicit we can use the left-hand side of \eqref{eq_freeprod}, that is 
\begin{equation}\label{eq_prodrep1} 
    \aeps,\quad \aeps\ast\aeps, \quad \aeps\ast(\aeps\ast\aeps) 
    ,\quad (\aeps\ast\aeps )\ast\aeps\,, \quad  \dots 
\end{equation}
% but note, strictly speaking the element in the base set of $\canmag_\aeps$ corresponding to, say $\aeps\st(\aeps\st\aeps) $ is $ (\aeps,(\aeps,\aeps))$.  
The form \eqref{eq_prodrep1} is also called infix order. 
There are two other common notations for binary products: prefix order and postfix order.  For prefix and postfix order   the brackets can be omitted without any ambiguity.
For example, using prefix notation \eqref{eq_prodrep1} becomes
\begin{equation}\label{eq_prodrep2} 
    \aeps,\quad 
    \ast  \aeps\aeps, \quad 
    \ast\aeps\!  \ast\!  \aeps\aeps,\quad 
    \ast\! \ast\! \aeps\aeps\aeps\,, \quad  \dots 
\end{equation}
or using postfix-order they become 
\begin{equation}\label{eq_prodrep3} 
    \aeps,\quad 
    \aeps\aeps\ast, \quad 
    \aeps\aeps\aeps\!   \ast\! \ast,\quad 
     \aeps \aeps\!  \ast\!  \aeps   \ast \,, \quad \dots \,.
\end{equation}
Thus  $\aeps\aeps\aeps\! \ast\! \ast$, $\ast \aeps\!\ast\!\aeps\aeps$ and $(\aeps\ast(\aeps\ast\aeps))$  are all the \emph{same} element of $\canmag_\aeps$, namely $(\aeps,(\aeps,\aeps))$.  Note, the set of `valid' product expressions can themselves be used as a base set to define additional Catalan families.

We   give a direct proof based on the universal mapping definition of free. 
This proof should be contrasted with  another (much simpler) proof based on \thmref{thm_free} in \secref{sec_factor}.
Any proof that a Catalan magma is free based on \defref{def_free} would be essentially the same as the proof of \propref{prop_cartfree}.
%=
\begin{prop}\label{prop_cartfree}
    The Cartesian magma $(\canmag_X,\ast)$  is a free magma generated by $X$.
\end{prop}
%=
\begin{proof} 
Let $Y$ be any set such that $|Y|=|X|$. 
Let $i: Y\to \canmag_X$ be a bijection from $Y$ to $X$. 
Let $(N,\bullet)$ be any magma and $f: Y\to N$ any map. 
Then, define a map $\theta: \magma\to N$ using: 
i) for all $y\in Y$, $\theta(i(y))=f(y)$ (this defines $\theta$ on $X$) and ii) for all $m\in \magma\setminus X$ recursively define $\theta(m)=\theta(m_1)\bullet \theta(m_2)$ where $m=(m_1,m_2)$. 
This uniquely defines $\theta$  as the factorisation $m=(m_1,m_2)$ is unique. 
From ii) we have $\theta(m_1\ast  m_2)=\theta(m_1)\bullet \theta(m_1)$. Thus $\theta$ is a uniquely defined magma morphism. 
Thus by \defref{def_free} the Cartesian magma, $(\canmag_X,\ast)$, is free.
\end{proof}

%==============================================
\section{Normed  magmas and Catalan numbers}
%==============================================

In this section we introduce a `size' map which we will call a norm. This will partition the base set of any free magma into subsets counted by  free magma  $p$-Catalan numbers,  \eqref{eq_gencatnumbers}. 
The norm, in conjunction with factorisation, will also give us an alternative characterisation of free magmas.

\begin{definition}[Norm] \label{def_norm}
    Let $(\magma,\star)$ be a magma.  A  \textbf{\valns} is  a super-additive map   $\norm{\cdot}  : \magma\to \nat$.
    If  $(\magma,\star)$ has a \val it will be called a \textbf{\vald  magma}.
 \end{definition}   
A norm is \textbf{super-additive} if it satisfies:  
$
   \text{for all $m_1,m_2\in\magma$},\,\norm{m_1\st m_2}\ge \norm{m_1}+\norm{m_2} 
$ and called \textbf{additive} if the equality always holds.

For example,  the conventional norm defined on a Cartesian magma is   the map $\norm{\cdot}\,:\, \canmag_X\to \nat$  defined by 
\begin{equation}\label{eq_cartnorm}
     \norm{m}=n\quad \text{when}\quad  m\in  W_n(X) . 
\end{equation}
With this definition all generators have norm one and the norm of other elements $m$ correspond  to the number of occurrences of the generators in the nested $2$-tuple form of $m$. From \eqref{eq_fm2} it is clear the norm is additive.

The connection with $p$-Catalan numbers is given in the following proposition. 

%=
\begin{prop} \label{thm_cat}
 Let $(\canmag_X,\ast)$  be the Cartesian magma of \defref{def_fm} with norm \eqref{eq_cartnorm}. If 
\[
    W_\ell=\{m\in \canmag_\aeps : \norm{m}=\ell \},\qquad \ell \ge 1\,,
\]
then the size of $W_\ell$ is,
\begin{equation}\label{eq_catnum2}
    |W_\ell|=|X|^\ell\, C_{\ell-1}=|X|^\ell\,\, \frac{1}{\ell}\binom{2\ell-2}{\ell-1}=C_{\ell-1}(p)\,.
\end{equation}

\end{prop}
%=
We will refer to $|X|^\ell\, C_{\ell-1}$ as a \textbf{free magma} $p$-\textbf{Catalan number} with $p=|X|$ and $\ell\ge1$.
The case $|X|=1$ is well known \cite{Segner:1761aa} and gives  the Catalan numbers \eqref{eq_catnumbers}.
The proof of \eqref{eq_catnum2} is a trivial generalisation of the conventional $|X|=1$ proof  which uses the recurrence 
\begin{equation}\label{eq_seg}
    |W_n| = \sum_{k=1}^{n-1} |W_k|\,|W_{n-k}|,\qquad n > 1
\end{equation}
which follows from equation \eqref{eq_freemagrec}, and then verifying  the Catalan number  \eqref{eq_catnumbers} satisfies the same recurrence with $|W_1|=1$. 
This is  usually done by using the generating function for Catalan numbers. For $|X|>1$   the same recurrence \eqref{eq_seg} holds but the initial value changes to $|W_1|=|X|$ which results in a factor   of $|X|^\ell$ in  \eqref{eq_catnum2}.

Note, unlike the conventional form of the Catalan recurrence, when using the \valns\ there is no shift of $+1$ in the subscripts of \eqref{eq_seg}. Result \eqref{eq_catnum2} shows that magmas generated by more than one generator are combinatorially a simple extension of the Catalan case (sequences A025225 ($|X|=2$),  and A025226 ($|X|=3$) in \cite{A02522X:aa}).

For each Catalan family, the \val is usually a simple function of the conventional Catalan  objects' ``size''  parameter   eg.\ for triangulations of a regular $n$ sided polygon (see  family   \fref{fam_pt} in the Appendix) the size  is usually the number of triangles ($=n-2$) or polygon edges ($=n$), but   neither size parameter is   additive over the  product. 
To achieve additivity the \val of the triangulation has to be $n-1$. For Dyck paths with $2n$ steps (see  family   \fref{fam_dp} in the Appendix) the norm is $n+1$.

We conclude this section with two examples of single generator Catalan families for which the recursive structure is well known:   Dyck paths  and  triangulations of regular $n$-sided polygons. We defer the proof that both families are free magmas until we have the method stated in \thmref{thm_free} of \secref{sec_factor}.

\medskip
\noindent\textit{Dyck Paths } (see \fref{fam_dp} in the Appendix). The following is the list of all of Dyck paths with six or less steps:
% \begin{center}
%     \includegraphics[scale=1]{figs/dyckPathExamples_B_lr.pdf}  
% \end{center}
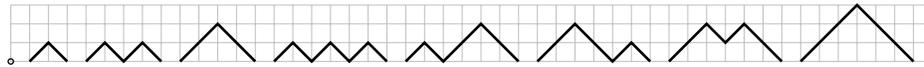
\begin{center}
    \def\up{-- ++(1,1)}
    \def\dn{-- ++(1,-1)}
    \begin{tikzpicture}[scale=0.25,line width=1pt]
        \draw [help lines, gray!50] (0,0) grid(49,3);
        \draw[line width=0.5pt]  (0,0)    circle (4pt);    
        \draw  (1,0)     \up   \dn  ;
        \draw  (4,0)     \up   \dn \up   \dn;    
        \draw  (9,0)     \up   \up \dn   \dn; 
        \draw  (14,0)     \up   \dn \up   \dn \up \dn;   
        \draw  (21,0)     \up   \dn \up   \up \dn \dn;  
        \draw  (28,0)     \up   \up \dn   \dn \up \dn;  
        \draw  (35,0)     \up   \up \dn   \up \dn \dn;  
        \draw  (42,0)     \up   \up \up   \dn \dn \dn;  
    \end{tikzpicture}
\end{center}
Note, the single vertex is considered a zero step path.
The Dyck path magma generator  is
\[
\eps_\dpn=\tikz\draw[line width=0.5pt]  (0,0)    circle (2pt);   
\]
where $\circ$ represents a single vertex or, if preferred, the empty path. 
The \val of the generator is defined to be one, and the \val of a path $d$ is defined as    $\norm{d}=(\text{number of up steps})+1$.

The product $d_1\st_\dpn d_2$ of two paths $  d_1 $ and $ d_2$ is illustrated schematically as 
\begin{equation}\label{eq_dyck_prod}
    \begin{tikzpicture}[line width=0.75pt,scale=1] 
        \fill [color_left,draw=black] (0,0) arc [radius=0.75,start angle = 180, end angle =0] ;
        \draw (0,0) -- (1.5,0);
        \end{tikzpicture}
    \quad\star_\dpn\quad 
    \begin{tikzpicture}[line width=0.75pt,scale=1] 
        \fill [color_right,draw=black] (0,0) arc [radius=0.75,start angle = 180, end angle =0] ;
        \draw (0,0) -- (1.5,0);
    \end{tikzpicture}
    \quad=\quad
    \begin{tikzpicture}[line width=0.75pt,scale=1] 
        \fill [color_left,draw=black] (0,0) arc [radius=0.75,start angle = 180, end angle =0] ;
        \draw (0,0) -- (1.5,0);
        \draw[color=orange, line width=1pt] (1.5,0) -- (2,0.5);
        \fill [color_right,draw=black] (2,0.5) arc [radius=0.75,start angle = 180, end angle =0] ;
        \draw (2,0.5) -- (3.5,0.5);
        \draw[color=orange, line width=1pt] (3.5,0.5) -- (4,0);
    \end{tikzpicture}
\end{equation}
%
% \begin{equation}\label{eq_dyck_prod}
%     \includegraphics[scale=\fscale]{figs/DyckPathProd_lr.pdf}  
% \end{equation}
Thus the  product takes the left path $d_1$, appends  an up step on its right, then appends the right path and finally appends a down step. The different colour of the the two added steps is for clarity only.
Note, this is the standard (right) factorisation of Dyck paths but now used to define  a magma product \footnote{The left factorisation can be used to define a related (same base set), but different, magma -- see \secref{sec_opprev}.}.
For example, the product of  two generators is
\begin{equation*}
      \text{ \tikz \draw[line width=0.5pt] (0,0) circle (1.5pt); } 
      \star_\dpn
      \text{ \tikz \draw[line width=0.5pt] (0,0) circle (1.5pt); }
      =
      \begin{tikzpicture}[scale=0.4]
        \draw [help lines, gray!50] (0,0) grid(2,1);
        \draw[line width=1pt]  (0,0)   -- (1,1) -- (2,0);  
        \path[draw,fill=white] (0,0) circle (4pt);
         \path[draw,fill=white] (1,1) circle (4pt);
      \end{tikzpicture}
\end{equation*}
% \begin{center}
%  \includegraphics[scale=\fscale]{figs/dyckProductExample_B_lr.pdf}
% \end{center}
-- see \fref{fam_dp} in the appendix for more examples.
Note, the sum of the \vals of the left two paths ($=1+1$)  equals that of the product path. 

\medskip
\noindent\textit{Triangulations of an $n$-gon (see \fref{fam_pt} in the Appendix).}   
  The second Catalan family is the triangulation of regular polygons.  The following is the list of all triangulations with three or less triangles:
  \begin{equation*}
    \def\pa{72}
    \def\lv{54}
    \def\rd{4pt}
    \def\sc{0.4}
    \def\ro{0.7} %two
    \def\rt{1.0} % three
    \def\rf{1.27} %four
    \def\rv{1.2} %five
    %
    %two
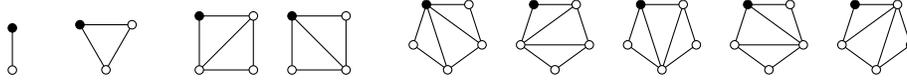
    \begin{tikzpicture}[scale=\sc]
        \draw (90:\ro) -- (270:\ro);
        \draw[fill=white] (270:\ro) circle (\rd);
        \draw[fill=black] (90:\ro) circle (\rd);
    \end{tikzpicture}\quad\quad
    %
    % three
    \begin{tikzpicture}[scale=\sc]
    \draw (30:\rt) -- (150:\rt) -- (-90:\rt) -- cycle;
        \draw[fill=white] (30:\rt) circle (\rd);
        \draw[fill=black] (150:\rt) circle (\rd);
        \draw[fill=white] (-90:\rt) circle (\rd);
    \end{tikzpicture} \quad\quad
    %
    % four
    \begin{tikzpicture}[scale=\sc]
        \draw (45:\rf) -- (135:\rf) -- (225:\rf) -- (315:\rf) -- cycle;
        \draw (45:\rf) -- (225:\rf);
        \draw[fill=white] (45:\rf) circle (\rd);
        \draw[fill=black] (135:\rf) circle (\rd);
        \draw[fill=white] (225:\rf) circle (\rd);
        \draw[fill=white] (315:\rf) circle (\rd);
    \end{tikzpicture}\quad
    \begin{tikzpicture}[scale=\sc]
        \draw (45:\rf) -- (135:\rf) -- (225:\rf) -- (315:\rf) -- cycle;
        \draw (135:\rf) -- (315:\rf);        
        \draw[fill=white] (45:\rf) circle (\rd);
        \draw[fill=black] (135:\rf) circle (\rd);
        \draw[fill=white] (225:\rf) circle (\rd);
        \draw[fill=white] (315:\rf) circle (\rd);
    \end{tikzpicture}\quad\quad
    %
    % five
    \begin{tikzpicture}[scale=\sc]
        \draw (\lv:\rv) -- ( \lv+\pa:\rv) -- (\lv+2*\pa:\rv) -- (\lv+3*\pa:\rv) -- (\lv+4*\pa:\rv)  -- cycle;
        \draw (\lv+6*\pa:\rv) -- (\lv+ 8*\pa:\rv);
        \draw (\lv+6*\pa:\rv) -- (\lv+ 9*\pa:\rv);
        \draw[fill=white] (\lv:\rv) circle (\rd);
        \draw[fill=black] (\lv+\pa:\rv) circle (\rd);
        \draw[fill=white] (\lv+2*\pa:\rv) circle (\rd);
        \draw[fill=white] (\lv+3*\pa:\rv) circle (\rd);
        \draw[fill=white] (\lv+4*\pa:\rv) circle (\rd);
    \end{tikzpicture}    \quad
    \begin{tikzpicture}[scale=\sc]
        \draw (\lv:\rv) -- ( \lv+\pa:\rv) -- (\lv+2*\pa:\rv) -- (\lv+3*\pa:\rv) -- (\lv+4*\pa:\rv)  -- cycle;
        \draw (\lv+2*\pa:\rv) -- (\lv+ 4*\pa:\rv);
        \draw (\lv+2*\pa:\rv) -- (\lv+ 5*\pa:\rv);
        \draw[fill=white] (\lv:\rv) circle (\rd);
        \draw[fill=black] (\lv+\pa:\rv) circle (\rd);
        \draw[fill=white] (\lv+2*\pa:\rv) circle (\rd);
        \draw[fill=white] (\lv+3*\pa:\rv) circle (\rd);
        \draw[fill=white] (\lv+4*\pa:\rv) circle (\rd);
    \end{tikzpicture}\quad
    \begin{tikzpicture}[scale=\sc]
        \draw (\lv:\rv) -- ( \lv+\pa:\rv) -- (\lv+2*\pa:\rv) -- (\lv+3*\pa:\rv) -- (\lv+4*\pa:\rv)  -- cycle;
        \draw (\lv+3*\pa:\rv) -- (\lv+ 5*\pa:\rv);
        \draw (\lv+3*\pa:\rv) -- (\lv+ 6*\pa:\rv);
        \draw[fill=white] (\lv:\rv) circle (\rd);
        \draw[fill=black] (\lv+\pa:\rv) circle (\rd);
        \draw[fill=white] (\lv+2*\pa:\rv) circle (\rd);
        \draw[fill=white] (\lv+3*\pa:\rv) circle (\rd);
        \draw[fill=white] (\lv+4*\pa:\rv) circle (\rd);
    \end{tikzpicture}\quad
    \begin{tikzpicture}[scale=\sc]
        \draw (\lv:\rv) -- ( \lv+\pa:\rv) -- (\lv+2*\pa:\rv) -- (\lv+3*\pa:\rv) -- (\lv+4*\pa:\rv)  -- cycle;
        \draw (\lv+4*\pa:\rv) -- (\lv+ 6*\pa:\rv);
        \draw (\lv+4*\pa:\rv) -- (\lv+ 7*\pa:\rv);
        \draw[fill=white] (\lv:\rv) circle (\rd);
        \draw[fill=black] (\lv+\pa:\rv) circle (\rd);
        \draw[fill=white] (\lv+2*\pa:\rv) circle (\rd);
        \draw[fill=white] (\lv+3*\pa:\rv) circle (\rd);
        \draw[fill=white] (\lv+4*\pa:\rv) circle (\rd);
    \end{tikzpicture}\quad
    \begin{tikzpicture}[scale=\sc]
        \draw (\lv:\rv) -- ( \lv+\pa:\rv) -- (\lv+2*\pa:\rv) -- (\lv+3*\pa:\rv) -- (\lv+4*\pa:\rv)  -- cycle;
        \draw (\lv+5*\pa:\rv) -- (\lv+ 7*\pa:\rv);
        \draw (\lv+5*\pa:\rv) -- (\lv+ 8*\pa:\rv);
        \draw[fill=white] (\lv:\rv) circle (\rd);
        \draw[fill=black] (\lv+\pa:\rv) circle (\rd);
        \draw[fill=white] (\lv+2*\pa:\rv) circle (\rd);
        \draw[fill=white] (\lv+3*\pa:\rv) circle (\rd);
        \draw[fill=white] (\lv+4*\pa:\rv) circle (\rd);
    \end{tikzpicture}\quad
\end{equation*}
% \begin{center}
%     \includegraphics[scale=\fscale]{figs/triangulations_Examples4_lr.pdf}
% \end{center}
Note, the polygons are drawn so that the polygon has a    marked vertex in a fixed position and the single edge object is considered a `$2$-gon' triangulation with zero triangles.

We can construct a magma from triangulations as follows.  
The generator, $\eps_\ptn$ is the above `$2$-gon' (with the top vertex marked), denoted 
\[
\eps_\ptn=
\begin{tikzpicture}[scale=1,baseline=-3pt]
    \def\ro{0.2} 
    \draw (90:\ro) -- (270:\ro);
    \draw[fill=white] (270:\ro) circle (2pt);
    \draw[fill=black] (90:\ro) circle (2pt);
\end{tikzpicture}
%\raisebox{-0.4\height}{\includegraphics[scale=\fscale]{figs/triagulationGen_B_lr.pdf}}
\]
which is defined to have  \val one. 
The \val of a triangulation of a regular $n$-gon for $n>2$ is   $n-1$ or $ (\text{number of triangles})+1$.

We represent the product\footnote{This product appears in Conway and Coxeter \cite{Conway:1973aa} as a way of constructing frieze patterns and in Brown \cite{Brown:1964aa}}, $t_1\st_\ptn t_2$, of  two triangulations $t_1$ and $t_2$  schematically  as
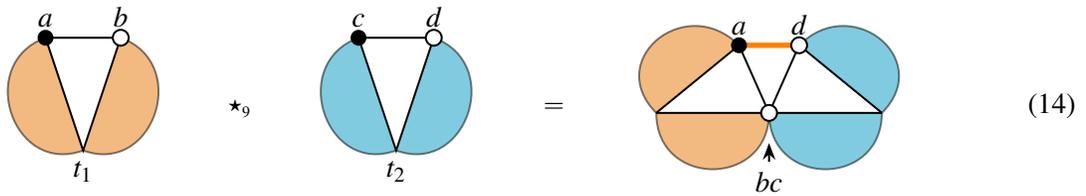
\begin{equation}\label{triangulations}
\def\BL{0.5cm}
    \begin{tikzpicture}[line width=0.75pt,scale=0.5,baseline=\BL,>=Stealth] 
    \def\A{(1,3)} \def\B{(2,0)} \def\C{(3,3)}
    \def\L{200} \def\R{-20}
        %  \draw [help lines, gray!50] (0,0) grid(6,3);
        \fill [color_left,draw=black] \A .. controls +(\L:2cm) and +(\L:2cm) .. \B;
        \fill [color_left,draw=black]  \B .. controls +(\R:2cm) and +(\R:2cm) .. \C;
        \draw \B -- \A -- \C -- cycle;
        \fill [black] \A circle (6pt) ;
        \node [anchor=south] at \A {$a$};
        \fill [white,draw=black] \C circle (6pt);
        \node [anchor=south] at \C {$b$};
        \node [anchor=north] at \B {$t_1$};
    \end{tikzpicture}
    \quad\star_\ptn\quad 
    \begin{tikzpicture}[line width=0.75pt,scale=0.5,baseline=\BL] 
    \def\A{(1,3)} \def\B{(2,0)} \def\C{(3,3)}
    \def\L{200} \def\R{-20}
        %  \draw [help lines, gray!50] (0,0) grid(6,3);
        \fill [color_right,draw=black] \A .. controls +(\L:2cm) and +(\L:2cm) .. \B;
        \fill [color_right,draw=black]  \B .. controls +(\R:2cm) and +(\R:2cm) .. \C;
        \draw \B -- \A -- \C -- cycle;
        \fill [black] \A circle (6pt) ;
        \node [anchor=south] at \A {$c$};
        \fill [white,draw=black] \C circle (6pt);
        \node [anchor=south] at \C {$d$};
        \node [anchor=north] at \B {$t_2$};
    \end{tikzpicture}
    \quad = \quad
        \begin{tikzpicture}[line width=0.75pt,scale=0.5,baseline=\BL,>=Stealth] 
        \def\A{(2.2,2.8)} \def\B{(0,1)} \def\C{(3,1)} \def\D{(6,1)}\def\E{(3.8,2.8)}
        \def\F{(3,-0.3)} \def\G{(3,0.2)}
        % \draw [help lines, gray!50] (0,0) grid(6,3);
        \fill [color_left,draw=black] \A .. controls +(135:2cm) and +(135:2cm) .. \B;
        \fill [color_left,draw=black]  \B .. controls +(270:2cm) and +(270:2cm) .. \C;
        \draw \B-- \A; \draw \B -- \C;
        \draw \C -- \A;
        \draw [orange,line width = 2pt] \A -- \E;
        \fill [color_right,draw=black] \C .. controls +(270:2cm) and +(270:2cm) .. \D;
        \fill [color_right,draw=black]  \D .. controls +(45:2cm) and +(45:2cm) .. \E;  
        \draw \E-- \D; \draw \D -- \C;
        \draw \C -- \E;
        \fill [black] \A circle (6pt) ;
        \node [anchor=south] at \A {$a$};
        \fill [white,draw=black] \E circle (6pt);
        \node [anchor=south] at \E {$d$};
        \fill [white,draw=black] \C circle (6pt);
        \node [anchor=north] at \F{$bc$};
        \draw [->] \F -- \G;

    \end{tikzpicture}
\end{equation}
% \begin{equation}\label{triangulations}
%     \eqfig{\fscale}{triangulations_lr.pdf}
% \end{equation} 
which represents how  the vertices $a$ and  $d$ are joined by a new edge (shown orange) and the vertices $b$ and $c$ are combined into a single vertex.
The vertex $a$ remains the only marked vertex.

For example, the product  of two generators  is 
\begin{equation*}
\def\BL{-3pt}  \def\rd{2pt}\def\ro{0.9} \def\rt{0.35}  
\begin{tikzpicture}[ scale=1,baseline=\BL] 
    \def\ro{0.2} 
    \draw (90:\ro) -- (270:\ro);
    \draw[fill=white] (270:\ro) circle (2pt);
    \draw[fill=black] (90:\ro) circle (2pt);
\end{tikzpicture}
\quad\star_\ptn \quad
\begin{tikzpicture}[ scale=1,baseline=\BL] 
    \def\ro{0.2} 
    \draw (90:\ro) -- (270:\ro);
    \draw[fill=white] (270:\ro) circle (\rd);
    \draw[fill=black] (90:\ro) circle (\rd);
\end{tikzpicture}
\quad = \quad
\begin{tikzpicture}[ scale=1,baseline=\BL] 
    \draw (30:\rt) -- (150:\rt) -- (-90:\rt) -- cycle;
    \draw[fill=white] (30:\rt) circle (\rd);
    \draw[fill=black] (150:\rt) circle (\rd);
    \draw[fill=white] (-90:\rt) circle (\rd);
\end{tikzpicture}
\end{equation*}
% \begin{center}
%   \includegraphics[scale=\fscale]{figs/triangulationProdExprod22_lr.pdf}
% \end{center}
and the product $(\eps_\ptn\st_\ptn\eps_\ptn)\st_\ptn (\eps_\ptn\st_\ptn\eps_\ptn) $ is
\begin{equation*}
\def\BL{-3pt}  \def\rd{2pt}\def\ro{0.9} \def\rt{0.35}  
\begin{tikzpicture}[ scale=1,baseline=\BL] 
    \draw (30:\rt) -- (150:\rt) -- (-90:\rt) -- cycle;
    \draw[fill=white] (30:\rt) circle (\rd);
    \draw[fill=black] (150:\rt) circle (\rd);
    \draw[fill=white] (-90:\rt) circle (\rd);
\end{tikzpicture}
\quad\star_\ptn \quad
\begin{tikzpicture}[ scale=1,baseline=\BL] 
    \draw (30:\rt) -- (150:\rt) -- (-90:\rt) -- cycle;
    \draw[fill=white] (30:\rt) circle (\rd);
    \draw[fill=black] (150:\rt) circle (\rd);
    \draw[fill=white] (-90:\rt) circle (\rd);
\end{tikzpicture}
\quad = \quad
\begin{tikzpicture}[ scale=0.5,baseline= 6pt] 
\def\A{(1.2,1)} \def\B{(0.5,0)} \def\C{(2,-0.2)} \def\D{(3.5,0)} \def\E{(2.8,1)} 
\def\rd{4pt}
 \draw [orange] \A -- \E;
    \draw \A -- \B -- \C -- cycle;
    \draw \C -- \D -- \E -- cycle;
    \draw[fill=white] \B circle (\rd);
     \draw[fill=white] \B circle (\rd);
      \draw[fill=white] \C circle (\rd);
       \draw[fill=white] \D circle (\rd);
        \draw[fill=white] \E circle (\rd);
   
    \draw[fill=black] \A circle (\rd);
\end{tikzpicture}
\quad \leadsto \quad
    \def\pa{72} \def\lv{54} \def\rd{4pt} \def\sc{0.5}
    \def\ro{0.9}  \def\rt{1.2}  \def\rf{1.27}  \def\rv{1}  
    \begin{tikzpicture}[scale=\sc,baseline= 0pt]
        \draw (\lv:\rv) -- ( \lv+\pa:\rv) -- (\lv+2*\pa:\rv) -- (\lv+3*\pa:\rv) -- (\lv+4*\pa:\rv)  -- cycle;
        \draw (\lv+3*\pa:\rv) -- (\lv+ 5*\pa:\rv);
        \draw (\lv+3*\pa:\rv) -- (\lv+ 6*\pa:\rv);
        \draw[fill=white] (\lv:\rv) circle (\rd);
        \draw[fill=black] (\lv+\pa:\rv) circle (\rd);
        \draw[fill=white] (\lv+2*\pa:\rv) circle (\rd);
        \draw[fill=white] (\lv+3*\pa:\rv) circle (\rd);
        \draw[fill=white] (\lv+4*\pa:\rv) circle (\rd);
    \end{tikzpicture}
\end{equation*}
% \begin{center}
%   \includegraphics[scale=\fscale]{figs/triangulationProdExprod33_lr.pdf} 
% \end{center}
where the product representation is shown in more conventional form in the rightmost diagram (see \fref{fam_pt} in the Appendix for more examples).

For some families identifying the geometry (or symbols) that represent the generator is simple (eg.\ the vertex for Dyck paths) whilst for others  less so (usually when the generator is the `empty' object).  
For the latter case the following equivalent process defines the product. Let $\ffm^+=\ffm\setminus \set{\eps}$, then do the following:
\begin{enumerate} 
	
\item Define the object: 	$\eps\st\eps$ 
\item  For all $f\in \ffm^+  $ define the object: $\eps \st f$
\item  For all $f\in \ffm^+  $ define the object: $f\st \eps$
\item  For all $f_1,f_2\in \ffm^+  $ define the object: $  f_1\st f_2 $.
 
\end{enumerate}
An alternative strategy is to add some auxiliary geometry or mark to the objects which is uniquely associated with the generator. 
This geometry is not strictly part of the the  standard definition of the set of objects but used to clarify the magma structure. 
It can also be used when considering the embedding bijections discussed in \secref{sec_embedd}.

%==============================================
\section{Factorisation in magmas}
%==============================================
\label{sec_factor}

In this section we consider the general problem of factorisation in magmas and how it is related to any norm defined on the magma. 
The existence of a norm will imply any recursive factorisation of an element of a magma will terminate. 
The latter property, in conjunction with the injectivity of the product,  will be used in the proof of \thmref{thm_free} which provides  an equivalent characterisation of a free magma. 
In this combinatorial context \thmref{thm_free} provides a more direct way of determining if a magma is free or not.
To this end we  make the following definition.
\begin{definition}[Irreducible elements, unique factorisation] \label{def_irred}
Let  $(\magma,\star)$ be a magma with product map
$\st:\magma\times \magma\to  \magma $.
The range of the map $\st$ is called the set of \textbf{reducible elements} and usually denoted by $\magp$. 
The elements of the set $\mprime= \magma\setminus\magp$ are called  \textbf{irreducible elements} and the set $\mprime$ called the \textbf{set of irreducibles}.   
If  the product map is injective then we will call the magma a \textbf{unique factorisation magma}.
\end{definition}
Note, there may be no irreducible elements, also if $\magma$ is finite then the product cannot be injective. 
The proof of \thmref{thm_free} will require an injective product map. 
Thus, to simplify this section we will restrict the discussion to unique factorisation magmas.

Consider an arbitrary unique factorisation magma   $(\magma,\st)$. 
For all reducible elements  $m\in \magp$ there exists exactly one factorisation  into  two  parts: $m= m_1\st m_2 $.
Call $m_1$ a \textbf{left factor} and $m_2$ a \textbf{right   factor}.  
Define $\pdni$  recursively as 
\begin{equation}\label{eq_ffin}
\pdni(m)=\begin{cases}
\bigl(\pdni(m_1),\pdni(m_2)\bigr) & \text{if $m\in \magp$ with $ m_1\st m_2 = m $ }\\ 
   m    & \text{if $m \in \mprime$. }
\end{cases}
\end{equation}

Of primary interest is when the recursion   \eqref{eq_ffin} terminates. 
Note, if \eqref{eq_ffin} terminates then it  defines a map of $\magma$ into the Cartesian magma's base set.

This motivates the following definition.
\begin{definition}[Finite decomposition]
    Let $(\magma,\st)$ be a unique factorisation  magma.  
    If for all elements,   $m\in \magp$,  the recursion \eqref{eq_ffin} terminates  then $(\magma,\st)$ will be called a \textbf{finite decomposition} magma.
    If $(\magma,\st)$ is a finite decomposition magma then
       $\pdni(m)$ will be called the \textbf{decomposition} of $m$. 
\end{definition}

\newcommand{\dd}{\iddots}
\begin{figure}[ht]
    \centering
    \begin{subfigure}[b]{0.3\textwidth}
        \centering
        \begingroup\makeatletter\def\f@size{6}\check@mathfonts
        \begin{equation*}
            \begin{array}{c|cccccc}
                \st  &1  &  2 &  3 &  4&  5&   \dots \\
                \hline 
                1  &5  &  7 &  10 &  3&  16&  22   \\
                2  &6  &  9 &  4 &  15&  21&    \dd\\
                3  &8  &  4 &  14 &  20&  27 &    \dd  \\
                4  &11  &  13 &  19 & 26  &   &    \dd \\
                5  &12  &  18 &  25  &   &   &    \dd  \\
                \vdots   &17  &  24 &  \dd   &  \dd  &  \dd  &   \dd 
            \end{array} 
        \end{equation*}
        \endgroup
    \caption{  i) Not a finite decomposition magma.    ii) Not  a unique factorisation magma.   iii) Two irreducible elements. \\ }
    \end{subfigure}
    \hfill
    \begin{subfigure}[b]{0.3\textwidth}
        \centering
        \begingroup\makeatletter\def\f@size{6}\check@mathfonts
        \begin{equation*}
            \begin{array}{c|cccccc}
                \st  &1  &  2 &  3 &  4&  5&   \dots \\
                \hline 
                1  &3  &  4 &  7 &  11&  16&  22   \\
                2  &5  &  1 &  10 &  15&  21&    \dd\\
                3  &6  &  9 &  14 &  20&  27 &    \dd  \\
                4  &3  &  13 &  19 & 26  &   &    \dd \\
                5  &12  &  18 &  25  &   &   &    \dd  \\
                \vdots   &17  &  24 &  \dd   &  \dd  &  \dd  &   \dd 
            \end{array} 
        \end{equation*}
        \endgroup
    \caption{   i) A finite decomposition magma.   ii)  Not  a unique factorisation magma.   ii) Two irreducible elements.\\ }
    \end{subfigure}
    \hfill
    \begin{subfigure}[b]{0.3\textwidth}
        \centering
        \begingroup\makeatletter\def\f@size{6}\check@mathfonts
        \begin{equation*}
            \begin{array}{c|cccccc}
                \st  &1  &  2 &  3 &  4&  5&   \dots \\
                \hline 
                1  &3  &  4 &  7 &  11&  16&  22   \\
                2  &5  &  1 &  10 &  15&  21&    \dd\\
                3  &6  &  9 &  14 &  20&  27 &    \dd  \\
                4  &8  &  13 &  19 & 26  &   &    \dd \\
                5  &12  &  18 &  25  &   &   &    \dd  \\
                \vdots   &17  &  24 &  \dd   &  \dd  &  \dd  &   \dd 
            \end{array} 
        \end{equation*}
        \endgroup
    \caption{  i) A finite decomposition magma.   ii) A unique factorisation magma.  iii)  One irreducible element. }
    \label{fig:y equals x}
    \end{subfigure}
    \caption{The magmas above have base set $\mathbb{N}$ and product rule defined by the table. Only a portion of the table is shown explicitly.  The rest of the table is filled by the remaining elements of $\mathbb{N}$ consecutively in the diagonal pattern illustrated (ie.\ the north-east diagonals  starting with  $12$, $13$, $14, \cdots$.  }
    \label{fig:magmaEx}
\end{figure}

The recursion terminates  if all factors always result in an   irreducible element after a finite number of iterations of \eqref{eq_ffin}.  
There are several  mechanisms that would result in the recursion \eqref{eq_ffin} not terminating. 
For example, if there are no irreducible elements, or, if at some stage of the recursion a factor of some element $m$  is $m$ itself 
ie.\ $\pdni(m) =\dots (\pdni(m) , \pdni(m') )\dots $ resulting in an infinitely nested tuple.  Three example magmas, showing different factorisation properties, are given in \figref{fig:magmaEx}.
Another example is the  Cartesian magma, $(\canmag_X,\ast)$, for which it is clear it  is a unique factorisation magma and that the generators (ie.\ the elements of $X$) are the only irreducible elements.

If $(\magma,\st)$ has a \val   then the recursion must terminate  as given by the following proposition.
\begin{prop}\label{thm_fff}
   Let   $(\magma,\st)$  be a unique factorisation magma with non-empty set of irreducibles. 
   Then  $(\magma,\st)$ is a \vald magma  if and only if   $(\magma,\st)$ is a finite  decomposition magma. 
\end{prop}
This proposition is in fact true for an arbitrary (ie.\ not necessarily a unique factorisation) magma  but since we do not require this more general result we dispense with the more notationally complex proof (the right-hand side of \eqref{eq_ffin} has to be replaced by the set of all pairs whose product is $m$).
To prove   \propref{thm_fff} we will use the following lemma:
\begin{lem}\label{thm_lenlem}
    Let $(\magma,\st)$ be a \vald magma  and  
    let  $X_{\min}\subset\magma$ be the set of elements with minimal \valns.  
    Then $X_{\min}$ is non-empty and all elements of $X_{\min}$ are irreducible.
\end{lem}
\begin{proof}[Proof of Lemma]
    The base set $\magma$ is assumed non-empty  and hence the \val has non-empty range $A\subseteq\mathbb{N}$. 
    Since   $\mathbb{N}$ is  a well ordered set the subset $A$ has a least element, $r$, and thus the  pre-image set of $r$,  $X_{\min}$, is not empty.
    Assume $m\in\magp\cap X_{\min}$. 
    Since $m\in\magp$ it has a pre-image pair  $(m_1,m_2)$  ie.\ $\st:(m_1,m_2)\mapsto m$, and  by assumption $m\in X_{\min}$ hence $\norm{m}$ is minimal. 
    But $\norm{m}\ge  \norm{m_1}+\norm{m_2}$ with $\norm{m_1},\norm{m_2}>0$  so   $\norm{m}$ is not minimal --  a contradiction. Thus $\magp\cap X_{\min}$ is empty. 
\end{proof} 
Note, the converse is false: a normed magma may have irreducible elements that do not have minimal norm. This lemma also shows why the norm has to be a positive integer.

\begin{proof}[Proof of \propref{thm_fff}]

    \textit{Forward:} If $(\magma,\st)$ has a \valns, by  \lemref{thm_lenlem}, the set of irreducibles, $\mprime $   is non-empty and all elements of minimal \val are in $\mprime$ . 
    Let $R$ be the range of the \valns. 
    Since $\mathbb{N}$ is well ordered the subset $R$ has a minimal element, say $r_0$.
    If $m\in \mprime$ the recursion terminates. 
    If $m\not\in \mprime$ then     $m=m_1 \star m_2$ (for unique $m_1$ and $m_2$) and $\norm{m}=\norm{m_1\star m_2}\ge \norm{m_1} +\norm{m_2}$ and  since the \val takes values in $\mathbb{N}$ we have that $\norm{m_1}<\norm{m_1\star m_2}$ and $\norm{m_2}<\norm{m_1\star m_2}$. 
    Thus every factor of $m$ has strictly smaller \val than $m$. 
    This process continues until a factor is an element of $\mprime$ (in which case the recursion terminates) or the \val of the factor is $r_0$. 
    In the latter case, by \lemref{thm_lenlem} the factor must be in $\mprime$ (and hence the recursion terminates). Thus in all cases the recursion terminates and hence $(\magma,\st)$ is a finite  decomposition magma.

    \textit{Converse:} If $(\magma,\st)$ is a finite  decomposition magma then we can construct a norm as follows. 
    Define a map $\norm{\cdot}: \magma\to\mathbb{N}$ as follows: 
    i) If $m\in \mprime$ then $\norm{m}=1$ and ii) if   $m= m_1\st m_2$ then  by definition, $ \norm{m_1\st m_2}=\norm{m_1}+\norm{m_2} $. Thus $\norm{\cdot}\in\nat$, is additive and is finite (since the decomposition of $m$ is finite). 
    Thus $\norm{\cdot}$ is a norm  and thus $(\magma,\st)$ is a \vald  magma.

\end{proof}

We can now state an  equivalent characterisation of free magmas.
 \begin{theorem}\label{thm_free}
    If   $(\magma,\st)$  is a unique factorisation \vald  magma with non-empty finite set of irreducibles,  
    then $(\magma,\st)$ is   a free magma   generated by the irreducible elements.
\end{theorem}

\thmref{thm_free} shows  that the following is sufficient to show a set $ \ffm$  with a product, $\st$, is a free magma:
\begin{itemize}
    \item Define   a product map $\st:\ffm\times\ffm\to \ffm$  and show it is injective. 
    \item Define    a  map $\norm{\cdot} : \ffm\to \mathbb{N}$  and show $\norm{\cdot}$ is additive over the product.
    \item Determine the range of the product map, $\ffm^+$, and hence  the set of generators  $\ffm^0=\ffm\setminus \ffm^+$ and $p=|\ffm^0|$.
\end{itemize}

The appendix lists several Catalan families, defining their product and a norm. 
In most cases it  is straightforward to check the product is injective (usually seen by the way the Catalan object factorises into a unique pair) and that the norm is additive. 

\begin{proof}[Proof of \thmref{thm_free}]
The proof is a direct proof based on \defref{def_free}. 
Given $(\magma,\st)$ the range of the product map defines $\magp$ and thus the non-empty set of irreducibles, $\mprime=\magma\setminus \magp$.
Let $\mathcal{Y}$ be any set such that $|\mathcal{Y}|=|\mprime|$ and $i: \mathcal{Y}\to \mprime $ be any bijection. Let $(\magman,\bullet)$ be any magma and $f$ any map $f:\mathcal{Y}\to \magman$. 
Define a map $\theta :\magma\to \mathcal{Y}$ as follows:
The value of $\theta$ for all elements in $ \mprime$ is defined  via 
\begin{equation}\label{eq_t1gen}
  \theta(i(y))=f(y), \qquad y\in \mathcal{Y}\,.  
\end{equation}
The value of $\theta$ for all elements   $m\in\magp$ is defined recursively by 
\begin{equation}\label{eq_t1rec}
\theta(m)=\theta(m_1)\bullet \theta(m_2)
\end{equation}
where $m = m_1\st m_2$. 
This uniquely defines $\theta$ for all   $m\in\magma$ and any $f$ as for all $m\in\magp$, $m$ uniquely factorises $m = m_1\st m_2$ (since $(\magma,\st)$ is a unique factorisation magma).
Furthermore  the existence of a norm implies the recursion  in the definition of $\theta(m)=\theta(m_1)\bullet \theta(m_2)$ will always terminate  (by \propref{thm_fff})  on an irreducible element (where $\theta$ is already defined) and thus the multiplication in $(\magman,\bullet)$ is well defined. 
Thus,   $i$ and $\mathcal{Y}$ exist  and for all maps $f$  and all magmas $(\magman,\bullet)$,  $\theta$ is a uniquely defined map. Furthermore, the recursive part of the definition $\theta(m)=\theta(m_1)\bullet \theta(m_2)$ shows that $\theta(m_1\st m_2)=\theta(m_1)\bullet \theta(m_2)$, that is, $\theta$ is also a (unique) magma morphism. Thus, by \defref{def_free}, $(\magma,\st)$ is free.
\end{proof}

This motivates the following definition.
\begin{definition}[Catalan  normed magma]\label{def_catmag}
Let $(\ffm,\st)$ be a normed magma. If
\begin{enumerate}
    \item the magma is free, and
    \item the norm is additive, and
    \item there are $p\in\nat$ irreducible elements,  and
    \item  the  irreducible elements all have the same norm (conventionally the norm is one),
\end{enumerate}
 then   $(\ffm,\st)$ is called a \textbf{$\boldsymbol{p}$-Catalan normed magma}. If $p=1$ it will be  called a \textbf{Catalan normed magma}.
\end{definition}
%
% The norm of the generators are conventionally set to one. 
We conclude this section by using \thmref{thm_free} to prove the two Catalan families considered above, Dyck paths and triangulations, are free magmas. 
These two proofs should be compared with the `from the definition' proof of  \propref{prop_cartfree} (ie.\ that the Cartesian magma is free). 

\medskip 
\noindent\textit{Dyck Paths -- \fref{fam_dp}.} \\
\textit{Product:} The product defined by \eqref{eq_dyck_prod} is injective: take any Dyck path $d\ne \circ$. Then the rightmost step of  $d$  to leave the surface always exists and is unique and  thus the two factors $d_1$ and $d_2$ are unique. \\
\textit{Norm:} The map \val of a path is defined as  $\norm{d}=(\text{number of up steps in $d$})+1$. This is well defined   and   in all cases  is additive: 
Let $2n_1$ (respec.\ $2n_2$) be the number of steps in $d_1$ (respec.\ $d_2$).
The number of steps in $d_1\st_\dpn d_2$ is $2(n_1+n_2)+2$. Thus $\norm{d_i}=n_i+1$, $\norm{d_1\st_\dpn\, d_2}=(n_1 +n_2+1)+1$ and thus $\norm{d_1\st_\dpn\, d_2}=\norm{d_1}+\norm{d_2} $. \\
\textit{Generator:} There is only one irreducible element in \fref{fam_dp}: the element with no rightmost up step, that is, the single vertex $\circ$.  Thus \fref{fam_dp} has one generator, $\eps_\dpn=\circ$.
Thus, by \thmref{thm_free} and \defref{def_catmag}, \fref{fam_dp} is a Catalan magma.  

\medskip 
\noindent\textit{Triangulations -- \fref{fam_pt}:}  \\
\textit{Product:} The product defined by \eqref{triangulations} is clearly injective: if $t$ is any triangulation with at least one triangle, then the triangle $a-d-bc$ to the right of the marked vertex   always exists, is unique, and uniquely splits   $t$ into a left, $t_1$ and right triangulation, $t_2$. \\
\textit{Norm:} The \val of a triangulation of a regular $n$-gon for $n>2$ is   $n-1$ (and $\norm{\TransformHoriz}=1)$.  The \val is well defined and adds:  if $\norm{t_1}=n_1-1$,  $\norm{t_2}=n_2-1$ then $t_1\st_\ptn t_2$ is an $((n_1+n_2-2)+1)$-gon and thus $\norm{t_1\st_\ptn t_2}=(n_1+n_2-1) -1 =\norm{t_1}+\norm{t_2}$. \\
\textit{Generator:} There is only one irreducible element in \fref{fam_pt}: the element with no triangle to the right of the marked vertex, that is, the single  edge $\TransformHoriz$.  Thus \fref{fam_pt} has one generator, $\eps_\ptn=\,\,\TransformHoriz$.
Thus, by \thmref{thm_free} and \defref{def_catmag}, \fref{fam_pt} is a Catalan magma.

%==============================================
\section{Free magma isomorphisms and a universal Catalan bijection}
%==============================================
\label{sec_unibij}

It is well known that any free algebraic structures defined by a universal mapping principle are isomorphic. Algebraically this is usually the end of the line: pick the simplest free structure and work with that. Thus all Catalan magmas are isomorphic with, arguably, the Cartesian magma having one of the simplest products.

However, in this combinatorial context a  large part of the interest in Catalan structures is in the bijections between them. 
The bijections range from simple (eg.\ between Dyck paths and matching brackets) to the very complex  (eg.\ between Dyck paths and Kepler Towers \cite{Knuth:2005aa}).  Thus the combinatorial interest is in the nature of the   numerous algorithms defining the isomorphisms.

We will restate the isomorphism theorem and give the  standard proof. The reason for doing this is that the theorem is constructive and shows the overall structure of the  isomorphism \emph{does not depend on the details of  any particular free magma}. Thus the theorem will provide a universal bijection between any pair of Catalan families. Furthermore, being an isomorphism, it also respects the recursive structure of all Catalan families, an essential property of all ``nice'' bijections. Thus we begin with the isomorphism theorem and its proof.

\begin{prop} \label{uni_iso}
Given magmas $(F,\st)$ and $(G,\bullet)$ and set $Y$  with maps $i:Y\to F$ and $j:Y\to G$ each satisfying respective free magma universal mapping diagrams  
\begin{equation}\label{eq_fgp}
 \begin{tikzcd}
  F    \arrow[r, dotted,"\theta"]   & N \\
   Y \arrow[u,"i"] \arrow[ru,"f"'] &                           
\end{tikzcd}   \qquad 
 \begin{tikzcd}
  G    \arrow[r,dotted, "\phi"]   & N' \\
   Y \arrow[u ,"j"] \arrow[ru,"g"'] &                           
\end{tikzcd} 
\end{equation}
there is a unique   magma isomorphism $\Gamma : F\to G$ such that $\Gamma\circ i =j $ and $\Gamma^{-1} \circ j=i $.
\end{prop}
Note, $\Gamma$ is unique, given $i$ and $j$, however $i$ and $j$ are are only unique when $|Y|=1$.
\begin{proof}
Since the diagrams \eqref{eq_fgp} commute for all $N$, $f$ and $g$ we can make the following choices: $f=j$, $N=G$, $g=i$ and $N'=F$. This gives two commutative diagrams which are drawn together as
\begin{equation}\label{eq_commtri}
 \begin{tikzcd}
  F    \arrow[r,dotted, "\theta"]   & G  \arrow[r,dotted,"\phi"] & F \\
 &   Y \arrow[lu,"i"]  \arrow[u,"j"] \arrow[ru,"i"'] &     \\        \end{tikzcd} 
\end{equation}
Commutativity gives, $j=\theta\circ i$ and $i=\phi\circ j$. Thus $i=(\phi \circ \theta)\circ i $ and hence $\phi \circ \theta = \id_{F}$. Similarly,  $j=( \theta\circ \phi )\circ j $ and hence $\theta\circ \phi = \id_{G}$. Thus $\Gamma=\theta  $ and $\Gamma^{-1}=  \phi $  are   isomorphisms.
\end{proof}

In the case $|Y|=1$   the two maps $i$ and $j$ ``point'' to the respective generators, $\eps_1\in F$ and $\eps_2\in G$, in the respective magmas $F$ and $G$. The values of the isomorphisms $\theta$ and $\phi$ are defined in the proof of  \thmref{thm_free}, that is, recursively from  $F$  to $G$
\begin{subequations}\label{eq_ftog}
\begin{align}
\theta(\eps_1)&=\eps_2\,,\\
    \theta(m)&=\theta(m_1)\bullet\theta(m_1)\quad \text{where $ m=m_1\st m_2$,}
\end{align}
\end{subequations}
 and from  $G$ to $F$
 \begin{subequations}\label{eq_gtof}
\begin{align}
    \phi(\eps_2)&=\eps_1\,,\\
    \phi(p)&=\phi(p_1)\st\phi(p_1)\quad \text{where $ p=p_1\bullet p_2$.}
\end{align}
\end{subequations}
From an algebraic perspective the maps are expected, however from a combinatorial perspective of trying to make sense of the large diversity of bijections between Catalan families they are   significant.  
The maps defined by equations \eqref{eq_ftog} (or \eqref{eq_gtof})   say that between \emph{any} pair of Catalan families the \emph{same} algorithm can be used to define a bijection, furthermore this bijection respects the recursive structure of both families. Since the bijection arises out of the universal mapping principle defining free magmas we will call it a universal bijection and give it a formal definition.

 \begin{definition}[Universal Bijection]
\label{def_unibij}
    \noindent Let $(\ffm_{i},\st_i)$ and $(\ffm_{j},\st_j)$ be free magmas with generator  sets $X_i$ and $X_j$ respectively, with  $|X_i|=|X_j|$. 
    Let $\sigma: X_i\to X_j$ be any bijection. 
    Define the   map  $\unimap_{i,j}: \ffm_i\to \ffm_j$   as  follows: For all $x\in X_i$, $\unimap_{i,j}(x)=\sigma(x)$. For all $u\in \ffm_i\setminus X_i$
    \begin{enumerate}
        \item Decompose $u$ into a product of   generators $\eps_i\in X_i$.
        \item In the decomposition of $u$ replace every occurrence of $\eps_i$ with $\sigma(\eps_i)$ and every occurrence of  $\st_i$ with   $\st_j$ to give an expression $w(u)$.
        \item The value of $\unimap_{i,j}(u)$ is defined to be  $w(u)$,  that is, evaluate all the products in $w(u)$ to give an element of $\ffm_{j}$.  
    \end{enumerate} 

\end{definition}

We then have the simple consequence following from \eqref{eq_ftog}:
\begin{cor} \label{thm_gamma}
Let $\unimap_{i,j}: \ffm_i\to \ffm_j$ be the map of \defref{def_unibij}, then  $\unimap_{i,j}$ is  a free  magma  isomorphism.
\end{cor}

The map  $\unimap_{i,j}$  can written schematically    as
\begin{equation}\label{eq_unibi}
f \qquad \stackrel{\text{\tiny decompose}}{\longrightarrow} \qquad  \mathop{\text{\tiny substitute} }^{ \eps_i\to \eps_j}_{\star_i\to \star_j}  \qquad \stackbin{\text{\tiny multiply}}{\longrightarrow}\qquad g\,.
\end{equation}
where $ \eps_j=\sigma(\eps_i)$.
Since $\unimap_{i,j}$ is an isomorphism we get two outcomes: firstly, it gives us a bijection between the base sets  of $\ffm_i$ and $\ffm_j$, and secondly, that the bijection is recursive, that is
if $u=u_1\st_i u_2 $ then 
$\unimap_{i,j}(u)=\unimap_{i,j}(u_1)\st_j \unimap_{i,j}(u_2)$.   
ie.\  if the images under $\unimap_{i,j}$ are already known for  the two factors then decomposition is not necessary -- only a single factorisation is required (and a single multiplication). 
 
We now consider  several examples to illustrate  the  bijection/isomorphism $\unimap_{i,j}$. We will use the convention that if it is clear from the context the subscripts $i$ and $j$ will be omitted.
The first examples use    the two   Catalan families discussed above: Dyck paths and triangulations and is given in a little more detail. This is followed by several examples where more detail (product definitions etc.) can be obtained from the Appendix.

Consider a Dyck path and decompose the  path down to a product of generators:
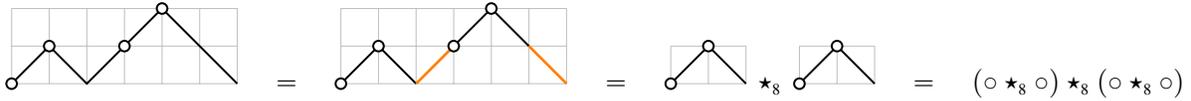
\begin{equation*}
\begin{tikzpicture}[style_size3]
    \def\U{++(1,1)} \def\D{++(1,-1)}
    \draw [help lines, gray!50] (0,0) grid(6,2);
    \draw  (0,0)   -- \U -- \D --  \U -- \U  -- \D -- \D;  
    \path[draw,fill=white] (0,0) circle (4pt) \U circle (4pt) \D \U circle (4pt) \U circle (4pt) ;
\end{tikzpicture}
\quad = \quad
\begin{tikzpicture}[style_size3]
    \def\U{++(1,1)} \def\D{++(1,-1)}
    \draw [help lines, gray!50] (0,0) grid(6,2);
    \draw  (0,0)   -- \U -- \D ;
    \draw [orange,line width =1pt] (2,0) --  \U;
    \draw (3,1) -- \U  -- \D ; 
    \draw [orange,line width =1pt] (5,1) --  \D;
    \path[draw,fill=white] (0,0) circle (4pt) \U circle (4pt) \D \U circle (4pt) \U circle (4pt) ;
\end{tikzpicture}      
 \quad = \quad
 \begin{tikzpicture}[style_size3]
    \def\U{++(1,1)} \def\D{++(1,-1)}
    \draw [help lines, gray!50] (0,0) grid(2,1);
    \draw  (0,0)   -- \U -- \D ;
    \path[draw,fill=white] (0,0) circle (4pt) \U circle (4pt) ;
\end{tikzpicture}   
 \,\star_\dpn \,
  \begin{tikzpicture}[style_size3]
    \def\U{++(1,1)} \def\D{++(1,-1)}
    \draw [help lines, gray!50] (0,0) grid(2,1);
    \draw  (0,0)   -- \U -- \D ;
    \path[draw,fill=white] (0,0) circle (4pt) \U circle (4pt) ;
\end{tikzpicture} 
 \quad = \quad
 (\circ \star_\dpn \circ) \star_\dpn (\circ \star_\dpn \circ)
\end{equation*}
%   \begin{center}
%  	\includegraphics[scale=\fscale]{figs/DyckPath_to_Tri_lr.pdf}
%  \end{center}
then change generators 
\[
\eps_\dpn=\circ\quad \to \quad \trigen=\eps_\ptn
\]
and product rule $\star_\dpn\to \star_\ptn$, that is,
% \begin{equation*}
%     (\circ \star_\dpn \circ)\star_\dpn  (\circ \star_\dpn  \circ) 
%   \quad  \mapsto \quad 
% \trigen
% \end{equation*}
\begin{equation*}
(\circ \star_\dpn \circ)\star_\dpn (\circ \star_\dpn \circ) 
\quad  \mapsto \quad 
\left( \trigen \, \star_\ptn \, \trigen \right)
\, \star_\ptn \,
\left( \trigen \, \star_\ptn \, \trigen \right)
\end{equation*}
% \begin{center}
%   \includegraphics[scale=\fscale]{figs/DyckPath_to_Tri_prodA1_lr}	
% \end{center} 
then   multiply:
\begin{equation*}
\def\BL{-0.2cm}  \def\rd{2pt}\def\ro{0.9} \def\rt{0.35}  
\left( \trigen \, \star_\ptn \, \trigen \right)
\, \star_\ptn \,
\left( \trigen \, \star_\ptn \, \trigen \right)
\quad = \quad
\begin{tikzpicture}[ scale=1,baseline=\BL] 
%   \draw [help lines, gray!50] (0,0) grid(1,1);
    \draw (30:\rt) -- (150:\rt) -- (-90:\rt) -- cycle;
    \draw[fill=white] (30:\rt) circle (\rd);
    \draw[fill=black] (150:\rt) circle (\rd);
    \draw[fill=white] (-90:\rt) circle (\rd);
\end{tikzpicture}
\,\star_\ptn \,
 \begin{tikzpicture}[ scale=1,baseline=\BL] 
    \draw (30:\rt) -- (150:\rt) -- (-90:\rt) -- cycle;
    \draw[fill=white] (30:\rt) circle (\rd);
    \draw[fill=black] (150:\rt) circle (\rd);
    \draw[fill=white] (-90:\rt) circle (\rd);
\end{tikzpicture}
\quad = \quad
    \begin{tikzpicture}[scale=0.5,baseline= -0.2cm]
    %   \draw [help lines, gray!50] (0,0) grid(1,1);
        \def\pa{72} \def\lv{54} \def\rd{4pt}  
    \def\ro{0.9}  \def\rt{1.2}  \def\rf{1.27}  \def\rv{1}  
        \draw (\lv:\rv) -- ( \lv+\pa:\rv) -- (\lv+2*\pa:\rv) -- (\lv+3*\pa:\rv) -- (\lv+4*\pa:\rv)  -- cycle;
        \draw (\lv+3*\pa:\rv) -- (\lv+ 5*\pa:\rv);
        \draw (\lv+3*\pa:\rv) -- (\lv+ 6*\pa:\rv);
        \draw[fill=white] (\lv:\rv) circle (\rd);
        \draw[fill=black] (\lv+\pa:\rv) circle (\rd);
        \draw[fill=white] (\lv+2*\pa:\rv) circle (\rd);
        \draw[fill=white] (\lv+3*\pa:\rv) circle (\rd);
        \draw[fill=white] (\lv+4*\pa:\rv) circle (\rd);
    \end{tikzpicture}
\end{equation*}
% \begin{center}
%   \includegraphics[scale=\fscale]{figs/DyckPath_to_Tri_prodA2_lr.pdf}	
% \end{center} 
which gives the bijection  
\begin{equation*}
    \begin{tikzpicture}[scale=0.5,baseline= 0.5cm]
    \def\U{++(1,1)} \def\D{++(1,-1)}
    \draw [help lines, gray!50] (0,0) grid(6,2);
    \draw[line width=0.5pt]  (0,0)   -- \U -- \D --  \U -- \U  -- \D -- \D;  
    \path[draw,fill=white] (0,0) circle (4pt) \U circle (4pt) \D \U circle (4pt) \U circle (4pt) ;
\end{tikzpicture} 
\quad\mapsto\quad
    \begin{tikzpicture}[scale=0.5,baseline= 0cm]
        % \draw [help lines, gray!50] (0,0) grid(1,1);
        \def\pa{72} \def\lv{54} \def\rd{4pt}  
    \def\ro{0.9}  \def\rt{1.2}  \def\rf{1.27}  \def\rv{1}  
        \draw (\lv:\rv) -- ( \lv+\pa:\rv) -- (\lv+2*\pa:\rv) -- (\lv+3*\pa:\rv) -- (\lv+4*\pa:\rv)  -- cycle;
        \draw (\lv+3*\pa:\rv) -- (\lv+ 5*\pa:\rv);
        \draw (\lv+3*\pa:\rv) -- (\lv+ 6*\pa:\rv);
        \draw[fill=white] (\lv:\rv) circle (\rd);
        \draw[fill=black] (\lv+\pa:\rv) circle (\rd);
        \draw[fill=white] (\lv+2*\pa:\rv) circle (\rd);
        \draw[fill=white] (\lv+3*\pa:\rv) circle (\rd);
        \draw[fill=white] (\lv+4*\pa:\rv) circle (\rd);
    \end{tikzpicture}
 \end{equation*}
% \begin{center}
%  	 \includegraphics[scale=\fscale]{figs/DyckPath_to_Tri_prod_bij_lr}
% \end{center}
Similarly, if we perform the same multiplications for matching brackets (see family \fref{fam_mb} in the appendix), we get
\[ 
(\emptyset\st_{\ref{fam_mb}} \emptyset)\st_{\ref{fam_mb}} (\emptyset\st_{\ref{fam_mb}} \emptyset) 
=\{\} \st_{\ref{fam_mb}}\{\}=\{\}\{\{\}\}
\]
(where $\eps_{\ref{fam_mb}}=\emptyset$, the empty expression),
or for Catalan floor plans (family \fref{fam_fp} in the appendix))
 \begin{equation*}
     (\floorgen \star \floorgen) \star (\floorgen \star \floorgen)
     \,=\,
     \tikz\draw (0,0) rectangle (0.5,0.5); 
     \,\star\,
     \tikz\draw (0,0) rectangle (0.5,0.5); 
     \,=\,
      \tikz\draw (0,0) rectangle (1,0.5) (0,1) rectangle (0.5,0.5) (1,1) rectangle (0.5,0.5); 
 \end{equation*}
% \[ 
% (\edge \st_{\ref{fam_fp}} \edge)\st_{\ref{fam_fp}} (\edge\st_{\ref{fam_fp}} \edge)
% = \fbox{\phantom{x}} \st_{\ref{fam_fp}} \fbox{\phantom{x}}= \raisebox{-0.3\height}{\includegraphics[scale=0.5]{figs/floorPlanExB_lr.pdf}}
% \]
(the marks have been removed from the generator geometry on the right) or for nested matchings (family \fref{fam_ld} in the appendix),
 \begin{equation*}
     (\matchgen \star \matchgen) \star (\matchgen \star \matchgen)
 \,=\, 
\begin{tikzpicture}[line width=0.5pt,scale=0.5] 
\def\S{++(1,0)} \def\R{4pt}
    \draw (0,0) -- (3,0);
    \fill (0,0) circle (\R) \S circle (\R) \S circle (\R);
    \draw  (1,0) arc [start angle=180, end angle = 0, radius=0.5];
\end{tikzpicture}
\star
\begin{tikzpicture}[line width=0.5pt,scale=0.5] 
\def\S{++(1,0)} \def\R{4pt}
    \draw (0,0) -- (3,0);
    \fill (0,0) circle (\R) \S circle (\R) \S circle (\R);
    \draw  (1,0) arc [start angle=180, end angle = 0, radius=0.5];
\end{tikzpicture}
\quad = \quad
\begin{tikzpicture}[line width=0.5pt,scale=0.5] 
\def\S{++(1,0)} \def\R{4pt} \def\C{circle (\R)}
    \draw (0,0) -- (7,0);
    \fill (0,0) \C \S \C \S \C \C \S \C \S \C \S \C \S \C   ;
    \draw  (1,0) arc [start angle=180, end angle = 0, radius=0.5];
    \draw  (3,0) arc [start angle=180, end angle = 0, radius=1.5];
    \draw  (4,0) arc [start angle=180, end angle = 0, radius=0.5];
\end{tikzpicture}
\end{equation*}
% \[ 
% (\raisebox{0ex}{\includegraphics[scale=0.6]{figs/nestedMatchingsGen_lr.pdf}} \st_{\ref{fam_fp}} \raisebox{0ex}{\includegraphics[scale=0.6]{figs/nestedMatchingsGen_lr.pdf}})\st_{\ref{fam_fp}} (\raisebox{0ex}{\includegraphics[scale=0.6]{figs/nestedMatchingsGen_lr.pdf}}\st_{\ref{fam_fp}} \raisebox{0ex}{\includegraphics[scale=0.6]{figs/nestedMatchingsGen_lr.pdf}})
% = \raisebox{-0.5ex}{\includegraphics[scale=2]{figs/nestedMatchingEx2g_lr.pdf}}
% \st_{\ref{fam_fp}} 
% \raisebox{-0.5ex}{\includegraphics[scale=2]{figs/nestedMatchingEx2g_lr.pdf}}
% = 
% \raisebox{-0.6ex}{ \includegraphics[scale=2]{figs/nestedMatchingEx2g_A2_lr.pdf} }
% \]
Thus we have the bijections:
 \begin{equation*}
 \def\BL{0.5cm}
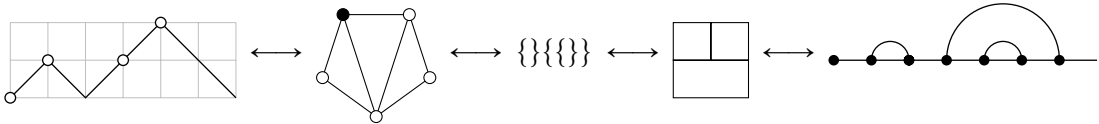
     \begin{tikzpicture}[scale=0.5,line width=0.5pt,baseline=\BL]
    \def\U{++(1,1)} \def\D{++(1,-1)}
    \draw [help lines, gray!50] (0,0) grid(6,2);
    \draw   (0,0)   -- \U -- \D --  \U -- \U  -- \D -- \D;  
    \path[draw,fill=white] (0,0) circle (4pt) \U circle (4pt) \D \U circle (4pt) \U circle (4pt) ;
\end{tikzpicture}
\,\longleftrightarrow \,
\begin{tikzpicture}[scale=0.75,baseline= 0cm]
    % \draw [help lines, gray!50] (0,0) grid(1,1);
    \def\pa{72} \def\lv{54} \def\rd{3pt}  
    \def\ro{0.9}  \def\rt{1.2}  \def\rf{1.27}  \def\rv{1}  
    \draw (\lv:\rv) -- ( \lv+\pa:\rv) -- (\lv+2*\pa:\rv) -- (\lv+3*\pa:\rv) -- (\lv+4*\pa:\rv)  -- cycle;
    \draw (\lv+3*\pa:\rv) -- (\lv+ 5*\pa:\rv);
    \draw (\lv+3*\pa:\rv) -- (\lv+ 6*\pa:\rv);
    \draw[fill=white] (\lv:\rv) circle (\rd);
    \draw[fill=black] (\lv+\pa:\rv) circle (\rd);
    \draw[fill=white] (\lv+2*\pa:\rv) circle (\rd);
    \draw[fill=white] (\lv+3*\pa:\rv) circle (\rd);
    \draw[fill=white] (\lv+4*\pa:\rv) circle (\rd);
\end{tikzpicture}
\,\longleftrightarrow \,
\{\}\{\{\}\}
\,\longleftrightarrow \,
\begin{tikzpicture}[baseline=\BL]
    \draw (0,0) rectangle (1,0.5) (0,1) rectangle (0.5,0.5) (1,1) rectangle (0.5,0.5); 
\end{tikzpicture}
\,\longleftrightarrow \,
 \begin{tikzpicture}[line width=0.5pt,scale=0.5,baseline=0cm] 
\def\S{++(1,0)} \def\R{4pt} \def\C{circle (\R)}
    \draw (0,0) -- (7,0);
    \fill (0,0) \C \S \C \S \C \C \S \C \S \C \S \C \S \C   ;
    \draw  (1,0) arc [start angle=180, end angle = 0, radius=0.5];
    \draw  (3,0) arc [start angle=180, end angle = 0, radius=1.5];
    \draw  (4,0) arc [start angle=180, end angle = 0, radius=0.5];
\end{tikzpicture}
\end{equation*}
% \begin{center}
%  	 \includegraphics[scale=\fscale]{figs/DyckPath_to_Tri_prod_bij_A2_lr.pdf}
% \end{center}

Some well known Catalan bijections are the same map  as defined by $\unimap_{i,j}$ but the algorithm stated to define the codomain element given a domain element can be dramatically different.  
For example, the bijection from Dyck paths to Staircase polygons (family \fref{fam_sp} in the appendix) defined by Delest and Viennot \cite{delest84} (where staircase polygons are called polyominoes) used the heights of the peaks and valleys of a path to define the heights of the staircase columns (ie.\ number of cells) and their overlaps (ie.\ number of adjacent cells), for example: 
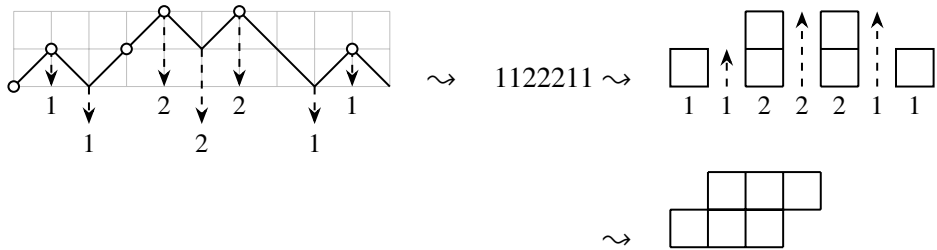
\begin{align*}
\begin{tikzpicture}[style_size3,>=Stealth,baseline=0cm]
\def\U{++(1,1)} \def\D{ ++(1,-1) }
\draw [help lines, gray!50] (0,0) grid(10,2);
\draw  (0,0)   -- \U -- \D --  \U -- \U --\D -- \U -- \D -- \D -- \U -- \D;  
\draw[dashed ] (1,1) edge[->]  (1,0) (1,-0.5) node {1};
\draw[dashed ] (2,0) edge[->]  (2,-1) (2,-1.5) node {1};
\draw[dashed ] (4,2) edge[->]  (4,0) (4,-0.5) node {2};
\draw[dashed ] (5,1) edge[->]  (5,-1) (5,-1.5) node {2};
\draw[dashed ] (6,2) edge[->]  (6,0) (6,-0.5) node {2};
\draw[dashed ] (8,0) edge[->]  (8,-1) (8,-1.5) node {1};
\draw[dashed ] (9,1) edge[->]  (9,0) (9,-0.5) node {1};
\path (0,0) pic{circle_white} -- (1,1) pic{circle_white} -- (3,1) pic{circle_white} -- (4,2) pic{circle_white} -- (6,2) pic{circle_white} -- (9,1) pic{circle_white};
\end{tikzpicture}
\quad\leadsto\quad
1122211
& \leadsto\quad
\begin{tikzpicture}[style_size3,>=Stealth,baseline=0cm]
\draw (0,0) rectangle (1,1);
\draw (1.5,0) edge[dashed,->] (1.5,1);
\draw (2,0) grid (3,2);
\draw (3.5,0) edge[dashed,->] (3.5,2);
\draw (4,0) grid (5,2);
\draw (5.5,0) edge[dashed,->] (5.5,2);
\draw (6,0) rectangle (7,1);
\draw (0.5,-0.5)  node {1} (1.5,-0.5) node{1}  (2.5,-0.5) node{2} (3.5,-0.5) node{2} (4.5,-0.5) node{2} (5.5,-0.5) node{1} (6.5,-0.5) node{1};
\end{tikzpicture} \\
&\leadsto\quad
\begin{tikzpicture}[style_size3,>=Stealth,baseline=0cm]
\draw (0,0) grid (3,1) (1,1) grid (4,2);
\end{tikzpicture}
\end{align*}
%
% \begin{equation}\label{eq:dvbij}
%   \raisebox{-0.5\height}{  \includegraphics[scale=0.8]{figs/pathToStaircaseBij_A1_lr.pdf}}
% \end{equation}
whilst factorising the path,

\begin{align*}
\begin{tikzpicture}[style_size3,>=Stealth,baseline=0cm]
\def\U{++(1,1)} \def\D{ ++(1,-1) }
\draw [help lines, gray!50] (0,0) grid(10,2);
\draw  (0,0)   -- \U -- \D --  \U -- \U --\D -- \U -- \D -- \D -- \U -- \D;  
\path (0,0) pic{circle_white} -- (1,1) pic{circle_white} -- (3,1) pic{circle_white} -- (4,2) pic{circle_white} -- (6,2) pic{circle_white} -- (9,1) pic{circle_white};
\end{tikzpicture}
& \quad = \quad ((\circ\star_{\ref{fam_dp}}\circ )\star_{\ref{fam_dp}}((\circ \star_{\ref{fam_dp}} \circ )\star_{\ref{fam_dp}} \circ ))\star_{\ref{fam_dp}} \circ \\
\intertext{and using $\unimap_{{\ref{fam_dp}},{\ref{fam_pt}}}$ gives}
((\circ\star_{\ref{fam_dp}}\circ )\star_{\ref{fam_dp}}((\circ \star_{\ref{fam_dp}} \circ )\star_{\ref{fam_dp}} \circ ))\star_{\ref{fam_dp}} \circ 
& \quad\mapsto\quad
((\gensp\star_{\ref{fam_pt}}\gensp )\star_{\ref{fam_pt}}((\gensp \star_{\ref{fam_pt}} \gensp )\star_{\ref{fam_pt}} \gensp ))\star_{\ref{fam_pt}} \gensp \\
& \\
&\quad = \quad\left(
\begin{tikzpicture}[baseline=0.15cm] 
\draw (0,0) rectangle (0.5,0.5);
\end{tikzpicture} 
\star_{{\ref{fam_pt}}}
\left(
\begin{tikzpicture}[baseline=0.15cm] 
\draw (0,0) rectangle (0.5,0.5);
\end{tikzpicture}
\star_{{\ref{fam_pt}}}\gensp
\right)\right) 
\star_{{\ref{fam_pt}}}\gensp \\
& \\
& \quad=\quad
\begin{tikzpicture}[style_size3,>=Stealth,baseline=0cm]
\draw (0,0) grid (3,1) (1,1) grid (4,2);
\end{tikzpicture}
\end{align*}
%
% \begin{center}
%     \includegraphics[scale=\fscale]{figs/pathToStaircaseBij_A2_lr.pdf}
% \end{center}
that is, the same staircase (the staircase product definition is given in the Appendix).
The proof of the equivalence of the definitions can be seen combinatorially by comparing  the embedding bijections discussed in   \secref{sec_embedd}.

%=============================================
\section{Opposite, reverse and reflected families}
%=============================================
\label{sec_opprev}

We now consider new products defined  by combining a given product $\st: \magma\times\magma\to \magma$ with a bijection $\gamma : \magp\to \magp$ whose domain is the set of reducible elements, $\magp$, of $\st$. 
If necessary it can be extended to $\magma$ by defining $\gamma(\eps)=\eps$ for all $\eps\in\mprime $.
Thus we define the product $\bst^\gamma$ as
\begin{equation}
    \bst^\gamma: \magma\times\magma\to \magma,\quad  \quad u\bst^\gamma v:= \gamma(u\st v)\,.
\end{equation}
If $\st$ is injective then since $\gamma$ is a bijection,  $\bst^\gamma$ will also be injective and thus $(\magma,\bst^\gamma)$ is also a unique factorisation magma. 

Choosing $\gamma$ appropriately can show how different products defined on the same base set (eg.\ left and right factorisation of Dyck paths) are related. 
% They will also be useful in understanding embedding bijections discussed in \secref{sec_embedd}.

We are primarily interested in three examples of $\gamma$:
\begin{enumerate}

\item The \textbf{opposite} map. This map is  defined by $\opm(u\st v)=v\st u$ giving the product
    \[
        \opp{\st}\,:\,    (u,  v)\mapsto  v\st u\,,
    \]
    called the \textbf{opposite product}. 
    Since the opposite map is defined for every magma, every Catalan family, $\ffm_i$ has an  \textbf{opposite  family}, denoted $\opp{\ffm}_i$. 
    For example, for Dyck paths
    \begin{equation*}
\def\S{0.3}
\begin{tikzpicture}[scale=\S]
    \def\U{++(1,1)} \def\D{++(1,-1)}
    \draw [help lines, gray!50] (0,0) grid(2,1);
    \draw[line width=1pt]  (0,0)   -- \U -- \D  ;  
\end{tikzpicture}
\,\star \,
\begin{tikzpicture}[scale=\S]
    \def\U{++(1,1)} \def\D{++(1,-1)}
    \draw [help lines, gray!50] (0,0) grid(4,1);
    \draw[line width=1pt]  (0,0)   -- \U -- \D --  \U --  \D;  
\end{tikzpicture}
\quad = \quad 
\begin{tikzpicture}[scale=\S]
    \def\U{++(1,1)} \def\D{++(1,-1)}
    \draw [help lines, gray!50] (0,0) grid(8,2);
    \draw[line width=1pt]  (0,0)   -- \U -- \D --  \U -- \U -- \D -- \U -- \D -- \D;  
\end{tikzpicture}
\quad\text{and}\quad 
\begin{tikzpicture}[scale=\S]
    \def\U{++(1,1)} \def\D{++(1,-1)}
    \draw [help lines, gray!50] (0,0) grid(2,1);
    \draw[line width=1pt]  (0,0)   -- \U -- \D  ;  
\end{tikzpicture}
\, \star^\text{op} \,
\begin{tikzpicture}[scale=\S]
    \def\U{++(1,1)} \def\D{++(1,-1)}
    \draw [help lines, gray!50] (0,0) grid(4,1);
    \draw[line width=1pt]  (0,0)   -- \U -- \D --  \U --  \D;  
\end{tikzpicture}
\qquad = \qquad 
\begin{tikzpicture}[scale=\S]
    \def\U{++(1,1)} \def\D{++(1,-1)}
    \draw [help lines, gray!50] (0,0) grid(8,2);
    \draw[line width=1pt]  (0,0)   -- \U -- \D --  \U -- \D -- \U -- \U -- \D -- \D;  
\end{tikzpicture}
\end{equation*}    
    % \begin{center}
   	%      \includegraphics[scale=\fscale]{figs/oppositeProd_art1_lr.pdf}
    % \end{center}
    Note, the opposite map, as defined above, is not a   morphism: $\opm(u\st v)\ne \opm(u)\opp{\st} \opm(v)$ since it is not defined recursively.
    The identity map $\id \,:\, \ffm\to \opp{\ffm}\,;\, u\mapsto u$ is however an anti-isomorphism:  $\id(u\st v)=\id(v)\opp{\st}\id(u)$.

\item The \textbf{reverse} map. This map is defined recursively by
    \begin{equation}\label{eq_revmap}
        \revm (u\st v)=\revm(v)\st \revm(u)\,.
    \end{equation}
    (with $\revm(\eps)=\eps$,  a generator) which    defines the reverse magma
    \[
        \st^\text{rev}\,:\,    (u,  v)\mapsto   \revm(v)\st \revm(u)\,.
    \]    
    The reverse map, $\text{rev} : \magma\to \magma$, is an anti-isomorphism.
    Every Catalan family, $\ffm_i$ has a  \textbf{reverse family}, denoted $\revp{\ffm}_i$.  For example, for Dyck paths
\begin{align*}
\def\S{0.3}
    \begin{tikzpicture}[scale=\S]
        \def\U{++(1,1)} \def\D{++(1,-1)}
        \draw [help lines, gray!50] (0,0) grid(2,1);
        \draw[line width=1pt]  (0,0)   -- \U -- \D  ;  
    \end{tikzpicture}
     \,\star^\text{rev} \,
    \begin{tikzpicture}[scale=\S]
        \def\U{++(1,1)} \def\D{++(1,-1)}
        \draw [help lines, gray!50] (0,0) grid(4,1);
        \draw[line width=1pt]  (0,0)   -- \U -- \D --  \U --  \D;  
    \end{tikzpicture} 
    &=  
    \text{rev}
    \left[ 
        (\circ \star \circ) \star ((\circ \star \circ )\star \circ )
    \right] \\
     &= (\circ \star (\circ \star \circ))\star (\circ \star \circ )\\
     &=
\def\S{0.3}
      \begin{tikzpicture}[scale=\S]
        \def\U{++(1,1)} \def\D{++(1,-1)}
        \draw [help lines, gray!50] (0,0) grid(8,2);
        \draw[line width=1pt]  (0,0)   -- \U --\U -- \D -- \D -- \U --  \U --  \D -- \D; 
      \end{tikzpicture}
\end{align*}
%      \begin{align*}
%   	    & \raisebox{-0.5\height}{\includegraphics[scale=1]{figs/oppositeProd_art2_lr.pdf}}\\
%   	     &=\text{rev}\bigl[(\circ\star\circ)\star( (\circ\star\circ)\star \circ)\bigr] = (\circ\star (\circ\star\circ))\star  (\circ\star\circ)\\
%   	      & =\raisebox{-0.5\height}{\includegraphics[scale=1]{figs/oppositeProd_art3_lr.pdf}}
% \end{align*}

\item  \textbf{Reflection} maps. 
    A reflection map, $\refm$,  is motivated by Catalan families that have some sort of geometrical structure for which a vertical or horizontal geometrical reflection can be well defined. 
    For families constructed on circles the reflected structure is (if well defined) the original structure read clockwise, but drawn counter clockwise.  As a product it is denoted by $\refp{\st}$, that is
    \[
        \refp{\st}\,:\,    (u,  v)\mapsto \refm(u\st v)\,.
    \]
    It is only well defined for certain Catalan families.
    Examples of the reflection maps for  Dyck paths and a binary trees, are (illustrated with one example)
\begin{equation*}
\def\S{0.25}
\text{ref}\left(
\begin{tikzpicture}[scale=\S]
    \def\U{++(1,1)} \def\D{++(1,-1)}
    \draw [help lines, gray!50] (0,0) grid(6,2);
    \draw[line width=1pt]  (0,0)   -- \U -- \D -- \U -- \U -- \D -- \D ;  
\end{tikzpicture}\right)
 \quad = \quad
\begin{tikzpicture}[scale=\S]
    \def\U{++(1,1)} \def\D{++(1,-1)}
    \draw [help lines, gray!50] (0,0) grid(6,2);
    \draw[line width=1pt]  (0,0)   -- \U -- \U -- \D -- \D -- \U -- \D ;  
\end{tikzpicture}
\qquad\text{and}\qquad
\text{ref}\left(
 \begin{tikzpicture}[scale=\S]
    \def\L{++(-1,-1)} \def\R{++(1,-1)}\def\Z{8pt}
    % \draw [help lines, gray!50] (0,0) grid(3,2);
    \draw[line width=1pt]  (2,2) -- \L (2,2) --\R (1,1) -- \L (1,1) -- \R  ;  
    \fill[white,draw=black,line width=1pt] (2,2) circle (\Z) ;
    \fill (1,1)  circle (\Z) (0,0) circle (\Z)  (2,0) circle (\Z) (3,1) circle (\Z)  ;  
\end{tikzpicture}\right)
 \quad = \quad
 \begin{tikzpicture}[scale=\S]
    \def\L{++(-1,-1)} \def\R{++(1,-1)}\def\Z{8pt}
    % \draw [help lines, gray!50] (0,0) grid(3,2);
    \draw[line width=1pt]  (1,2) -- \L (1,2) --\R (2,1) -- \L (2,1) -- \R  ;  
    \fill[ white,draw=black,line width=1pt] (1,2)  circle (\Z) ;
     \fill (0,1) circle (\Z)  (1,0) circle (\Z)  (2,1) circle (\Z) (3,0) circle (\Z)  ; 
\end{tikzpicture}
\end{equation*}
    % \begin{center}
    %     \includegraphics[scale=0.8]{figs/refmap_art1_lr.pdf}
    % \end{center}
    defined by a  geometrical reflection across a vertical line.
    If the  Catalan family $\ffm_i$ has a reflection map defined  then the family with the same base set but with the reflection product will be called the  \textbf{reflection family}, denoted $\refp{\ffm}_i$. 
\end{enumerate}

Of primary interest is the relationship between the above  families and any other product that might be defined on the same base set. 
For example, many Catalan families have a left and right version of the product. 
If we start from, say the right product, is it or its opposite/reverse/reflected version related to the left product? \noindent We will consider two examples:

\vspace{1em}\noindent\emph{Dyck Paths.}
It is well know that Dyck paths can be factored in two ways giving rise to two different magma products, that is,  a left-product, which in this section we will denote $\lst$  and a right product, $\rst$. 
 Thus,  schematically, if $d_1$ and $d_2$ are two Dyck paths, then the two products are
 
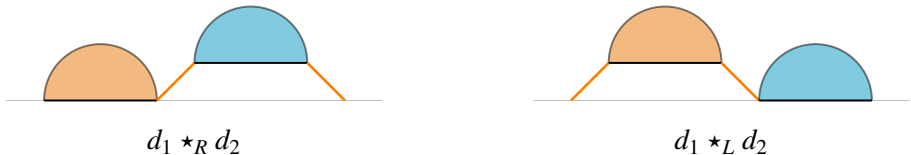
\begin{equation*}
    \begin{tikzpicture}[line width=0.75pt,scale=1] 
        \draw [gray!50,line width=0.25pt] (-0.5,0) -- (4.5,0);
        \fill [color_left,draw=black] (0,0) arc [radius=0.75,start angle = 180, end angle =0] ;
        \draw (0,0) -- (1.5,0);
        \draw[color=orange, line width=1pt] (1.5,0) -- (2,0.5);
        \fill [color_right,draw=black] (2,0.5) arc [radius=0.75,start angle = 180, end angle =0] ;
        \draw (2,0.5) -- (3.5,0.5);
        \draw[color=orange, line width=1pt] (3.5,0.5) -- (4,0);
        \node [below ] at (2,-0.25) {$d_1\star_R d_2$};
    \end{tikzpicture}
    \hspace{2cm}
    \begin{tikzpicture}[line width=0.75pt,scale=1] 
         \draw [gray!50,line width=0.25pt] (-0.5,0) -- (4.5,0);
        \draw[color=orange, line width=1pt] (0,0) -- (0.5,0.5);
        \fill [color_left,draw=black] (0.5,0.5) arc [radius=0.75,start angle = 180, end angle =0] ;
         \draw (0.5,0.5)-- (2,0.5);
        \draw[color=orange, line width=1pt] (2,0.5) -- (2.5,0);
        \fill [color_right,draw=black] (2.5,0) arc [radius=0.75,start angle = 180, end angle =0] ;
       \draw (2.5,0) -- (4,0);
        \node [below ] at (2,-0.25) {$d_1\star_L d_2$};
    \end{tikzpicture}
\end{equation*}
% \begin{center}
%  \includegraphics[scale=\fscale]{figs/DyckPathProdLeftRight_lr.pdf} 
% \end{center}
Clearly these two products define two different free magmas (with the same base set), denote these  $ (\text{\fref{fam_dp}},\st_L)$ and $(\text{\fref{fam_dp}},\st_R)$. 
 
It would appear the the left product and right products are related by a vertical reflection, however it is not quite as simple. In fact the relationship between the left and right Dyck path products is
\begin{equation}
    \opp{\st}_L=\refp{\st}_R
\end{equation}
or equivalently,  $ \st_L=\opp{(\refp{\st}_R)}$. 
This is because a reflection swaps $d_1$ and $d_2$ which the opposite product swaps back: $\refm{(d_1\rst d_2)}=\refm{(d_2)}\lst \refm{(d_1)}=\opm( \refm{(d_1)}\lst \refm{(d_2)})$.

\vspace{1em}\noindent\emph{Complete Binary Trees.}
For complete binary trees there is no natural left and right products but there is an  obvious reflection map obtained by `reflecting' the tree across a vertical axis through the root. 
This is equivalent to recursively swapping left and right sub-trees, that is,  we can define the reflection map recursively by
\begin{equation} \label{eq_treeref}
     \refm{(t_1\st t_2)} =\refm{(t_2)}\st \refm{(t_1)}\quad\text{and}\quad \refm{(\circ )} = \circ
\end{equation}
where $\circ$ is the generator  and $\st$ the product for \fref{fam_bt}. From this definition its clear that for the complete binary tree product
\begin{equation}\label{eq_treerefp}
    \refp{\st} = \revp{\st}\ne \opp{\st} \,.
\end{equation}
 
%=============================================
\section{Embedding bijections}
\label{sec_embedd}
%=============================================
%
 
% In order to use the universal bijection one has to first factorise an object.  
% Factorisation of a Catalan object is clearly a recursive algorithm. 
% Whilst such an algorithm can always be programmed it is also desirable to implement the process `by hand'. 
% For some Catalan families this a substantially easier that others. For example, the factorisation of a binary tree is trivial whilst for other families, such as Kepler towers it is not.

% However, it may be that a family with a complicated factorisation algorithm may have a simple bijection to a family where the factorisation is a lot simpler.

In this section we address the problem of factorisation in a particular way, which is to   embed  the   objects of one family (whose factorisation is simple) into the objects of another (whose factorisation may be  more difficult).  
% We will define a process which can always be implemented for any pair of families. 

Embedding bijections have previously been defined   (for example p58 of \cite{Stanley:2015aa} and \cite{Comtet:1974aa}) however, in those cases the primary aim   is prove that the given set of objects is indeed Catalan by establishing a bijection to a known Catalan family.
Here the primary  motivation is factorisation. 
If the family is `geometrical' (ie.\ defined by sets of points rather than by, say,  a sequence) we assume we know the product rule  but now want, if possible, a geometrical algorithm for factorisation.
A geometric family with one of the simplest    factorisation algorithms is the family of  complete binary trees (CBTs) -- see \fref{fam_bt} in the appendix. 
For this reason we will focus on embedding CBTs into other geometric families (there is also a possible Category Theory motivation -- see below).
% \footnote{CBTs are also closely related to the definition of a product in Category Theory.}. 

Tree traversal of the embedded CBT will give any of the three decomposition   representations \eqref{eq_prodrep1}, \eqref{eq_prodrep2} and \eqref{eq_prodrep3}   as illustrated in \figref{fig:treeTraversal}.
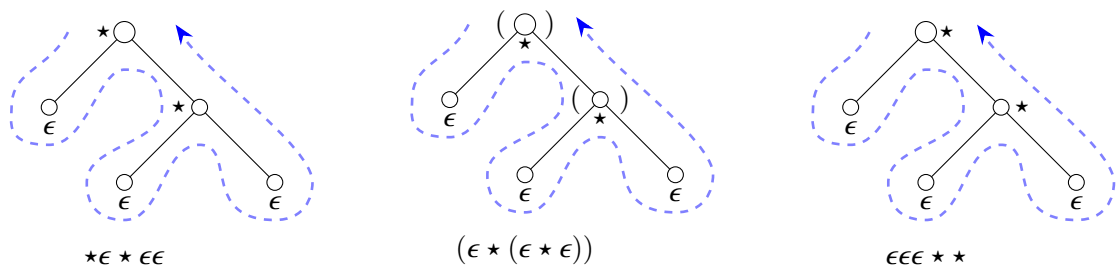
\begin{figure}[ht]
    \centering
\begin{tikzpicture}[>=Stealth]
\def\in{++(0,-0.5)}
\def\pre{++(-0.5,0)}
\def\post{++(0.5,0)}
%nodes
\coordinate (root) at (1,2);
\coordinate (L) at (0,1);
\coordinate (R) at (2,1);
\coordinate (RL) at (1,0);
\coordinate (RR) at (3,0);
% \draw[help lines] (0,0) grid (3,3);
\draw (root) -- (L) (root)--(R)--(RR) (R)--(RL);
\fill[white,draw=black] (root) circle (4pt);
\fill[white,draw=black] (L) circle (3pt);
\fill[white,draw=black] (R) circle (3pt);
\fill[white,draw=black] (RL) circle (3pt);
\fill[white,draw=black] (RR) circle (3pt);
\path (root)  node[left,inner sep=5pt] {$\star$};
\path (R) node[left,inner sep=5pt] {$\star$};
\path (L) node[below,inner sep=5pt] {$\epsilon$}; 
\path (RL) node[below,inner sep=5pt] {$\epsilon$}; 
\path (RR) node[below,inner sep=5pt] {$\epsilon$}; 
\draw[blue!50,dashed,line width=1pt] (0.25,2) 
.. controls (0 ,1.5) and (-0.5,1.5)..   (-0.5,1 )  
.. controls (-0.5 ,0.8) and (-0.6,0.5)..   (0,0.5)  
.. controls (0.5,0.5) and (0.65,1.5).. (1,1.5) 
.. controls (1.2,1.5) and (1.5,1.5).. (1.5,1) 
.. controls (1.5,0.5) and (0.5,0.5).. (0.5,0) 
.. controls (0.5,-0.1) and (0.5,-0.5).. (1 ,-0.5) 
.. controls (1.5,-0.5) and (1.5, 0.5).. (2 , 0.5)
.. controls (2.5,0.5) and (2,-0.5).. (3 ,-0.5)
.. controls (3.4,-0.5) and (3.5,-0.2).. (3.5,0) 
.. controls (3.5,0.5) and (2 ,1.5).. (1.75 ,2);
\draw[line width=1pt,blue] (1.75,2) edge [->] (1.68,2.1);
\path (1,-1) node {$\star\epsilon\star\epsilon\epsilon $};
\end{tikzpicture}
\hspace{1cm}
\begin{tikzpicture}[>=Stealth]
\def\in{++(0,-0.5)}
\def\pre{++(-0.5,0)}
\def\post{++(0.5,0)}
%nodes
\coordinate (root) at (1,2);
\coordinate (L) at (0,1);
\coordinate (R) at (2,1);
\coordinate (RL) at (1,0);
\coordinate (RR) at (3,0);
% \draw[help lines] (0,0) grid (3,3);
\draw (root) -- (L) (root)--(R)--(RR) (R)--(RL);
\fill[white,draw=black] (root) circle (4pt);
\fill[white,draw=black] (L) circle (3pt);
\fill[white,draw=black] (R) circle (3pt);
\fill[white,draw=black] (RL) circle (3pt);
\fill[white,draw=black] (RR) circle (3pt);

\path (root)  node[left,inner sep=7pt] {$($};
\path (root)  node[below,inner sep=5pt] {$\star$};
\path (root)  node[right,inner sep=7pt] {$)$};
\path (R) node[left,inner sep=7pt] {$($};
\path (R) node[below,inner sep=5pt] {$\star$};
\path (R) node[right,inner sep=7pt] {$)$};
\path (L) node[below,inner sep=5pt] {$\epsilon$}; 
\path (RL) node[below,inner sep=5pt] {$\epsilon$}; 
\path (RR) node[below,inner sep=5pt] {$\epsilon$}; 
\draw[blue!50,dashed,line width=1pt] (0.25,2) 
.. controls (0 ,1.5) and (-0.5,1.5)..   (-0.5,1 )  
.. controls (-0.5 ,0.8) and (-0.6,0.5)..   (0,0.5)  
.. controls (0.5,0.5) and (0.65,1.5).. (1,1.5) 
.. controls (1.2,1.5) and (1.5,1.5).. (1.5,1) 
.. controls (1.5,0.5) and (0.5,0.5).. (0.5,0) 
.. controls (0.5,-0.1) and (0.5,-0.5).. (1 ,-0.5) 
.. controls (1.5,-0.5) and (1.5, 0.5).. (2 , 0.5)
.. controls (2.5,0.5) and (2,-0.5).. (3 ,-0.5)
.. controls (3.4,-0.5) and (3.5,-0.2).. (3.5,0) 
.. controls (3.5,0.5) and (2 ,1.5).. (1.75 ,2);
\draw[line width=1pt,blue] (1.75,2) edge [->] (1.68,2.1);
\path (1,-1) node {$(\epsilon\star(\epsilon\star\epsilon)) $};
\end{tikzpicture}
\hspace{1cm}
\begin{tikzpicture}[>=Stealth]
\def\in{++(0,-0.5)}
\def\pre{++(-0.5,0)}
\def\post{++(0.5,0)}
%nodes
\coordinate (root) at (1,2);
\coordinate (L) at (0,1);
\coordinate (R) at (2,1);
\coordinate (RL) at (1,0);
\coordinate (RR) at (3,0);
% \draw[help lines] (0,0) grid (3,3);
\draw (root) -- (L) (root)--(R)--(RR) (R)--(RL);
\fill[white,draw=black] (root) circle (4pt);
\fill[white,draw=black] (L) circle (3pt);
\fill[white,draw=black] (R) circle (3pt);
\fill[white,draw=black] (RL) circle (3pt);
\fill[white,draw=black] (RR) circle (3pt);

\path (root)  node[right,inner sep=5pt] {$\star$};

\path (R) node[right,inner sep=5pt] {$\star$};

\path (L) node[below,inner sep=5pt] {$\epsilon$}; 
\path (RL) node[below,inner sep=5pt] {$\epsilon$}; 
\path (RR) node[below,inner sep=5pt] {$\epsilon$}; 
\draw[blue!50,dashed,line width=1pt] (0.25,2) 
.. controls (0 ,1.5) and (-0.5,1.5)..   (-0.5,1 )  
.. controls (-0.5 ,0.8) and (-0.6,0.5)..   (0,0.5)  
.. controls (0.5,0.5) and (0.65,1.5).. (1,1.5) 
.. controls (1.2,1.5) and (1.5,1.5).. (1.5,1) 
.. controls (1.5,0.5) and (0.5,0.5).. (0.5,0) 
.. controls (0.5,-0.1) and (0.5,-0.5).. (1 ,-0.5) 
.. controls (1.5,-0.5) and (1.5, 0.5).. (2 , 0.5)
.. controls (2.5,0.5) and (2,-0.5).. (3 ,-0.5)
.. controls (3.4,-0.5) and (3.5,-0.2).. (3.5,0) 
.. controls (3.5,0.5) and (2 ,1.5).. (1.75 ,2);
\draw[line width=1pt,blue] (1.75,2) edge [->] (1.68,2.1);
\path (1,-1) node {$\epsilon\epsilon\epsilon\star\star $};
\end{tikzpicture}
    \caption{Counter-clockwise traversal of complete binary trees and the three ways to write the product: (left) pre-fix form, (middle) in-fix form and (right) post-fix form}
    \label{fig:treeTraversal}
\end{figure}
These three forms correspond to placing the product symbol $\star$, in one of  three positions relative to an internal  node of the tree: the \textbf{pre-fix position} (to the left), the \textbf{in-fix position} (below) and the \textbf{post-fix position} (to the right). For the in-fix order product representation, brackets must also be inserted (using the pre- and post-fix positions) to remove ambiguity. The generator $\eps$ is placed below each leaf.

Even though the embedding algorithm can be defined for any pair of  families, currently the embedding gives a simpler factorisation process only in some cases\footnote{There are over  22791 pairs - not all of which have been checked.}.
The case where the Catalan family is not geometrical (eg.\ a sequence) will be discussed further below.

%=================================================================
\subsection*{Geometrical families}
%=================================================================
\label{sec_geofam}

\newcommand{\abet}{\mathbb{A}}
\newcommand{\mon}{\mathcal{W}}
\newcommand{\monin}{\mathbb{W}^\text{in}}
\newcommand{\monpre}{\mathbb{W}^\text{pre}}
\newcommand{\monpost}{\mathbb{W}^\text{post}}
\newcommand{\embed}{\alpha}

\newcommand{\prodgeo}{\mathbb{P}}
\newcommand{\gengeo}{\mathbb{G}}
\newcommand{\geomap}{\Gamma}
\newcommand{\geopart}{\Pi}

\newcommand{\pshape}{\rho}
\newcommand{\gshape}{\gamma}
\newcommand{\smap}{\Theta}

In this section we take a more informal approach, defining by examples rather than with formal definitions.
We assume all free magmas $(\ffm,\st)$, have a single generator and the elements have some sort of ``geometrical'' representation -- defined by sets of points rather than sequences of numbers. 

If the set of objects is geometrical, the definition of the  magma generator and product gives us two important sets of information: 1) from the geometrical definition of the generator, we know  which parts of the geometry of the object are associated with the generators and 2) from the product definition we know which parts of the geometry of the object are added each time two factors are multiplied. 

For example, for CBTs, the product definition is:
\begin{equation}\label{eq_treescpro}
    \begin{tikzpicture}[style_size3,baseline=0.75cm]
        % \draw [help lines, gray!50] (0,0) grid(3,3);
        \fill [color_left,draw=black] (2,3) -- (3,0) -- (0,1) -- cycle;
        \fill [white,draw=black] (2,3) pic{circle_white} ;
    \end{tikzpicture}
    \qquad\star\quad
    \begin{tikzpicture}[style_size3,baseline=0.75cm]
        % \draw [help lines, gray!50] (0,0) grid(3,3);
        \fill [color_right,draw=black] (1,3) -- (0,0) -- (3,1) -- cycle;
        \fill [white,draw=black] (1,3) pic{circle_white} ;
    \end{tikzpicture}
    \qquad=\quad
    \begin{tikzpicture}[style_size3  ,baseline=0.75cm]
        %   \draw [help lines, gray!50] (0,0) grid(7,4);
        %
        \draw[orange, line width=1pt] (3.5,4) -- (2,3) (3.5,4) -- (5,3);
        \path (3.5,4) pic{cbt_root_orange} ;
        \node at (3.5,4.5) {$a$}; 
        \fill [color_left,draw=black] (2,3) -- (3,0) -- (0,1) -- cycle;
        \fill [draw=black] (2,3) pic{circle_white} ;
        \node at (1.5,3) {$b$};
        \fill [color_right,draw=black] (5,3) -- (4,0) -- (7,1) -- cycle;
        \fill [draw=black] (5,3) pic{circle_white}  ;
        \node at (5.5,3) {$c$};
    \end{tikzpicture}
\end{equation}
% \begin{equation} 
%     \begin{tikzpicture}[scale=.5,line width=0.5pt,baseline=0.75cm]
%         % \draw [help lines, gray!50] (0,0) grid(3,3);
%         \fill [color_left,draw=black] (2,3) -- (3,0) -- (0,1) -- cycle;
%         \fill [white,draw=black] (2,3) circle (5pt);
%     \end{tikzpicture}
%     \quad\star\quad
%     \begin{tikzpicture}[scale=.5,line width=0.5pt,baseline=0.75cm]
%         % \draw [help lines, gray!50] (0,0) grid(3,3);
%         \fill [color_right,draw=black] (1,3) -- (0,0) -- (3,1) -- cycle;
%         \fill [white,draw=black] (1,3) circle (5pt);
%     \end{tikzpicture}
%     \quad=\quad 
%     \begin{tikzpicture}[scale=.5,line width=0.5pt,baseline=0.75cm]
%         %   \draw [help lines, gray!50] (0,0) grid(7,4);
%         %
%         \draw[orange, line width=1pt] (3.5,4) -- (2,3) (3.5,4) -- (5,3);
%         \fill [white, draw=orange] (3.5,4) circle (5pt) ;
%         \node[above] at (3.5,4) {$a$};
%         %
%         \fill [color_left,draw=black] (2,3) -- (3,0) -- (0,1) -- cycle;
%         \fill [draw=black] (2,3) circle (5pt);
%         \node[above] at (2,3) {$b$};
%         %
%         \fill [color_right,draw=black] (5,3) -- (4,0) -- (7,1) -- cycle;
%         \fill [draw=black] (5,3) circle (5pt) ;
%         \node[above] at (5,3) {$c$};
%     \end{tikzpicture}  
% % \begin{equation}\label{eq_treesgep}
% %      \eqfig{0.7}{cbt-prod-geo.pdf}
% \end{equation}
%
from which it is clear that the new geometry added by the product is the root node, $a$, and the two edges $(a,b)$ and $(a,c)$.  The generator geometry, $\eps=\tikz{\path pic{gencbt}}$  is a leaf node of the tree.
 
In order to clearly state which parts of the object arise from the generators and which from the products it is necessary to choose a particular way of drawing (representing) the objects of the family.
For any particular Catalan family there may be several very similar  ways of  drawing  the objects.  
Thus, for Dyck paths all the   diagrams below clearly  define the \emph{same} path:
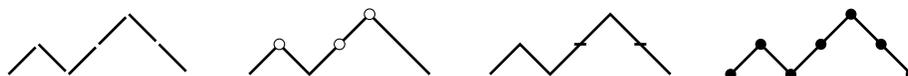
\begin{center}
\begin{tikzpicture}[scale=0.4]
    \def\ln{0.9} \def\gp{0.1}
    \def\UPF{-- ++(1,1)    }
    \def\DNF{-- ++(1,-1)    }
    \def\UP{-- ++(\ln,\ln) ++(\gp,\gp)  }
    \def\DN{-- ++(\ln,-\ln) ++(\gp,-\gp)  }
    \def\UPC{  -- ++(\ln,\ln) ++(\gp,\gp) [fill=black!10] circle [radius=8pt ]  }
    \def\DNC{  -- ++(\ln,-\ln) ++(\gp,-\gp) circle [radius=8pt]   }
    \def\UPD{-- ++(1, 1)   ++(-0.2,0) -- ++( 0.4,0)  ++( -0.2,0)}
    \def\DND{-- ++(1,-1)    ++(-0.2,0) -- ++( 0.4,0)  ++( -0.2,0) }
    \def\UPB{-- ++(1,1)   circle [radius=4pt] [fill=black] }
    \def\DNB{-- ++(1,-1)   circle [radius=4pt] [fill=black] }
    \def\UW{  ++(1,1)   circle [radius=5pt]  }
    \def\DW{  ++(1,-1)   circle [radius=5pt]   }
    \path [draw,line width=1pt,cap=rounded] (0,0)  \UP \DN \UP \UP \DN \DN ;
    \path [draw,line width=1pt] (8,0)  \UPF \DNF \UPF \UPF \DNF \DNF;
    \path  [fill=white,draw=black]  (8,0) \UW ++(1,-1) \UW \UW ++(1,-1) ++(1,-1) ;
    %
    % \path [draw,line width=1pt] (8,0)  circle [radius=4pt]   \UPC \DNC \UPC \UPC \DNC \DNC ;
    \path [draw,line width=1pt] (16,0)  \UPF  \DNF  \UPD \UPF  \DND \DNF  ;
    \path [draw,line width=1pt] (24,0)  circle [radius=4pt] [fill=black] \UPB \DNB \UPB \UPB \DNB \DNB ;
%   \path [line width=1pt,draw] (0,0) [stepsty] -- ++(1,1) [fill=blue] circle [ radius=9pt] [draw] -- ++(1,-1) -- ++(1,1) -- ++(1, 1) -- ++(1,-1) -- ++(1,-1);;
%   \draw [shide] circle [ radius=9pt];
%   \draw [stepsty] -- ++(1,-1) -- ++(1,1) -- ++(1, 1) -- ++(1,-1) -- ++(1,-1); 
\end{tikzpicture} 
\end{center}
similarly, for staircase polygons (family \fref{fam_sp} in the appendix) all the diagrams below represent the same staircase polygon:
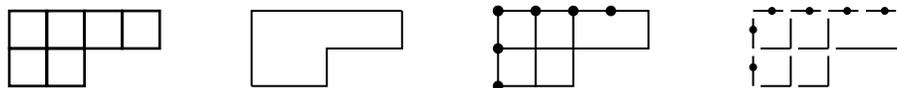
\begin{center}
    \def\Sp{1cm}
    \begin{tikzpicture}[style_size1,baseline=0cm] 
        \def\R{ rectangle ++(0.5,0.5)} \def\T{++(0,-0.5)}
        % \draw [help lines, gray!50] (0,0) grid(4,2);
        \draw  (0 ,0)  \R \T \R;    
        \draw (0,0.5) \R \T \R \T \R \T \R;
    \end{tikzpicture}
    \hspace{\Sp}
    \begin{tikzpicture}[line width=0.75pt,scale=1,baseline=0cm]  
        % \draw [help lines, gray!50] (0,0) grid(4,2);
        \draw (0,0) -- (0,1) --(2,1);
        \draw (0,0) -- (1,0) -- (1,0.5) -- (2,0.5) -- (2,1);
    \end{tikzpicture}
    \hspace{\Sp}
    \begin{tikzpicture}[line width=0.75pt,scale=1,baseline=0cm] 
        \def\R{ rectangle ++(0.5,0.5)} \def\T{++(0,-0.5)} \def\D{2pt}
        % \draw [help lines, gray!50] (0,0) grid(4,2);
        \draw  (0 ,0)  \R \T \R;    
        \draw (0,0.5) \R \T \R ;
        \draw (1,0.5) rectangle (2,1);
        \fill (0,0) circle (\D) (0,0.5) circle (\D) (0,1) circle (\D);
        \fill (0,1) circle (\D) (0.5,1) circle (\D) (1,1) circle (\D) (1.5,1) circle (\D);
    \end{tikzpicture}
     \newcommand{\Lc}[2]{\draw (#1,#2) -- (#1-0.4,#2) (#1,#2) -- (#1 ,#2+0.4);}
    \hspace{\Sp}
\begin{tikzpicture}[line width=0.75pt,scale=1,baseline=0cm]  
    \def\R{ rectangle ++(0.5,0.5)} \def\T{++(0,-0.5)} \def\D{2pt} % \draw [help lines, gray!50] (0,0) grid(4,2);
    \fill[draw=white,line width=1pt]  (0,0.25) circle (\D) (0,0.75) circle (\D) (0.25,1) circle (\D);
    \fill[draw=white,line width=1pt]  (0.75,1) circle (\D) (1.25,1) circle (\D) (1.75,1) circle (\D);
    \draw (0,0) --(0,0.4) (0,0.5) -- (0,0.9);
    \draw (0,1) -- (0.4,1) (0.5,1) -- (0.9,1) (1,1) -- (1.4,1) (1.5,1) -- (1.9,1);
    \draw [white,line width=2pt] (0.4,0.4) -- (2,0.4);
    \Lc{0.5}{0} \Lc{1}{0} \Lc{0.5}{0.5} \Lc{1}{0.5} 
    \draw (2,0.5) -- (1.1,0.5) (2,0.5) -- (2,0.9);
\end{tikzpicture}
\end{center}
% \begin{center}
%     \includegraphics[height=1cm]{figtemp/staricase_equiv.png}
% \end{center}
In both examples the diagrams only differ by trivial bijections.
Clearly a given family can be considered an  equivalence class of diagrams and thus we are   free to  choose the most convenient  canonical representative. 

For this section we will choose a canonical form such that the points of the object  partition into subsets, each subset corresponding to elements of either generator geometry or product geometry. 
We will call this canonical form the \textbf{magma form}.
% which is informally defined as the form of the object that arises by defining the product such that the points contributed by the generator geometry and the points contributed by the product geometry remain disjoint. 
% To simplify the discussion below we assume the magma product has been defined so that the partitioning of the points of the object is clear. 
For example, for the nested matchings (\fref{fam_ld} in the appendix) the generator geometry arises from $\eps=\tikz{\path pic{gennm}}$  and the product definition is 
\begin{equation}\label{eq_mproddef}   
\def\bsl{0pt}
\begin{tikzpicture}[style_size4,baseline=\bsl] 
    \fill [color_left,draw=black] (1,0) arc [radius=1,start angle = 180, end angle =0];
    \draw (0,0) -- (4,0);
    \fill (0,0) pic{circle_black}  ;
    \node[below,anchor=mid  ] at (0,-0.7) {$a$};
    % \fill[color=white] (4,0) pic{circle_black} ;
    \node[below,anchor=mid ] at (4,-0.7) {$b$};
    \end{tikzpicture}
\star 
\begin{tikzpicture}[style_size4,baseline=\bsl] 
    \fill [color_right,draw=black] (6,0) arc [radius=1,start angle = 180, end angle =0];
    \draw (5,0) -- (9,0);
    \fill (5,0) pic{circle_black}  ;
    \node[below,anchor=mid  ] at (5,-0.7) {$c$};
    % \fill[color=white] (9,0) pic{circle_white} ;
    \node[below,anchor=mid  ] at (9,-0.7) {$d$};
\end{tikzpicture}
= 
\begin{tikzpicture}[style_size4,baseline=\bsl] 
    \node[below ,anchor=mid ] at (10,-0.7) {$a$};
    \fill [color_left,draw=black]  (11,0) arc [radius=1,start angle = 180, end angle =0];
    \fill [color_right,draw=black]  
        (15,0) arc [radius=1,start angle = 180, end angle =0];
    % \fill (5,0) circle [radius=4pt];
    \draw (10,0) -- (18,0);
    \fill (10,0) pic{circle_black}  ;   
    \draw[color=productColor,line width=1pt] (14,0) arc [start angle=180, end angle = 0, radius=2];
    \fill (14,0) pic{circle_black}   ;
    \node[below,anchor=mid  ] at (14,-0.7) {$bc$};
    \fill [color=productColor] (18,0) pic{circle_product} ;
    \draw [color=productColor,line width=1pt] (18,0) -- (19,0);
    % \fill[color=white] (19,0) pic{circle_black}  ;
    \node[below,anchor=mid  ] at (18,-0.7) {$d$};
    \node[below,anchor=mid  ] at (19,-0.7) {$e$};
\end{tikzpicture}
\end{equation}
with  the product geometry   shown in orange.  It is clear that the nested matching below (left) can be partitioned as shown below (right):
\begin{equation}\label{eq_nmparv} 
\begin{tikzpicture}[style_size3, node font=\small,inner sep=5pt] 
    \draw (0,0) -- (7,0);
    \foreach \x in {1,...,8} {
        \path (\x,0) ++(-1,0) node[below,color=gray]   {\x};
    };
        \foreach \x in {0,...,6} {
        \fill (\x,0) pic{circle_black};
    };
    \draw  (2,0) arc [start angle=0, end angle = 180, radius=0.5];
    \draw  (5,0) arc [start angle=0, end angle = 180, radius=0.5];
    \draw (6,0) arc [start angle=0, end angle = 180, radius=1.5];
\end{tikzpicture}
        \hspace{2cm}
\begin{tikzpicture}[style_size3, node font=\small,inner sep=5pt] 
    \fill (0,0) pic{circle_black} [draw] (0,0) -- +(0.8,0);
    \fill (1,0) pic{circle_black} [draw] (1,0) -- +(0.8,0);
    % start prod geo 
    \draw[size_prod,productColor]  (2,0) arc [start angle=0, end angle = 160, radius=0.5] ;
    \fill  (2,0) pic{circle_product} ;
    \path (2,0) [draw=productColor,line width=1.5pt]  -- +(0.8,0);
    %end prod geo
    \fill (3,0) pic{circle_black} [draw] (3,0) -- +(0.8,0);
     \fill (4,0) pic{circle_black} [draw] (4,0) -- +(0.8,0);
    \foreach \x in {1,...,8} {
        \path (\x,0) ++(-1,0) node[below,color=gray]   {\x};
    };
        % start prod geo 
    \draw[productColor,line width=1.5pt]  (5,0) arc [start angle=0, end angle = 160, radius=0.5] ;
    \fill  (5,0) pic{circle_product} ;
    \path (5,0) [draw=productColor,line width=1.5pt]  -- +(0.8,0);
    %end prod geo
    % \draw  (5,0) arc [start angle=0, end angle = 180, radius=0.5];
    % start prod geo 
    \draw[productColor,line width=1.5pt]  (6,0) arc [start angle=0, end angle = 172, radius=1.5] ;
    \fill  (6,0) pic{circle_product} ;
    \path (6,0) [draw=productColor,line width=1.5pt]  -- +(0.8,0);
    %end prod geo
    % \draw (6,0) arc [start angle=0, end angle = 180, radius=1.5];
\end{tikzpicture}
\end{equation}
where gaps in the geometry are used to clarify the partitioning (the labelling of the nodes is used below).
% \begin{center}
%     \includegraphics[width=8cm]{figtemp/nestematcg_geo.png}
% \end{center}

\newcommand{\scgenshape}[1]{
\begin{tikzpicture}[scale=#1]
 \fill (0.4,0) arc [radius=0.1,start angle = 180, end angle= 0]   ;
  \draw [thick] (0,0)   -- (1,0);
\end{tikzpicture}
}
\newcommand{\scprodshape}[1]{
\begin{tikzpicture}[scale=#1]
  \draw [thick] (0,0)   -- (1,0) -- (1,0.5);
   \fill (1,0) circle [radius=0.1]   ;
\end{tikzpicture}
}
% With labelled roots
\newcommand{\scgenshapelab}[2]{
\begin{tikzpicture}[scale=#1,baseline=-2ex]
 \fill (0.35,0) arc [radius=0.15,start angle = 180, end angle= 0]   ;
 \node [below] at  (0.5,0) {#2};
  \draw [thick] (0,0)   -- (1,0);
\end{tikzpicture}
}
\newcommand{\scvgenshapelab}[2]{
\begin{tikzpicture}[scale=#1,baseline=-4ex]
 \fill (0,0.35) arc [radius=0.15,start angle = 270, end angle= 90]   ;
  \node [left] at  (0,0.5) {#2};
  \draw [thick] (0,0)   -- (0,1);
\end{tikzpicture}
}
\newcommand{\scprodshapelab}[2]{
\begin{tikzpicture}[scale=#1,baseline=-0ex]
  \draw [thick] (0,0)   -- (1,0) -- (1,0.5);
   \fill (1,0) circle [radius=0.15]   ;
    \node [below] at (1,0) {#2};
\end{tikzpicture}
}

Another example are staircase polygons (\fref{fam_sp} in the Appendix). The generator geometry arises from  $\eps=\tikz{\path pic[scale=0.8]{gensp}}$ and the product definition is:
\begin{equation} \label{eq_scpmf} 
    \def\tsc{0.5}
    \def\bsl{2ex}
    \def\lwd{0.5pt}
    \begin{tikzpicture}[style_size3,baseline=\bsl]
      \path [fill,color_left,draw=black]  
      (0,0) -- (0,1) -- (1,2) -- (2,2) --(2,1) --(1,0)--++(-1,0);
    \end{tikzpicture}
   \quad \star\quad
    \begin{tikzpicture}[style_size3,baseline=\bsl]
    % \draw [help lines, gray!50] (0,0) grid(3,3);
      \path [fill,color_right,draw=black ] 
      (0,0) -- (0,1) -- (1,2) -- (2.5,2) --(2.5,1) --(1.5,0)--++(-1.5,0) ;
    \end{tikzpicture}
   \quad = \quad
    \begin{tikzpicture}[style_size3,baseline=\bsl]
    % \draw [help lines, gray!50] (0,0) grid(5,5);    
      \path [fill,color_left,draw=black ] 
        (0,0) -- (0,1) -- (1,2) -- (2,2) --(2,1) --(1,0) --++(-1,0);
      \path (2,1) [ draw=productColor, line width=1pt]  -- ++(1.5,0) --++(0,1);
        \path[fill=productColor]  (3.5,1) pic{circle_product} ;
      \path [fill,color_right,draw=black ]  
        (2,2) -- ++(0,1) -- ++(1,1) -- ++(1.5,0) --++(0,-1) -- ++(-1 ,-1) --++(-1.5,0);
    \end{tikzpicture}
\end{equation}
% \begin{equation}\label{eq_scpmf} 
%     \def\tsc{0.5}
%     \def\bsl{2ex}
%     \def\lwd{0.5pt}
%     \begin{tikzpicture}[scale=\tsc,baseline=\bsl]
%       \path [fill,color=red!30,draw=black, line width=\lwd]  
%       (0,0) -- (0,1) -- (1,2) -- (2,2) --(2,1) --(1,0)--++(-1,0);
%     \end{tikzpicture}
%     \star
%     \begin{tikzpicture}[scale=\tsc,baseline=\bsl]
%     % \draw [help lines, gray!50] (0,0) grid(3,3);
%       \path [fill,color=blue!30,draw=black, line width=\lwd] 
%       (0,0) -- (0,1) -- (1,2) -- (2.5,2) --(2.5,1) --(1.5,0)--++(-1.5,0) ;
%     \end{tikzpicture}
%     =
%     \begin{tikzpicture}[scale=\tsc,baseline=\bsl]
%     % \draw [help lines, gray!50] (0,0) grid(5,5);    
%       \path [fill,color=red!30,draw=black, line width=\lwd] 
%         (0,0) -- (0,1) -- (1,2) -- (2,2) --(2,1) --(1,0) --++(-1,0);
%       \path (2,1) [ draw=orange, line width=1pt]  -- ++(1.5,0) --++(0,1);
%         \path[fill=orange]  (3.5,1) circle [radius=4pt];
%       \path [fill,color=blue!30,draw=black, line width=\lwd]  
%         (2,2) -- ++(0,1) -- ++(1,1) -- ++(1.5,0) --++(0,-1) -- ++(-1 ,-1) --++(-1.5,0);
%     \end{tikzpicture}
% \end{equation}
%
with  the product geometry   shown in orange. 
This leads to a staircase  magma form as shown in the example below with the conventional form (left) and magma form (right):
\begin{equation}\label{eq_spmf}
\begin{tikzpicture}[style_size2,baseline=0cm]
    \draw (0,0) rectangle (1,1);
    \draw (0,1) rectangle (1,2);
    \draw (1,1) rectangle (2,2);
%   \draw [help lines, gray!50] (0,0) grid(6,4);
\end{tikzpicture}
\hspace{3cm}
\begin{tikzpicture}[style_size2,inner sep=2pt,baseline=0cm]
\node (a) at (0,0){ };
\node (b) at (0,1){ };
\node (c) at (0,2){ };
\node (d) at (1,2){ };
\node (e) at (1,1){ };
\node (f) at (1,0){ };
\node (g) at (2,1){ };
\node (h) at (2,2){ };

% \draw  (a) -- (b); 
\fill  (a) pic[rotate=90,scale=0.7]{gensp};
% \draw  (b) -- (c); 
\fill (b) pic[rotate=90,scale=0.7]{gensp};
% \draw  (c) -- (d); 
\fill (c) pic[rotate=0,scale=0.7]{gensp};
% \draw  (d) -- (h); 
\fill (d) pic[rotate=0,scale=0.7]{gensp};

\draw[productColor] (b)--(1,1)--(d);
\draw[productColor]  (a)--(f)-- ++(0,0.8);
 \draw[productColor]  (e) ++(0.2,0) --(g)--(h);
\begin{scope}[inner sep=4pt,color=gray,node font=\small] 
\path (barycentric cs:a=1,b=1)  node[left]  {1};
\path  (barycentric cs:b=1,c=1)  node[left]  {2};
\path  (barycentric cs:c=1,d=1)  node[above]  {3};
\path  (barycentric cs:d=1,h=1)  node[above]  {4};

\path (g) node [right]{5} ;
\path (e) node [anchor=north west]{6} ;
\path (f) node [right]{7} ;
\end{scope}

\fill (1,1) pic{circle_product};
\fill (1,0) pic{circle_product};
\fill (2,1) pic{circle_product};

    %   \draw [help lines, gray!50] (0,0) grid(6,4);
    \end{tikzpicture}
\end{equation}
Again, gaps in the geometry are used to clarify the partitioning and the labels will be used below. 
Note, for the staircase family the generators can be drawn vertically or horizontally.
% \begin{center}
%     \includegraphics[width=5cm]{figtemp/staircase_magmaform.png}
% \end{center}

\newcommand{\cbtprod}[2]{
\begin{tikzpicture}[scale=#1,baseline=5pt]
  \draw (0,0) -- (1,1) -- (2,0);
  \fill (1,1) circle [radius=7pt];
  \node[left] at (1,1) {#2};
\end{tikzpicture}
}
\newcommand{\cbtgen}[2]{
\begin{tikzpicture}[scale=#1,baseline=5pt]
  \fill (1,1) circle [radius=7pt];
  \node[below] (dot)  at (1,1) {#2};
\end{tikzpicture}
}

Thus, for any  geometrical family  we assume that the product definition is such that the  new  geometry is point-wise disjoint from the geometry of the left and right factors. 
% The geometry associated with the generators is   denoted $\gengeo_\ffm(m)$. 
If the object is defined using an ``empty''   generator we will assume the generator can still be associated with a particular position in the diagram and thus can be represented by a single point (which might be drawn as a circle for clarity). 

For any given object, $m$ in some Catalan family $ \ffm$,  let   $\prodgeo_\ffm(m)$ denote  the set of disjoint parts of \textbf{product geometry} of $m$ and let   $\gengeo_\ffm(m)$ denote  the set of disjoint parts of \textbf{generator geometry} of $m$.  
% Rather than give a formal definition of the sets  we illustrate   them with three examples.
We illustrate these sets with several examples.

The first example is the family of CBTs.  From the product definition \eqref{eq_treescpro},
%
% \begin{equation}\ 
%     \begin{tikzpicture}[scale=.5,line width=0.5pt,baseline=0.75cm]
%         % \draw [help lines, gray!50] (0,0) grid(3,3);
%         \fill [color_left,draw=black] (2,3) -- (3,0) -- (0,1) -- cycle;
%         \fill [white,draw=black] (2,3) circle (5pt);
%     \end{tikzpicture}
%     \quad\star\quad
%     \begin{tikzpicture}[scale=.5,line width=0.5pt,baseline=0.75cm]
%         % \draw [help lines, gray!50] (0,0) grid(3,3);
%         \fill [color_right,draw=black] (1,3) -- (0,0) -- (3,1) -- cycle;
%         \fill [white,draw=black] (1,3) circle (5pt);
%     \end{tikzpicture}
%     \quad=\quad 
%     \begin{tikzpicture}[scale=.5,line width=0.5pt,baseline=0.75cm]
%         %   \draw [help lines, gray!50] (0,0) grid(7,4);
%         %
%         \draw[orange, line width=1pt] (3.5,4) -- (2,3) (3.5,4) -- (5,3);
%         \fill [white, draw=orange] (3.5,4) circle (5pt) ;
%         \node[above] at (3.5,4) {$a$};
%         %
%         \fill [color_left,draw=black] (2,3) -- (3,0) -- (0,1) -- cycle;
%         \fill [draw=black] (2,3) circle (5pt);
%         \node[above] at (2,3) {$b$};
%         %
%         \fill [color_right,draw=black] (5,3) -- (4,0) -- (7,1) -- cycle;
%         \fill [draw=black] (5,3) circle (5pt) ;
%         \node[above] at (5,3) {$c$};
%     \end{tikzpicture}  
% % \begin{equation}\label{eq_treesgep}
% %      \eqfig{0.7}{cbt-prod-geo.pdf}
% \end{equation}
%
it is clear that the new geometry added by the product is the root node, $a$, and the two edges $(a,b)$ and $(a,c)$. 
The generator geometry is a leaf node. The elements of the  sets, $\gengeo_\ffm(m)$ and 
$\prodgeo_\ffm(m)$ are the product and generator geometry at all levels of factorisation. For example, for the CBT on the left, the partitioning is shown on the right
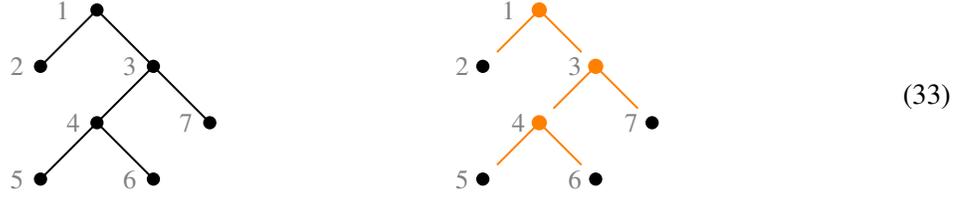
\begin{equation}\label{eq_treesgepff3} 
\begin{tikzpicture}[style_size2,baseline=1cm,inner sep=6pt]
    % \def\L{++(-1,-1)} \def\R{++(1,-1)}\def\Z{4pt}
%  \draw [help lines, gray!50] (0,-1) grid(6,4);
\coordinate  (a) at (1,3)  ;
\coordinate (b) at (0,2) ;
\coordinate (c) at (2,2) ;
\coordinate (d) at (1,1) ;
\coordinate (e) at (0,0) ;
\coordinate (f) at (2,0) ;
\coordinate (g) at (3,1) ; 

\begin{scope}[color=gray,anchor=east ]
\node[ circle] (na) at (1,3) {1};
\node   (nb) at (0,2){2};
\node  (nc) at (2,2){3};
\node (nd) at (1,1){4};
\node  (ne) at (0,0){5};
\node  (nf) at (2,0){6};
\node  (ng) at (3,1){7};
\end{scope}

\draw (a) -- (c);\fill (c) pic{circle_black};
\draw (c) -- (d);\fill (d) pic{circle_black};
\draw (c) -- (g);\fill (g) pic{circle_black};
\draw (d) -- (e);\fill (e) pic{circle_black};
\draw (d) -- (f);\fill (f) pic{circle_black};
\fill (b) pic{circle_black};

\draw   (a)--(b);\fill[white,draw=black, line width=1pt] (a) pic{circle_black};

\end{tikzpicture}
   \hspace{3cm}
\begin{tikzpicture}[style_size2,baseline=1cm,  >={Stealth[length=1.5mm]},inner sep=5pt  ]
%  \draw [help lines, gray!50] (0,-1) grid(6,4);
\coordinate  (a) at (1,3)  ;
\coordinate (b) at (0,2) ;
\coordinate (c) at (2,2) ;
\coordinate (d) at (1,1) ;
\coordinate (e) at (0,0) ;
\coordinate (f) at (2,0) ;
\coordinate (g) at (3,1) ; 

\node (ap) at (1,3) {} ;
\node (bp) at (0,2){} ;
\node (cp) at (2,2){} ;
\node (dp) at (1,1){} ;
\node (ep) at (0,0){} ;
\node (fp) at (2,0){} ;
\node (gp) at (3,1){} ; 

\begin{scope}[color=gray ]
\node[left,circle] (na) at (1,3) {1};
\node[left] (nb) at (0,2){2};
\node[left] (nc) at (2,2){3};
\node[left](nd) at (1,1){4};
\node[left] (ne) at (0,0){5};
\node[left] (nf) at (2,0){6};
\node[left] (ng) at (3,1){7};
\end{scope}

\draw[productColor] (a) -- (cp);
\draw[productColor] (c) -- (dp);
\draw[productColor] (c) -- (gp);
\draw[productColor] (d) -- (ep);
\draw[productColor] (d) -- (fp);

\draw[productColor]  (a)--(bp);

\fill (b) pic{circle_black};
\fill  (g) pic{circle_black};
\fill (e) pic{circle_black};
\fill  (f) pic{circle_black};

\fill (a) pic{circle_product};
\fill (c) pic{circle_product};
\fill (d) pic{circle_product};

\end{tikzpicture}
\end{equation}
%
% \begin{equation}\label{eq_treesgepff}
%     \includegraphics[width=4cm]{figtemp/cbt_prodgeo.png}
% \end{equation}
and thus the partition sets are
\tikzset{tree_prod/.pic={
\draw[style_size4] (0,0) -- (1,1) -- (2,0) ;
\path[style_size4] (1,1) pic{circle_product};
}}
\begin{equation}  
\prodgeo_\ffm = \Bigl\{ 
\tikz{\path[productColor] (0,0) pic{tree_prod};
\path[style_size4,  node font=\small,color=gray] (1,1) node[above] {1};
},\,
\tikz{\path[productColor]  (0,0) pic{tree_prod};
\path[style_size4,  node font=\small,color=gray] (1,1) node[above] {3};
},\,
\tikz{\path[productColor]  (0,0) pic{tree_prod};
\path[style_size4,  node font=\small,color=gray] (1,1) node[above] {4};
}
\Bigr\}
\qquad\text{ and }\qquad
%     \prodgeo_\ffm = 
% \bigl\{ \tikz{\path (0,0) pic{tree_prod};},
% \cbtprod{0.3}{\tiny 1},\, \cbtprod{0.3}{\tiny 3},\, \cbtprod{0.3}{\tiny 4} 
% \bigr\}
% \qquad\text{ and }\qquad
 \gengeo_\ffm =\Bigl\{
\tikz{\path[style_size4,  node font=\small,color=gray] 
pic{circle_black}  node[above]{2}; },\,
\tikz{\path[style_size4,  node font=\small,color=gray] 
pic{circle_black}  node[above]{5}; },\,
\tikz{\path[style_size4,  node font=\small,color=gray] 
pic{circle_black}  node[above]{6}; },\,
\tikz{\path[style_size4,  node font=\small,color=gray] 
 pic{circle_black}  node[above]{7}; }
\Bigr\}\,.
\end{equation}

For the second example consider the family of nested matchings, with product given by \eqref{eq_mproddef}. We see that the product geometry is the arc $(bc,d)$ and edge $(d,e)$.  
%========================================
% Staircase geo
%========================================
Thus for the nested matching of \eqref{eq_nmparv} we get the partition sets 
\tikzset{nm_prod/.pic={
% \draw[productColor,style_size4] (0,0) -- (1,1) -- (2,0) ;
% \path[style_size4] (1,1) pic{circle_product};
\draw[productColor,style_size4]  (0,0) arc [start angle=0, end angle = 160, radius=1] ;
\fill[style_size4]  (0,0) pic{circle_product} ;
\path (0,0) [draw=productColor,style_size4]  -- +(2,0);
}}
\begin{equation}  
\prodgeo_\ffm = \Bigl\{ 
\tikz{\path[style_size4 ] (0,0) pic{nm_prod};
\path[style_size4,node font=\small,color=gray ] (0,0) node[anchor=south west] {3};
},\,
\tikz{\path[style_size4 ] (0,0) pic{nm_prod};
\path[style_size4,node font=\small,color=gray,baseline=1cm] (0,0) node[anchor=south west] {6};
},\,
\tikz{\path[style_size4 ] (0,0) pic{nm_prod};
\path[style_size4,node font=\small,color=gray ] (0,0) node[anchor=south west] {7};
}
\Bigr\}
\qquad\text{ and }\qquad
 \gengeo_\ffm =\Bigl\{
\tikz{\path[style_size4,  node font=\small] 
pic{gennm}  node[above,color=gray]{1}; },\,
\tikz{\path[style_size4,  node font=\small ] 
pic{gennm}  node[above,color=gray]{2}; },\,
\tikz{\path[style_size4,  node font=\small ] 
pic{gennm}  node[above,color=gray]{4}; },\,
\tikz{\path[style_size4,  node font=\small ] 
 pic{gennm}  node[above,color=gray]{5}; }
\Bigr\}\,.
\end{equation}%
For the third example consider  staircase polygons. 
From the product definition, \eqref{eq_scpmf},
and generator definition $\eps =\tikz{\path[style_size4  ] 
 pic{gensp} } $, 
the partition sets of \eqref{eq_spmf} are
\tikzset{sp_prod/.pic={
\draw[productColor,style_size4]  (0,0) -- (1,0) -- (1,1) ;
\fill[style_size4]  (1,0) pic{circle_product} ;
% \path (0,0) [draw=productColor,style_size4]  -- +(2,0);
}}
\begin{equation}  
\prodgeo_\ffm = \Bigl\{ 
\tikz{\path[style_size4 ] (0,0) pic{sp_prod};
\path[style_size4,node font=\small,color=gray ] (1,0) node[anchor=south west] {5};
},\,
\tikz{\path[style_size4 ] (0,0) pic{sp_prod};
\path[style_size4,node font=\small,color=gray,baseline=1cm] (1,0) node[anchor=south west] {6};
},\,
\tikz{\path[style_size4 ] (0,0) pic{sp_prod};
\path[style_size4,node font=\small,color=gray ] (1,0) node[anchor=south west] {7};
}
\Bigr\}
\qquad\text{ and }\qquad
 \gengeo_\ffm =\Bigl\{
\tikz{\path[style_size4,  node font=\small]   
pic{gensp} ++(0.72,0) node[above,color=gray]{1}; },\,
\tikz{\path[style_size4,  node font=\small ] 
pic{gensp} ++(0.72,0) node[above,color=gray]{2}; },\,
\tikz{\path[style_size4,  node font=\small ] 
pic{gensp} ++(0.72,0) node[above,color=gray]{3}; },\,
\tikz{\path[style_size4,  node font=\small ] 
 pic{gensp} ++(0.72,0) node[above,color=gray]{4}; }
\Bigr\}\,.
\end{equation}
%
% \prodgeo_\ffm = 
%   \{
%  \scprodshapelab{0.5}{\tiny 5}  ,
%  \scprodshapelab{0.5}{\tiny 6}  ,
%  \scprodshapelab{0.5}{\tiny 7}  
% \}
% \qquad\text{ and }\qquad
% \gengeo_\ffm =
% \{
% \mbox{  \raisebox{-0.75\height}{\scvgenshapelab{0.5}{\tiny 1}} },
% \mbox{  \raisebox{-0.75\height}{\scvgenshapelab{0.5}{\tiny 2}} },
% \mbox{  \raisebox{-0.75\height}{\scgenshapelab{0.5}{\tiny 3}} },
% \mbox{  \raisebox{-0.75\height}{\scgenshapelab{0.5}{\tiny 4}} }
% \}\,.
% \]
%
From these examples it is clear that the partition sets are repetitions  of  two  basic `shapes' (possibly rotated and/or scaled), one, denoted $\pshape(\ffm)$, associated with $\prodgeo_\ffm $  and  other, denoted $\gshape(\ffm)$, with  $\gengeo_\ffm $. Each object $m\in\ffm$ is thus constructed   by a particular number of  suitable `shape preserving' transformations of each of the two shapes (the number being related to the norm of the object).

For each shape we will choose (or mark) one point. 
The marked point in $\pshape$ will be called the \textbf{root} of $\pshape$ and the marked point in $\gshape$ will be called the \textbf{root} of $\gshape$. 
% The root point and leaf point will be called the \textbf{mark}.
For the above three examples the choices are as follows.
For the complete binary tree the marked shapes are
\begin{equation}\label{eq_cbtmarks}
    \gshape_\bullet = \bullet
\qquad\text{and}\qquad  
\pshape_\bullet = 
\tikz{\draw[style_size4] (0,0) -- (1,1) -- (2,0); \path[style_size4]   (1,1) pic{circle_black};}
\end{equation}
where the marked point is shown as a solid circle and the $\bullet$ subscript indicates the shape has a mark.
For the nested matchings the marked shapes  are
\begin{equation}\label{eq_nmmarks}
   \gshape_\bullet =\nmgenshape 
\qquad\text{and}\qquad  
\pshape_\bullet =\nmprodshape   
\end{equation}
For the staircase polygons we choose
\begin{equation}\label{eq_spmarks}
    \gshape_\bullet =\begin{tikzpicture}[scale=1]
 \fill (0,0) pic[scale=1,line width=1pt]{gensp}   ;
%   \draw [thick] (0,0)   -- (1,0);
\end{tikzpicture}
% \raisebox{-0.75\height}{\scgenshapelab{xx} }  
\qquad\text{and}\qquad  
\pshape_\bullet =\scprodshape{1}  
\end{equation}

% The above two examples show that the sets $\gengeo_\ffm $ and $\prodgeo_\ffm $ partition the object $m$ and this partition is equivalent to the existence of a map
% $$
% \smap_m: m\to \set{ \gshape_\bullet ,\pshape_\bullet } 
% $$ 
% where each subsets of points of $m$ must be mapped by a similarity transformation to an element of  $\set{ \gshape_\bullet ,\pshape_\bullet } $.

If an object $m$ is a product $m=m_1\star m_2$, then each factor can be considered as an object which, if not a generator, will be a product and hence have product geometry and hence have a root. These roots will be called \textbf{sub-roots}.
 
 The marked shapes are   used to define the embedding of a CBT into an object. 
 This is done by  recursively embedding triples
 \footnote{This notation comes from the Category Theory definition of a product diagram. },
 $\triple$, into the geometry of the Catalan object $m$.
The symbol $P$ denotes the root of the product geometry. 
The embedding is defined  as follows:
If $m=m_1\st m_2$   then 
    \begin{itemize}
        \item attach $P$ to the root of the product geometry of $\st$, and
        \item the left arrow of the triple points from $P$ to the sub-root of the product geometry of $m_1$, and
        \item the right arrow of the triple points from $P$ to the sub-root of the product geometry of $m_2$.
    \end{itemize} 
% \parbox{14cm}{
% \begin{itemize}
%     \item if $m=m_1\st m_2$ and $m_1\ne \eps$ and $m_2\ne \eps$ then 
%     \begin{itemize}
%         \item attach $P$ to the root of the product geometry of $\st$, and
%         \item the left arrow of the triple points from $P$ to the sub-root of the product geometry of $m_1$, and
%         \item the right arrow of the triple points from $P$ to the sub-root of the product geometry of $m_2$.
%     \end{itemize} 
%     \item if $m=m_1\st \eps$  and $m_1\ne\eps$ then the right arrow points to the leaf of the generator geometry of $\eps$  and the left arrow of the triple points from $P$ to the sub-root of the product geometry of $m_1$.
%     \item if $m=\eps \st m_2$  and $m_2\ne\eps$ then the left arrow points to the leaf of the generator geometry of $\eps$  and the right arrow of the triple points from $P$ to the subroot of the product geometry of $m_1$.
%     \item if $m=\eps \st \eps $    then the left arrow points to the leaf of the generator geometry of the left $\eps$  and   the right arrow points to the leaf of the generator geometry of the right $\eps$.
% \end{itemize}
% }\\
 By definition, the  $L$ labelled arrow always points to the sub-root in the left factor   and similarly the $R$ labelled arrow always points to the sub-root in the right factor of the product.  
The $\triple$ embedding can thus  be represented schematically by drawing $P$ on the root of the product geometry with the left and right arrows drawn from $P$ to the respective left and right sub-roots.

For example, for CBTs it is (by design) trivial to perform the embedding.
From the product definition we get the embedding:
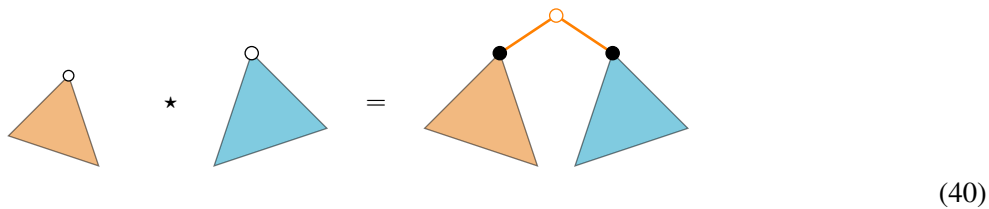
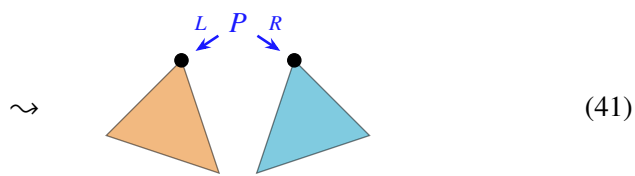
\begin{align} 
    \begin{tikzpicture}[scale=.4,line width=0.5pt,baseline=0.75cm]
        % \draw [help lines, gray!50] (0,0) grid(3,3);
        \fill [color_left,draw=black] (2,3) -- (3,0) -- (0,1) -- cycle;
        \fill [white,draw=black] (2,3) circle (5pt);
    \end{tikzpicture}
    \qquad\star\quad
    \begin{tikzpicture}[scale=.5,line width=0.5pt,baseline=0.75cm]
        % \draw [help lines, gray!50] (0,0) grid(3,3);
        \fill [color_right,draw=black] (1,3) -- (0,0) -- (3,1) -- cycle;
        \fill [white,draw=black] (1,3) circle (5pt);
    \end{tikzpicture}
    \quad&=\quad 
    \begin{tikzpicture}[scale=.5,line width=0.5pt,baseline=0.75cm]
        %   \draw [help lines, gray!50] (0,0) grid(7,4);
        %
        \draw[orange, line width=1pt] (3.5,4) -- (2,3) (3.5,4) -- (5,3);
        \fill [white, draw=orange] (3.5,4) circle (5pt) ;
        \fill [color_left,draw=black] (2,3) -- (3,0) -- (0,1) -- cycle;
        \fill [draw=black] (2,3) circle (5pt);
        \fill [color_right,draw=black] (5,3) -- (4,0) -- (7,1) -- cycle;
        \fill [draw=black] (5,3) circle (5pt) ;
    \end{tikzpicture} \notag   \\
    &\\
   \qquad& \leadsto\qquad
    \begin{tikzpicture}[scale=.5,line width=0.5pt,baseline=0.75cm,>={Stealth[length=2mm]}  ]
        %  \draw [help lines, gray!50] (0,0) grid(7,4);
        %
         \node[color=tripleColor] (a) at (3.5,4){$ P $};
        \node (b) at (2,3) {.};
        \node (c) at (5,3) {.};
        \draw[tripleColor, line width=1pt,->] (a) -- (b) ;
        \draw[tripleColor, line width=1pt,->]  (a) -- (c);
        % \fill [white, draw=orange] (3.5,4) circle (5pt) ;
        %
        \fill [color_left,draw=black] (2,3) -- (3,0) -- (0,1) -- cycle;
        \fill [draw=black] (2,3) circle (5pt);
        \fill [color_right,draw=black] (5,3) -- (4,0) -- (7,1) -- cycle;
        \fill [draw=black] (5,3) circle (5pt) ;
        \node[color=tripleColor] at (2.5,4) {$\scriptstyle L$};
        \node[color=tripleColor] at (4.5,4) {$\scriptstyle R$};
    \end{tikzpicture} 
\end{align}
Repeating the embedding recursively  gives, for example: 
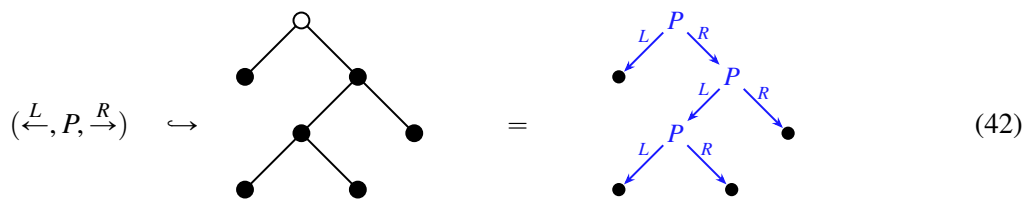
\begin{equation}\label{eq_treescprro55}
\triple\quad  \hookrightarrow\quad
\begin{tikzpicture}[style_size2,baseline=0cm]
    \def\L{++(-1,-1)} \def\R{++(1,-1)}\def\Z{4pt}
    % \draw [help lines, gray!50] (0,0) grid(6,4);
    \draw   (1,2) -- \L (1,2) --\R (2,1) -- \L (2,1) -- \R  (1,0) -- \R (1,0) -- \L;  
    \fill[draw=black]  (0,1) circle (\Z)  (1,0) circle (\Z)  (2,1) circle (\Z) (3,0) circle (\Z)  ; 
    \fill[draw=black]  (0,-1) circle (\Z) (2,-1) circle (\Z);
    \fill[white, draw=black] (1,2)  circle (\Z);
\end{tikzpicture}
\hspace{1cm} = \hspace{1cm} 
\begin{tikzpicture}[tripleColor,fill=black,style_size2,baseline=0cm,>={Stealth[length=1.5mm]}]
    \def\R{ pic{circle_nocolor} } \def\P{$  P$}
    % \draw [help lines, gray!50] (0,0) grid(6,4);
    \begin{scope}[inner sep=1pt]
     \node (a) at (1,2) {\P};
    \node (b) at (0,1) {.};
    \node (c) at (2,1) {\P};
    \node (d) at (1,0) {\P};
    \node (e) at (0,-1) {.};
    \node (f) at (2,-1) {.};
    \node (g) at (3,0) {.};
    \end{scope}
   
    \draw [->] (a) edge node[above] {$\scriptstyle L$} (b);
    \draw [->] (a) edge node[above] {$\scriptstyle R$} (c);
    \draw [->] (c) edge node[above] {$\scriptstyle L$} (d);
    \draw [->] (c) edge node[above] {$\scriptstyle R$} (g);
    \draw [->] (d) edge node[above] {$\scriptstyle L$} (e);
    \draw [->] (d) edge node[above] {$\scriptstyle R$} (f);
    \fill (b) \R;
    \fill  (e) \R;
    \fill  (f) \R; 
    \fill (g) \R;
\end{tikzpicture}
\end{equation}
The resulting structure is trivially an alternative  way of drawing CBTs
\footnote{However, this form shows that the tree can be considered to generate a free category where all  the triples $ A \leftarrow P \rightarrow B $ of the generating graph are product diagrams.}.
In the remainder of this section the $L$ and $R$ labels will generally be omitted when drawing the tree.

Since the recursive embedding of the triple, $\triple$, into a CBT is essentially the CBT itself, a recursive   embedding of the triple $\triple$ into an object, $m$, of some other family, is effectively an embedding of a CBT into that object. 
This recursive embedding of $\triple$ into some object $m$ giving rise to the CBT,  $t$, will be written:
\begin{equation}
    \triple\hookrightarrow m=t\,.
\end{equation}
If $m$ is a schematic form of a product definition then the notation is referring to how the product form defines the embedding of $\triple$, for example,   see \eqref{eq_nmtrdef}.

If $\unimap$ is the universal bijection between the family $\ffm$ containing $m$ and the CBT family, then clearly $\unimap(m)=t$. Thus the decomposition of $t$ gives, via $\unimap$, the decomposition of $m$. The decomposition of a CBT is then simple to obtain by using any of the three tree traversals  shown in \figref{fig:treeTraversal}.

In the recursive embedding \eqref{eq_treescpro}, the $L$ and $R$ labels are clearly superfluous (and thus can be omitted), however this is not always the case. 
In some embeddings   the orientation of the arrows (once embedded) is not always clearly left-to-right, thus the $L$ and $R$ arrow labels are necessary to make it clear which is the left pointing arrow and which the right and hence which are the left  and right  sub-trees of $t$.
%=========================

Dyck paths illustrate why it is necessary to use the $L$ and $R$ labels. The paths   are also  an example of when the choice of root on the product geometry is not immediately clear as the geometry  occurs in two places giving rise to some ambiguity as to  which point  to choose for the root. 
Using the recursive structure of the left and right sub-paths and noting the locations of the left and right product geometry in the sub-paths, as shown below,
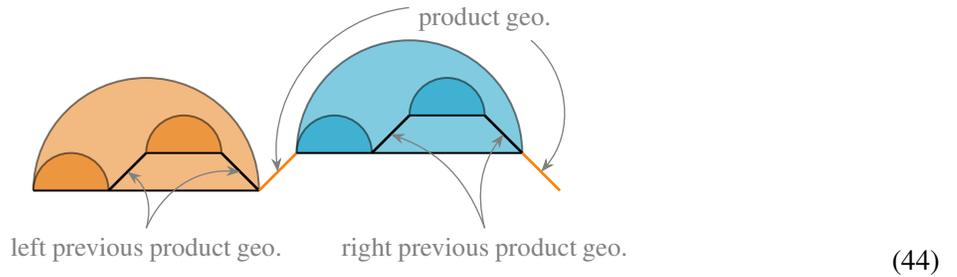
\begin{equation}
\begin{tikzpicture}[style_size3,>=Stealth] 
\fill [color_left,draw=black] (0,0) arc [radius=3,start angle = 180, end angle =0] ;
    \fill [color_left,draw=black] (0,0) arc [radius=1,start angle = 180, end angle =0] ;
    \draw  (0,0) -- (6,0);
    \draw[color=black, line width=1pt] (2,0) -- (3,1);
    \fill [color_left,draw=black] (3,1) arc [radius=1,start angle = 180, end angle =0] ;
    \draw (3,1) -- (5,1);
    \draw[color=black, line width=1pt] (5,1) -- (6,0);
    
    \draw [color=orange, line width=1pt] (6,0)--(7,1);
    
    \fill [color_right,draw=black] (7,1) arc [radius=3,start angle = 180, end angle =0] ;
    \fill [color_right,draw=black] (7,1) arc [radius=1,start angle = 180, end angle =0] ;
    \draw  ((7,1) -- (13,1);
    \draw[color=black, line width=1pt] (9,1) -- (10,2);
    \fill [color_right,draw=black] (10,2) arc [radius=1,start angle = 180, end angle =0] ;
    \draw (10,2) -- (12,2);
    \draw[color=black, line width=1pt] (12,2) -- (13,1);
    
    \draw [color=orange, line width=1pt] (13,1) -- (14,0);

    \coordinate (D) at (5,1);
    \coordinate (E) at (6,0);
    \coordinate (G) at (2,0);
    \coordinate (F) at (3,1);
    \coordinate (H) at (3,-1);
    \coordinate (I) at (barycentric cs:G=1,F=1);
    \coordinate (J) at (barycentric cs:D=1,E=1);
    \draw[  color=gray,line width=0.5pt]  node[below] at (H) {left previous product geo.}   
    (H) edge [->, bend right=35] (I)  (H) edge[->, bend left=35] (J) ;
    
    \coordinate (a) at (9,1);
    \coordinate (b) at (10,2);
    \coordinate (c) at (12,2);
    \coordinate (d) at (13,1);
    \node[below,  color=gray,line width=0.5pt]   (e) at (12,-1) {right previous product geo.};
    \coordinate (h) at (12,-1);
    \coordinate (f) at (barycentric cs:a=1,b=1);
    \coordinate (g) at (barycentric cs:c=1,d=1);
    \draw[  color=gray,line width=0.5pt]   
    (h) edge [->, bend right=25] (f)  (h) edge[->, bend left=25] (g) ;

    \coordinate (r) at (6,0);
    \coordinate (s) at (7,1);
    \coordinate (t) at (13,1);
    \coordinate (u) at (14,0);
    \node[above,  color=gray,line width=0.5pt]   (v) at (12,4) {product geo.};
    \coordinate (w) at (barycentric cs:r=1,s=1);
    \coordinate (x) at (barycentric cs:t=1,u=1);
    \draw[  color=gray,line width=0.5pt]   
    (v) edge [->, bend right=45] (w)  (v) edge[->, bend left=45] (x) ;
    %
    % \draw [help lines, gray!60] (0,0) grid(16,4);
\end{tikzpicture}
\end{equation}
it should be clear that the right-most points of the the respective product geometries gives a simple choice for the embedding of $\triple$ as shown below:
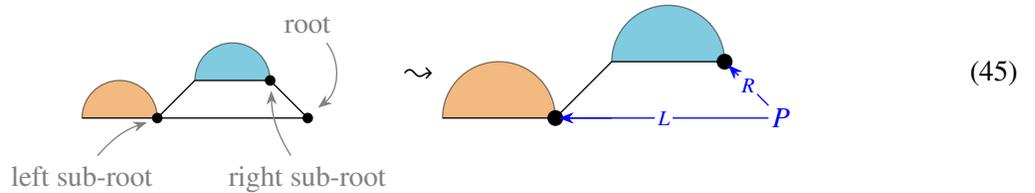
\begin{equation}\label{eq_cbtPathem}
\begin{tikzpicture}[line width=0.5pt,scale=.5,>=Stealth,baseline=0.5cm] 
    \node (a) at (2,0){};
    \node (b) at (5,1){};
    \node (c) at (6,0){};
    \fill [color_left,draw=black] (0,0) arc [radius=1,start angle = 180, end angle =0] ;
    \draw  (0,0) -- (6,0);
    \draw[color=black] (2,0) -- (3,1);
    \fill [color_right,draw=black] (3,1) arc [radius=1,start angle = 180, end angle =0] ;
    \draw (3,1) -- (5,1);
    \draw[color=black] (5,1) -- (6,0);
    
    \fill (a) circle  (4pt);
    \fill (b) circle  (4pt);
    \fill (c) circle  (4pt);
    \node [below, color=gray,line width=0.5pt] (n) at (0,-1) {left sub-root};
    \node [below, color=gray,line width=0.5pt] (m) at (6,-1) {right sub-root};
    \node [above, color=gray,line width=0.5pt] (p) at (6,2) { root};
     \draw[  color=gray,line width=0.5pt]   
    (n) edge [->, bend left=15] (a)  (m) edge[->, bend left=15] (b)  (p) edge[->, bend left=45] (c);
    %
    % \draw [help lines, gray!60] (0,0) grid(7,4);
\end{tikzpicture}
\leadsto
\begin{tikzpicture}[line width=0.5pt,scale=.75,>={Stealth[length=2mm]},baseline=0.5cm,inner sep=1pt] 
    \node (a) at (2,0){};
    \node (b) at (5,1){};
    \node[color=blue,fill=white] (c) at (6,0){$P$};
    \fill [color_left,draw=black] (0,0) arc [radius=1,start angle = 180, end angle =0] ;
    \draw (0,0) -- (3,0);
 
    \draw[color=black] (2,0) -- (3,1);
    \fill [color_right,draw=black] (3,1) arc [radius=1,start angle = 180, end angle =0] ;
    \draw (3,1) -- (5,1);

    \fill (a) circle  (4pt);
    \fill (b) circle  (4pt);
    
    \draw[blue,->] (c)edge node[fill=white, inner sep=1pt] {$\scriptstyle L$}(a);
    \draw[blue,->] (c)edge node[fill=white, inner sep=1pt] {$\scriptstyle R$}(b);

    %
    % \draw [help lines, gray!60] (0,0) grid(7,4);
\end{tikzpicture}
\end{equation}
The labelling of the left and right arrows is important in obtaining the correct orientation of the resulting embedded CBT as shown by the following example:
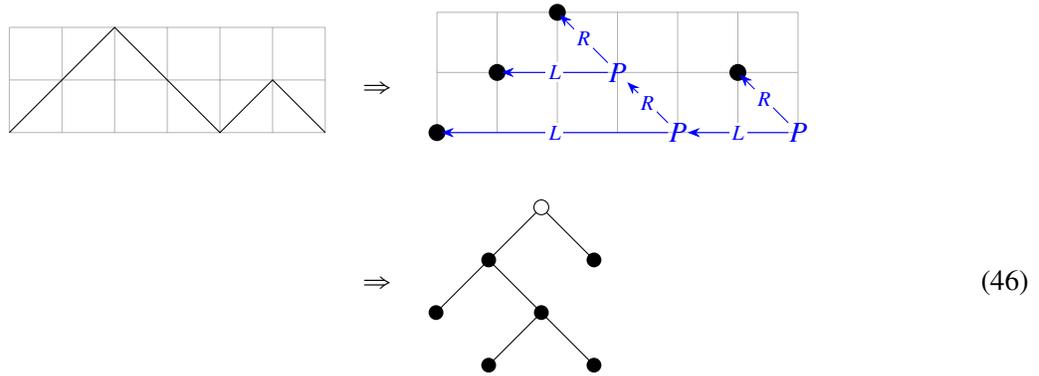
\begin{align}\label{eq_dyckcbtembed}
 \begin{tikzpicture}[scale=.7,baseline=0.5cm]
    \draw [help lines, gray!60] (0,0) grid(6,2);  
    \coordinate (a) at (0,0)  ;
    \coordinate (b) at (1,1) ;
    \coordinate (c) at (2,2) ;
    \coordinate (d) at (3,1) ;
    \coordinate (e) at (4,0) ;
    \coordinate (f) at (5,1) ;
    \coordinate (g) at (6,0) ;   
    \draw (a)-- (b) -- (c) -- (d) -- (e) -- (f) -- (g);
\end{tikzpicture}
\quad&\Rightarrow\quad
\begin{tikzpicture}[scale=.8,>=Stealth,baseline=0.5cm,inner sep=0pt]
\def\C{circle (4pt)}
    \draw [help lines, gray!60] (0,0) grid(6,2); 
    \node (a) at (0,0) {.};
    \node (b) at (1,1){.};
    \node (c) at (2,2){.};
    \node[color=blue] (d) at (3,1){$P$};
    \node[color=blue] (e) at (4,0){$P$};
    \node (f) at (5,1){.};
    \node[color=blue] (g) at (6,0){$P$};
    \fill (a) \C (b) \C (c) \C (f) \C;
    \draw [blue,->] (g) edge node[fill=white, inner sep=1pt] {$\scriptstyle R$}  (f)  ;
    \draw [blue,->]  (g) edge node[fill=white, inner sep=1pt ] {$\scriptstyle L$}  (e);
    \draw [blue,->]  (e)edge node[fill=white, inner sep=1pt] {$\scriptstyle R$}(d);
    \draw [blue,->]  (d)edge node[fill=white, inner sep=1pt] {$\scriptstyle R$}(c);
    \draw [blue,->]  (d) edge node[fill=white, inner sep=1pt ] {$\scriptstyle L$} (b);
    \draw [blue,->]  (e) edge node[fill=white, inner sep=1pt ] {$\scriptstyle L$} (a);
\end{tikzpicture} \notag\\
&\notag\\
\quad& \Rightarrow\quad
\begin{tikzpicture}[scale=.7,baseline=1cm]
\def\C{circle (4pt)}
    % \draw [help lines, gray!60] (0,0) grid(6,2);  
        \coordinate (a) at (2,3) ;
        \coordinate (b) at (1,2) ;
        \coordinate (c) at (0,1) ;
        \coordinate (d) at (2,1) ;
        \coordinate (e) at (1,0) ;
        \coordinate (f) at (3,0) ;
        \coordinate (g) at (3,2) ;  
 \draw (a)--(b) (a)--(g);
 \draw (b)--(c) (b)--(d);  
  \draw (d)--(e) (d)--(f);
  \fill[white,draw=black] (a) \C ;
  \fill (b) \C (c) \C (d) \C (e) \C (f) \C (g) \C;
\end{tikzpicture}
\end{align}
Note the correct CBT (shown last) is a reflected version of the tree that would have resulted by a simple counter-clockwise rotation (to orientate vertically) of the embedded tree (shown second).
% \begin{equation}\label{eq_}
%     \includegraphics[width=14cm]{figtemp/dyckpath_embed.pdf}
% \end{equation}

% ========================
We illustrated the recursive $\triple$ embedding and resulting CBT  with a few more examples.
%
% ========================
\subsubsection*{Nested matchings -- \fref{fam_ld}}
For nested matchings, the product definition is shown in \eqref{eq_scpmf} and mark choices in \eqref{eq_nmmarks}, which leads  to an embedding of $\triple$ into the nested matching which can be represented schematically as
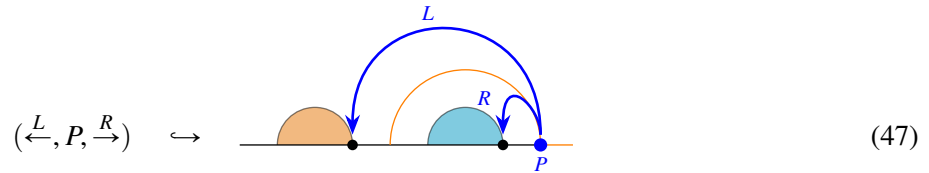
\begin{equation}\label{eq_nmtrdef}  
\def\bsl{0pt}
\triple \quad \hookrightarrow\quad  
\begin{tikzpicture}[line width=0.5pt,scale=0.5,>=Stealth,baseline=\bsl] 
    % \draw [help lines, gray!50] (0,-2) grid(9,4);
    \fill [color_left,draw=black] (1,0) arc [radius=1,start angle = 180, end angle =0];
    \fill [color_right,draw=black] (5,0) arc [radius=1,start angle = 180, end angle =0];
    \draw (0,0) -- (8,0);
    \draw[color=orange ] (4,0) arc [start angle=180, end angle = 0, radius=2];
    \draw [color=orange ] ( 8,0) -- ( 9,0);
    \fill[color=white] ( 9,0) circle [radius=4pt];
    \node (L) at  (3,0) { };
    \node (R) at (7,0) { } ;
    \node (P) at (8,0) { } ;
    \fill (L) circle [radius=4pt];
    \fill (R) circle [radius=4pt];
    \fill [color=blue] (P) circle [radius=5pt];   
    \draw[blue,line width=1pt,->]  (P) .. controls +(90:4cm) and +(90:4cm) .. (L);
    \draw[blue,line width=1pt,->]  (P) .. controls +(90:1cm) and +(90:2cm) .. (R);   
    \node[color=blue] at (5,3.5) {$\scriptstyle L$};
    \node[color=blue] at (6.5,1.25) {$\scriptstyle R$};
    \node[color=blue] at (8,-0.5) {$\scriptstyle P$};
\end{tikzpicture}
\end{equation}
Repeating the embedding recursively gives, for example,
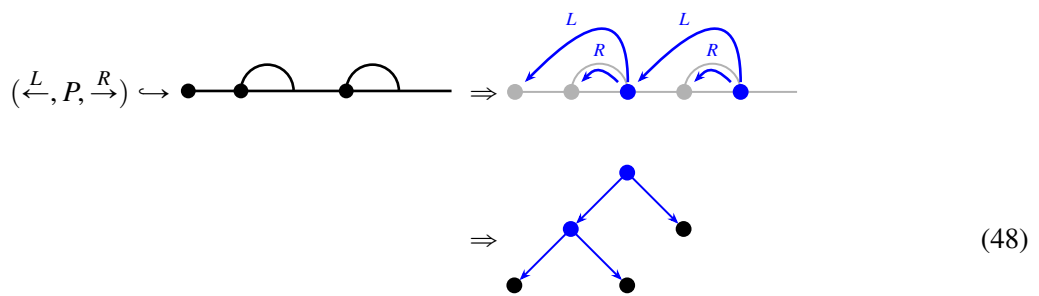
\begin{align}  \label{eq_nmembex} 
\def\S{0.7}
\triple \hookrightarrow 
\begin{tikzpicture}[line width=1pt,scale=\S,>=Stealth] 
\def\R{4pt}
% \draw [help lines, gray!50] (0,0) grid (5,1);
\node (a) at (0,0) { };
\node (b) at (1,0) { };
\node (c) at (2,0) { };
\node (d) at (3,0) { };
\node (e) at (4,0) { };
\node (f) at (5,0) { };
\draw (0,0) -- (5,0);
\draw (b) arc [radius=0.5, start angle = 180, end angle =0];
\draw (d) arc [radius=0.5, start angle = 180, end angle =0];
\fill (a) circle (\R) (b) circle (\R) ;
\fill  (d) circle (\R)  ;
\end{tikzpicture} 
 &\Rightarrow 
\begin{tikzpicture} [style_size2,>={Stealth[length=1.5mm]}] 
\def\R{4pt}
%  \draw [help lines, gray!50] (0,0) grid (5,1);
\node (a) at (0,0) { };
\node (b) at (1,0) { };
\node (c) at (2,0) { };
\node (d) at (3,0) { };
\node (e) at (4,0) { };
\node (f) at (5,0) { };
\draw[color=gray!60]  (0,0) -- (5,0);
\draw[color=gray!60] (b) arc [radius=0.5, start angle = 180, end angle =0];
\draw[color=gray!60]  (d) arc [radius=0.5, start angle = 180, end angle =0];
\fill[color=gray!60]  (a) circle (\R) (b) circle (\R)  (d) circle (\R) ;
\fill[blue]  (c) circle (\R) (e) circle (\R);
\draw[blue,line width=1pt,->]  (c) .. controls +(90:2cm) and +(45:1cm) .. (a);
\draw[blue,line width=1pt,->]  (c) .. controls +(135:.75cm) and +(45:.5cm) .. (b);
\node[color=blue] at (1,1.3) {$\scriptstyle L$};
\node[color=blue] at (1.5,0.75) {$\scriptstyle R$};
\draw[blue,line width=1pt,->]  (e) .. controls +(90:2cm) and +(45:1cm) .. (c);
\draw[blue,line width=1pt,->]  (e) .. controls +(135:.75cm) and +(45:.5cm) .. (d);
\node[color=blue] at (3,1.3) {$\scriptstyle L$};
\node[color=blue] at (3.5,0.75) {$\scriptstyle R$};    
\end{tikzpicture} 
\notag\\
&\notag\\
&\Rightarrow 
\begin{tikzpicture}[style_size2,baseline=0.5cm,>={Stealth[length=1.5mm]},inner sep=1pt]
\def\R{ circle (4pt) }  
% \draw [help lines, gray!50] (0,0) grid(3,2);
\node (a) at (2,2) {\small .};
\node (b) at (1,1) {.};
\node (c) at (0,0) {.};
\node (d) at (2,0) {.};
\node (e) at (3,1) {.};
\draw[blue ] [->] (a) -- (b);
\draw[blue ] [->] (a) -- (e);
\draw[blue ] [->] (b) -- (c);
\draw[blue ] [->] (b) -- (d);
\fill[blue ] (a) \R;
\fill[blue ](b) \R;
\fill[black ](c) \R; 
\fill[black ](d) \R;
\fill[black ](e) \R;
\end{tikzpicture}
\end{align}
% \begin{equation}\label{eq_}
%     \includegraphics[width=10cm]{figtemp/nmtriprecembed.png}
% \end{equation}
Note, attaching the $L$ and $R$ labels in \eqref{eq_nmtrdef} and \eqref{eq_nmembex} makes it clear how the CBT is orientated. In-fix tree traversal gives
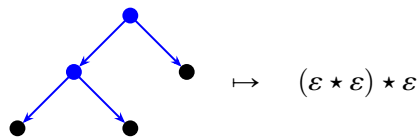
\begin{equation*}
\begin{tikzpicture}  [style_size2,baseline=0.5cm,  >={Stealth[length=1.5mm]},inner sep= 1pt  ]
  \def\L{++(-1,-1)} \def\R{++(1,-1)}
\def\R{ circle (4pt) }  
% \draw [help lines, gray!50] (0,0) grid(3,2);
\node (a) at (2,2) {\small .};
\node (b) at (1,1) {.};
\node (c) at (0,0) {.};
\node (d) at (2,0) {.};
\node (e) at (3,1) {.};
\draw[blue ] [->] (a) -- (b);
\draw[blue ] [->] (a) -- (e);
\draw[blue ] [->] (b) -- (c);
\draw[blue ] [->] (b) -- (d);
\fill[blue ] (a) \R;
\fill[blue ](b) \R;
\fill[black ](c) \R; 
\fill[black ](d) \R;
\fill[black ](e) \R;
\end{tikzpicture}
\quad\mapsto \quad (\eps\star\eps)\star\eps
\end{equation*}
and thus the decompostion
\begin{equation*}
\begin{tikzpicture}[line width=1pt,scale=0.7,>=Stealth,baseline=0cm] 
\def\R{4pt}
    % \draw [help lines, gray!50] (0,0) grid (5,1);
    \node (a) at (0,0) { };
    \node (b) at (1,0) { };
    \node (c) at (2,0) { };
    \node (d) at (3,0) { };
    \node (e) at (4,0) { };
    \node (f) at (5,0) { };
    \draw (0,0) -- (5,0);
    \draw (b) arc [radius=0.5, start angle = 180, end angle =0];
    \draw (d) arc [radius=0.5, start angle = 180, end angle =0];
    \fill (a) circle (\R) (b) circle (\R) ;
    \fill  (d) circle (\R)  ;
\end{tikzpicture}
\quad = \quad 
(\nmgenshape  \star \nmgenshape )\star \nmgenshape \,.
\end{equation*}
From  example \eqref{eq_nmembex}, it is clear the embedded CBT internal nodes (and root)  always occur at the right ends of  the arcs   and the leaves always occur   (except for the leftmost) at the left end of an arc. This rule gives a simple way to embed the CBT in any nested matching.

%=========================================== 
\subsubsection*{Triangulations -- \fref{triangulations}}
Considering the product geometry   of the left and right factors gives the $\triple$ embedding:
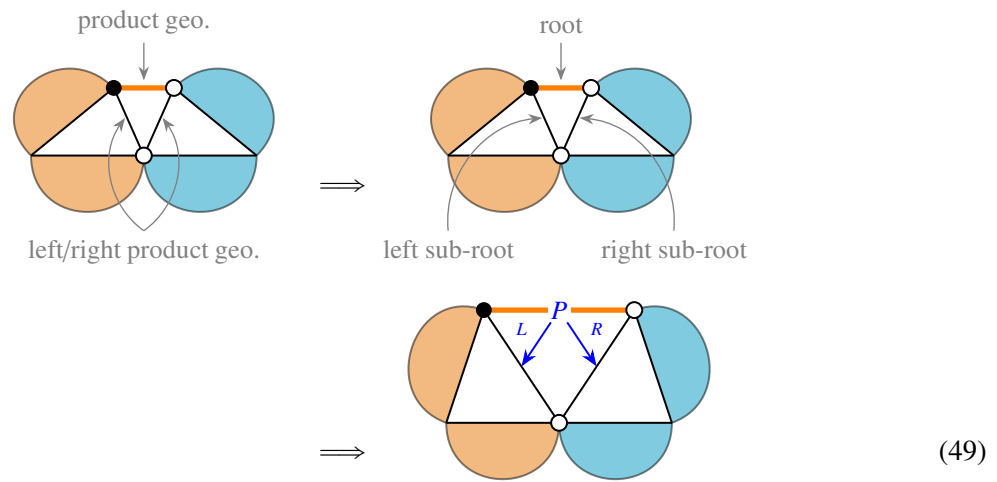
\begin{align}\label{cbtEmbedTri7}
\begin{tikzpicture}[style_size3,baseline=0cm,>=Stealth] 
\coordinate (A)  at (2.2,2.8);
    \coordinate (B)  at (0,1);
    \coordinate (C)  at (3,1);
    \coordinate (D) at  (6,1);
    \coordinate (E) at  (3.8,2.8);
    \coordinate (F)  at (3,3);
    \coordinate (G)  at (3,4);
    \coordinate (H)  at (3, -1);
    \coordinate (I)  at (2.5,2);
    \coordinate (J)  at (3.5,2);
    \fill [color_left,draw=black] (A).. controls +(135:2cm) and +(135:2cm) .. (B);
    \fill [color_left,draw=black]  (B) .. controls +(270:2cm) and +(270:2cm) .. (C);
    \draw (B)-- (A); \draw (B)-- (C);
    \draw (C) -- (A);
    \draw [orange,line width = 2pt] (A) -- (E);
    \draw[color=gray,line width=0.5pt] node[above] at (G) {product geo.}  (G) edge [->] (F);     
    \fill [color_right,draw=black] (C) .. controls +(270:2cm) and +(270:2cm) .. (D);
    \fill [color_right,draw=black]  (D) .. controls +(45:2cm) and +(45:2cm) .. (E);  
    \draw (E)-- (D); \draw (D) -- (C);
    \draw (C) -- (E);
    \fill [black] (A) circle (6pt) ;
    \fill [white,draw=black] (E) circle (6pt);
    \fill [white,draw=black] (C) circle (6pt);  
    \draw[  color=gray,line width=0.5pt]  node[below] at (H) {left/right product geo.}   
    (H) edge [->, bend left=45] (I)  (H) edge[->, bend right=45] (J) ;
    % \draw [help lines, gray!60] (0,0) grid(6,4);
\end{tikzpicture}
    & \Longrightarrow
\begin{tikzpicture}[style_size3,baseline=0cm,>=Stealth] 
        \coordinate (A) at (2.2,2.8);
        \coordinate (B)  at (0,1);
        \coordinate (C)  at (3,1);
        \coordinate (D) at  (6,1);
        \coordinate (E) at  (3.8,2.8);
        \coordinate (F)  at (3,3);
        \coordinate (G)  at (3,4);
        \coordinate (H)  at (0, -1);
        \coordinate (I)  at (2.5,2);
        \coordinate (J)  at (3.5,2);
        \coordinate (K)  at (1, -1);
        \coordinate (L)  at (6, -1);
        \fill [color_left,draw=black] (A).. controls +(135:2cm) and +(135:2cm) .. (B);
        \fill [color_left,draw=black]  (B) .. controls +(270:2cm) and +(270:2cm) .. (C);
        \draw (B)-- (A); \draw (B)-- (C);
        \draw (C) -- (A);
        \draw [orange,line width = 2pt] (A) -- (E);
        \draw[color=gray,line width=0.5pt] node[above] at (G) {root}  (G) edge [->] (F);        
        \fill [color_right,draw=black] (C) .. controls +(270:2cm) and +(270:2cm) .. (D);
        \fill [color_right,draw=black]  (D) .. controls +(45:2cm) and +(45:2cm) .. (E);  
        \draw (E)-- (D); \draw (D) -- (C);
        \draw (C) -- (E);
        \fill [black] (A) circle (6pt) ;
        \fill [white,draw=black] (E) circle (6pt);
        \fill [white,draw=black] (C) circle (6pt);  
        \draw[  color=gray,line width=0.5pt]  node[below] at (H) {left sub-root}   ;
        \draw[  color=gray,line width=0.5pt] (H) edge [->, bend left=45] (I) ;
        \draw[  color=gray,line width=0.5pt]  node[below] at (L) {right sub-root}   ;
        \draw[  color=gray,line width=0.5pt] (L) edge[->, bend right=45] (J) ;
        % \draw [help lines, gray!60] (0,0) grid(6,4);
\end{tikzpicture} \notag\\
      &  \Longrightarrow
\begin{tikzpicture}[style_size3,baseline=0cm,>=Stealth] 
    \coordinate (A)  at (1, 4);
    \coordinate (B)  at (0,1 );
    \coordinate (C)  at (3,1);
    \coordinate (D) at  (6,1);
    \coordinate (E) at  (5,4);
    \coordinate (F)  at (3,3);
    \coordinate (G)  at (3,4);
    \def\Ag{20}
    \fill [color_left,draw=black] (A).. controls +(160:2cm) and +(160:2cm) .. (B);
    \fill [color_left,draw=black]  (B) .. controls +(-90:2cm) and +(-90:2cm) .. (C);
    \draw (B)-- (A); \draw (B)-- (C);
    \draw (C) -- (A);
    \draw [orange,line width = 2pt] (A) -- (E);
    % \draw[color=gray,line width=0.5pt] node[above] at (G) {root}  (G) edge [->] (F); 
    \fill [color_right,draw=black] (C) .. controls +(270:2cm) and +(270:2cm) .. (D);
    \fill [color_right,draw=black]  (D) .. controls +(20:2cm) and +(20:2cm) .. (E);  
    \draw (E)-- (D); \draw (D) -- (C);
    \draw (C) -- (E);
    \fill [black] (A) circle (6pt) ;
    \fill [white,draw=black] (E) circle (6pt);
    \fill [white,draw=black] (C) circle (6pt);  
    \coordinate (lsr) at (barycentric cs:A=1,C=1);
    \coordinate (rsr) at (barycentric cs:E=1,C=1);
    \node[color=blue,fill=white,inner sep=1pt] (R) at (3,4){$P$};
    \node (S) at (2 ,2){};
    \node (T) at (4 ,2){};
    % \node[color=orange,fill=white,inner sep=1pt] at (H) {$P$};
    \draw[->,blue] (R) -- (lsr);
    \draw[->,blue] (R) -- (rsr);
    \path (2.,3.5) node [blue] {$\scriptstyle L$};
    \path (4.0,3.5) node [blue] {$\scriptstyle R$};
    %   \draw [help lines, gray!60] (0,0) grid(6,4);
\end{tikzpicture}
\end{align}
% \begin{equation}\label{eq_}
%     \includegraphics[width=10cm]{figtemp/triag_embed.png}
% \end{equation}
An example of recursively applying the embedding  $\triple$ is shown below  which results in the CBT shown below right:
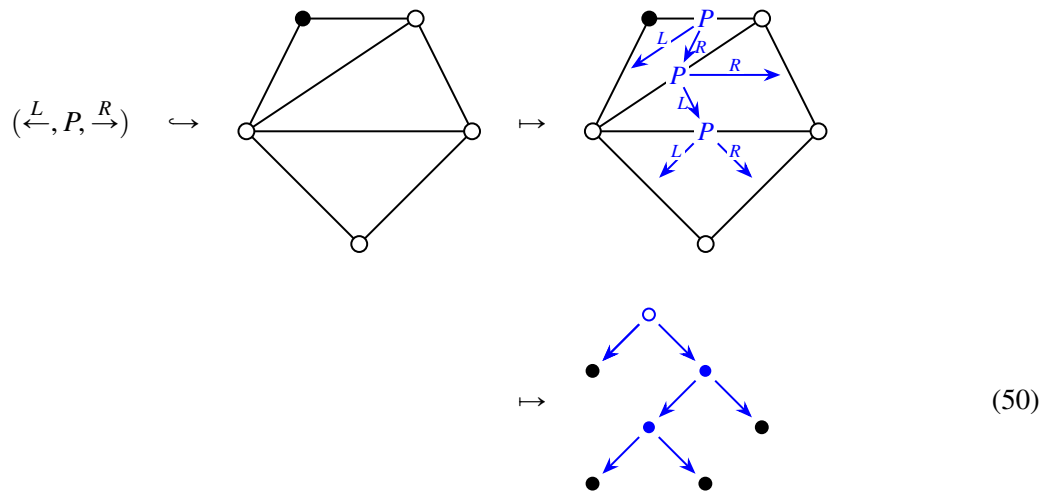
\begin{align}\label{eq_cbttriemb}
\triple  \quad\hookrightarrow \quad
 \begin{tikzpicture}[style_size2,baseline= 1.5cm,>=Stealth]
    \def\A{circle (4pt)}
    \coordinate (A) at (1,4);
    \coordinate (B) at (0,2);
    \coordinate (C) at (2,0);
    \coordinate (D) at (4,2);
    \coordinate (E) at (3,4);
    \draw (A) -- (B) -- (C) -- (D) -- (E) -- (A);
    \draw (B) -- (E) (B) -- (D);
    \fill (A) \A;
    \fill[white,draw=black] (B) \A (C) \A (D) \A (E) \A;        
    \node (b) at (barycentric cs:A=1,B=1) { };
    \node (d) at (barycentric cs:E=1,D=1) {};  
    \node (f) at (barycentric cs:B=1,C=1) {}; 
    \node (g) at (barycentric cs:D=1,C=1) {};   
    \end{tikzpicture}
\quad & \mapsto \quad 
\begin{tikzpicture}[style_size2,baseline= 1.5cm,>=Stealth]
    \def\A{circle (4pt)}
    \coordinate (A) at (1,4);
    \coordinate (B) at (0,2);
    \coordinate (C) at (2,0);
    \coordinate (D) at (4,2);
    \coordinate (E) at (3,4);
    \draw (A) -- (B) -- (C) -- (D) -- (E) -- (A);
    \draw (B) -- (E) (B) -- (D);
    \fill (A) \A;
    \fill[white,draw=black] (B) \A (C) \A (D) \A (E) \A;        
    \node[color=blue,fill=white,inner sep=1pt ] (a) at (barycentric cs:A=1,E=1) {$P$};
    \node (b) at (barycentric cs:A=1,B=1) { };
    \node[color=blue,fill=white,inner sep=1pt ] (c) at (barycentric cs:B=1,E=1) {$P$};    
    \node (d) at (barycentric cs:E=1,D=1) {};  
    \node[color=blue,fill=white,inner sep=1pt ] (e) at (barycentric cs:B=1,D=1) {$P$};  
    \node (f) at (barycentric cs:B=1,C=1) {}; 
    \node (g) at (barycentric cs:D=1,C=1) {}; 
    \draw[ blue] (a) edge [->] node[fill=white,inner sep=1pt,above] {$\scriptstyle L$} (b);
    \draw[ blue] (a) edge [->] node[ inner sep=1pt,right] {$\scriptstyle R$} (c);
    \draw[ blue] (c) edge [->] node[fill=white,inner sep=1pt,above] {$\scriptstyle R$} (d);
    \draw[ blue] (c) edge [->] node[ inner sep=1pt,left] {$\scriptstyle L$} (e);    
    \draw[ blue] (e) edge [->] node[fill=white,inner sep=1pt,above] {$\scriptstyle R$} (g);
    \draw[ blue] (e) edge [->] node[fill=white, inner sep=1pt,above] {$\scriptstyle L$} (f);    
    \end{tikzpicture} \notag\\
&\notag\\
&\mapsto \quad 
\begin{tikzpicture}[style_size2,baseline=1.0cm,>=Stealth]
        \def\R{circle (3pt)}
        \node[blue] (a) at (1,3) {};
        \node (b) at (0,2){};
        \node (c) at (2,2){};
        \node (d) at (1,1){};
        \node (e) at (0,0){};
        \node (f) at (2,0){};
        \node (g) at (3,1){};
        \draw[blue] [->] (a)--(b);\fill[white,draw=blue] (a) \R;
        \draw[blue] [->] (a)--(b);\fill[ draw=black] (b) \R;
        \draw[blue] [->] (a)--(c);\fill[blue] (c) \R;
        \draw[blue] [->] (c)--(d);\fill[blue] (d) \R;
        \draw[blue] [->] (c)--(g);\fill[ draw=black] (g) \R;
        \draw[blue] [->] (d)--(e);\fill[ draw=black] (e) \R;
        \draw[blue] [->] (d)--(f);\fill[ draw=black] (f) \R;
    \end{tikzpicture}
\end{align}

This embedding is a well known embedding of a CBT into a triangulation and thus   shows that the  previously known embedding can be deduced from the magma structure of triangulations.

\subsubsection*{Staircase polygons -- \fref{fam_sp}}
% The staircase polygons illustrate a family where the position of the left and right sub-roots does not necessarily coincide with a  point of the product geometry. 
The position of the root and sub-roots is determined by the `necks' of the structure. 
A neck is the common segment occurring between   two consecutive columns of the polygon which overlap by exactly one cell -- the left side of the leftmost lowest cell is also a neck. 
Three necks occur in the following example:
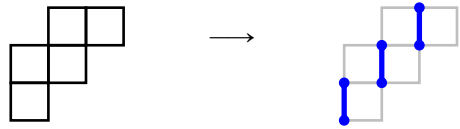
\begin{equation*}
\begin{tikzpicture}[scale=0.5,line width =1pt,baseline=0.5cm]
% \draw [help lines, gray!50] (0,0) grid(3,2); 
\draw (0,-1) rectangle (1,0);
\draw (0,0) rectangle (1,1);
\draw (1,0) rectangle (2,1);
\draw (1,1) rectangle (2,2);
\draw (2,1) rectangle (3,2);
\end{tikzpicture}
\hspace{1cm}\longrightarrow\hspace{1cm}
\begin{tikzpicture}[scale=0.5,line width =1pt,draw=gray!50,baseline=0.5cm]
% \draw [help lines, gray!50] (0,0) grid(3,2); 
\draw (0,-1) rectangle (1,0);
\draw (0,0) rectangle (1,1);
\draw (1,0) rectangle (2,1);
\draw (1,1) rectangle (2,2);
\draw (2,1) rectangle (3,2);
\draw [line width=2pt,blue] (0,-1) circle (2pt) -- (0,0) circle (2pt);
\draw [line width=2pt,blue] (1,0) circle (2pt) -- (1,1) circle (2pt);
\draw [line width=2pt,blue] (2,1) circle (2pt) -- (2,2) circle (2pt);
\end{tikzpicture}
\end{equation*}
The polygons factorise at the rightmost neck. This neck defines the product geometry and hence the root. 
The neck in the left factor (respec.\ right factor) defines the left sub-root (respec. right sub-root) as illustrated below:
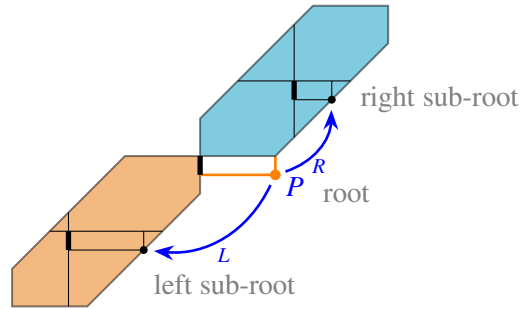
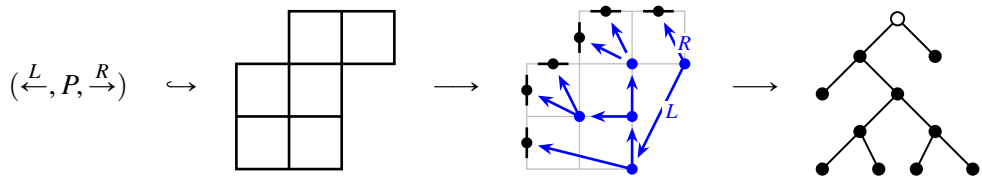
\begin{equation}  
\begin{tikzpicture}[scale=0.5,baseline=0cm,>=Stealth ]
% \draw [help lines, gray!50] (0,0) grid(10,8);    
\path [fill,color_left,draw=black, line width=0.75pt] 
(0,0) -- (0,1) -- (3,4) -- (5,4) --(5,3) --(2,0) -- cycle;
\draw (1.5,0) -- (1.5,2.5)  (1,2) -- (4,2);
\draw (1.5,1.5) -- (3.5,1.5) -- (3.5,2) ;
\draw[line width=2pt] (1.5,1.5) -- (1.5,2) ; %neck
\fill (3.5,1.5) circle (3pt);
\path (5,3.5) [ draw=orange, line width=1pt]  --  (7,3.5) -- (7,4); % product geo
\path[fill=orange]  (7,3.5) circle [radius=4pt]; % product geo
\path [fill,color_right,draw=black, line width=0.75pt]  
(5,4) -- (7,4) -- (10,7) -- (10,8) -- (8,8) -- (5,5) -- cycle;
\draw[line width=2pt] (5,3.5) -- (5,4) ; %neck
\draw (7.5,4.5) -- (7.5,7.5);
\draw (6,6) -- (9,6);
\draw (7.5,5.5) -- (8.5,5.5) -- (8.5,6);
\fill (8.5,5.5) circle (3pt);
\draw[line width=2pt] (7.5,5.5) -- (7.5,6) ; %neck
\node (L) at (3.5,1.5){};
\node  (P) at (7,3.5) { };
\node (U) at (8.5,5.5){};

\draw[ blue,line width=1pt] (P) edge[->, bend left=35] node [below]{$\scriptstyle L$}(L);
\draw[ blue,line width=1pt](P) edge[->, bend right=35] node [below] {$\scriptstyle R$} (U);
\node[blue,right]  (Pt) at (7,3.2) {$P$};

\node[right, color=gray] at (9,5.5) {right sub-root};
\node[right, color=gray] at (8,3) {root};
\node[right, color=gray] at (3.5,0.6) {left sub-root};

\end{tikzpicture}
\end{equation}
For example,
\begin{equation*}
\triple\quad\hookrightarrow\quad
\begin{tikzpicture}[line width=1pt,scale=0.7,baseline=1.0cm] 
% \draw [help lines, gray!50] (0,0) grid(3,3);  
%
\coordinate (1) at (0,0);
\coordinate (2) at (0,1);
\coordinate (3) at (0,2);
\coordinate (4) at (1,2);
\coordinate (5) at (1,3);
\coordinate (6) at (2,3);
\coordinate (7) at (3,3);
\coordinate (8) at (3,2);
\coordinate (9) at (2,2);
\coordinate (10) at (2,1);
\coordinate (11) at (2,0);
\coordinate (12) at (1,0);

\draw (1) -- (3) -- (4) -- (5) -- (7) -- (8) -- (9) -- (11) -- cycle;
\draw  (12) -- (4) (2)--(10) (9) -- (6) (4) -- (9);

\end{tikzpicture}
\quad\longrightarrow\quad
\begin{tikzpicture}[line width=1pt,>={Stealth[length=2mm]},scale=0.7,baseline=1.0cm] 
\def\R{3pt}
% \draw [help lines, gray!50] (0,0) grid(3,3);  
%
\coordinate (1) at (0,0);
\coordinate (2) at (0,1);
\coordinate (3) at (0,2);
\coordinate (4) at (1,2);
\coordinate (5) at (1,3);
\coordinate (6) at (2,3);
\coordinate (7) at (3,3);
\coordinate (8) at (3,2);
\coordinate (9) at (2,2);
\coordinate (10) at (2,1);
\coordinate (11) at (2,0);
\coordinate (12) at (1,0);

\node (a) at (barycentric cs:1=1,2=1){};
\node (b) at (barycentric cs:2=1,3=1){};
\node (c) at (barycentric cs:3=1,4=1){};
\node (d) at (barycentric cs:4=1,5=1){};
\node (e) at (barycentric cs:5=1,6=1){};
\node (f) at (barycentric cs:6=1,7=1){};

\node (n1) at (0,0){};
\node (n2) at (0,1){};
\node (n3) at (0,2){};
\node (n4) at (1,2){};
\node (n5) at (1,3){};
\node (n6) at (2,3){};
\node (n7) at (3,3){};

\node (g) at (1,1){};
\node (h) at (2,2){};
\node (i) at (8){};
\node (j) at (10){};
\node (k) at (11){};

\draw[gray!50,line width=0.5pt] (1) -- (3) -- (4) -- (5) -- (7) -- (8) -- (9) -- (11) -- cycle;
\draw[gray!50,line width=0.5pt]   (12) -- (4) (2)--(10) (9) -- (6) (4) -- (9);
 
\path (n1)    edge   (n2) ; 
\fill (a) circle (\R);
\path (n2)    edge   (n3) ; 
\fill (b) circle (\R);
\path (n3)    edge   (n4) ; 
\fill (c) circle (\R);
\path (n4)    edge   (n5) ; 
\fill (d) circle (\R);
\path (n5)    edge   (n6) ; 
\fill (e) circle (\R);
\path (n6)    edge   (n7) ; 
\fill (f) circle (\R);

\begin{scope}[blue]
\fill (11) circle (\R);
\draw[->] (11)--(a);
\draw[->] (11)--(j);

\fill (1,1) circle (\R);
\draw[->] (1,1)--(b);
\draw[->] (1,1)--(c);

\fill (2,2) circle (\R);
\draw[->] (h)--(d);
\draw[->] (h)--(e);

\fill (8) circle (\R);
\draw  (8) edge[->] node [right, fill=white,inner sep=1pt] {$\scriptstyle R$}   (f);
\draw  (8)  edge[->] node [right, fill=white,inner sep=1pt] {$\scriptstyle L$}  (k);

\fill (10) circle (\R);
\draw[->] (10)--(h);
\draw[->] (10)--(g);

\end{scope}

\end{tikzpicture}
\quad\longrightarrow\quad
\begin{tikzpicture}[line width=0.75pt,scale=0.5,baseline=1.0cm]
% \draw [help lines, gray!50] (0,0) grid(5,4);  
\coordinate (1) at (3,4);
\coordinate (2) at (2,3);
\coordinate (3) at (1,2);
\coordinate (4) at (3,2);
\coordinate (5) at (2,1);
\coordinate (6) at (1,0);
\coordinate (7) at (2.5,0);
\coordinate (8) at (4,1);
\coordinate (9) at (3.5,0);
\coordinate (10) at (5,0);
\coordinate (11) at (4,3);

\draw (1)-- (2) -- (3) (1)--(11);
\draw (2) -- (4) -- (8) -- (9);
\draw (4) -- (5) -- (6) ;
\draw (5) -- (7) (8) -- (10);

\foreach \p in {2,...,11} \fill (\p) circle (5pt); 
\fill[white,draw=black] (1) circle (5pt);

\end{tikzpicture}
\end{equation*}
% \begin{equation}\label{eq_}
%     \includegraphics[width=14cm]{figtemp/sctripemb.png}
% \end{equation}
 
% \noindent\emph{Planar Trees -- \fref{???}}
% \begin{equation}\label{eq_}
%     \includegraphics[width=14cm]{figtemp/planartree_embed.png}
% \end{equation}
% For planar trees the root of the product geometry is represented by a `dash' to avoid confusing it with a node of the tree.

%=================================================================
\subsubsection*{Tree traversal embeddings}
%=================================================================

Assume we have objects $m$ and $n$ each from a different Catalan family but in universal bijection $m=\unimap_{i,j}(n) $.  Then $m$ and $n$ biject to the  same CBT, $t=\unimap_{\ref{fam_bt},i}(m)=\unimap_{\ref{fam_bt},j}(n)$.
% , embedding  $m\hookleftarrow t \hookrightarrow  n$.
It is then formally possible to use $t$ as an intermediate object to effect the  embedding of $m$ directly into  $n$ or vice versa.  
It is rare that this produces a ``natural'' embedded bijection, however there is  at least  one situation where some utility is obtained.
If parts of the generator and product geometry of $m$ can be uniquely associated with any of the pre-fix, in-fix or post-fix positions of $t$ and similarly for $n$, then one of the tree traversals (see \figref{fig:treeTraversal}) will effect the embedded  bijection. 

We will illustrate this  by showing how a Dyck path is embedded in a triangulation. 
This is a well known bijection (see \cite{Stanley:2015aa})  but now we see that it can be deduced from the magma structure of the objects.
First note the up steps, denoted $u$, of a Dyck path are uniquely associated with the in-fix position of the embedded CBT and the down steps with the post-fix position, for example:
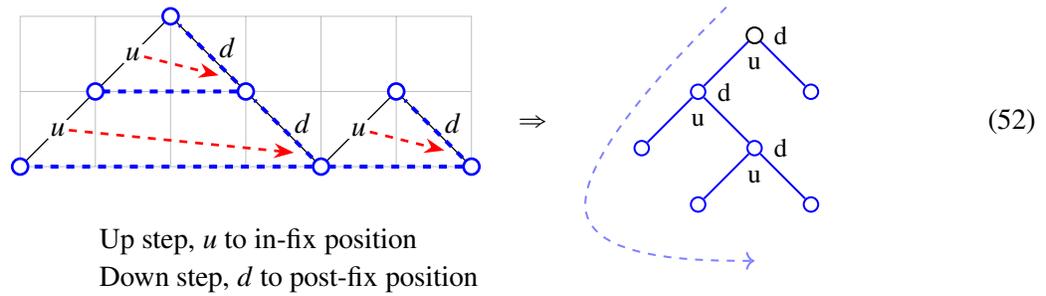
\begin{equation}\label{eq_tte1}
\begin{tikzpicture}[style_size1,baseline=0.5cm,inner sep =1pt,>=Stealth]
\draw [help lines, gray!50] (0,0) grid(6,2);  
\coordinate (a) at (0,0)  ;
\coordinate (b) at (1,1) ;
\coordinate (c) at (2,2) ;
\coordinate (d) at (3,1) ;
\coordinate (e) at (4,0) ;
\coordinate (f) at (5,1) ;
\coordinate (g) at (6,0) ;   
\draw[line width=0.5pt] (a)-- (b) -- (c) -- (d) -- (e) -- (f) -- (g);

\node  (na) at (a) { };

 \node (nb) at (b) {\phantom{f}};
 \node  (nc) at (c) {\phantom{f}};
 \node  (nd) at (d) {\phantom{f}};
 \node  (nf) at (f) {\phantom{f}};
 \node (ng) at (g) {\phantom{f}};
 \node  (ne) at (e) {\phantom{f}};

\begin{scope}[blue,dashed,line width=1.5pt] 
\draw(ng)--(nf);
\draw  (ng)--(ne);
\draw  (ne)--(nd);
\draw  (nd)--(nc);
\draw  (nd)--(nb);
\draw  (ne)--(na);
\end{scope}
\begin{scope}[fill=white,draw=blue,line width=1pt,radius=3pt]
\fill [white,draw=blue] (a) circle;
\fill[white,draw=blue] (b) circle;
\fill[white,draw=blue] (c) circle;
\fill[white,draw=blue] (d) circle;
\fill[white,draw=blue] (e) circle;
\fill[white,draw=blue] (g) circle;
\fill[white,draw=blue] (f) circle;       
\end{scope}

\node[above,inner sep=5pt] (gf) at (barycentric cs:ng=3,nf=1) {$d$ };
\node[above,inner sep=5pt] (de) at (barycentric cs:ne=3,nd=1) {$d$ };
\node[above,inner sep=5pt] (dc) at (barycentric cs:nd=3,nc=1) {$d$ };

\node[inner sep=1pt,fill=white] (ab) at (barycentric cs:na=1,nb=1) {$u$};
\node[inner sep=5pt] (edb) at (barycentric cs:ne=5,nd=1) {};
\draw[->,red,line width=1pt,dashed] (ab)-- (edb);

\node[inner sep=1pt,fill=white] (bc) at (barycentric cs:nb=1,nc=1) {$u$};
\node[inner sep=5pt] (dcb) at (barycentric cs:nd=5,nc=1) {};
\draw[->,red,line width=1pt,dashed] (bc)-- (dcb);

\node[inner sep=1pt,fill=white] (ef) at (barycentric cs:ne=1,nf=1) {$u$};
\node[inner sep=5pt] (gfb) at (barycentric cs:ng=5,nf=1) {};
\draw[->,red,line width=1pt,dashed] (ef)-- (gfb);

        \node[right] at (1,-1) {Up step, $u$ to in-fix position};
        \node[right] at (1,-1.5) {Down step, $d$ to post-fix position};

    %  \coordinate (lsr) at (barycentric cs:A=1,C=1);
\end{tikzpicture}
\quad\Rightarrow\quad
\begin{tikzpicture}[style_size2,baseline=1cm,inner sep=1pt ]
\def\C{circle (4pt)}

    \node[circle,draw=blue,fill=white] (a) at (2,3) {\phantom{.}};
    \node [circle,draw=blue,fill=white] (b) at (1,2) {\phantom{.}};
    \node[circle,draw=blue,fill=white] (c) at (0,1) {\phantom{.}};
    \node[circle,draw=blue,fill=white] (d) at (2,1) {\phantom{.}};
    \node[circle,draw=blue,fill=white] (e) at (1,0) {\phantom{.}};
    \node[circle,draw=blue,fill=white] (f) at (3,0) {\phantom{.}};
    \node[circle,draw=blue,fill=white] (g) at (3,2) {\phantom{.}};
        
    \draw [blue](a)--(b) (a)--(g);
    \draw [blue](b)--(c) (b)--(d);  
    \draw[blue] (d)--(e) (d)--(f);
    \fill[white,draw=black] (a) \C ;

    \node[right,inner sep=7pt] at (a) {d};
    \node[right,inner sep=7pt] at (b) {d};
    \node[right,inner sep=7pt] at (d) {d};
    
    \node[below,inner sep=8pt] at (b) {u};
    \node[below,inner sep=8pt] at (d) {u};
    \node[below,inner sep=8pt] at (a) {u};
    \draw[->,color=blue!50,dashed] 
    (1.5,3.5) .. controls +(-135:1cm) and +(90:1cm) .. 
    (-1,0) .. controls +(-90:1cm) and +(180:1cm) .. (2,-1);
    
    % \draw [help lines, gray!60] (0,0) grid(6,4);  
\end{tikzpicture}
\end{equation}
Anti-clockwise traversal gives the well known bijection from a CBT to a Dyck word and hence trivially a Dyck path.

Secondly, for triangulations, if the chords (and the edge right adjacent to the marked node) are oriented  by adding an arrow  away from the node first encountered whilst traversing anti-clockwise around the outside of the triangulation then we can uniquely associate parts of the triangulation with the in-fix and pre-fix positions of the embedded CBT. 
Using the embedding \eqref{eq_cbttriemb} it can be seen that the corner of the triangle opposite an internal node of the CBT is uniquely associated with the in-fix position and the post-fix position is uniquely associated with the end of the chord pointed to by the orientation arrow, for example:
\begin{equation}\label{eq_tte2}
\begin{tikzpicture}[scale=1,baseline=1cm,>=Stealth,color_cbt,line width=1pt]
\def\R{circle (3pt)}
\node (a) at (1,3) {};
\node (b) at (0,2){};
\node (c) at (2,2){};
\node (d) at (1,1){};
\node (e) at (0,0){};
\node (f) at (2,0){};
\node (g) at (3,1){};

\draw   (a)--(b);\fill[white, draw=blue,line width=1pt] (a) \R;
\draw   (a)--(b); 
\draw   (a)--(c);\fill[  color_cbt] (c) \R;
\draw   (c)--(d);\fill[  color_cbt] (d) \R;
\draw  (c)--(g); 
\draw  (d)--(e); 
\draw   (d)--(f);

\fill[red] (1 ,2.7) circle (2pt);  
\fill[red] (2 ,1.7) circle (2pt);   
\fill[red] (1 ,0.7) circle (2pt);  
\end{tikzpicture}
\qquad \leftrightarrow \qquad    
\begin{tikzpicture}[scale=1,baseline= 2cm,>={Stealth[length=2mm]}]
\def\A{circle (4pt)}
\coordinate (A) at (1,4);
\coordinate (B) at (0,2);
\coordinate (C) at (2,0);
\coordinate (D) at (4,2);
\coordinate (E) at (3,4);

\draw (A) -- (B) -- (C) -- (D) -- (E) -- (A);
\draw (B) -- (E) (B) -- (D);
\draw[->] (A) -- (barycentric cs:A=3,E=1) ;
\draw[->] (B) -- (barycentric cs:B=3,E=1) ;
\draw[->] (B) -- (barycentric cs:B=3,D=1) ;
\fill (A) \A;
\fill[white,draw=black] (B) \A (C) \A (D) \A (E) \A;

% CBT nodes    
\node[color=blue,fill=white,inner sep=3pt ] (a) at (barycentric cs:A=1,E=1) { };

\node (b) at (barycentric cs:A=1,B=1) {  };
\node[color=blue,fill=white,inner sep=1pt ] (c) at (barycentric cs:B=1,E=1) { };

\node (d) at (barycentric cs:E=1,D=1) {} ;  
\node[color=blue,fill=white,inner sep=1pt ] (e) at (barycentric cs:B=1,D=1) { };  

\node (f) at (barycentric cs:B=1,C=1) { }; 
\node (g) at (barycentric cs:D=1,C=1) { }; 

% CBT edges
\draw[ color_cbt,line width=1pt] (a) edge     (b);
\draw[ color_cbt,line width=1pt] (a) edge     (c);
\draw[ color_cbt,line width=1pt] (c) edge    (d);
\draw[ color_cbt,line width=1pt] (c) edge    (e);    
\draw[ color_cbt,line width=1pt] (e) edge     (g);
\draw[ color_cbt,line width=1pt] (e) edge    (f);   

% red nodes
\node (An) at (1,4) {.};
\node[inner sep = 3pt] (Bn)   at (barycentric cs:a=1,B=5) {};
\fill [red] (Bn) circle (2pt) ;
\node[inner sep = 3pt] (Cn)   at (barycentric cs:e=1,C=5) {};
\fill [red] (Cn) circle (2pt) ;
\node[inner sep = 3pt] (Dn)   at (barycentric cs:c=1,D=5) {};
\fill [red] (Dn) circle (2pt) ;
% red arrows
\draw[->,red,dashed] (a) -- (Bn);
\draw[->,red,dashed] (c) -- (Dn);
\draw[->,red,dashed] (e) -- (Cn);
%
% light blue nodes
\node[inner sep = 1pt] (En)   at (barycentric cs:A=1,E=5) {};
\fill [blue!30] (En) circle (2pt) ;
\node[inner sep = 1pt] (Em)   at (barycentric cs:B=1,E=5) {};
\fill [blue!30] (Em) circle (2pt) ;
\node[inner sep = 1pt] (Dm)   at (barycentric cs:B=1,D=5) {};
\fill [blue!30] (Dm) circle (2pt) ;

\draw[->,blue!30,dashed] (a) .. controls +(35:.4cm) and +(145:.4cm) .. (En);
\draw[->,blue!30,dashed] (c) .. controls +(35:.4cm) and +(145:.4cm) .. (Em);
\draw[->,blue!30,dashed] (e) .. controls +(35:.4cm) and +(145:.4cm) .. (Dm); 
\fill[blue!30] (5,4) circle (2pt);
\fill[red] (5,3.5) circle (2pt);
\node[right,inner sep = 3pt] at (5,4) {Post-fix association};
\node[right,inner sep = 3pt] at (5,3.5) {In-fix association};

\fill[blue!30] (1.3,1) circle (2pt);  
\fill[blue!30] (2.3,2) circle (2pt);   
\fill[blue!30] (1.3,3) circle (2pt);   
 % blue nodes
\fill[color_cbt] (e) circle (3pt);
\fill[color_cbt] (c) circle (3pt);
\fill[white, draw=blue,line width=1pt] (a) circle (3pt);

%  \draw [help lines, gray!60] (0,0) grid(4,4);   

\end{tikzpicture}
\end{equation}
% \begin{equation}
%     \includegraphics[width=8cm]{figtemp/treetravbij_B.png}
% \end{equation}
Combining \eqref{eq_tte1} and \eqref{eq_tte2} gives the traversal embedded bijection between triangulations and Dyck words (and hence paths). For example
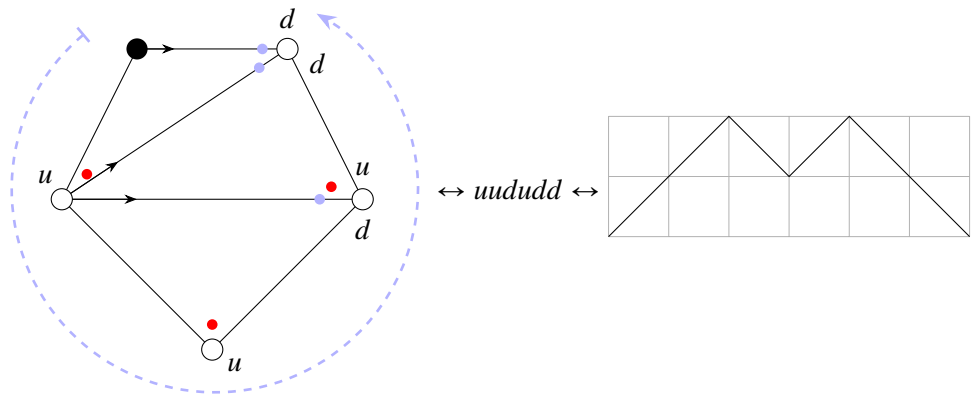
\begin{equation*}
 \begin{tikzpicture}[scale=1,baseline= 2cm,>=Stealth]
    \def\A{circle (4pt)}
    \coordinate (A) at (1,4);
    \coordinate (B) at (0,2);
    \coordinate (C) at (2,0);
    \coordinate (D) at (4,2);
    \coordinate (E) at (3,4);

    \draw (A) -- (B) -- (C) -- (D) -- (E) -- (A);
    \draw (B) -- (E) (B) -- (D);
        \draw[->] (A) -- (barycentric cs:A=3,E=1) ;
    \draw[->] (B) -- (barycentric cs:B=3,E=1) ;
    \draw[->] (B) -- (barycentric cs:B=3,D=1) ;
    \fill (A) \A;
    \fill[white,draw=black] (B) \A (C) \A (D) \A (E) \A;
    
    % CBT nodes    
    \coordinate(a) at (barycentric cs:A=1,E=1) ;
    % \fill[white, draw=blue,line width=1pt] (a) circle (3pt);
    \coordinate (b) at (barycentric cs:A=1,B=1) ;
    \coordinate (c) at (barycentric cs:B=1,E=1) ;
    % \fill[blue] (c) circle (3pt);
    \coordinate (d) at (barycentric cs:E=1,D=1) ;  
    \coordinate (e) at (barycentric cs:B=1,D=1) ;  
    % \fill[blue] (e) circle (3pt);
    \coordinate (f) at (barycentric cs:B=1,C=1) ; 
    \coordinate (g) at (barycentric cs:D=1,C=1) ;

    % red nodes
    \node (An) at (1,4) {.};
    \node[inner sep = 3pt] (Bn)   at (barycentric cs:a=1,B=5) {};
    \fill [red] (Bn) circle (2pt) ;
    \node[inner sep = 3pt] (Cn)   at (barycentric cs:e=1,C=5) {};
    \fill [red] (Cn) circle (2pt) ;
    \node[inner sep = 3pt] (Dn)   at (barycentric cs:c=1,D=5) {};
    \fill [red] (Dn) circle (2pt) ;
 
    % light blue nodes
    \node[inner sep = 1pt] (En)   at (barycentric cs:A=1,E=5) {};
    \fill [blue!30] (En) circle (2pt) ;
    \node[inner sep = 1pt] (Em)   at (barycentric cs:B=1,E=7) {};
    \fill [blue!30] (Em) circle (2pt) ;
    \node[inner sep = 1pt] (Dm)   at (barycentric cs:B=1,D=6) {};
    \fill [blue!30] (Dm) circle (2pt) ;
    
    \path  (-0.22,2.3) node {$u$};
    \path  (2.3,-0.2) node {$u$};
    \path  (4.0,1.6) node {$d$};
    \path  (4.0,2.4) node {$u$};
    \path  (3.4,3.8) node {$d$};
    \path  (3,4.4) node {$d$};

    % \draw[->,blue!30,dashed,line width=1pt] (0.5,4.4) .. controls +(200:2cm) and +(145:2cm) .. (0,-0.4); 
%  \fill [color_right,draw=black] (0,0) arc [radius=0.75,start angle = 180, end angle =0] ;
  \draw [|->,blue!30,dashed,line width=1pt] (0.3,4.2) arc [radius=2.7cm, start angle=130,delta angle=290];
    
% \draw [help lines, gray!60] (0,0) grid(4,4);   

    \end{tikzpicture}
    \leftrightarrow
    uududd
    \leftrightarrow
\begin{tikzpicture}[scale=.8,baseline=0.5cm]
    \draw [help lines, gray!60] (0,0) grid(6,2);  
    \coordinate (a) at (0,0)  ;
    \coordinate (b) at (1,1) ;
    \coordinate (c) at (2,2) ;
    \coordinate (d) at (3,1) ;
    \coordinate (e) at (4,2) ;
    \coordinate (f) at (5,1) ;
    \coordinate (g) at (6,0) ;   
    \draw (a)-- (b) -- (c) -- (d) -- (e) -- (f) -- (g);
\end{tikzpicture}    
\end{equation*}

As a final example we show how the CBT embedding into a 321-avoiding permutation leads to the Dyck path embedding of Rotem \cite{Rotem:1975aa}. From the product definition in the Appendix (family \fref{fam_ap}) we get the simplified representation 
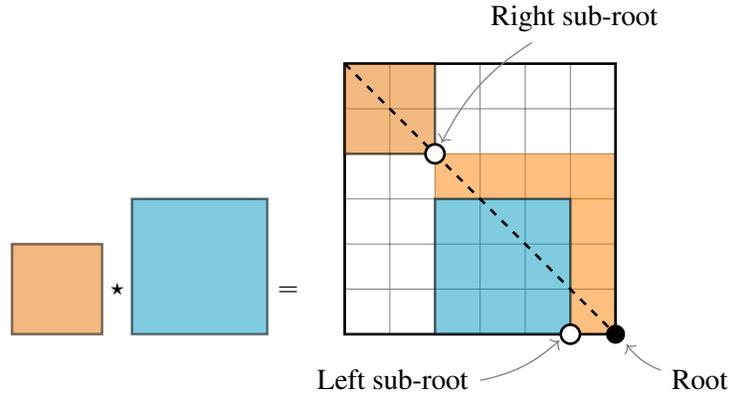
\begin{equation*}
    \begin{tikzpicture}[scale=0.6,line width=1pt,baseline = 0.5cm]
\fill [color_left,draw=black]  (0,0) rectangle (2,2) ;
\end{tikzpicture}
\star
\begin{tikzpicture}[scale=0.6,line width=1pt,baseline = 0.5cm]
\fill [color_right,draw=black]  (0,0) rectangle (3,3) ;
\end{tikzpicture}
=
\begin{tikzpicture}[scale=0.6,line width=1pt,baseline = 0.5cm]
\draw [help lines, gray ] grid(6,6);
\fill[orange, opacity=0.5] (2,3) rectangle (6,4);
\fill[orange, opacity=0.5] (5,0) rectangle (6,3);
\draw (0,0) rectangle (6,6);
\fill [color_left,draw=black]  (0,4) rectangle (2,6) ;
\fill [color_right,draw=black]  (2,0) rectangle (5,3) ;
\draw[dashed] (0,6) -- (6,0);
\coordinate (P) at (6,0) ;
\coordinate (L) at (5,0);
\coordinate (R) at (2,4);
\fill (P) circle (6pt); 
\node (An) at (6,0) { };
\node[anchor=west] (Ar) at (7,-1) {Root };
\draw[  color=gray,line width=0.5pt]  (Ar) edge [->, bend left=15] (An) ;
\fill[white, draw=black] (L) circle (6pt); 
\node (Ln) at (5,0) { };
\node[anchor=east] (Lr) at (3,-1) {Left sub-root };
\draw[  color=gray,line width=0.5pt]  (Lr) edge [->, bend right=15] (Ln) ;
\fill[white, draw=black]  (R) circle (6pt); 
\node (Rn) at (2,4) { };
\node[anchor=west] (Rr) at (3,7) {Right sub-root };
\draw[  color=gray,line width=0.5pt]  (Rr) edge [->, bend right=15] (Rn) ;
\end{tikzpicture}
\end{equation*}
which more clearly shows the ``product'' geometry  (in orange) ie.\ the new row a dot is added and the new column created by the final cascaded dot. From this we choose the root of the product geometry to be the south-east corner of the permutation, giving
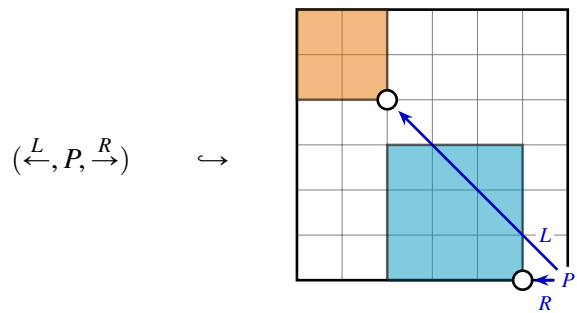
\begin{equation*}
\triple \qquad \hookrightarrow \qquad
\begin{tikzpicture}[scale=0.6,line width=1pt,baseline = 1.5cm,>={Stealth[length=2mm]}]
\draw [help lines, gray ] grid(6,6);
\draw (0,0) rectangle (6,6);
\fill [color_left,draw=black]  (0,4) rectangle (2,6) ;
\fill [color_right,draw=black]  (2,0) rectangle (5,3) ;
\coordinate (P) at (6,0) ;
\coordinate (L) at (5,0);
\coordinate (R) at (2,4);
\fill[white, draw=black] (L) circle (6pt); 
\node (Ln) at (5,0) { };
\fill[white, draw=black]  (R) circle (6pt); 
\node (Rn) at (2,4) { };
\draw[color_cbt] (P) node[circle, fill=white,inner sep = 1pt] (An)   {$\scriptstyle P$};
\draw[color_cbt,->] (An) edge (Ln);
\draw[color_cbt,->] (An)edge (Rn);
\draw[color_cbt] (5.5,1) node[fill=white, inner sep=1pt]  {$\scriptstyle L$} ;
\draw[color_cbt] (5.5,-0.5) node[fill=white, inner sep=1pt]  {$\scriptstyle R$} ;
\end{tikzpicture}
\end{equation*}
As an example consider the permutation $1342$ which results in the CBT embedding
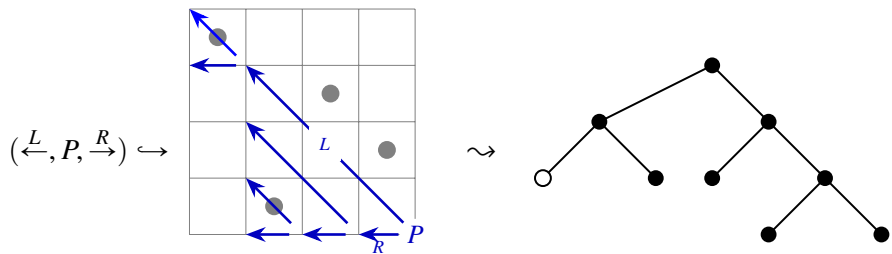
\begin{equation*}
\triple \hookrightarrow
\begin{tikzpicture}[>=Stealth,style_size2,baseline=1cm,] 

\draw [help lines, gray ] grid(4,4);
\draw[fill=white,gray] (1.5,0.5) circle (4pt)  (3.5,1.5) circle (4pt)  (2.5,2.5) circle (4pt) (0.5,3.5) circle (4pt) ;
\coordinate (A) at (4,0) ;
\coordinate (B) at (3,0) ;
\node (Bn) at (B){\phantom{.}};
\node[circle,color_cbt,fill=white,inner sep=1pt] (An) at (A)  {$P$};
 \draw [color_cbt,->,line width=1pt]  (An) edge node[below,  inner sep=2pt] {$\scriptstyle R$}(B);
\coordinate (C) at (1,3);

 \node  (Cn) at (C) {};
  \draw [color_cbt,->,line width=1pt]  (An) edge node[  fill=white, inner sep=4pt] {$\scriptstyle L$}(C);
\coordinate (D) at (2,0);

\node (Dn) at (D) {\phantom{.}};
\draw [color_cbt,->,line width=1pt]  (Bn) edge  (D);
\coordinate (E) at (1,0);
\node (En) at (E) {\phantom{.}};
\draw [color_cbt,->,line width=1pt]  (Dn) edge  (E);
\coordinate (F) at (1,1);
\node (Fn) at (F) {\phantom{.}};
\draw [color_cbt,->,line width=1pt]  (Dn) edge  (F);
\coordinate (G) at (1,2);
\node (Gn) at (G) {\phantom{.}};
\draw [color_cbt,->,line width=1pt]  (Bn) edge  (G);
\coordinate (H) at (0,3);
\node (Hn) at (H) {\phantom{.}};
\draw [color_cbt,->,line width=1pt]  (Cn) edge  (H);
\coordinate (I) at (0,4);
\node (In) at (I) {\phantom{.}};
\draw [blue,->,line width=1pt]  (Cn) edge  (I);
%
%
% \node[above] at (A) {A};
% \node[above] at (B) {B};
% \node[above] at (C) {C};
% \node[above] at (D) {D};
% \node[above] at (E) {E};
% \node[above] at (F) {F};
% \node[above] at (G) {G};
% \node[above] at (H) {H};
% \node[above] at (I) {I};
\end{tikzpicture}
\quad \leadsto \quad 
\begin{tikzpicture}[baseline=1cm,style_size2]
\def\C{circle (4pt)}
% \draw [help lines, gray ] grid(5,3);
\coordinate (1) at (0,1) ;
\coordinate (2) at (1,2) ;
\coordinate (3) at (2,1) ;
\coordinate (4) at (3,3) ;
\coordinate (5) at (3,1) ;
\coordinate (6) at (4,2) ;
\coordinate (7) at (4,0) ;
\coordinate (8) at (5,1) ;  
\coordinate (9) at (6,0) ;  
 \draw (4)--(2) (4)--(6);
 \draw (2)--(1) (2)--(3);  
  \draw (6)--(5) (6)--(8);
   \draw (8)--(7) (8)--(9);
  \fill[white,draw=black] (1) \C ;
  \fill (2) \C (3) \C (4) \C (5) \C (6) \C (7) \C (8) \C (9) \C;
\end{tikzpicture}
\end{equation*}
From the counter clockwise traversal of the  CBT we get the factorisation
\begin{equation}\label{eq_perfact}
  1342= (\emptyset \st  \emptyset) \star (\emptyset\star(\emptyset\star\emptyset))\,.
\end{equation}
where $\emptyset$ is the 321-avoiding generator (the empty permutation).
Comparing this example with the Dyck path example \eqref{eq_dyckcbtembed} shows clearly the bijection to the Dyck path $uduuuddd$  (which can also be obtained by multiply out the product \eqref{eq_perfact} using the Dyck path generator).

\subsection*{Sequence families}
%=================================================================

If the Catalan object is a sequence, for example frieze patterns \fref{fam_fz},  then associating a CBT with the elements of the sequence may not be  of much utility as the factorisation process may be sufficiently simple to perform on the sequence itself.  
For example, for matching brackets (or Dyck words) finding the rightmost bracket matching the last bracket (to the right) is straightforward. 

However not all `sequence' families are this easy to factorise. In these cases there may exist a simple version of the universal bijection to some geometric family and then for the geometric family one can construct the CBT embedding and hence obtain the factorisation.  
Clearly this is the case for matching brackets: biject to Dyck paths and then use the CBT tree embedding \eqref{eq_dyckcbtembed}.  
A less trivial example is the two row  standard tableaux family (\fref{fam_st} in the appendix), considered as a pair of sequences:
\[
(\text{top row},\text{bottom row})=(t_1\dots t_n,b_1\dots b_n)\,.
\]
In this case the universal  bijection to Dyck paths takes a simple  form: the top row indexes the up steps and the bottom row indexes the down steps. Thus the   CBT embedding (and hence decomposition) of 
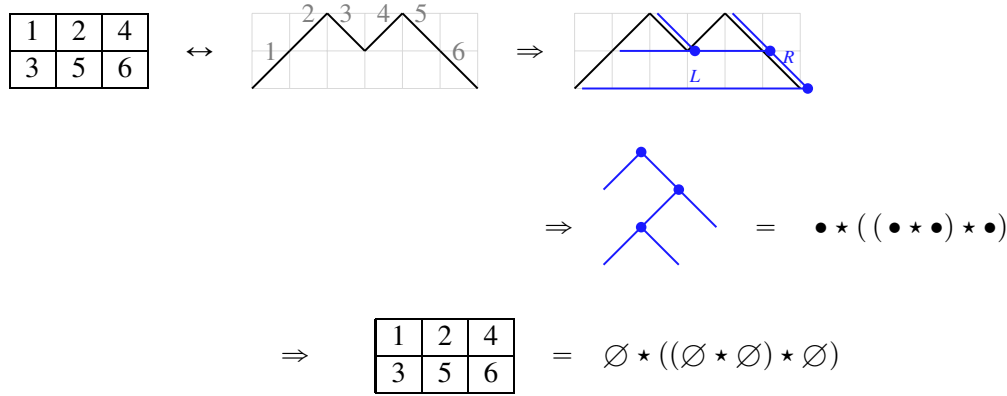
\begin{equation*}
\begin{tabular}{|c|c|c|}
    \hline
    1 & 2 & 4\\
    \hline
    3 & 5 & 6\\
    \hline
\end{tabular}
\end{equation*}
becomes
\begin{align*}
\begin{tabular}{|c|c|c|}
\hline  
1 &  2 & 4\\
\hline  
3 &  5 & 6 \\
\hline 
\end{tabular} 
\quad  \leftrightarrow \quad 
\begin{tikzpicture}[style_size3,baseline=0.4cm]  
\def\E{0.2}
\draw [help lines,color=gray!30] (0,0) grid (6,2);
\draw (0,0) edge node[above,color=gray] {1} (1,1) 
(1,1) edge node[above,color=gray] {2} (2,2) 
(2,2) edge node[above,color=gray] {3} ( 3,1) 
( 3,1) edge node[above,color=gray] {4} (4,2) 
(4,2) edge node[above,color=gray] {5} (5,1) 
(5,1) edge node[above,color=gray] {6} (6,0);
\end{tikzpicture}
\quad \Rightarrow\quad  &
\begin{tikzpicture}[style_size3,baseline=0.4cm,inner sep = 2pt]  
\def\E{0.2}
\draw [help lines,color=gray!30] (0,0) grid (6,2);
\draw (0,0) edge (1,1) 
(1,1) edge  (2,2) 
(2,2) edge  ( 3,1) 
( 3,1) edge  (4,2) 
(4,2) edge    (5,1) 
(5,1) edge   (6,0);
\draw[tripleColor] (6+\E,0) edge node[above] {$\scriptstyle R$} (5+\E,1) (5+\E,1) -- (4+\E,2)
(6+\E,0) edge node[above] {$\scriptstyle L$} (0+\E,0) (5+\E,1) -- (1+\E,1) (3+\E,1) -- (2+\E,2);
\begin{scope}[fill=tripleColor]
\path (6+\E,0) pic{circle_nocolor};
\path (5+\E,1) pic{circle_nocolor}; 
\path (3+\E,1) pic{circle_nocolor};
\end{scope}
\end{tikzpicture} \\
& \\
\quad  \Rightarrow  &
\quad 
\begin{tikzpicture}[style_size3,baseline=0.4cm]  
%   \draw [help lines,color=gray!30] (0,0) grid (3,3);
\draw[tripleColor]   (1,3) -- (2,2) -- (1,1) -- (0,0) (2,2)--(3,1) (1,3) -- (0,2) (1,1) -- (2,0);
\begin{scope}[fill=tripleColor]
\path (1 ,3) pic{circle_nocolor};
\path (1 ,1) pic{circle_nocolor}; 
\path (2 ,2) pic{circle_nocolor};
\end{scope}
\end{tikzpicture}
\quad  = 
\quad\bullet  \star  (\,  (\, \bullet  \star \bullet  )\star  \bullet   )\\
& \\
\quad  \Rightarrow    \qquad 
\begin{tabular}{|c|c|c|}
\hline  
1 &  2 & 4\\
\hline  
3 &  5 & 6 \\
\hline 
\end{tabular} \quad 
= & \quad 
\emptyset \star ((\emptyset\star\emptyset)\star\emptyset)
% \begin{tikzpicture}[ baseline=0.15cm] \path (0,0) pic[scale=0.6,radius=2pt]{gentrt};  \end{tikzpicture}\star
% \Bigl(\,
% \Bigl(\,
% \begin{tikzpicture}[ baseline=0.15cm] \path (0,0) pic[scale=0.6,radius=2pt]{gentrt};  \end{tikzpicture}
% \star
% \begin{tikzpicture}[ baseline=0.15cm] \path (0,0) pic[scale=0.6,radius=2pt]{gentrt};  \end{tikzpicture}
% \,\Bigr)
% \star
% \begin{tikzpicture}[ baseline=0.15cm] \path (0,0) pic[scale=0.6,radius=2pt]{gentrt};  \end{tikzpicture}
% \Bigr)\,.
\end{align*}
Another example is frieze sequences. 
These biject to  $n$-gon triangulations \cite{Conway:1973aa}: the frieze sequence is the same as the sequence obtained by a counter clockwise traversal of the polygon noting the degree of each vertex (less one). 
For example, if only a single level of factorisation is required then using the    $n$-gon triangulation bijection  gives
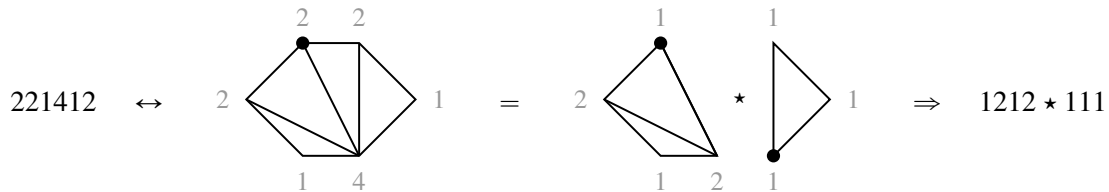
\begin{equation*}
 221412\quad \leftrightarrow\quad
\begin{tikzpicture}[style_size2,inner sep = 6pt,baseline =0.6cm,, node font=\small]
%  \draw[help lines] (0,0) grid (3,2);
  \coordinate (a) at (1,0)  ;
 \coordinate (b) at (0,1)  ;
 \coordinate (c) at (1,2)  ;
 \coordinate (d) at (2,2)  ;
 \coordinate(e) at (3,1) ;
 \coordinate (f) at (2,0) ;
 \draw (1,0) -- (0,1) -- (1,2) -- (2,2) -- (3,1) -- (2,0) -- cycle;
 \path (c) pic{circle_black};
\begin{scope}[color=gray!80]
 \node[below] (an) at (a) {1};
 \node[left] (bn) at (b) {2};
 \node[above] (cn) at (c) {2};
 \node[above] (dn) at (d) {2};
 \node[right] (en) at (e) {1};
 \node[below] (fn) at (f) {4};
\end{scope}
 \draw (b) -- (f);
 \draw (c) -- (f);
 \draw (d) -- (f);
\end{tikzpicture}
\quad = \quad 
\begin{tikzpicture}[style_size2,inner sep = 6pt,baseline =0.6cm,, node font=\small]
%   \draw[help lines] (0,0) grid (3,2);
  \coordinate (a) at (1,0)  ;
 \coordinate (b) at (0,1)  ;
 \coordinate (c) at (1,2)  ;
 \coordinate (d) at (3,2)  ;
 \coordinate(e) at (4,1) ;
 \coordinate (f) at (2,0) ;
  \coordinate (f2) at (3,0) ;
 \draw (a) -- (b) -- (c) -- (f) -- cycle;
 \draw (f2) -- (d) -- (e) -- cycle;
 
 \node (m) at (2.4,1){$\star$};
 
 \path (c) pic{circle_black};
  \path (f2) pic{circle_black};
   \begin{scope}[color=gray!80]
 \node[below] (an) at (a) {1};
 \node[left] (bn) at (b) {2};
 \node[above] (cn) at (c) {1};
 \node[above] (dn) at (d) {1};
 \node[right] (en) at (e) {1};
 \node[below] (fn) at (f) {2};
  \node[below] (fn2) at (f2) {1};
  \end{scope}
 \draw (b) -- (f);
 \draw (c) -- (f);
 
\end{tikzpicture}
\quad \Rightarrow \quad 1212\star 111
\end{equation*}
%
% \begin{align*}
%     &221412 \mapsto \raisebox{-0.5\height}{\includegraphics[scale=1]{figs/friezeEmd2-.pdf}}\\
%     &  = 1212\st 111\,.
%     %  (00\st(00\st00))\st(00\st00)
% \end{align*}
or, for a full decomposition the embedded CBT can be used:
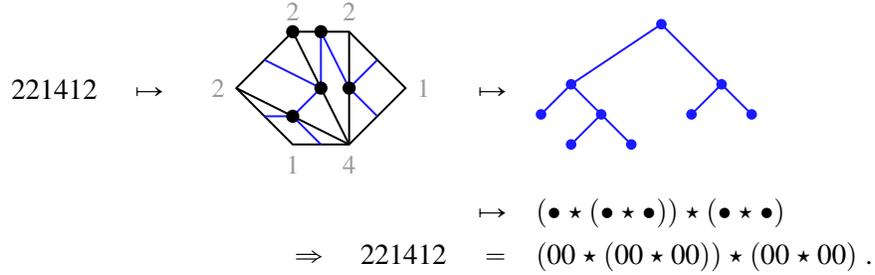
\begin{align*}
221412   \quad\mapsto\quad   
\begin{tikzpicture}[style_size2,inner sep = 4pt,baseline =0.6cm,, node font=\small]  
%  \draw[help lines] (0,0) grid (3,2);
  \coordinate (a) at (1,0)  ;
 \coordinate (b) at (0,1)  ;
 \coordinate (c) at (1,2)  ;
 \coordinate (d) at (2,2)  ;
 \coordinate(e) at (3,1) ;
 \coordinate (f) at (2,0) ;
 \draw (1,0) -- (0,1) -- (1,2) -- (2,2) -- (3,1) -- (2,0) -- cycle;
 \path (c) pic{circle_black};
 \begin{scope}[color=gray!80]
  \node[below ] (an) at (a) {1};
 \node[left] (bn) at (b) {2};
 \node[above] (cn) at (c) {2};
 \node[above] (dn) at (d) {2};
 \node[right] (en) at (e) {1};
 \node[below] (fn) at (f) {4};
 \end{scope}
 \draw (b) -- (f);
 \draw (c) -- (f);
 \draw (d) -- (f);
 \coordinate (t4) at (barycentric cs:c=1,d=1) ;
 \coordinate (t3) at (barycentric cs:d=1,f=1) ;
 \coordinate (t2) at (barycentric cs:c=1,f=1) ;
 \coordinate (t1) at (barycentric cs:f=1,b=1) ;
  \coordinate (t5) at (barycentric cs:d=1,e=1) ;
 \coordinate (t6) at (barycentric cs:e=1,f=1) ;
 \coordinate (t7) at (barycentric cs:f=1,a=1) ;
 \coordinate (t8) at (barycentric cs:a=1,b=1) ;
  \coordinate (t9) at (barycentric cs:b=1,c=1) ;
 \draw[tripleColor] (t4) -- (t3) ;
\draw[tripleColor] (t4) -- (t2) ;
\draw[tripleColor] (t3) -- (t5) ;
\draw[tripleColor] (t3) -- (t6) ;
\draw[tripleColor] (t2) -- (t1) ;
\draw[tripleColor] (t2) -- (t9) ;
\draw[tripleColor] (t1) -- (t7) ;
\draw[tripleColor] (t1) -- (t8) ;
\path[fill=tripleColor] (t4)  pic{circle_nocolor} ;
\path[fill=tripleColor] (t3)  pic{circle_nocolor} ;
\path[fill=tripleColor] (t2)  pic{circle_nocolor} ;
\path[fill=tripleColor] (t1)  pic{circle_nocolor} ;
\end{tikzpicture}
\quad \mapsto\quad  
                 &  
                 \begin{tikzpicture}[style_size4,tripleColor,baseline =0.6cm ]  
% \draw[help lines] (0,0) grid (7,4);
\draw (0,1) -- (1,2) -- (4,4) -- (6,2) -- (7,1);
\draw (1,2) -- (2,1) (2,1) -- (3,0) (2,1) -- (1,0) (6,2) -- (5,1);
\path (4,4) pic{circle_nocolor}  
(1,2) pic{circle_nocolor}  
(0,1) pic{circle_nocolor}
(2,1) pic{circle_nocolor}
(1,0) pic{circle_nocolor}
(3,0) pic{circle_nocolor} 
(6,2) pic{circle_nocolor} 
(5,1) pic{circle_nocolor} 
(7,1) pic{circle_nocolor} 
;
\end{tikzpicture}\\
    \quad\mapsto\quad & (\bullet \star (\bullet \star \bullet))\star (\bullet\star\bullet)\\
    % \quad\mapsto\quad & (00 \star (00 \star 00))\star (00\star 00)\\
\Rightarrow \quad 221412   \quad= \quad &     (00 \star (00 \star 00))\star (00\star 00) \;.
\end{align*}
%

% This process can be restated as an algorithm acting directly on the sequence.
% where $00$  is the frieze pattern generator. 
% Similarly to the tableau, if only a single level of factorisation is required then the triangulation   can be factored and each factor mapped back to a frieze sequence.

%============================================= 
\section{Narayana connection}  
 %=============================================
\label{sec_naray}

The distribution of  the number of right  (or left)  multiplications by the generator in the   decomposition of any element of some Catalan family $\ffm_i$ is of interest because, as will be shown below, it is given by the Narayana distribution. 
Let $N_{n,k}$ be the number of elements in $\canmag_\eps$ with norm $n$ and $k$ right multiplications by the generator.  
The $N_{n,k}$ distribution can also be considered to arise from a map 
$$
N_r:\canmag_{\eps}\to \mathbb{N}\,;\, w\mapsto N_r(w),
$$
where $N_r(w)$ is the number of right multiplications by the generator in the infix form of $w$, then $N_{n,k}=|\set{w\in \canmag_{\eps} : \norm{w}=n,\, N_r(w)=k}|$. 
In the following example the right multiplications  by the generator are boxed and the image under $N_r$ is shown:
\begin{align*}
	 \eps\st(\eps\st(\eps\st\fbox{$\eps$}))&\mapsto 1\\
	  (\eps\st(\eps\st\fbox{$\eps$}) )\st \fbox{$\eps$}&\mapsto 2\\
	 (\eps\st\fbox{$\eps$})\st(\eps\st\fbox{$\eps$})&\mapsto 2 \\
     \eps \st((\eps\st \fbox{$\eps$})\st \fbox{$\eps$})&\mapsto 2\\
	  ((\eps\st \fbox{$\eps$})\st \fbox{$\eps$})\st\fbox{$\eps$}&\mapsto 3. 
\end{align*} 
Thus $N_{4,1}=1$, $N_{4,2}=3$ and $N_{4,3}=1$.
The $N_r$ map on $\canmag_{\eps}$ is clearly additive over the product, that is, it is an additive morphism:
\begin{equation}
    N_r(w_1\st w_2) = N_r(w_1)+N_r( w_2)\,.
\end{equation}

From the counter-clockwise traversal of a complete binary tree -- see \figref{fig:treeTraversal} -- it is clear that right leaves correspond to right multiplications by the generator. 
Furthermore, from the embedding bijection of complete binary trees into Dyck paths (see \eqref{eq_cbtPathem}) it is clear that peaks in the path (ie.\ an up step immediately followed by a down step) correspond to right leaves in the tree and hence the number of peaks in a $2n$ step path is also given by $N_{n,k}$. 
Finally,   Narayana showed \cite{narayana:1959ix} that the number of $2n$ step paths with $k$ peaks  is given by  
\begin{equation}\label{eq_nara}
 N_{n,k}=\frac{1}{n-1}\binom{n-1}{k}\binom{n-1}{k-1}\,,\qquad 1\le k\le n-1\,.
\end{equation}
which tells us that $N_{n,k}$ is the Narayana distribution.

Since the universal map $\unimap: \ffm_i \to \ffm_j$ does not change the number of right multiplications the map, $N_r$  is invariant under the action of $\unimap$, that is,   for any pair $w\in\ffm_i$ and $u\in\ffm_j$ with $u=\unimap(w)$ we have that  
$$
N_r(w)=N_r(\unimap(w)).
$$
Thus we have the following result.
\begin{prop} \label{thm_nara}
Let $(\ffm_i,\st_i)$ be a Catalan magma with generator, $\eps_i$,  and $\ffm_{i,n}$ the subset of elements with \val $n$. Then the number of elements, $N_{n,k}$, in $\ffm_{i,n}$ whose decomposition has $k$ right multiplications by the generator,  is given by the Narayana distribution \eqref{eq_nara}.
\end{prop}
We are interested in knowing whether  there is  any structure or pattern in the Catalan object that is directly associated with right multiplication and hence will have a Narayana distribution.
The previous examples: the right leaves of the complete binary tree and the number of  peaks  in a  Dyck path  \cite{sulanke:1998mz,sulanke:2002lc} are two such structures. 

The previous argument connecting peaks in Dyck paths to right multiplications is a specific instance of the following method: 
if the universal bijection   embeds $w$ into $u=\unimap(w)$ and we can   identify the Narayana structures in $w$, then their image under the embedding form of $\unimap$   will give the corresponding structure in the image object.

Using the embedding of the complete binary tree into a triangulation  -- illustrated in \eqref{cbtEmbedTri7} -- shows that the Narayana structure  is (schematically),
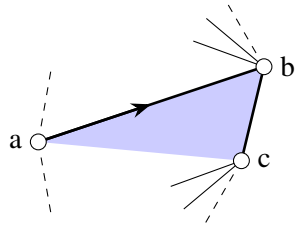
\begin{center}
\begin{tikzpicture}[>=Stealth]
\coordinate (A) at (0,0);
\coordinate (B) at (3,1);
\coordinate (C) at (2.7,-0.25);
\coordinate (M) at (barycentric cs:A=1,B=1);
\fill[blue,opacity=0.2] (A) -- (B) -- (C) -- cycle;
\draw[line width=1pt] (A) edge[->] (M) -- (B) (B) -- (C) ;
\draw[dashed] (A) -- (80:1cm) (A) -- (-80:1cm);
\draw[dashed] (B) -- ++(120:1cm) (C) -- ++(-120:1cm) ;
\draw (B) -- ++(140:1cm) (B) -- ++(160:1cm) ;
\draw (C) -- ++(-140:1cm) (C) -- ++(-160:1cm) ;
\draw[fill=white] (A) circle (3pt) node[anchor=east,inner sep = 6pt] {a} ;
\draw[fill=white] (B) circle (3pt) node[anchor=west,inner sep = 6pt] {b} ;
\draw[fill=white] (C) circle (3pt) node[anchor=west,inner sep = 6pt] {c} ;    
\end{tikzpicture}
\end{center}
that is, the polygon edge $cb$ which is defined by the existence of a  chord $ab$ that is incident on a node $b$ and adjacent to an earlier (traversing the polygon counter-clockwise from the marked node) node $a$. 
Note, the chord $ab$ can also be the polygon edge to the right of the marked vertex.

A related approach is to consider the schematic form of the magma product and to track the effect of a right multiplication.
For example, for the Catalan floor plans (\fref{fam_fp} in Appendix) a right multiplication by the generator (an edge) results in a rectangle spanning  the floor plan (ie.\ a new row), which upon further multiplication on the left  by an arbitrary floor plan,  moves the `row' (rectangle 3) to the right hand side of the floor plan, that is, 
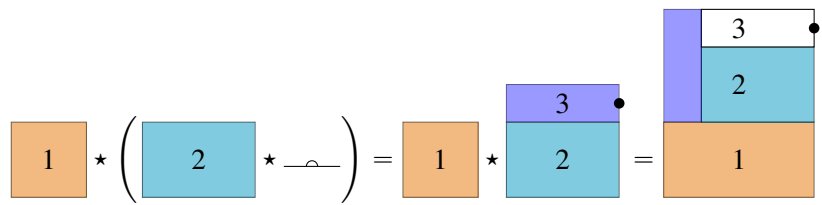
\begin{equation*}
\begin{tikzpicture}[baseline=0.4cm]
\fill[color_left,draw=black] (0,0) rectangle (1,1) ; 
\node at (0.5,0.5) {1};
\end{tikzpicture}
\star
\Biggl(\,
 \begin{tikzpicture}[baseline=0.4cm]
\fill[color_right,draw=black] (0,0) rectangle (1.5,1) ; 
\node at (0.75,0.5) {2};
\end{tikzpicture}
\star
\gensp
\Biggr)
=
\begin{tikzpicture}[baseline=0.4cm]
\fill[color_left,draw=black] (0,0) rectangle (1,1) ; 
\node at (0.5,0.5) {1};
\end{tikzpicture}
\star
\begin{tikzpicture}[baseline=0.4cm]
\fill[color_right,draw=black] (0,0) rectangle (1.5,1); 
\fill[color_product,draw=black] (0,1) rectangle (1.5,1.5); 
  \node at (0.75  ,1.25) {3};
\fill[black] (1.5,1.25) circle (2pt);
\node at (0.75,0.5) {2};
\end{tikzpicture}
=
\begin{tikzpicture}[baseline=0.4cm]
\fill[color_left,draw=black] (0,0) rectangle (2,1) ; 
\node at (1,0.5) {1};
 \fill[color_right,draw=black] (0.5,1) rectangle (2,2); 
 \node at (1 ,1.5) {2};
 \fill[white, draw=black] (0.5,2) rectangle (2,2.5); 
  \node at (1 ,2.25) {3};
 \fill[black] (2,2.25) circle (2pt);
 \fill[color_product,draw=black] (0,1) rectangle (0.5,2.5); 
\end{tikzpicture}
\end{equation*}
%
% \begin{center}
%     \includegraphics[scale=\bbfscale]{figs/floorPlans_nara_lr.pdf}
% \end{center}
Multiplication on the right by an arbitrary plan does not change this right alignment. Furthermore, the only way to add a new row is by right multiplication by the generator.
Thus the right multiplication rectangle always attaches to the right side of the floor plan   at any stage of multiplication which shows that the Narayana structure corresponds to the  number of `rows' on the right side of the floor plan.

Other properties of the distribution can be seen in the magma structure.
There is an obvious symmetry $k\to n-k$ obtained algebraically from the binomial identity $\binom{a}{b}=\binom{a}{a-b}$ showing that
$
N_{n,k}=N_{n,n-k}\,.
$
The map $k\to n-k$  corresponds to reversing the decomposition word (ie.\ the map  defined by \eqref{eq_revmap}) thus replacing right multiplication by a generator by left multiplication by a generator. Thus the number of left generators $N_\ell$ takes  on the `complementary' value $N_\ell(w)=n-N_{r}(w)$ but  has the same distribution $N_{n,\ell}$ where  $n=\norm{w}$.

%=============================================
\section{The symbolic enumeration method and magmatisation} 
%=============================================
\label{sec_symbol}

In this section we briefly discuss the relation between single generator Catalan families considered as magmas and the symbolic enumeration method  for counting Catalan families. The latter method is  explained and illustrated in the book by Flajolet and Sedgewick \cite{flajolet:2001ys} and used in Goulden and Jackson \cite{goulden:1983yg}.  Bracketing words are discussed in Comtet \cite{Comtet:1974aa} and their representation of complete binary trees and polygon triangulations.

The symbolic method is primarily a technique for finding an equation satisfied by the  generating function  of a  combinatorial family  and thereby solving the enumeration problem. 
The generating function equation is obtained by a   translation of a ``set equation'' satisfied by the set of combinatorial objects, $\struct$ (or a set in bijection with the objects of the family). 
The set $\struct$ is defined as a solution to a set equation which is an equation  containing $\struct$,  sets of `atomic' objects $\set{z_1,z_2,\dots}$, disjoint unions  $\dot{+}$ , Cartesian products  $\times$,  markings etc. As will be seen in the examples below, in some sense the atoms ``generate'' the elements of the combinatorial set but they do not always correspond to the magma generator(s).
The symbolic method is well suited to finding the generating function  of a \textit{single  family} as a large amount of detail can be placed in the elements of $\struct$ which then translates  directly to a multi-variable generating function. 

Different families (with the same generating functions)  can have different defining set equations.  
For example, the set equation for complete binary trees \cite{flajolet:2001ys} can be written with atoms, $z_1=\Box$ (representing leaves) and $z_2=\bullet$ (representing internal nodes). The set $T$, in simple bijection with the trees (assuming trees are defined graphically), satisfies the set  equation
\begin{equation}\label{eq:tsseq}
 T = \set{\Box}\,\dot{+}\,\set{\bullet}\times T\times T
\end{equation}
and  begins
\begin{equation}\label{eq:tss}
    T=\set{\Box,\, (\bullet,\Box,\Box),\, (\bullet, \Box,(\bullet,\Box,\Box)),\,\dots }\,.
\end{equation}
For Dyck paths there are three atoms $z_1=\circ$ (representing  `left' vertices), $z_2=\usp$ (representing up steps) and $z_3=\dsp $ (representing  down steps). 
The set $D$, in simple bijection with the set of  Dyck paths (assuming paths are defined graphically), satisfies the set  equation  
\begin{equation}\label{eq:dsseq}
 D = \set{\circ}\,\dot{+}\,    D\times\set{\usp} \times D\times \set{\dsp}
\end{equation}
and  begins
\begin{equation}\label{eq:dss}
     D=\set{\circ,\, ( \circ,\usp,\circ,\dsp),\,  
 (\circ,\usp,( \circ,\usp,\circ,\dsp),\dsp),\,\cdots }\,.
\end{equation}
A given family can be associated with several different set equations, each equation defining a slightly different combinatorial set. For example, complete binary trees can equally well be associated with the set equation  
\begin{equation}\label{eq:tsseq2}
    T' = \set{\Box}\,\dot{+}\,T'\times \set{\bullet}\times T' 
\end{equation} 
with
\begin{equation}\label{eq:tss2}
    T'=\set{\Box,\, (\Box,\bullet,\Box),\, (\Box,\bullet, (\Box,\bullet,\Box)),\,\cdots }\, 
\end{equation}
Clearly, in this case, there is a simple bijection relating $T$ and $T'$.
Because of this type of  flexibility, the symbolic enumeration method  can  result in many different (but similar) symbolic expressions representing the \emph{same} combinatorial family and thus bijecting one Catalan family to another using the elements defined by the associated set equation becomes  sensitive to the precise details of the set equation.

The situation with a magma is different -- it is an algebraic structure.  
It \textit{assumes} there exists a base set and then defines a product map on top of the base set. Furthermore,   there is a natural notion of  maps \textit{between} magmas (the morphisms). 
The magma product map then determines the properties (is it free?, what are the irreducible elements? etc.). 
The morphisms give  us relations between different families ie.\  recursive bijections.
In the single generator Catalan case, each family is an isomorphic concrete representation of the set of elements of an algebraic structure (a free magma generated by a single element of the base set). 
The standard form of the symbolic method lacks the notion of irreducible elements (the `atoms' may or may not be magma generators) and of morphisms (constraining the maps \textit{between} families with the same generating functions). 
 
 It is however constructive to think of the set equations arising in the symbolic method as equations that   define the elements of the base set. 
 One can then place a product structure on top of this base set. 
 For example,   for complete binary trees, rather than starting with a graphical definition of the base set (and hence define the product as with $\ffm_{\ref{fam_bt}}$), we can start with the base set defined by the set equation, \eqref{eq:tsseq}, then define a  product $\st_1: T\times T\to T$   as 
 \[
  t_1\st_1 t_2 = (\bullet, t_1,t_2)\qquad \text{for all $t_1,t_2\in T$} 
 \]
 and define the norm as the number of $\Box$'s in the element.  
 Then  the pair $(T,\st_1)$ is   a Catalan magma (proved by showing \thmref{thm_free} is satisfied) with generator $\Box$ (the only irreducible element in $T$). 
 Thus in this case one of the atoms is the generator whilst the other is  part of the construction defining  the elements of the base set.  
 
Similarly, for the Dyck path example with base set $D$ given by \eqref{eq:dss}, a possible  product is
 \[
  d_1\st_2 d_2 = ( d_1,\usp, d_2, \dsp) 
 \]
with  norm defined as  the number of $\circ$'s in the element. Then the pair $(D,\st_2)$  is a Catalan magma with generator $\circ$ (the only irreducible element in $D$), again,  the two other atoms $\usp$ and $\dsp$ are not generators but only part of the definition of the base set. Clearly in both these examples the definition of the products is closely related to the Cartesian product  in the set equations as both are related to the recursive structure of the family.
 
By adding products to the base sets defined by the symbolic method (and proving they are Catalan magmas)  they have been effectively ``magmatised''.
The notion of an isomorphism between them is now well defined  and can be used to define  bijections between the base sets. 
For example, using  $T$ and $D$ the universal bijection, \defref{def_unibij}, gives,
\begin{equation}\label{eq:tdssb}
 (\bullet, \Box,(\bullet,\Box,\Box))=\Box\st_1 (\Box\st_1 \Box))
\quad \mapsto\quad  
\circ\st_2 (\circ\st_2 \circ)=(\circ,\usp,( \circ,\usp,\circ,\dsp),\dsp)   
\end{equation}
thus $(\bullet, \Box,(\bullet,\Box,\Box))$ bijects to $(\circ,\usp,( \circ,\usp,\circ,\dsp),\dsp)$ as expected. 

Note, changing the details of the complete binary tree set equation eg.\ \eqref{eq:tsseq} to \eqref{eq:tsseq2}, changes the product definition in a minor way -- it becomes $t_1\st_1 t_2 = ( t_1,\bullet,t_2)$ --  and similarly  changes the details of the bijection \eqref{eq:tdssb} in a minor way (note, the form of the left element changes to $( \Box,\bullet,(\Box,\bullet,\Box)) )$. 
 
In the  set equations \eqref{eq:tsseq} and \eqref{eq:dsseq} the base set only appears twice in each Cartesian product. 
In other applications of the symbolic method higher order Cartesian products occur as well as other set constructions. It is of interest to see how the ideas associated with free magmas carry across to these cases. 
Some progress with this generalisation has been made in \cite{Brak:2018aa} where a unary map is added in addition to a product map to study Motzkin and Schr\"oder families.

%=============================================
\section{Acknowledgement}
%=============================================

 I would like to   thank the Australian Research Council (ARC) and the Centre of Excellence for Mathematics and Statistics of Complex Systems (MASCOS) for financial support. I would also like to acknowledge the contributions of A.\ Alvey-Price, V.\ Chockalingam and N.\ Mahony. I would also  like to thank V.\ Vatter for useful  feedback and finally, X.\ Viennot who introduced me to Catalan numbers many years ago.

\newpage

%=============================================
\section{Appendix}
%=============================================

In this appendix we list several Catalan families and a list of information defining their magma structure. 

We use the convention of Stanley \cite{stanley:1999vw2,Stanley:2015aa}  of giving a brief description of the family followed by the list of the  five   objects with \val four. If the problem (or trivially equivalent) occurs in \cite{Stanley:2015aa} then it is referenced in the section header.

For each family the following information is itemised:
\begin{itemize}

\item The generator of the magma. \\
-- It is not strictly necessary to define exactly what $\eps$ maps to, so long as the left and right multiplication by $\eps$ is defined (and the element $\eps\st\eps$).
\item A schematic definition of the product. \\
	-- We will use the same symbol, $\st$ for the product for every family, but of course, the product is different in each case, so strictly a different  product symbol should be used in each case.

\item  \Valns. \\
-- In all cases the \val of the generator is one: $\norm{\eps}=1 $. This, together with the additive property (see \defref{def_norm}) defines the \val of every object. 
However, the \val will map to a conventional or convenient parameter associated with the objects in each particular family. This correspondence is stated here.
\item Examples of the product. \\
-- Products where the LHS or RHS is a generator can be poorly defined when using only the schematic diagram. 
An alternative way to define the product is to consider each case separately as illustrated in the  binary tree family \fref{fam_bt}. 
However, doing this in every family can obscure the simplicity of the general case product, thus we clarify the product when one side is a generator by giving  examples. 
For several of the examples the geometry (if appropriate) associated with the generator is emphasised as is the product geometry. 
 
\end{itemize}
Given the above data for any family it is then possible to biject  any  object of one family to its image in another family using the universal bijection \eqref{eq_unibi}.

\newpage

%%=============================================
%% Catalan Familie
%%=============================================

\section*{Catalan Families}

\setcounter{family}{0}

\begin{itemize}
    \item[\fref{fam_cm}] -- Cartesian magma
    \item[\fref{fam_mb}] -- Matching brackets and Dyck words
    \item[\fref{fam_nc}] -- Non-crossing chords -– the circular form of nested matchings
    \item[\fref{fam_bt}] -- Complete Binary Trees and Binary Trees
    \item[\fref{fam_pt}] -- Planar Trees
    \item[\fref{fam_ld}] -- Nested matchings or Link Diagrams
    \item[\fref{fam_np}] -- Non-crossing partitions
    \item[\fref{fam_dp}] -- Dyck paths
    \item[\fref{fam_pt}] -- Polygon triangulations
    \item[\fref{fam_ap}] -- 321-avoiding permutations
    \item[\fref{fam_sp}] -- Staircase polygons
    % \item[\fref{fam_hs}] -- Pyramid of heaps of segments
    \item[\fref{fam_st}] -- Two row standard tableaux
    \item[\fref{fam_fp}] -- Floor Plans
    \item[\fref{fam_fz}] -- Frieze Patterns
    
\end{itemize}

%=============================================
%  Matching brackets/Dyck words
%=============================================
\newcommand{\stan}[1]{(Ex#1 of \cite{Stanley:2015aa})}
\newcommand{\tstan}[2]{(Ex#1, #2 of \cite{Stanley:2015aa})}
\newcommand{\apphead}[1]{ \subsection{\textbf{#1}  \hspace{3em}}}

\newcommand{\tempnp}{\vspace{2em}}

\tempnp

\raggedbottom

%>>>>>>>>>>>>>>>>>>>>>>>>>>>>>>>>>>>>>>>>>>>>>>>>>>>>>>>>>>>>>>>>>>>>>>>>>>
 \begin{family}[Cartesian magma] \label{fam_cm} 
%>>>>>>>>>>>>>>>>>>>>>>>>>>>>>>>>>>>>>>>>>>>>>>>>>>>>>>>>>>>>>>>>>>>>>>>>>>

 \addcontentsline{toca}{subsection}{\fref{fam_cm}: Cartesian Magma}
 
The single generator  Cartesian magma is the set of nested 2-tuples defined in \defref{def_catmag} with $X=\set{\aeps}$.
\begin{equation*}
     (\aeps,(\aeps,(\aeps,\aeps))), 
     \quad
     ((\aeps,(\aeps,\aeps) ),\aeps),
    \quad
    (\aeps, ((\aeps,\aeps),\aeps)),  
    \quad
    (((\aeps,\aeps),\aeps ), \aeps), 
    \quad
    ((\aeps,\aeps),(\aeps,\aeps)) \,.
\end{equation*}
\begin{itemize}
\item Generator: The symbol $\eps$.
\item Product: $m_1\st m_2= (m_1,m_2)$.
\item Norm: Number of occurrences of $\eps$ in the nested 2-tuple.
\item Examples:
\[
\eps\st\eps = (\aeps,\aeps)
\]
\[
(\aeps,\aeps)\st\aeps=((\aeps,\aeps),\aeps)
\]
\[
\aeps\st(\aeps,\aeps)  =(\aeps,(\aeps,\aeps))
\]

\end{itemize}
Note, the product is as close to syntactic concatenation of $m_1$ and $m_2$ as possible.

\end{family}
\tempnp
%>>>>>>>>>>>>>>>>>>>>>>>>>>>>>>>>>>>>>>>>>>>>>>>>>>>>>>>>>>>>>>>>>>>>>>>>>>
 \begin{family}[Matching brackets and Dyck words  \tstan{77}{211}] \label{fam_mb} 
 %>>>>>>>>>>>>>>>>>>>>>>>>>>>>>>>>>>>>>>>>>>>>>>>>>>>>>>>>>>>>>>>>>>>>>>>>>>
 
\addcontentsline{toca}{subsection}{\fref{fam_mb}: Matching brackets and Dyck words}

A language in the alphabet containing the two symbols $\{$ and $\}$ (or $x$ and $\bar{x}$ for Dyck words), such that there is an equal number of left as right brackets and  the  brackets match  ``pairwise'' (ie.\ every prefix  of the word contains no  more  right brackets than left).
%\[
%	  \eps\st(\eps\st(\eps\st\eps))   \qquad (\eps\st(\eps\st\eps) )\st \eps \qquad 
%	  (\eps\st(\eps\st\eps) )\st \eps   \qquad  \eps\st(\eps\st(\eps\st\eps))  \qquad   
% (\eps\st\eps)\st(\eps\st\eps)    
% \]
\[ 
 \{\}\{\{\}\}\qquad \{\{\}\{\}\}  \qquad \{\{\{\}\}\}\qquad \{\{\}\}\{\} \qquad \{\}\{\}\{\}
\]
\begin{itemize}
\item Generator: $\eps $ = $\emptyset\qquad$ (ie.\ the empty word)

\item Product:  $b_1 \st  b_2 =  b_1\{\,b_2\, \}\qquad $  ie.\ word concatenation (and add two  brackets).
\item $\text{\Valns}=\text{(Half the number of brackets)}+1$.
\item Examples:
\[
\emptyset\st\emptyset = \{\}
\]
\[
\{\}\st\emptyset=\{\}\{\}
\]
\[
\emptyset\st\{\} =\{\{\}\}
\]
\[
\{\}\st\{\} =\{\}\{\{\}\}
\]

% \item Narayana parameter: The number of $\{\}$ subwords.
\end{itemize}

% \begin{list}{$\bullet\,\,F_{\arabic{family}}$:}{\usecounter{family}}

\end{family}

%=============================================

\tempnp
%>>>>>>>>>>>>>>>>>>>>>>>>>>>>>>>>>>>>>>>>>>>>>>>>>>>>>>>>>>>>>>>>>>>>>>>>>>
\begin{family}[Non-crossing chords -- the circular form of nested matchings \stan{59}]
%>>>>>>>>>>>>>>>>>>>>>>>>>>>>>>>>>>>>>>>>>>>>>>>>>>>>>>>>>>>>>>>>>>>>>>>>>>
\label{fam_nc} 

\addcontentsline{toca}{subsection}{\fref{fam_nc}: Non-crossing chords}

A circle  with $2n$ nodes and one marked segment of the circumference. Each node is joined to  exactly one other node by a chord such that the chords do not intersect. 
\begin{equation*}
\begin{tikzpicture}
\coordinate (A) at (60:1cm);
\coordinate (B) at  (0:1cm);
\coordinate (C) at  (-60:1cm);
\coordinate (D)  at (120:1cm);
\coordinate (E)  at (180:1cm);
\coordinate (F)  at (240:1cm);
\draw (0,0) circle (1cm) (90:1cm) pic{tick2};
\fill[black] (60:1cm) circle (2pt)  (0:1cm) circle (2pt) (-60:1cm) circle (2pt) ;
\fill[black] (120:1cm) circle (2pt)  
(180:1cm) circle (2pt) 
(240:1cm) circle (2pt) ;
\draw[black] (A) -- (D)  (B)--(E) (C)--(F);
\end{tikzpicture}
\qquad
\begin{tikzpicture}
\coordinate (A) at (60:1cm);
\coordinate (B) at  (0:1cm);
\coordinate (C) at  (-60:1cm);
\coordinate (D)  at (120:1cm);
\coordinate (E)  at (180:1cm);
\coordinate (F)  at (240:1cm);
\draw (0,0) circle (1cm) (90:1cm) pic{tick2};
\fill[black] (60:1cm) circle (2pt)  (0:1cm) circle (2pt) (-60:1cm) circle (2pt) ;
\fill[black] (120:1cm) circle (2pt)  
(180:1cm) circle (2pt) 
(240:1cm) circle (2pt) ;
\draw[black] (A) -- (B)  (C)--(D) (E)--(F);
\end{tikzpicture}
\qquad
\begin{tikzpicture}
\coordinate (A) at (60:1cm);
\coordinate (B) at  (0:1cm);
\coordinate (C) at  (-60:1cm);
\coordinate (D)  at (120:1cm);
\coordinate (E)  at (180:1cm);
\coordinate (F)  at (240:1cm);
\draw (0,0) circle (1cm) (90:1cm) pic{tick2};
\fill[black] (60:1cm) circle (2pt)  (0:1cm) circle (2pt) (-60:1cm) circle (2pt) ;
\fill[black] (120:1cm) circle (2pt)  
(180:1cm) circle (2pt) 
(240:1cm) circle (2pt) ;
\draw[black] (A) -- (D)  (B)--(C) (E)--(F);
\end{tikzpicture}
\qquad
\begin{tikzpicture}
\coordinate (A) at (60:1cm);
\coordinate (B) at  (0:1cm);
\coordinate (C) at  (-60:1cm);
\coordinate (D)  at (120:1cm);
\coordinate (E)  at (180:1cm);
\coordinate (F)  at (240:1cm);
\draw (0,0) circle (1cm) (90:1cm) pic{tick2};
\fill[black] (60:1cm) circle (2pt)  (0:1cm) circle (2pt) (-60:1cm) circle (2pt) ;
\fill[black] (120:1cm) circle (2pt)  
(180:1cm) circle (2pt) 
(240:1cm) circle (2pt) ;
\draw[black] (A) -- (F)  (B)--(C) (D)--(E);
\end{tikzpicture}
\qquad
\begin{tikzpicture}
\coordinate (A) at (60:1cm);
\coordinate (B) at  (0:1cm);
\coordinate (C) at  (-60:1cm);
\coordinate (D)  at (120:1cm);
\coordinate (E)  at (180:1cm);
\coordinate (F)  at (240:1cm);
\draw (0,0) circle (1cm) (90:1cm) pic{tick2};
\fill[black] (60:1cm) circle (2pt)  (0:1cm) circle (2pt) (-60:1cm) circle (2pt) ;
\fill[black] (120:1cm) circle (2pt)  
(180:1cm) circle (2pt) 
(240:1cm) circle (2pt) ;
\draw[black] (A) -- (B)  (C)--(F) (E)--(D);
\end{tikzpicture}
\end{equation*} 
%
% \begin{center}
%  \includegraphics[scale=\fscale]{figs/non-crossingChordsExample_lr.pdf}
% %\includegraphics{figs/ex-ncc}	
% \end{center}

\begin{itemize}
\item Generator: $\eps=\tikz{\path pic{genncc}} $
\item Product:   
 \begin{equation*}
\begin{tikzpicture}[baseline=0cm] 
\coordinate (A) at (34.23:1cm);
\coordinate (B) at  (145.77:1cm);

\coordinate (C) at  (-50:1cm);
\coordinate (D)  at (-130:1cm);

\draw (0,0) circle (1cm) (90:1cm) pic{tick2};

\begin{scope}
\draw[clip] (0,0) circle (1cm); 
\fill[color_left,draw=black ] (0,0) circle (1cm);
\fill[white, draw=black,densely dashed,line width=1pt] (0,1.25) circle (1.05cm);
\end{scope}

\fill (A) circle (3pt) (B) circle (3pt);

\path (A) node[anchor=south west] {$b$};
\path (B) node[anchor=south east] {$a$};

\end{tikzpicture}
\quad\star\quad
\begin{tikzpicture}[baseline=0cm]  
\coordinate (A) at (34.23:1cm);
\coordinate (B) at  (145.77:1cm);

\coordinate (C) at  (-50:1cm);
\coordinate (D)  at (-130:1cm);

\draw (0,0) circle (1cm) (90:1cm) pic{tick2};

\begin{scope}
\draw[clip] (0,0) circle (1cm); 
\fill[color_right,draw=black ] (0,0) circle (1cm);
\fill[white, draw=black,densely dashed,line width=1pt] (0,1.25) circle (1.05cm);
\end{scope}

\fill (A) circle (3pt) (B) circle (3pt);
\path (A) node[anchor=south west] {$d$};
\path (B) node[anchor=south east] {$c$};
\end{tikzpicture}
\quad =\quad
\begin{tikzpicture}[baseline=0cm]  
\coordinate (A) at (50.23:1cm);
\coordinate (B) at  (116.77:1cm);

\coordinate (C) at  (-55:1cm);
\coordinate (D)  at (-116:1cm);

\coordinate (E) at  (70:1cm);
\coordinate (F)  at (-90:1cm);

\draw (0,0) circle (1cm) (90:1cm) pic{tick2};

\begin{scope}
\draw[clip] (0,0) circle (1cm); 

\fill[color_left,draw=black,densely dashed] (0,-2) to [bend left =25] (0,2) -- (-3,2) -- (-3,-2) -- cycle;
\fill[color_right,draw=black,densely dashed] (0.1,-2) to [bend right =25] (0.25,2) -- (3,2) -- (3,-2) -- cycle;
 
\end{scope}

\fill (A) circle (3pt) (B) circle (3pt);
\fill (C) circle (3pt) (D) circle (3pt);

\fill[orange] (E) circle (3pt) (F) circle (3pt) ;
\draw[orange] (E)--(F);

\path (A) node[anchor=south west] {$d$};
\path (B) node[anchor=south east] {$a$};
\path (C) node[anchor=north west] {$c$};
\path (D) node[anchor=north east] {$b$};

\end{tikzpicture}
 \end{equation*}
% \begin{center}
% \includegraphics[scale=\fscale]{figs/non-crossingChords_lr.pdf} 
% \end{center}

% \item Product:  \raisebox{-2 ex}{ \includegraphics[width=7cm]{figtemp/non-crossingChordsProd.png}}

\item $\text{\Valns}=\text{(Number of chords)}+1$ 

\item Examples:
\begin{align*}
\begin{tikzpicture} \path pic{genncc}; \end{tikzpicture}
\quad \star \quad
\begin{tikzpicture} \path[scale=2] pic{genncc}; \end{tikzpicture}
&    \quad = \quad 
\begin{tikzpicture}
    \coordinate (A) at (45:0.5cm) ;
    \coordinate (B) at (45-180:0.5cm);
    \path (135:0.5cm) pic{tick};
    \draw  (0,0) circle (0.5cm);
    \draw[orange,line width=1pt] (A) -- (B);
    \fill (A) circle (2pt) (B) circle (2pt);
\end{tikzpicture}\\
& \\
\begin{tikzpicture}
    \coordinate (A) at (45:0.5cm) ;
    \coordinate (B) at (-135:0.5cm);
    \path (135:0.5cm) pic{tick};
    \draw  (0,0) circle (0.5cm);
    \draw  (A) -- (B);
    \fill (A) circle (2pt) (B) circle (2pt);
\end{tikzpicture}
\quad \star \quad
\begin{tikzpicture} \path pic{genncc}; \end{tikzpicture}
&    \quad = \quad 
\begin{tikzpicture}
    \coordinate (A) at (63:0.5cm) ;
    \coordinate (B) at (-9:0.5cm);
    \coordinate (C) at (-81:0.5cm) ;
    \coordinate (D) at (-153:0.5cm);
    \path (135:0.5cm) pic{tick};
    \draw  (0,0) circle (0.5cm);
    \draw (C) -- (D) ;
    \draw[orange,line width=1pt] (A) -- (B);
    \fill (A) circle (2pt) (B) circle (2pt);
    \fill (D) circle (2pt) (C) circle (2pt);    
\end{tikzpicture}\\
& \\
\begin{tikzpicture} \path pic{genncc}; \end{tikzpicture}
\quad \star \quad
\begin{tikzpicture}
    \coordinate (A) at (45:0.5cm) ;
    \coordinate (B) at (-135:0.5cm);
    \path (135:0.5cm) pic{tick};
    \draw  (0,0) circle (0.5cm);
    \draw  (A) -- (B);
    \fill (A) circle (2pt) (B) circle (2pt);
\end{tikzpicture}
&    \quad = \quad 
\begin{tikzpicture}
    \coordinate (A) at (90:0.5cm) ;
    \coordinate (B) at (0:0.5cm);
    \coordinate (C) at (-90:0.5cm) ;
    \coordinate (D) at (-180:0.5cm);
    \path (135:0.5cm) pic{tick};
    \draw  (0,0) circle (0.5cm);
    \draw[orange,line width=1pt] (A) -- (D) ;
    \draw (C) -- (B);
    \fill (A) circle (2pt) (B) circle (2pt);
    \fill (D) circle (2pt) (C) circle (2pt);    
\end{tikzpicture}\\
& \\
\begin{tikzpicture}
    \coordinate (A) at (15:0.5cm) ;
    \coordinate (B) at (255:0.5cm);
    \path (135:0.5cm) pic{tick};
    \draw  (0,0) circle (0.5cm);
    \draw  (A) -- (B);
    \fill (A) circle (2pt) (B) circle (2pt);
\end{tikzpicture}
\quad \star \quad
\begin{tikzpicture}
    \coordinate (A) at (15:0.5cm) ;
    \coordinate (B) at (255:0.5cm);
    \path (135:0.5cm) pic{tick};
    \draw  (0,0) circle (0.5cm);
    \draw  (A) -- (B);
    \fill (A) circle (2pt) (B) circle (2pt);
\end{tikzpicture}
&    \quad = \quad 
\begin{tikzpicture}
    \coordinate (A) at (45:0.5cm) ;
    \coordinate (B) at (-15:0.5cm);
    \coordinate (C) at (-75:0.5cm) ;
    \coordinate (D) at (-135:0.5cm);
    \coordinate (E) at (-195:0.5cm) ;
    \coordinate (F) at (-255.5:0.5cm);    
    \path (135:0.5cm) pic{tick};
    \draw  (0,0) circle (0.5cm);
    \draw[orange,line width=1pt] (F) -- (C);
    \draw (A) -- (B) (E) -- (D);
    \fill (A) circle (2pt) (B) circle (2pt);
    \fill (D) circle (2pt) (C) circle (2pt);    
    \fill (E) circle (2pt) (F) circle (2pt);  
\end{tikzpicture}\\
\end{align*}
% \begin{center}
% 	\includegraphics[scale=\fscale]{figs/non-crossingChordsProdEx_lr.pdf}
% \end{center}

% \item Narayana parameter: The number of  chords $ab$ that have no other nodes (or the mark) on the arc of the circle between $a$ and $b$ (going counter clockwise from $a$ to $b$).
 
% \item Embedded Complete Binary Tree:
% \begin{itemize}
%     \item[] $\rho : \text{root}\to \text{centre of chord $rr^\prime $} $
%     \item[] $\lambda : \text{leaf}\to \text{centre of arc between black and white nodes} $
% \end{itemize}
% \begin{center}
% 	\includegraphics[scale=\fscale]{figs/non-crossingChords_CBT.pdf}
% \end{center}

\end{itemize}
\end{family}
%=============================================
%  Complete Binary tree
%=============================================

%=============================================

\tempnp

% \filbreak
%>>>>>>>>>>>>>>>>>>>>>>>>>>>>>>>>>>>>>>>>>>>>>>>>>>>>>>>>>>>>>>>>>>>>>>>>>>
\begin{family}[Complete binary trees and binary trees \tstan{
4}{5}] 
%>>>>>>>>>>>>>>>>>>>>>>>>>>>>>>>>>>>>>>>>>>>>>>>>>>>>>>>>>>>>>>>>>>>>>>>>>>
\label{fam_bt}

 \addcontentsline{toca}{subsection}{\fref{fam_bt}: Complete Binary trees and Binary trees}

Binary trees are trivially related to complete binary trees: just delete the leaves. 
The root of a complete binary tree is denoted by a triangle, the internal nodes with solid circles and the leaves with white circles.
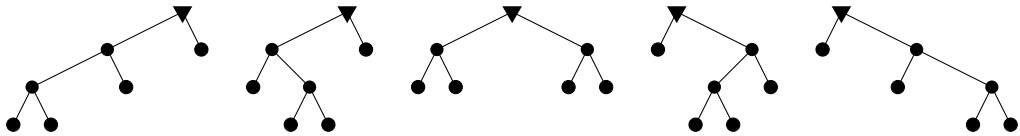
\begin{equation*}
\def\gap{0.5cm}
\begin{tikzpicture}[scale=0.5,baseline=0cm]
\coordinate (1) at (4,2);
\coordinate (2) at (2,1);
\coordinate (3) at (0,0);
\coordinate (4) at (3,0);
\coordinate (5) at (5,0);
\coordinate (6) at (6,1);
\coordinate (7) at (8,0);
%
% \draw [help lines] (0,0) grid (8,2);
\draw (1) -- (2) -- (3);
% \path (1) pic{root};
\fill (2) circle (5pt) (3) circle (5pt);
\path (3) pic[transform shape]{lleaf} pic[transform shape]{rleaf} (2)  pic[transform shape]{rleaf} (1) pic[scale=0.15]{root} (1) pic[transform shape]{rleaf};
\end{tikzpicture}
 \hspace{\gap}
\begin{tikzpicture}[scale=0.5,baseline=0cm]
\coordinate (1) at (4,2);
\coordinate (2) at (2,1);
\coordinate (3) at (0,0);
\coordinate (4) at (3,0);
\coordinate (5) at (5,0);
\coordinate (6) at (6,1);
\coordinate (7) at (8,0);
%
% \draw [help lines] (0,0) grid (8,2);
\draw (1) -- (2) -- (4);
% \path (1) pic{root};
\fill (2) circle (5pt) (4) circle (5pt);
\path (4) pic[transform shape]{lleaf} pic[transform shape]{rleaf} (2)  pic[transform shape]{lleaf} (1) pic[scale=0.15]{root} (1) pic[transform shape]{rleaf};
\end{tikzpicture}
 \hspace{\gap}
\begin{tikzpicture}[scale=0.5,baseline=0cm]
\coordinate (1) at (4,2);
\coordinate (2) at (2,1);
\coordinate (3) at (0,0);
\coordinate (4) at (3,0);
\coordinate (5) at (5,0);
\coordinate (6) at (6,1);
\coordinate (7) at (8,0);
%
% \draw [help lines] (0,0) grid (8,2);
\draw (1) -- (2) (1) -- (6);
% \path (1) pic{root};
\fill (2) circle (5pt) (6) circle (5pt);
\path (2) pic[transform shape]{lleaf} pic[transform shape]{rleaf} 
(6)  pic[transform shape]{lleaf} pic[transform shape]{rleaf} 
(1) pic[scale=0.15]{root};
\end{tikzpicture}
 \hspace{\gap}
\begin{tikzpicture}[scale=0.5,baseline=0cm]
\coordinate (1) at (4,2);
\coordinate (2) at (2,1);
\coordinate (3) at (0,0);
\coordinate (4) at (3,0);
\coordinate (5) at (5,0);
\coordinate (6) at (6,1);
\coordinate (7) at (8,0);
%
% \draw [help lines] (0,0) grid (8,2);
\draw (1) -- (6) -- (5);
% \path (1) pic{root};
\fill (6) circle (5pt) (5) circle (5pt);
\path (5) pic[transform shape]{lleaf} pic[transform shape]{rleaf} 
(6)  pic[transform shape]{rleaf} 
(1) pic[scale=0.15]{root} 
(1) pic[transform shape]{lleaf};
\end{tikzpicture}
 \hspace{\gap}
\begin{tikzpicture}[scale=0.5,baseline=0cm]
\coordinate (1) at (4,2);
\coordinate (2) at (2,1);
\coordinate (3) at (0,0);
\coordinate (4) at (3,0);
\coordinate (5) at (5,0);
\coordinate (6) at (6,1);
\coordinate (7) at (8,0);
%
% \draw [help lines] (0,0) grid (8,2);
\draw (1) -- (6) -- (7);
% \path (1) pic{root};
\fill (6) circle (5pt) (7) circle (5pt);
\path (7) pic[transform shape]{lleaf} pic[transform shape]{rleaf} 
(6)  pic[transform shape]{lleaf} 
(1) pic[scale=0.15]{root} 
(1) pic[transform shape]{lleaf};
\end{tikzpicture}
\end{equation*}
%
% \begin{center}
%  \includegraphics[scale=\fscale]{figs/CBTExamples_lr.pdf}
% %\includegraphics{figs/ex-cbt}	
% \end{center}

\begin{itemize}
\item Generator: $\eps = \circ$ (a leaf). %\raisebox{-0.25ex}{\includegraphics[scale=\fscale]{figs/cbt-gen.pdf}} 

\item Product\cite{Loday:2012aa}:
% \footnote{This is the original magma product -- see Appendix C of  \cite{Loday:2012aa} }%

\begin{equation}
    \begin{tikzpicture}[style_size3,baseline=0.75cm]
        % \draw [help lines, gray!50] (0,0) grid(3,3);
        \fill [color_left,draw=black] (2,3) -- (3,0) -- (0,1) -- cycle;
        \fill   (2,3) pic[scale=0.2]{root} ;
    \end{tikzpicture}
    \qquad\star\quad
    \begin{tikzpicture}[style_size3,baseline=0.75cm]
        % \draw [help lines, gray!50] (0,0) grid(3,3);
        \fill [color_right,draw=black] (1,3) -- (0,0) -- (3,1) -- cycle;
        \fill  (1,3) pic[scale=0.2]{root} ;
    \end{tikzpicture}
    \qquad=\quad
    \begin{tikzpicture}[style_size3  ,baseline=0.75cm]
        %   \draw [help lines, gray!50] (0,0) grid(7,4);
        %
        \draw[orange, line width=1pt] (3.5,4) -- (2,3) (3.5,4) -- (5,3);
        \path (3.5,4) pic[scale=0.2,fill=orange]{root} ;
        \node at (3.5,4.5) {$a$}; 
        \fill [color_left,draw=black] (2,3) -- (3,0) -- (0,1) -- cycle;
        \fill [draw=black] (2,3) pic{circle_black} ;
        \node at (1.5,3) {$b$};
        \fill [color_right,draw=black] (5,3) -- (4,0) -- (7,1) -- cycle;
        \fill [draw=black] (5,3) pic{circle_black}  ;
        \node at (5.5,3) {$c$};
    \end{tikzpicture}
\end{equation}

% \begin{center}
%     \includegraphics[scale=\fscale]{figs/cbt-prod_lr.pdf}
% \end{center}
% \raisebox{-4ex}{}
\item $\text{\Valns}=\text{Number of leaves.}$  
 
\item Examples:
\begin{align*}
\circ\quad \star \quad \circ  
& \quad=\quad      
\begin{tikzpicture}[scale=0.5]
\coordinate (1) at (4,2); \coordinate (2) at (2,1); \coordinate (3) at (0,0); 
\coordinate (4) at (3,0); \coordinate (5) at (5,0); \coordinate (6) at (6,1);
\coordinate (7) at (8,0);
\path   (1) pic[scale=0.15,fill=orange]{root} 
(1) pic[transform shape,draw=orange,fill=white ]{rleaf} 
pic[transform shape,draw=orange,fill=white ]{lleaf};
\end{tikzpicture} \\ 
&\\
\circ\quad  \star \quad 
\begin{tikzpicture}[scale=0.5]
\coordinate (1) at (4,2); \coordinate (2) at (2,1); \coordinate (3) at (0,0); 
\coordinate (4) at (3,0); \coordinate (5) at (5,0); \coordinate (6) at (6,1);
\coordinate (7) at (8,0);
\path   (1) pic[scale=0.15]{root} (1) pic[transform shape,fill=white]{rleaf} pic[transform shape,fill=white]{lleaf};
\end{tikzpicture}
& \quad=\quad    
\begin{tikzpicture}[scale=0.5]
\coordinate (1) at (4,2); \coordinate (2) at (2,1); \coordinate (3) at (0,0); 
\coordinate (4) at (3,0); \coordinate (5) at (5,0); \coordinate (6) at (6,1);
\coordinate (7) at (8,0);
\draw[orange] (1) -- (6)  ;
% \path (1) pic{root};
\fill (6) circle (5pt)  ;
\path (6) pic[transform shape,fill=white ]{lleaf} pic[transform shape,draw=orange,fill=white ]{rleaf} 
(1) pic[transform shape,draw=orange,fill=white  ]{lleaf} pic[scale=0.15,fill=orange]{root}  ;
\end{tikzpicture}  \\ 
&\\
\begin{tikzpicture}[scale=0.5]
\coordinate (1) at (4,2); \coordinate (2) at (2,1); \coordinate (3) at (0,0); 
\coordinate (4) at (3,0); \coordinate (5) at (5,0); \coordinate (6) at (6,1);
\coordinate (7) at (8,0);
\path   (1) pic[scale=0.15]{root} (1) pic[transform shape,fill=white]{rleaf} pic[transform shape,fill=white]{lleaf};
\end{tikzpicture}
\quad \star \quad \circ  
& \quad=\quad      
\begin{tikzpicture}[scale=0.5]
\coordinate (1) at (4,2); \coordinate (2) at (2,1); \coordinate (3) at (0,0); 
\coordinate (4) at (3,0); \coordinate (5) at (5,0); \coordinate (6) at (6,1);
\coordinate (7) at (8,0);
\draw[orange] (1) -- (2)  ;
% \path (1) pic{root};
\fill (2) circle (5pt)  ;
\path (2) pic[transform shape,draw=orange,fill=white ]{lleaf} pic[transform shape,draw=orange,fill=white ]{rleaf} 
(1) pic[transform shape,draw=orange ]{rleaf} pic[scale=0.15,fill=orange]{root}  ;
\end{tikzpicture} \\ 
&\\
\begin{tikzpicture}[scale=0.5]
\coordinate (1) at (4,2); \coordinate (2) at (2,1); \coordinate (3) at (0,0); 
\coordinate (4) at (3,0); \coordinate (5) at (5,0); \coordinate (6) at (6,1);
\coordinate (7) at (8,0);
\path   (1) pic[scale=0.15]{root} (1) pic[transform shape,fill=white]{rleaf} pic[transform shape,fill=white]{lleaf};
\end{tikzpicture}
\quad \star \quad 
\begin{tikzpicture}[scale=0.5]
\coordinate (1) at (4,2); \coordinate (2) at (2,1); \coordinate (3) at (0,0); 
\coordinate (4) at (3,0); \coordinate (5) at (5,0); \coordinate (6) at (6,1);
\coordinate (7) at (8,0);
\path   (1) pic[scale=0.15]{root} (1) pic[transform shape,fill=white]{rleaf} pic[transform shape,fill=white]{lleaf};
\end{tikzpicture}
& \quad=\quad      
\begin{tikzpicture}[scale=0.5]
\coordinate (1) at (4,2); \coordinate (2) at (2,1); \coordinate (3) at (0,0);
\coordinate (4) at (3,0); \coordinate (5) at (5,0); \coordinate (6) at (6,1);
\coordinate (7) at (8,0);
%
% \draw [help lines] (0,0) grid (8,2);
\draw[orange] (1) -- (2) (1) -- (6);
% \path (1) pic{root};
\fill (2) circle (5pt) (6) circle (5pt);
\path (2) pic[transform shape,draw=orange,fill=white ]{lleaf} pic[transform shape,draw=orange,fill=white ]{rleaf} 
(6)  pic[transform shape,fill=white]{lleaf} pic[transform shape,fill=white]{rleaf} 
(1) pic[scale=0.15,fill=orange]{root};
\end{tikzpicture}
\end{align*} 
% \begin{center}
% 	\includegraphics[scale=\fscale]{figs/CBT-Prod-Examples_lr.pdf}
% \end{center}

% \item Narayana parameter: Number of  right leaves. 
\end{itemize}

%=============================================
%  Binary trees
%=============================================

\end{family}

\tempnp
% \filbreak
%>>>>>>>>>>>>>>>>>>>>>>>>>>>>>>>>>>>>>>>>>>>>>>>>>>>>>>>>>>>>>>>>>>>>>>>>>>
\begin{family}[Planar trees \tstan{6}]
%>>>>>>>>>>>>>>>>>>>>>>>>>>>>>>>>>>>>>>>>>>>>>>>>>>>>>>>>>>>>>>>>>>>>>>>>>>
\label{fam_ps}

\addcontentsline{toca}{subsection}{\fref{fam_ps}: Planar Trees }

A graph with a marked vertex (the root) and no cycles (ie.\ a unique path from the root to  every leaf (degree one vertex)). The planar tree is conventionally drawn with the root at the top.
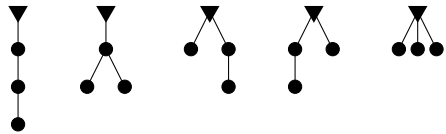
\begin{center}
\def\gap{0.5cm}
\begin{tikzpicture}[scale=0.5,baseline=0cm]
    \path pic[scale=0.15]{root} pic[transform shape,fill=black]{cleaf} ++(0,-1) pic[transform shape]{cleaf} ++(0,-1) pic[transform shape,fill=black]{cleaf};
\end{tikzpicture}
\hspace{\gap}
\begin{tikzpicture}[scale=0.5,baseline=0cm]
    \path pic[scale=0.15]{root} pic[transform shape,fill=black]{cleaf} ++(0,-1) pic[transform shape,fill=black]{lleaf}   pic[transform shape,fill=black]{rleaf};
\end{tikzpicture}
\hspace{\gap}
\begin{tikzpicture}[scale=0.5,baseline=0cm]
    \path pic[scale=0.15]{root} pic[transform shape,fill=black]{lleaf}   pic[transform shape,fill=black]{rleaf}  ++(0.5,-1) pic[transform shape,fill=black]{cleaf};
\end{tikzpicture}
\hspace{\gap}
\begin{tikzpicture}[scale=0.5,baseline=0cm]
    \path pic[scale=0.15]{root} pic[transform shape,fill=black]{rleaf}   pic[transform shape,fill=black]{lleaf}  ++(-0.5,-1) pic[transform shape,fill=black]{cleaf};
\end{tikzpicture}
\hspace{\gap}
\begin{tikzpicture}[scale=0.5,baseline=0cm]
    \path pic[scale=0.15]{root} pic[transform shape,fill=black]{rleaf}   pic[transform shape,fill=black]{lleaf}   pic[transform shape,fill=black]{cleaf};
\end{tikzpicture}
\end{center}
% \begin{center}
% \includegraphics[scale=\fscale]{figs/planarTrees-Examples_lr.pdf}
% %\includegraphics{figs/ex-pt}	
% \end{center}

\begin{itemize}
\item Generator: $\eps = \bullet$ (ie.\ a non-root node).

\item Product:  
\begin{equation*}
\begin{tikzpicture}[style_size3,baseline=0.75cm]
    % \draw [help lines, gray!50] (0,0) grid(3,3);
    \fill [color_left,draw=black] (2,3) -- (3,0) -- (0,1) -- cycle;
    \fill   (2,3) pic[scale=0.2]{root} ;
\end{tikzpicture}
\qquad\star\quad
\begin{tikzpicture}[style_size3,baseline=0.75cm]
    % \draw [help lines, gray!50] (0,0) grid(3,3);
    \fill [color_right,draw=black] (1,3) -- (0,0) -- (3,1) -- cycle;
    \fill  (1,3) pic[scale=0.2]{root} ;
\end{tikzpicture}
\qquad=\quad
\begin{tikzpicture}[style_size3  ,baseline=0.75cm]
    \draw[orange, line width=1pt]  (3.5,4) -- (6,3);
    \fill [color_left,draw=black] (3.5,4) -- (4.5,1) -- (1.5,2) -- cycle;
    \fill [color_right,draw=black] (6,3) -- (5,0) -- (8,1) -- cycle;
    \fill [draw=black] (6,3) pic{circle_black}  ;
     \path (3.5,4) pic[scale=0.2 ]{root} ;
\end{tikzpicture}
\end{equation*}
% \begin{center}
%     \includegraphics[scale=\fscale]{figs/planarTree-prod_lr.pdf}
% \end{center}

\item $\text{\Valns}=\text{Number of non-root nodes.}$
\item Examples:
\begin{align*}
    \begin{tikzpicture}[scale=0.5,baseline=0cm]
    \fill circle (5pt);
    \end{tikzpicture}
    \quad \star \quad
    \begin{tikzpicture}[scale=0.5,baseline=0cm]
    \fill circle (5pt);
    \end{tikzpicture}
    & \quad= \quad 
    \begin{tikzpicture}[scale=0.5,baseline=0cm]
    \path pic[transform shape,fill=black,draw=orange]{cleaf}  pic[scale=0.15]{root};
    \end{tikzpicture} \\
    &\\
    \begin{tikzpicture}[scale=0.5,baseline=0cm]
    \fill circle (5pt);
    \end{tikzpicture}
    \quad \star \quad
    \begin{tikzpicture}[scale=0.5,baseline=0cm]
    \path pic[scale=0.15]{root} pic[transform shape,fill=black ]{cleaf};
    \end{tikzpicture} 
    & \quad= \quad 
    \begin{tikzpicture}[scale=0.5,baseline=0cm]
    \path pic[scale=0.15]{root} pic[transform shape,fill=black,draw=orange ]{cleaf} 
    ++(0,-1) pic[transform shape,fill=black ]{cleaf} ;
    \end{tikzpicture} \\
        &\\
    \begin{tikzpicture}[scale=0.5,baseline=0cm]
    \path pic[scale=0.15]{root} pic[transform shape,fill=black ]{cleaf};
    \end{tikzpicture} 
    \quad \star \quad
    \begin{tikzpicture}[scale=0.5,baseline=0cm]
    \fill circle (5pt);
    \end{tikzpicture}
    & \quad= \quad 
    \begin{tikzpicture}[scale=0.5,baseline=0cm]
    \path pic[transform shape,fill=black ]{cleaf} 
     pic[transform shape,fill=black,draw=orange ]{rleaf} pic[scale=0.15]{root} ;
    \end{tikzpicture} \\
        &\\
    \begin{tikzpicture}[scale=0.5,baseline=0cm]
    \path pic[scale=0.15]{root} pic[transform shape,fill=black ]{cleaf};
    \end{tikzpicture} 
    \quad \star \quad
    \begin{tikzpicture}[scale=0.5,baseline=0cm]
    \path pic[scale=0.15]{root} pic[transform shape,fill=black ]{cleaf};
    \end{tikzpicture} 
    & \quad= \quad 
    \begin{tikzpicture}[scale=0.5,baseline=0cm]
    \path pic[scale=0.15]{root} 
     pic[transform shape,fill=black ]{rleaf} ++(0.5,-1) pic[transform shape,fill=black,draw=orange  ]{cleaf} pic[transform shape,fill=black ]{cleaf}  ;
    \end{tikzpicture} \\
\end{align*}
% \begin{center}
% 	\includegraphics[scale=\fscale]{figs/planarTrees-Prod-Examples_lr.pdf}
% \end{center}

% \item Narayana parameter:  Number of leaf nodes.

\end{itemize}
%=============================================
% Nested matchings or Link Diagrams
%=============================================
\end{family}

%=============================================

\tempnp
% \filbreak 
%>>>>>>>>>>>>>>>>>>>>>>>>>>>>>>>>>>>>>>>>>>>>>>>>>>>>>>>>>>>>>>>>>>>>>>>>>>
\begin{family}[Nested matchings or link diagrams \stan{61}]
%>>>>>>>>>>>>>>>>>>>>>>>>>>>>>>>>>>>>>>>>>>>>>>>>>>>>>>>>>>>>>>>>>>>>>>>>>>
\label{fam_ld}

 \addcontentsline{toca}{subsection}{\fref{fam_ld}: Nested matchings or Link Diagrams}

A line of $2n$ nodes with the  last not draw (ie.\ is drawn as a point).   connected pairwise by $n-1$ links such that the arcs representing the links (always drawn above the line) do not intersect.  Note, conventionally the first and last nodes are   drawn, but omitting the last simplifies the product definition.
\begin{center}
\def\gap{1cm}
\begin{tikzpicture}[style_size3] 
\def\S{++(1,0)} \def\R{4pt} \def\C{circle  }
    \draw (0,0) -- (7,0);
        \foreach \x in {0,...,6} {
        \fill (\x,0) pic{circle_black}  ;
    };
    \draw  (1,0) arc [start angle=180, end angle = 0, radius=2.5];
    \draw  (2,0) arc [start angle=180, end angle = 0, radius=1.5];
    \draw  (3,0) arc [start angle=180, end angle = 0, radius=0.5];
\end{tikzpicture}
\hspace{\gap} %------------------
\begin{tikzpicture}[style_size3] 
\def\S{++(1,0)} \def\R{4pt} \def\C{circle  }
    \draw (0,0) -- (7,0);
        \foreach \x in {0,...,6} {
        \fill (\x,0) pic{circle_black}  ;
    };
    \draw  (1,0) arc [start angle=180, end angle = 0, radius=2.5];
    \draw  (2,0) arc [start angle=180, end angle = 0, radius=0.5];
    \draw  (4,0) arc [start angle=180, end angle = 0, radius=0.5];
\end{tikzpicture}
\end{center}
\begin{center}
\def\gap{1cm}

\begin{tikzpicture}[style_size3] 
\def\S{++(1,0)} \def\R{4pt} \def\C{circle  }
    \draw (0,0) -- (7,0);
        \foreach \x in {0,...,6} {
        \fill (\x,0) pic{circle_black}  ;
    };
    \draw  (1,0) arc [start angle=180, end angle = 0, radius=0.5];
    \draw  (3,0) arc [start angle=180, end angle = 0, radius=1.5];
    \draw  (4,0) arc [start angle=180, end angle = 0, radius=0.5];
\end{tikzpicture}
\hspace{\gap} %------------------
\begin{tikzpicture}[style_size3] 
\def\S{++(1,0)} \def\R{4pt} \def\C{circle  }
    \draw (0,0) -- (7,0);
        \foreach \x in {0,...,6} {
        \fill (\x,0) pic{circle_black}  ;
    };
    \draw  (1,0) arc [start angle=180, end angle = 0, radius=1.5];
    \draw  (2,0) arc [start angle=180, end angle = 0, radius=0.5];
    \draw  (5,0) arc [start angle=180, end angle = 0, radius=0.5];
\end{tikzpicture}
\hspace{\gap} %------------------
\begin{tikzpicture}[style_size3] 
\def\S{++(1,0)} \def\R{4pt} \def\C{circle  }
    \draw (0,0) -- (7,0);
        \foreach \x in {0,...,6} {
        \fill (\x,0) pic{circle_black}  ;
    };
    \draw  (1,0) arc [start angle=180, end angle = 0, radius=0.5];
    \draw  (3,0) arc [start angle=180, end angle = 0, radius=0.5];
    \draw  (5,0) arc [start angle=180, end angle = 0, radius=0.5];
\end{tikzpicture}
% \hspace{\gap} %------------------
\end{center}
% \begin{center}
%     \includegraphics[scale=\fscale]{figs/nestedMatchingsExamples_lrr.pdf}
% \end{center}

\newcommand{\nmgen}{
\begin{tikzpicture}[baseline={([yshift=-.5ex]current bounding box.center)}]
  \draw (0,0) -- (0,1) -- (4,1) -- (4,0) -- (0,0);
\draw[fill] (0,0) circle [radius=0.05];
\end{tikzpicture}
}

\begin{itemize}

\item Generator: $\eps =\gennm$

% $\eps=\raisebox{0ex}{\includegraphics[scale=\fscale]{figs/nestedMatchingsGen_lr.pdf}}$

\item Product:

\begin{equation}   
\def\bsl{0pt}
\begin{tikzpicture}[style_size4,baseline=\bsl] 
    \fill [color_left,draw=black] (1,0) arc [radius=1,start angle = 180, end angle =0];
    \draw (0,0) -- (4,0);
    \fill (0,0) pic{circle_black}  ;
    \node[below,anchor=mid  ] at (0,-0.7) {$a$};
    % \fill[color=white] (4,0) pic{circle_black} ;
    \node[below,anchor=mid ] at (4,-0.7) {$b$};
    \end{tikzpicture}
\star 
\begin{tikzpicture}[style_size4,baseline=\bsl] 
    \fill [color_right,draw=black] (6,0) arc [radius=1,start angle = 180, end angle =0];
    \draw (5,0) -- (9,0);
    \fill (5,0) pic{circle_black}  ;
    \node[below,anchor=mid  ] at (5,-0.7) {$c$};
    % \fill[color=white] (9,0) pic{circle_white} ;
    \node[below,anchor=mid  ] at (9,-0.7) {$d$};
\end{tikzpicture}
= 
\begin{tikzpicture}[style_size4,baseline=\bsl] 
    \node[below ,anchor=mid ] at (10,-0.7) {$a$};
    \fill [color_left,draw=black]  (11,0) arc [radius=1,start angle = 180, end angle =0];
    \fill [color_right,draw=black]  
        (15,0) arc [radius=1,start angle = 180, end angle =0];
    % \fill (5,0) circle [radius=4pt];
    \draw (10,0) -- (18,0);
    \fill (10,0) pic{circle_black}  ;   
    \draw[color=productColor,line width=1pt] (14,0) arc [start angle=180, end angle = 0, radius=2];
    \fill (14,0) pic{circle_black}   ;
    \node[below,anchor=mid  ] at (14,-0.7) {$bc$};
    \fill [color=productColor] (18,0) pic{circle_product} ;
    \draw [color=productColor,line width=1pt] (18,0) -- (19,0);
    % \fill[color=white] (19,0) pic{circle_black}  ;
    \node[below,anchor=mid  ] at (18,-0.7) {$d$};
    \node[below,anchor=mid  ] at (19,-0.7) {$e$};
\end{tikzpicture}
\end{equation}
% \begin{center}
%     \includegraphics[scale=\fscale]{figs/nestedMatchings_lr.pdf}
% \end{center}
Nodes $b$ and $c$ become a single node.
\item $\text{\Valns}=\text{(Number of links)}+1=\text{Half the number of nodes.}$

\item Examples:

\begin{align*}
\gennm\star \gennm &= 
\begin{tikzpicture}[style_size3] 
\def\S{++(1,0)} \def\R{4pt} \def\C{circle  }
\draw[orange]  (1,0) arc [start angle=180, end angle = 0, radius=0.5];  
\draw (0,0) -- (3,0);
\foreach \x in {0,...,1} {
\fill (\x,0) pic{circle_black}  ;
};
\fill[orange, draw=orange] (2,0) circle (\R)   -- (3,0);
\end{tikzpicture}\\
& \\
%==========================
\begin{tikzpicture}[style_size3] 
\def\S{++(1,0)} \def\R{4pt} \def\C{circle  }
\draw  (1,0) arc [start angle=180, end angle = 0, radius=0.5];  
\draw (0,0) -- (3,0);
\foreach \x in {0,...,2} {
\fill (\x,0) pic{circle_black}  ;
};
\end{tikzpicture}\star \gennm & = 
\begin{tikzpicture}[style_size3] 
\def\S{++(1,0)} \def\R{4pt} \def\C{circle  }
\draw[orange]  (3,0) arc [start angle=180, end angle = 0, radius=0.5];  
\draw  (1,0) arc [start angle=180, end angle = 0, radius=0.5];
\draw (0,0) -- (4,0);
\foreach \x in {0,...,4} {
\fill (\x,0) pic{circle_black}  ;
};
\fill[orange,draw=orange] (4,0) circle (\R)   -- (5,0);
\end{tikzpicture}\\
& \\
%==========================
\gennm \star \begin{tikzpicture}[style_size3] 
\def\S{++(1,0)} \def\R{4pt} \def\C{circle  }
\draw  (1,0) arc [start angle=180, end angle = 0, radius=0.5];  
\draw (0,0) -- (3,0);
\foreach \x in {0,...,2} {
\fill (\x,0) pic{circle_black}  ;
};
\end{tikzpicture} & = 
\begin{tikzpicture}[style_size3] 
\def\S{++(1,0)} \def\R{4pt} \def\C{circle  }
\draw  (2,0) arc [start angle=180, end angle = 0, radius=0.5];
\draw[orange]  (1,0) arc [start angle=180, end angle = 0, radius=1.5];  
\draw (0,0) -- (4,0);
\foreach \x in {0,...,4} {
\fill (\x,0) pic{circle_black}  ;
};
\fill[orange,draw=orange] (4,0) circle (\R)   -- (5,0);
\end{tikzpicture}\\
& \\
%==========================
\begin{tikzpicture}[style_size3] 
\def\S{++(1,0)} \def\R{4pt} \def\C{circle  }
\draw  (1,0) arc [start angle=180, end angle = 0, radius=0.5];  
\draw (0,0) -- (3,0);
\foreach \x in {0,...,2} {
\fill (\x,0) pic{circle_black}  ;
};
\end{tikzpicture} \star 
\begin{tikzpicture}[style_size3] 
\def\S{++(1,0)} \def\R{4pt} \def\C{circle  }
\draw  (1,0) arc [start angle=180, end angle = 0, radius=0.5];  
\draw (0,0) -- (3,0);
\foreach \x in {0,...,2} {
\fill (\x,0) pic{circle_black}  ;
};
\end{tikzpicture} & = 
\begin{tikzpicture}[style_size3] 
\def\S{++(1,0)} \def\R{4pt} \def\C{circle  }
\draw  (1,0) arc [start angle=180, end angle = 0, radius=0.5];
\draw  (4,0) arc [start angle=180, end angle = 0, radius=0.5];
\draw[orange]  (3,0) arc [start angle=180, end angle = 0, radius=1.5];  
\draw (0,0) -- (6,0);
\foreach \x in {0,...,6} {
\fill (\x,0) pic{circle_black}  ;
};
\fill[orange,draw=orange] (6,0) circle (\R)   -- (7,0);
\end{tikzpicture}
\end{align*}
% \begin{center}
% 	 \includegraphics[scale=\fscale]{figs/nestedMatchingsProdExamples_lr.pdf}
% \end{center}

% \item Narayana parameter: Number of links matching adjacent nodes.
% \item Embedded Complete Binary Tree:
%     \begin{itemize}
%         \item[] $\rho : \text{root}\to \text{centre of rightmost matching} $
%         \item[] $\lambda : \text{leaf}\to \text{centre edge between black and white nodes} $
%     \end{itemize}
%     \begin{center}
%     	\includegraphics[width=7cm]{figs/nestedMatchings_CBT.pdf}
%     \end{center}
\end{itemize}

\end{family}
\tempnp

%>>>>>>>>>>>>>>>>>>>>>>>>>>>>>>>>>>>>>>>>>>>>>>>>>>>>>>>>>>>>>>>>>>>>>>>>>>
\begin{family}[Non-crossing partitions \stan{159}]
%>>>>>>>>>>>>>>>>>>>>>>>>>>>>>>>>>>>>>>>>>>>>>>>>>>>>>>>>>>>>>>>>>>>>>>>>>>
\label{fam_np}

 \addcontentsline{toca}{subsection}{\fref{fam_np}: Non-crossing partitions}

A partition of $\set{1, 2, \dots,  n}$ is non-crossing if whenever four elements, $1\le a<  b < c < d\le n$, are such that $a$, $c$ are in the same block and $b$, $d$ are in the same block, then the two blocks coincide \cite{simion:2000lr}. We will use the marked circular representation of the partitions. The five non-crossing partitions with for $n=3$ are:
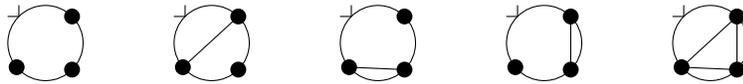
\begin{center}
\def\gap{1cm}
% 1
\begin{tikzpicture}
\coordinate (A) at (45:0.5cm) ;
\coordinate (B) at (-45:0.5cm);
\coordinate (C) at (-140:0.5cm);
\path (135:0.5cm) pic{tick};
\draw  (0,0) circle (0.5cm);
% \draw  (A) -- (B);
\fill (A) circle (3pt) (B) circle (3pt) (C) circle (3pt);
\end{tikzpicture}    
\hspace{\gap}
% 2
\begin{tikzpicture}
\coordinate (A) at (45:0.5cm) ;
\coordinate (B) at (-45:0.5cm);
\coordinate (C) at (-140:0.5cm);
\path (135:0.5cm) pic{tick};
\draw  (0,0) circle (0.5cm);
\draw  (A) -- (C);
\fill (A) circle (3pt) (B) circle (3pt) (C) circle (3pt);
\end{tikzpicture}    
\hspace{\gap}
% 3
\begin{tikzpicture}
\coordinate (A) at (45:0.5cm) ;
\coordinate (B) at (-45:0.5cm);
\coordinate (C) at (-140:0.5cm);
\path (135:0.5cm) pic{tick};
\draw  (0,0) circle (0.5cm);
\draw  (B) -- (C);
\fill (A) circle (3pt) (B) circle (3pt) (C) circle (3pt);
\end{tikzpicture}    
\hspace{\gap}
% 4
\begin{tikzpicture}
\coordinate (A) at (45:0.5cm) ;
\coordinate (B) at (-45:0.5cm);
\coordinate (C) at (-140:0.5cm);
\path (135:0.5cm) pic{tick};
\draw  (0,0) circle (0.5cm);
\draw  (A) -- (B);
\fill (A) circle (3pt) (B) circle (3pt) (C) circle (3pt);
\end{tikzpicture}    
\hspace{\gap}
% 5
\begin{tikzpicture}
\coordinate (A) at (45:0.5cm) ;
\coordinate (B) at (-45:0.5cm);
\coordinate (C) at (-140:0.5cm);
\path (135:0.5cm) pic{tick};
\draw  (0,0) circle (0.5cm);
\draw  (A) -- (B) -- (C) -- cycle;
\fill (A) circle (3pt) (B) circle (3pt) (C) circle (3pt);
\end{tikzpicture}    
\hspace{\gap}

\end{center}
% \begin{center}
%     \includegraphics[scale=\fscale]{figs/non-crossingPartitions_Examples_lr.pdf}
% \end{center}
To convert to a partition of $\set{1, 2, \dots,  n}$ assume the nodes are enumerated \emph{anti-clockwise} from the mark. Nodes joined by a chord are in the same block of the partition.

\begin{itemize}
\item Generator: 
$\eps = \genncp $ (a circle with no nodes).
%  \raisebox{-0.5\height}{\includegraphics[scale=\fscale]{figs/non-crossingPartitions-Prod_gen_lr.pdf}}$ (arc of circle).

In a general non-crossing partition the $i^\text{th}$ generator is the arc between node $i-1$ and $i$ (nodes labelled anti-clockwise). The arc between the last node and the mark is not a generator.

\item  Product:

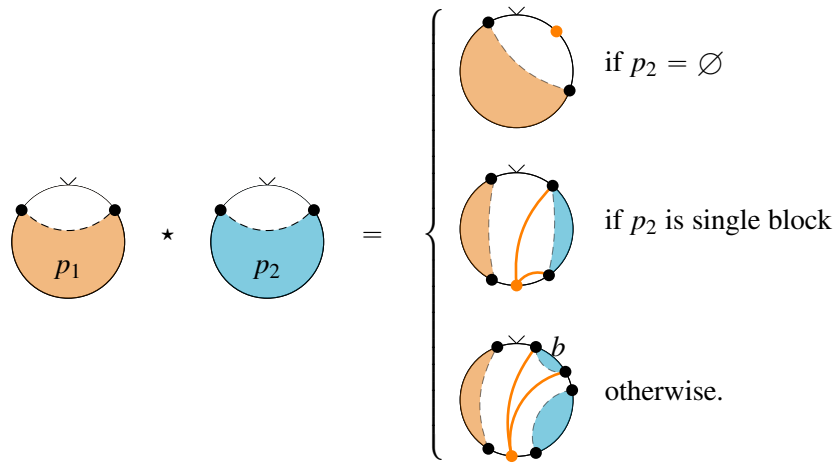
\begin{equation*}
\begin{tikzpicture}[baseline=0cm,scale=0.75] 
\coordinate (A) at (34.23:1cm);
\coordinate (B) at  (145.77:1cm);

\coordinate (C) at  (-50:1cm);
\coordinate (D)  at (-130:1cm);

\draw (0,0) circle (1cm) (90:1cm) pic{tick2};

\begin{scope}
\draw[clip] (0,0) circle (1cm); 
\fill[color_left,draw=black ] (0,0) circle (1cm);
\fill[white, draw=black,densely dashed ] (0,1.25) circle (1.05cm);
\end{scope}

\fill (A) circle (3pt) (B) circle (3pt);

\path (0,-0.5)  node{$p_1$}; 

\end{tikzpicture}
\quad\star\quad
\begin{tikzpicture}[baseline=0cm,scale=0.75]  
\coordinate (A) at (34.23:1cm);
\coordinate (B) at  (145.77:1cm);

\coordinate (C) at  (-50:1cm);
\coordinate (D)  at (-130:1cm);

\draw (0,0) circle (1cm) (90:1cm) pic{tick2};

\begin{scope}
\draw[clip] (0,0) circle (1cm); 
\fill[color_right,draw=black ] (0,0) circle (1cm);
\fill[white, draw=black,densely dashed ] (0,1.25) circle (1.05cm);
\end{scope}

\fill (A) circle (3pt) (B) circle (3pt);
% \path (A) node[anchor=south west] {$d$};
% \path (B) node[anchor=south east] {$c$};

\path (0,-0.5)  node{$p_2$}; 
\end{tikzpicture}
\quad =\quad 
\left\{ 
\begin{array}{ll} %==============
\begin{tikzpicture}[baseline=0cm,scale=0.75]  
\coordinate (A) at (120:1cm);
\coordinate (B) at  (-20:1cm);
\draw (0,0) circle (1cm) (90:1cm) pic{tick2};
\begin{scope}
\draw[clip] (0,0) circle (1cm); 
\fill[color_left,draw=black,densely dashed] (A) to [bend right =25] (B) -- (5,-22) -- (-4,-1) -- cycle;
\end{scope}
\fill (A) circle (3pt) (B) circle (3pt);
\fill[orange] (45:1cm)   circle (3pt);
\end{tikzpicture} 
& \text{if $p_2=\emptyset$}\\
& \\
\begin{tikzpicture}[baseline=0cm,scale=0.75]  
\coordinate (A) at (50.23:1cm);
\coordinate (B) at  (116.77:1cm);

\coordinate (C) at  (-55:1cm);
\coordinate (D)  at (-116:1cm);

\coordinate (E) at  (70:1cm);
\coordinate (F)  at (-90:1cm);

\draw (0,0) circle (1cm) (90:1cm) pic{tick2};
\begin{scope}
\draw[clip] (0,0) circle (1cm); 
\fill[color_left,draw=black,densely dashed] (0,-2) to [bend left =25] (0,2) -- (-3,2) -- (-3,-2) -- cycle;
\fill[color_right,draw=black,densely dashed] (0.1,-2) to [bend right =25] (0.25,2) -- (3,2) -- (3,-2) -- cycle;
\end{scope}
\fill[orange] (F) circle (3pt)  ;
\draw[orange,line width=1pt] (A) to [bend right =25] (F);
\draw[orange,line width=1pt] (F) to [bend left =45] (C);

\fill (A) circle (3pt) (B) circle (3pt);
\fill (C) circle (3pt) (D) circle (3pt);

\end{tikzpicture} 
&   \text{if $p_2$ is  single block}\\
& \\
\begin{tikzpicture}[baseline=0cm,scale=0.75]  
\coordinate (A) at (70:1cm);
\coordinate (An) at (52:1.2cm);
\coordinate (B) at  (30:1cm);

\coordinate (C) at  ( 10:1cm);
\coordinate (D)  at (-70:1cm);

\coordinate (E) at  (-95:1cm);

\coordinate (F)  at (-120:1cm);
\coordinate (G)  at (110:1cm);

\draw (0,0) circle (1cm) (90:1cm) pic{tick2};
\begin{scope}
\draw[clip] (0,0) circle (1cm); 
\fill[color_left,draw=black,densely dashed] (F) to [bend left =25] (G) -- (-3,0)  -- cycle;
\fill[color_right,draw=black,densely dashed] (A) to [bend right =55] (B) -- (5,5)   -- cycle;
\fill[color_right,draw=black,densely dashed] (C) to [bend right =55] (D)  -- (5,-5) -- cycle;
\end{scope}
\fill[orange] (E) circle (3pt)  ;
\draw[orange,line width=1pt] (E) to [bend left =25] (A);
\draw[orange,line width=1pt] (E) to [bend left =45] (B);

\fill (A) circle (3pt) (B) circle (3pt);
\fill (C) circle (3pt) (D) circle (3pt);
\fill (F) circle (3pt) (G) circle (3pt);
\node  at (An)  {$b$};
\end{tikzpicture} 
&  \text{otherwise.}\\
\end{array}%==============
\right.
\end{equation*}
% \begin{equation*}
%     \raisebox{-0.5\height}{\includegraphics[origin=lt,scale=\fscale]{figs/non-crossingPartitions-Prod_LHS_lr.pdf}}\quad =\quad 
%     \begin{cases}
%         \raisebox{-0.5\height}{\includegraphics[origin=lc,scale=\fscale]{figs/non-crossingPartitions-Prod_RHS1_lr.pdf}} &
%         \text{if $p_2=\emptyset$,}\\
%         \raisebox{-0.5\height}{\includegraphics[origin=lc,scale=\fscale]{figs/non-crossingPartitions-Prod_RHS2_lr.pdf}} & 
%         \text{if $p_2$ is a single block,}\\
%         \raisebox{-0.5\height}{\includegraphics[origin=lc,scale=\fscale]{figs/non-crossingPartitions-Prod_RHS3_lr.pdf}}&
%         \text{otherwise.}
%     \end{cases}
% \end{equation*}
In the last case $b$ is the block in $p_2$ containing the last (reading anti-clockwise) node.
In the second and last  cases if the last block is a single node then there is only a single product (orange) line.

\item $\text{\Valns}=\text{(Number of nodes)}+1$

\item Examples:

\begin{align*}
\begin{tikzpicture} \path pic{genncp}; \end{tikzpicture}
\quad \star \quad
\begin{tikzpicture} \path[scale=2] pic{genncp}; \end{tikzpicture}
&    \quad = \quad 
\begin{tikzpicture}
    \coordinate (A) at (-45:0.5cm) ;
    \path (135:0.5cm) pic{tick};
    \draw  (0,0) circle (0.5cm);
    \fill[orange] (A) circle (3pt);
\end{tikzpicture}\\
& \\
\begin{tikzpicture}
    \coordinate (A) at (-45:0.5cm) ;
    \path (135:0.5cm) pic{tick};
    \draw  (0,0) circle (0.5cm);
    \fill  (A) circle (3pt);
\end{tikzpicture}
\quad \star \quad
\begin{tikzpicture} \path pic{genncp}; \end{tikzpicture}
&    \quad = \quad 
\begin{tikzpicture}
    \coordinate (A) at (60:0.5cm) ;
    \coordinate (B) at (240:0.5cm);
    \path (135:0.5cm) pic{tick};
    \draw  (0,0) circle (0.5cm);
    \fill[orange] (A) circle (3pt);   
    \fill  (B) circle (3pt);   
\end{tikzpicture}\\
& \\
\begin{tikzpicture} \path pic{genncp}; \end{tikzpicture}
\quad \star \quad
\begin{tikzpicture}
    \coordinate (A) at (-45:0.5cm) ;
    \path (135:0.5cm) pic{tick};
    \draw  (0,0) circle (0.5cm);
    \fill  (A) circle (3pt);
\end{tikzpicture}
&    \quad = \quad 
\begin{tikzpicture}
    \coordinate (A) at (60:0.5cm) ;
    \coordinate (B) at (240:0.5cm);
    \path (135:0.5cm) pic{tick};
    \draw  (0,0) circle (0.5cm);
    \fill[orange] (B) circle (3pt);   
    \draw[orange,line width=1pt] (A)-- (B);
    \fill  (A) circle (3pt);     
\end{tikzpicture}\\
& \\
\begin{tikzpicture}
    \coordinate (A) at (-45:0.5cm) ;
    \path (135:0.5cm) pic{tick};
    \draw  (0,0) circle (0.5cm);
    \fill  (A) circle (3pt);
\end{tikzpicture}
\quad \star \quad
\begin{tikzpicture}
    \coordinate (A) at (45:0.5cm) ;
    \coordinate (B) at (-45:0.5cm);
    \coordinate (C) at (-140:0.5cm);
    \path (135:0.5cm) pic{tick};
    \draw  (0,0) circle (0.5cm);
    \draw  (A) -- (B);
    \fill (A) circle (3pt) (B) circle (3pt) (C) circle (3pt);
\end{tikzpicture}
&    \quad = \quad 
\begin{tikzpicture}
    \coordinate (A) at (83.5:0.5cm) ;
    \coordinate (B) at (0:0.5cm);
    \coordinate (C) at (-45:0.5cm) ;
    \coordinate (D) at (-120:0.5cm);
    \coordinate (E) at ( 180:0.5cm) ;
    \path (135:0.5cm) pic{tick};
    \draw  (0,0) circle (0.5cm);
    \draw[orange,line width=1pt] (A) -- (D) (B)-- (D);
    \draw (A) -- (B)  ;
    \fill (A) circle (3pt) (B) circle (3pt);
    \fill[orange] (D) circle (3pt) ;
    \fill (C) circle (3pt);    
    \fill (E) circle (3pt) ;  
\end{tikzpicture}\\
\end{align*}

% \begin{center}
% 	\includegraphics[scale=\fscale]{figs/non-crossing_examples_lr.pdf}
% \end{center}

% \item Narayana parameter: If the chords are orientated with a arrow pointing away from the node first encountered in a counter clockwise traversal from the mark, then the Narayana parameter is the number of nodes with only incoming chords (or no chords).

\end{itemize}

%=============================================
%  Dyck paths
%=============================================
\end{family}

%=============================================

\tempnp
%>>>>>>>>>>>>>>>>>>>>>>>>>>>>>>>>>>>>>>>>>>>>>>>>>>>>>>>>>>>>>>>>>>>>>>>>>>
\begin{family}[Dyck paths \stan{30}]
%>>>>>>>>>>>>>>>>>>>>>>>>>>>>>>>>>>>>>>>>>>>>>>>>>>>>>>>>>>>>>>>>>>>>>>>>>>
\label{fam_dp}

 \addcontentsline{toca}{subsection}{\fref{fam_dp}: Dyck paths}

A Dyck path is a sequence of $2n$ steps: $n$ `up' steps and $n$ `down' steps such that the path starts and ends at the same height and does not step below the height  of the leftmost vertex. There are five paths with six steps.
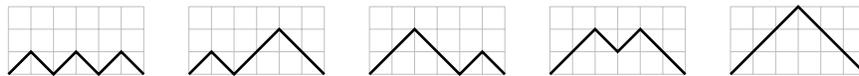
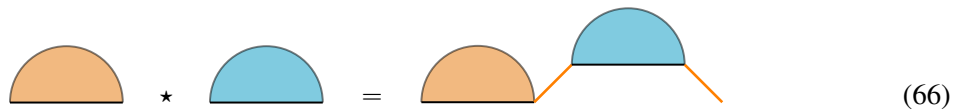
\begin{center}
    \def\up{-- ++(1,1)}
    \def\dn{-- ++(1,-1)}
\begin{tikzpicture}[scale=0.3,line width=1pt]
\draw [help lines, gray!50] (0,0) grid(6,3);
% \draw[line width=0.5pt]  (0,0)    circle (4pt);    
\draw  (0,0)     \up   \dn \up   \dn \up \dn;  
\draw [help lines, gray!50] (8,0) grid(14,3);
\draw  (8,0)     \up   \dn \up   \up \dn \dn;  
\draw [help lines, gray!50] (16,0) grid(22,3);
\draw  (16,0)     \up   \up \dn   \dn \up \dn;  
\draw [help lines, gray!50] (24,0) grid(30,3);
\draw  (24,0)     \up   \up \dn   \up \dn \dn;  
\draw [help lines, gray!50] (32,0) grid(38,3);
\draw  (32,0)     \up   \up \up   \dn \dn \dn;  
\end{tikzpicture}
\end{center}
% \begin{center}
% \includegraphics[scale=\fscale]{figs/DyckPaths_Examples_lr.pdf}
% \end{center}
\begin{itemize}

\item Generator: $\eps=\circ$ (a vertex).

% $\eps = \raisebox{-1ex}{\includegraphics[scale=\figscale]{figs/magma-dp-gen-Layer_1.pdf}}$ 

\item Product:

\begin{equation}
    \begin{tikzpicture}[line width=0.75pt,scale=1] 
        \fill [color_left,draw=black] (0,0) arc [radius=0.75,start angle = 180, end angle =0] ;
        \draw (0,0) -- (1.5,0);
        \end{tikzpicture}
    \quad\star\quad 
    \begin{tikzpicture}[line width=0.75pt,scale=1] 
        \fill [color_right,draw=black] (0,0) arc [radius=0.75,start angle = 180, end angle =0] ;
        \draw (0,0) -- (1.5,0);
    \end{tikzpicture}
    \quad=\quad
    \begin{tikzpicture}[line width=0.75pt,scale=1] 
        \fill [color_left,draw=black] (0,0) arc [radius=0.75,start angle = 180, end angle =0] ;
        \draw (0,0) -- (1.5,0);
        \draw[color=orange, line width=1pt] (1.5,0) -- (2,0.5);
        \fill [color_right,draw=black] (2,0.5) arc [radius=0.75,start angle = 180, end angle =0] ;
        \draw (2,0.5) -- (3.5,0.5);
        \draw[color=orange, line width=1pt] (3.5,0.5) -- (4,0);
    \end{tikzpicture}
\end{equation}
% \begin{center}
% \includegraphics[scale=\fscale]{figs/DyckPathProd_lr.pdf}
% \end{center}

% \raisebox{-0.45ex}{\includegraphics[scale=\figscale]{figs/magma-dp-prod-Layer_1.pdf}} 

\item $\text{\Valns}= \text{(Number of up steps)}+1$.

\item Examples:

\begin{align*}
\tikz{\fill[scale=0.4, white,draw=black] circle (5pt);}\star
\tikz{\fill[scale=0.4, white,draw=black] circle (5pt);}& =
\begin{tikzpicture}[scale=0.4]
\def\up{-- ++(1,1)} \def\dn{-- ++(1,-1)}
\draw [help lines, gray!50] (0,0) grid(2,1);
\draw[orange,line width=1pt ]  (0,0)     \up   \dn  ; 
\fill[white,draw=black]  (0,0) circle (5pt) (1,1) circle (5pt) ;
\end{tikzpicture}\\
&\\
% ======================
\tikz{\fill[scale=0.4, white,draw=black] circle (5pt);}
\star\begin{tikzpicture}[scale=0.4]
\def\up{-- ++(1,1)} \def\dn{-- ++(1,-1)}
\draw [help lines, gray!50] (0,0) grid(2,1);
\draw   (0,0)  \up   \dn  ;  
\fill[white,draw=black]    circle (5pt) (1,1)   circle (5pt)  ;
\end{tikzpicture} & =
\begin{tikzpicture}[scale=0.4]
\def\up{-- ++(1,1)} \def\dn{-- ++(1,-1)}
\draw [help lines, gray!50] (0,0) grid(4,2);
\draw[orange,line width=1pt]  (0,0)     \up     ;  
\draw   (1,1)     \up   \dn  ;  
\fill[white,draw=black]    (1,1) circle (5pt)  (2,2) circle (5pt)  ;
\draw[orange,line width=1pt]  (3,1)     \dn     ;  
\fill[white,draw=black]     (0,0) circle (5pt)   ;
\end{tikzpicture}\\
&\\
% ======================
\begin{tikzpicture}[scale=0.4]
\def\up{-- ++(1,1)} \def\dn{-- ++(1,-1)}
\draw [help lines, gray!50] (0,0) grid(2,1);
\draw   (0,0)  \up   \dn  ;  
\fill[white,draw=black]    circle (5pt) (1,1)   circle (5pt)  ;
\end{tikzpicture}\star\circ & = 
\begin{tikzpicture}[scale=0.4]
\def\up{-- ++(1,1)} \def\dn{-- ++(1,-1)}
\draw [help lines, gray!50] (0,0) grid(4,1);
\draw   (0,0)  \up   \dn  ;  
\fill[white,draw=black]    circle (5pt)  (1,1)   circle (5pt) ;
\draw[orange,line width=1pt]  (2,0)     \up   \dn  ; 
\fill[white,draw=black]    (2,0) circle (5pt) (3,1) circle (5pt) ;
\end{tikzpicture}
\\
&\\
% ======================
\begin{tikzpicture}[scale=0.4]
\def\up{-- ++(1,1)} \def\dn{-- ++(1,-1)}
\draw [help lines, gray!50] (0,0) grid(2,1);
\draw   (0,0)  \up   \dn  ;  
\fill[white,draw=black]    circle (5pt) (1,1)   circle (5pt)  ;
\end{tikzpicture}\star
\begin{tikzpicture}[scale=0.4]
\def\up{-- ++(1,1)} \def\dn{-- ++(1,-1)}
\draw [help lines, gray!50] (0,0) grid(2,1);
\draw   (0,0)  \up   \dn  ;  
\fill[white,draw=black]    circle (5pt) (1,1)   circle (5pt)  ;
\end{tikzpicture}
& = 
\begin{tikzpicture}[scale=0.4]
\def\up{-- ++(1,1)} \def\dn{-- ++(1,-1)}
\draw [help lines, gray!50] (0,0) grid(6,2);
\draw   (0,0)  \up   \dn  ;  
\fill[white,draw=black]    circle (5pt) (1,1)   circle (5pt)  ;
\draw[orange,line width=1pt]   (2,0)  \up    ;  
\draw   (3,1)     \up   \dn  ; 
\draw[orange,line width=1pt]   (5,1)  \dn    ;  
\fill[white,draw=black]    (2,0) circle (5pt) (3,1) circle (5pt) (4,2)   circle (5pt) ;
\end{tikzpicture}
\end{align*}

% \begin{center}
% 	\includegraphics[scale=\fscale]{figs/DyckPath_Prod_Examples_lr.pdf}
% \end{center}

% \item Narayana parameter: Number of peaks (ie.\ an up step immediately followed by a down step)

% \item Embedded \emph{reflected} Complete Binary Tree
% \begin{center}
%  	 \includegraphics[scale=\fscale]{figs/DyckPathOppTree.pdf}
% \end{center}
% Tree for $\eps\st((\eps\st\eps)\st\eps)$\todo{remove}
\end{itemize}

%=============================================
%  Polygon triangulations
%=============================================

%=============================================
\end{family}

%=============================================

\tempnp
% \filbreak 
%>>>>>>>>>>>>>>>>>>>>>>>>>>>>>>>>>>>>>>>>>>>>>>>>>>>>>>>>>>>>>>>>>>>>>>>>>>
\begin{family}[Polygon triangulations \stan{1},  \cite{motzkin:1948lr}]
%>>>>>>>>>>>>>>>>>>>>>>>>>>>>>>>>>>>>>>>>>>>>>>>>>>>>>>>>>>>>>>>>>>>>>>>>>>
\label{fam_pt}

 \addcontentsline{toca}{subsection}{\fref{fam_pt}: Polygon triangulations}

A triangulated $n$-gon is a partition of an $n$ sided polygon into $n-2$ triangles by means of $n-3$ non-crossing chords.
There are five triangulations of a $5$-gon:
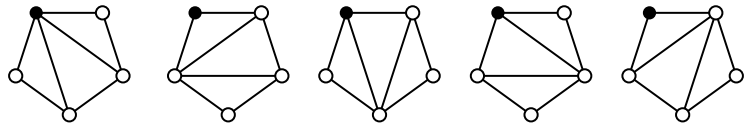
\begin{center}
    \def\pa{72}
    \def\lv{54}
    \def\rd{4pt}
    \def\sc{0.4}
    \def\ro{0.9} %two
    \def\rt{1.2} % three
    \def\rf{1.27} %four
    \def\rv{1} %five
\begin{tikzpicture}[style_size2]
        \draw (\lv:\rv) -- ( \lv+\pa:\rv) -- (\lv+2*\pa:\rv) -- (\lv+3*\pa:\rv) -- (\lv+4*\pa:\rv)  -- cycle;
        \draw (\lv+6*\pa:\rv) -- (\lv+ 8*\pa:\rv);
        \draw (\lv+6*\pa:\rv) -- (\lv+ 9*\pa:\rv);
        \draw[fill=white] (\lv:\rv) pic{circle_white};
        \draw[fill=black] (\lv+\pa:\rv) pic{circle_black}  ;
        \draw[fill=white] (\lv+2*\pa:\rv) pic{circle_white};
        \draw[fill=white] (\lv+3*\pa:\rv) pic{circle_white} ;
        \draw[fill=white] (\lv+4*\pa:\rv) pic{circle_white} ;
    \end{tikzpicture}    \quad
    \begin{tikzpicture}[style_size2]
        \draw (\lv:\rv) -- ( \lv+\pa:\rv) -- (\lv+2*\pa:\rv) -- (\lv+3*\pa:\rv) -- (\lv+4*\pa:\rv)  -- cycle;
        \draw (\lv+2*\pa:\rv) -- (\lv+ 4*\pa:\rv);
        \draw (\lv+2*\pa:\rv) -- (\lv+ 5*\pa:\rv);
        \draw[fill=white] (\lv:\rv) pic{circle_white};
        \draw[fill=black] (\lv+\pa:\rv) pic{circle_black} ;
        \draw[fill=white] (\lv+2*\pa:\rv) pic{circle_white};
        \draw[fill=white] (\lv+3*\pa:\rv) pic{circle_white};
        \draw[fill=white] (\lv+4*\pa:\rv) pic{circle_white};
    \end{tikzpicture}\quad
    \begin{tikzpicture}[style_size2]
        \draw (\lv:\rv) -- ( \lv+\pa:\rv) -- (\lv+2*\pa:\rv) -- (\lv+3*\pa:\rv) -- (\lv+4*\pa:\rv)  -- cycle;
        \draw (\lv+3*\pa:\rv) -- (\lv+ 5*\pa:\rv);
        \draw (\lv+3*\pa:\rv) -- (\lv+ 6*\pa:\rv);
        \draw[fill=white] (\lv:\rv) pic{circle_white};
        \draw[fill=black] (\lv+\pa:\rv) pic{circle_black} ;
        \draw[fill=white] (\lv+2*\pa:\rv) pic{circle_white};
        \draw[fill=white] (\lv+3*\pa:\rv) pic{circle_white};
        \draw[fill=white] (\lv+4*\pa:\rv) pic{circle_white};
    \end{tikzpicture}\quad
    \begin{tikzpicture}[style_size2]
        \draw (\lv:\rv) -- ( \lv+\pa:\rv) -- (\lv+2*\pa:\rv) -- (\lv+3*\pa:\rv) -- (\lv+4*\pa:\rv)  -- cycle;
        \draw (\lv+4*\pa:\rv) -- (\lv+ 6*\pa:\rv);
        \draw (\lv+4*\pa:\rv) -- (\lv+ 7*\pa:\rv);
        \draw[fill=white] (\lv:\rv) pic{circle_white};
        \draw[fill=black] (\lv+\pa:\rv) pic{circle_black} ;
        \draw[fill=white] (\lv+2*\pa:\rv) pic{circle_white};
        \draw[fill=white] (\lv+3*\pa:\rv) pic{circle_white};
        \draw[fill=white] (\lv+4*\pa:\rv) pic{circle_white};
    \end{tikzpicture}\quad
    \begin{tikzpicture}[style_size2]
        \draw (\lv:\rv) -- ( \lv+\pa:\rv) -- (\lv+2*\pa:\rv) -- (\lv+3*\pa:\rv) -- (\lv+4*\pa:\rv)  -- cycle;
        \draw (\lv+5*\pa:\rv) -- (\lv+ 7*\pa:\rv);
        \draw (\lv+5*\pa:\rv) -- (\lv+ 8*\pa:\rv);
        \draw[fill=white] (\lv:\rv) pic{circle_white};
        \draw[fill=black] (\lv+\pa:\rv) pic{circle_black} ;
        \draw[fill=white] (\lv+2*\pa:\rv) pic{circle_white};
        \draw[fill=white] (\lv+3*\pa:\rv) pic{circle_white};
        \draw[fill=white] (\lv+4*\pa:\rv) pic{circle_white};
    \end{tikzpicture}\ 
    
\end{center}
% \begin{center}
%     \includegraphics[scale=\fscale]{figs/triangulations_Examples_lr.pdf}
% \end{center}
The marked node is used in the product definition.

\begin{itemize}

\item Generator $\epsilon = \tikz[baseline=-0.1cm]{\path  pic{gentri};}$
% \raisebox{-2ex}{\includegraphics[scale=\fscale]{figs/triangulations_gen_lr.pdf}}

\item Product:  

\begin{equation}
\def\BL{0.5cm}
    \begin{tikzpicture}[style_size3, radius=5pt,baseline=\BL] 
    \def\A{(1,3)} \def\B{(2,0)} \def\C{(3,3)}
    \def\L{200} \def\R{-20}
        %  \draw [help lines, gray!50] (0,0) grid(6,3);
        \fill [color_left,draw=black] \A .. controls +(\L:2cm) and +(\L:2cm) .. \B;
        \fill [color_left,draw=black]  \B .. controls +(\R:2cm) and +(\R:2cm) .. \C;
        \draw \B -- \A -- \C -- cycle;
        \fill [black] \A circle ;
        \node [anchor=south] at \A {$a$};
        \fill [white,draw=black] \C circle  ;
        \node [anchor=south] at \C {$b$};
    \end{tikzpicture}
    \quad\star\quad 
    \begin{tikzpicture}[style_size3, radius=5pt,baseline=\BL] 
    \def\A{(1,3)} \def\B{(2,0)} \def\C{(3,3)}
    \def\L{200} \def\R{-20}
        %  \draw [help lines, gray!50] (0,0) grid(6,3);
        \fill [color_right,draw=black] \A .. controls +(\L:2cm) and +(\L:2cm) .. \B;
        \fill [color_right,draw=black]  \B .. controls +(\R:2cm) and +(\R:2cm) .. \C;
        \draw \B -- \A -- \C -- cycle;
        \fill [black] \A circle   ;
        \node [anchor=south] at \A {$c$};
        \fill [white,draw=black] \C circle  ;
        \node [anchor=south] at \C {$d$};
    \end{tikzpicture}
    \quad = \quad
        \begin{tikzpicture}[style_size3, radius=5pt,baseline=\BL,>={Stealth[length=2mm]}] 
        \def\A{(2.2,2.8)} \def\B{(0,1)} \def\C{(3,1)} \def\D{(6,1)}\def\E{(3.8,2.8)}
        \def\F{(3,-0.5)} \def\G{(3,0.2)}
        % \draw [help lines, gray!50] (0,0) grid(6,3);
        \fill [color_left,draw=black] \A .. controls +(135:2cm) and +(135:2cm) .. \B;
        \fill [color_left,draw=black]  \B .. controls +(270:2cm) and +(270:2cm) .. \C;
        \draw \B-- \A; \draw \B -- \C;
        \draw \C -- \A;
        \draw [orange,line width = 2pt] \A -- \E;
        \fill [color_right,draw=black] \C .. controls +(270:2cm) and +(270:2cm) .. \D;
        \fill [color_right,draw=black]  \D .. controls +(45:2cm) and +(45:2cm) .. \E;  
        \draw \E-- \D; \draw \D -- \C;
        \draw \C -- \E;
        \fill [black] \A circle  ;
        \node [anchor=south] at \A {$a$};
        \fill [white,draw=black] \E circle ;
        \node [anchor=south] at \E {$b$};
        \fill [white,draw=black] \C circle ;
        \node [anchor=north] at \F{$b=c$};
        \draw [->] \F -- \G;

    \end{tikzpicture}
\end{equation}

% \begin{center}
% \includegraphics[scale=\fscale]{figs/triangulations_lr.pdf}
% \end{center}

where nodes $b$ and $c$ are merged. This concatenation of triangulations appears in \cite{Conway:1973aa} in connection with frieze patterns -- $F_{\ref{fam_fz}}$.

\item $\text{\Valns}=\text{(Number of  triangles)}+1$ 

\item Examples:

\begin{align*}
 \def\rd{2pt}\def\ro{0.9} \def\rt{0.35}  
\tikz{\path pic{gentri};}
\quad\star \quad
\tikz{\path pic{gentri};}
\quad & = \quad
\begin{tikzpicture}[ scale=1 ] 
  \def\rd{2pt}\def\ro{0.9} \def\rt{0.35}  
    \draw (30:\rt) -- (150:\rt) -- (-90:\rt) -- cycle;
        \draw[orange,line width=1pt]  (30:\rt) -- (150:\rt) ;
    \draw[fill=white] (30:\rt) circle (\rd);
    \draw[fill=black] (150:\rt) circle (\rd);
    \draw[fill=white] (-90:\rt) circle (\rd);
\end{tikzpicture}
 \\
 &\\
 %========================
  \def\rd{2pt}\def\ro{0.9} \def\rt{0.35}  
\begin{tikzpicture}[ scale=1 ] 
  \def\rd{2pt}\def\ro{0.9} \def\rt{0.35}  
    \draw (30:\rt) -- (150:\rt) -- (-90:\rt) -- cycle;
    \draw[fill=white] (30:\rt) circle (\rd);
    \draw[fill=black] (150:\rt) circle (\rd);
    \draw[fill=white] (-90:\rt) circle (\rd);
\end{tikzpicture}
\quad\star \quad
\tikz{\path pic{gentri};}
\quad & = \quad
 \def\rd{2pt}\def\ro{0.9} \def\rt{0.35}  
\begin{tikzpicture}[ scale=1 ] 
  \def\rd{4pt} 
    \draw[scale=0.5] (0,0)--(1,0)--(1,1)--(0,1) -- cycle;
    \draw[scale=0.5] (0,1)--(1,0);
    \draw[orange,line width=1pt,scale=0.5] (0,1)--(1,1) ;
    \draw[scale=0.5,fill=white] (0,0) circle (\rd) (1,0) circle (\rd)  (1,1) circle (\rd)  ;
    \fill[scale=0.5,black] (0,1) circle (\rd);
\end{tikzpicture}\\
 &\\
 %==========================
  \def\rd{2pt}\def\ro{0.9} \def\rt{0.35}  
\tikz{\path pic{gentri};}
\quad\star \quad
\begin{tikzpicture}[ scale=1 ] 
  \def\rd{2pt}\def\ro{0.9} \def\rt{0.35}  
    \draw (30:\rt) -- (150:\rt) -- (-90:\rt) -- cycle;
    \draw[fill=white] (30:\rt) circle (\rd);
    \draw[fill=black] (150:\rt) circle (\rd);
    \draw[fill=white] (-90:\rt) circle (\rd);
\end{tikzpicture}
\quad & = \quad
\begin{tikzpicture}[ scale=1 ] 
  \def\rd{4pt} 
    \draw[scale=0.5] (0,0)--(1,0)--(1,1)--(0,1) -- cycle;
    \draw[scale=0.5] (0,0)--(1,1);
    \draw[orange,line width=1pt,scale=0.5] (0,1)--(1,1) ;
     \draw[scale=0.5,fill=white] (0,0) circle (\rd) (1,0) circle (\rd)  (1,1) circle (\rd)  ;
     \fill[scale=0.5,black] (0,1) circle (\rd);
\end{tikzpicture}\\
 &\\
 %==========================
\begin{tikzpicture}[ scale=1 ] 
  \def\rd{2pt}\def\ro{0.9} \def\rt{0.35}  
    \draw (30:\rt) -- (150:\rt) -- (-90:\rt) -- cycle;
    \draw[fill=white] (30:\rt) circle (\rd);
    \draw[fill=black] (150:\rt) circle (\rd);
    \draw[fill=white] (-90:\rt) circle (\rd);
\end{tikzpicture}
\quad\star  \quad
\begin{tikzpicture}[ scale=1 ] 
  \def\rd{2pt}\def\ro{0.9} \def\rt{0.35}  
    \draw (30:\rt) -- (150:\rt) -- (-90:\rt) -- cycle;
    \draw[fill=white] (30:\rt) circle (\rd);
    \draw[fill=black] (150:\rt) circle (\rd);
    \draw[fill=white] (-90:\rt) circle (\rd);
\end{tikzpicture}
\quad & = \quad
    \begin{tikzpicture}[scale=0.5] 
    \def\pa{72}
    \def\lv{54}
    \def\rd{4pt}
    \def\sc{0.4}
    \def\rd{4pt}
    \def\ro{0.9} %two
    \def\rt{1.2} % three
    \def\rf{1.27} %four
    \def\rv{1} %five
    \draw  (\lv:\rv) -- ( \lv+\pa:\rv) -- (\lv+2*\pa:\rv) -- (\lv+3*\pa:\rv) -- (\lv+4*\pa:\rv)  -- cycle;
    \draw (\lv+3*\pa:\rv) -- (\lv+ 5*\pa:\rv);
    \draw(\lv+3*\pa:\rv) -- (\lv+ 6*\pa:\rv);
    \draw[orange]  (\lv:\rv) -- ( \lv+\pa:\rv);
    \draw[fill=white] (\lv:\rv)  circle (\rd);
    \draw[fill=black] (\lv+\pa:\rv) circle (\rd) ;
    \draw[fill=white] (\lv+2*\pa:\rv) circle (\rd);
    \draw[fill=white] (\lv+3*\pa:\rv) circle (\rd);
    \draw[fill=white] (\lv+4*\pa:\rv) circle (\rd);
    \end{tikzpicture}
\end{align*}
% \begin{center}
% 	 \includegraphics[scale=\fscale]{figs/triangulations_Prod_Examples_lr.pdf}
% \end{center}

% \item Narayana parameter:   The number of polygon edges adjacent to a node which is also adjacent to a chord coming from an earlier node, ie.\ the edge $bc$ (or triangle $abc$) illustrated schematically:
% \begin{center}
%     \includegraphics[scale=\bbfscale]{figs/triag_nara.pdf}
% \end{center}
% This is discussed in more detail in \secref{sec_naray}.
% -- excluding the edge to the right of the marked node.
% number of triangles with two boundary sides and not containing.\todo{check} 

% \item Embedded Complete Binary Tree:
% \begin{itemize}
%     \item[] $\rho : \text{root}\to \text{centre of polygon edge $ab$} $
%     \item[] $\lambda : \text{leaf}\to \text{centre of polygon edge  (except $ab$)} $
% \end{itemize}
% \begin{center}
% 	\includegraphics[width=5cm]{figs/triangulations_CBT.pdf}
% \end{center}

\end{itemize}

%=============================================
%  
%=============================================
\end{family}

%=============================================

\tempnp
% \filbreak 
%>>>>>>>>>>>>>>>>>>>>>>>>>>>>>>>>>>>>>>>>>>>>>>>>>>>>>>>>>>>>>>>>>>>>>>>>>>
\begin{family}[321-avoiding permutations \stan{115}]
%>>>>>>>>>>>>>>>>>>>>>>>>>>>>>>>>>>>>>>>>>>>>>>>>>>>>>>>>>>>>>>>>>>>>>>>>>>
\label{fam_ap}

 \addcontentsline{toca}{subsection}{\fref{fam_ap}: 321-avoiding permutations}

A permutation $\sigma=\sigma_1\dots\sigma_n$ of $\set{1,\dots, n}$ is called a $321$-avoiding permutation if it does not contain the triple  $\sigma_k<\sigma_j<\sigma_i$ when $i<j<k$.  The five 321-avoiding permutations of $123$ are
\[
123,\quad 132 ,\quad 213,\quad 312,\quad 231\,,
\]
and as Rothe diagrams \cite{rothe:1800dj} (row $i$ has a dot in column $\sigma_i$ -- rows are labelled from the top):
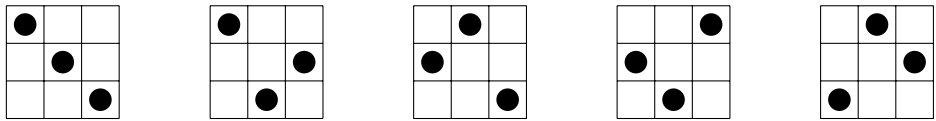
\begin{center}
\def\gap{1cm} % 1
\begin{tikzpicture}[scale=0.5]
\draw grid (3,3);
\fill (0+0.5,2+0.5) circle (0.3cm);
\fill (2+0.5,0+0.5) circle (0.3cm);
\fill (1+0.5,1+0.5) circle (0.3cm);
\end{tikzpicture}
\hspace{\gap} % 2
\begin{tikzpicture}[scale=0.5]
\draw grid (3,3);
\fill (1+0.5,0+0.5) circle (0.3cm);
\fill (2+0.5,1+0.5) circle (0.3cm);
\fill (0+0.5,2+0.5) circle (0.3cm);
\end{tikzpicture}
\hspace{\gap} % 3
\begin{tikzpicture}[scale=0.5]
\draw grid (3,3);
\fill (2+0.5,0+0.5) circle (0.3cm);
\fill (0+0.5,1+0.5) circle (0.3cm);
\fill (1+0.5,2+0.5) circle (0.3cm);
\end{tikzpicture}
\hspace{\gap} % 4
\begin{tikzpicture}[scale=0.5]
\draw grid (3,3);
\fill (1+0.5,0+0.5) circle (0.3cm);
\fill (0+0.5,1+0.5) circle (0.3cm);
\fill (2+0.5,2+0.5) circle (0.3cm);
\end{tikzpicture}
\hspace{\gap} % 5
\begin{tikzpicture}[scale=0.5]
\draw grid (3,3);
\fill (0+0.5,0+0.5) circle (0.3cm);
\fill (2+0.5,1+0.5) circle (0.3cm);
\fill (1+0.5,2+0.5) circle (0.3cm);
\end{tikzpicture}
\end{center}
% \begin{center} 
% 	    \includegraphics[scale=0.6]{figs/avoidingExamples_lr.pdf}
% \end{center}

%\includegraphics{figs/ex-perm}
\begin{itemize}  
\item Generator: $\eps = \emptyset$   ( the empty permutation). 
% On occasion the generator will be represented by a $\circ$ for clarity.
% $\eps = \raisebox{-1ex}{\includegraphics[scale=\figscale]{figs/magma-ap-gen.pdf}}$ 

\item Product: The dots of a 321-avoiding permutation can be partitioned into two subsets: 1) those on or below the diagonal -- called ``black dots'' and 2) those above the diagonal -- called ``white dots''. The product of two permutations $p_1$ and $p_2$ is illustrated schematically below:
% TIKZ
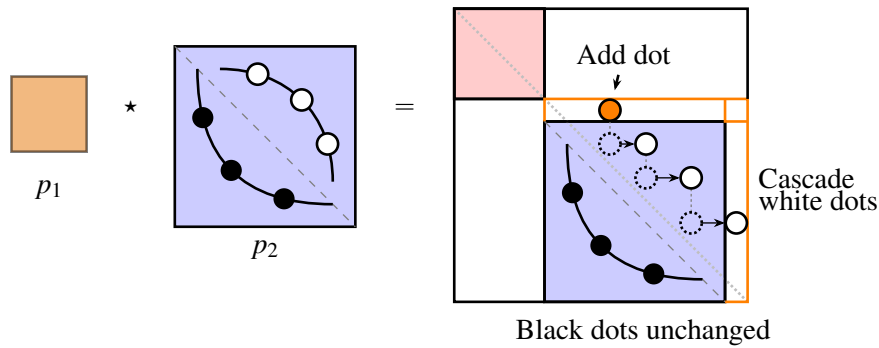
\tikzset{dot_black/.pic={\fill circle (4pt); }}
\tikzset{dot_white/.pic={\fill [white,draw=black] circle (4pt); }}
\tikzset{dot_orange_dashed/.pic={\fill [white,draw=black ] circle (4pt); }}
\tikzset{dot_blue_dashed/.pic={\fill [blue!20, draw=black, dash pattern=on 1pt off 1pt] circle (4pt); }}
\tikzset{dot_orange/.pic={\fill [orange,draw=black] circle (4pt); }}
\tikzset{dash_fine/.style={dash pattern = on 1pt off 1pt}}
\begin{equation*}
%========
\begin{tikzpicture}[scale=0.5,line width=1pt,baseline = 0.5cm]
\fill [color_left,draw=black]  (0,0) rectangle (2,2) ;
\node at (1,-1) {$p_1$};
\end{tikzpicture}
\quad\star\quad
%========
\begin{tikzpicture}[scale=0.3,baseline = 1.5cm,line width=1pt]

\fill[blue!20,draw=black] (0,0) rectangle (8,8);

\draw [dashed,line width=0.5pt,gray] (0,8) -- (8,0);
\draw (1,7) .. controls (1,3) and (3,1) .. (7,1) 
pic [pos=0.2] {dot_black}  pic [pos=0.5] {dot_black}  pic [pos=0.8] {dot_black};

\draw (2,7) .. controls (5,7) and (7,5) .. (7,2) 
pic [pos=0.2] {dot_white}  pic [pos=0.5] {dot_white}  pic [pos=0.8] {dot_white};
\node at (4,-1) {$p_2$};
% \draw [help lines, gray!50] (0,0) grid(8,8);
\end{tikzpicture}
\quad=\quad
%========
\begin{tikzpicture}[scale=0.3,baseline = 2.5cm,line width=1pt,>={Stealth[length=1.5mm]}]
\coordinate (A) at (2.9,7 );
\coordinate (B) at (4.5,5.5);
\coordinate (C) at (6.5,3.5);

\fill[red!20,draw=black] (-4,9) rectangle (0,13);
\fill[blue!20,draw=black] (0,0) rectangle (8,8);

\draw [dash_fine,gray!50] (-4,13) -- (9,0);
\draw[ black] (-4,9) rectangle (0,13);
\draw[ black] (0,0) rectangle (8,8);

\draw [orange] (0,8) -- (0,9) -- (9,9) -- (9,0) -- (8,0);
\draw [orange] (8,9) -- (8,8) -- (9,8);

\draw (0,0) -- (-4,0) -- (-4,13) -- (9,13) -- (9,9);

\draw [dashed,line width=0.5pt,gray] (0,8) -- (8,0);
\draw (1,7) .. controls (1,3) and (3,1) .. (7,1) 
pic [pos=0.2] {dot_black}  pic [pos=0.5] {dot_black}  pic [pos=0.8] {dot_black};

\draw [line width =0.5pt,->] (A) -- (4,7);
\draw [line width =0.5pt,->] (B) -- (6,5.5);
\draw [line width =0.5pt,->] (C) -- (8,3.5);
\draw[line width =0.5pt, gray,dash_fine] (B) -- (4.5,7);
\draw[line width =0.5pt, gray,dash_fine] (C) -- (6.5,5.5);
\draw[line width =0.5pt, gray,dash_fine] (2.9,8.5) -- (A); 

\draw [dashed,gray] (1.5,7.5)   (A) pic {dot_blue_dashed}   (B) pic {dot_blue_dashed}   (C) pic {dot_blue_dashed}   (7,2);
 
\path 
(2.9,8.5) pic {dot_orange} 
(4.5,7) pic {dot_orange_dashed}
(6.5,5.5) pic {dot_orange_dashed}
(8.5,3.5) pic {dot_orange_dashed};
\node[inner sep = 7pt] (on) at (2.9,8.5) {};
\node (an) at (3.5,11) {Add dot};
\draw [->] (an) -- (on);
\node (ld) at (9.5,3.5) {};
\node[right ] (an1) at (9,5.5) { Cascade };
\node[right ] (an2) at (9,4.5) { white dots};
\node[below] (an3) at (4.4,-0.3) {Black dots unchanged};
% \draw [help lines, gray ] (-4,0) grid(13,13);
\end{tikzpicture}
\end{equation*}
% END TIKZ
\noindent The permutation $p_1\star p_2$ is defined as follows (referring the the diagram above):
\begin{enumerate}
\item The permutations $p_1$ and (a modified) $p_2$ are placed on the diagonal of the new permutation (as illustrated above).
    \item The dots of $p_1$ are unchanged.
    \item The black dots of $p_2$ are unchanged.
    % \item A new row is inserted between $p_1$ and $p_2$ and an dot  added to the new row in the same column as the leftmost white dot of $p_2$.
    \item A new row is inserted between $p_1$ and $p_2$ and a new column is added on the right.
    \item  Assume $p_2=\sigma_1\dots \sigma_n$ and $p_2$ contains $k$ white dots. Label the rows and columns of $p_2$, $1,\dots ,n$ where cell $(0,0)$ is in the north-west corner.
    \begin{itemize}
    \item Case $k=0$:  A new dot is added at the intersection of the new row and new column.
        \item  Case: $k>0$: The white dots of $p_2$ are ``cascaded'' to the right as follows:
    \begin{enumerate}
        \item   Let the  white dots have
        have row-column  labels: $$(i_1,\sigma_{i_1}),(i_2,\sigma_{i_2}),\dots, (i_k,\sigma_{i_k})$$ with $i_1<i_2<\dots < i_k$.
        \item The ``cascade'' process moves all the white dots to the right as follows: 
        \begin{itemize}
            \item For all dots in rows $i_1$ to $i_{k-1}$ move dot $(i_\alpha,\sigma_{i_\alpha})$ to $(i_\alpha,\sigma_{i_{\alpha+1}})$ (ie.\ same row but column of next white dot).
            \item Move white dot at $(i_k,\sigma_{i_k})$ to a cell in the same row but in the new right column..
        \end{itemize}
    \end{enumerate}
    \end{itemize}
\end{enumerate}

	\item $\text{\Valns}=\text{(number of dots) }+1 $ 
	
%  \filbreak 
	
\item Examples:
% \begin{center}
% 	\includegraphics[scale=\fscale]{figs/avoidingProdExamples_lr.pdf}
% \end{center}
\tikzset{dot_dashed/.pic={\draw [line width=0.5pt,gray, draw=black, dash pattern=on 1pt off 1pt] circle (4pt); }}
\begin{align*}
\emptyset 
\quad\star \quad \emptyset 
\quad & = \quad
\begin{tikzpicture}[scale=0.5]
\draw rectangle (1,1);
\fill[orange,draw=black] (0.5,0.5) circle (0.3cm);
\end{tikzpicture}
\\
&\\
\emptyset 
\quad\star \quad 
\begin{tikzpicture}[scale=0.5]
\draw rectangle (1,1);
\fill  (0.5,0.5) circle (0.3cm);
\end{tikzpicture}
\quad & = \quad
\begin{tikzpicture}[scale=0.5]
\draw grid (2,2);
\fill[orange,draw=black] (0.5,1.5) circle (0.3cm);
\fill[black,draw=black] (1.5,0.5) circle (0.3cm);
\end{tikzpicture}
\\
&\\
\begin{tikzpicture}[scale=0.5]
\draw rectangle (1,1);
\fill  (0.5,0.5) circle (0.3cm);
\end{tikzpicture}
\quad\star \quad 
\emptyset
\quad & = \quad
\begin{tikzpicture}[scale=0.5]
\draw grid (2,2);
\fill[black,draw=black] (0.5,1.5) circle (0.3cm);
\fill[orange,draw=black] (1.5,0.5) circle (0.3cm);
\end{tikzpicture}
\\
&\\
\begin{tikzpicture}[scale=0.5]
\draw rectangle (1,1);
\fill  (0.5,0.5) circle (0.3cm);
\end{tikzpicture}
\quad\star \quad 
\begin{tikzpicture}[scale=0.5]
\draw rectangle (1,1);
\fill  (0.5,0.5) circle (0.3cm);
\end{tikzpicture}
\quad & = \quad
\begin{tikzpicture}[scale=0.5]
\draw grid (3,3);
\fill[black,draw=black] (0.5,2.5) circle (0.3cm);
\fill[black,draw=black] (2.5,0.5) circle (0.3cm);
\fill[orange,draw=black] (1.5,1.5) circle (0.3cm);
\end{tikzpicture}
\\
&\\
% \epsilon\star 13526487=
\emptyset\quad\star\quad 
\begin{tikzpicture}[scale=0.4,baseline=1.5cm]
\def\H{++(0.5,0.5)}
\draw (0,0)rectangle(8,8) ;
\draw [dashed,gray!40,line width=0.5pt] (0,8) -- (8,0);
\path (0,7) \H pic{dot_black} (1,4) \H pic {dot_black} (3,2)\H   pic {dot_black} (6,0)\H pic {dot_black};
\path (2,6) \H pic{dot_white} (4,5) \H pic {dot_white} (5,3)\H   pic {dot_white} (7,1)\H   pic {dot_white};
\draw [help lines,gray!50] (0,0) grid (8,8);
\end{tikzpicture}
\quad & =\quad 
\begin{tikzpicture}[scale=0.4,baseline=1.5cm]
\def\H{++(0.5,0.5)}
\draw (0,0)rectangle(8,8) ;
\draw [dashed,gray!50,line width=0.5pt] (0,8) -- (8,0);
\path (0,7) \H pic{dot_black} (1,4) \H pic {dot_black} (3,2)\H   pic {dot_black} (6,0)\H pic {dot_black};
\path (2,6) \H pic{dot_dashed} (4,5) \H pic {dot_dashed} (5,3)\H   pic {dot_dashed} (7,1)\H   pic {dot_dashed};
\draw [orange] (0,8) rectangle (9,9) rectangle (8,0);
\path (2,8)\H pic{dot_orange};
\path (4,6) \H pic{dot_white} (5,5) \H pic {dot_white} (7,3)\H   pic {dot_white} (8,1)\H   pic {dot_white};
\draw [help lines,gray!50] (0,0) grid (9,9);
\end{tikzpicture}
\end{align*} 
% \item Narayana parameter: In the above Rothe diagrams it is    the number of dots on or below the diagonal (from the top left corner to the bottom right).

% \item Embedded complete binary tree
% \begin{center}
%  	 \includegraphics[scale=0.3]{figtemp/123avoidingSkel_01.jpeg}
% \end{center}

% \item Note, the product rules for all other three number  pattern avoiding permutations can be obtained from the 321 rule by various compositions  of i) a vertical reflection, ii) a horizontal reflection or iii) a single cyclic permutation of the product diagram \eqref{eq_avoid}. This gives six patterns which pair off with their corresponding opposite family. Thus, the opposite family to 321-avoiding permutations are 123-avoiding permutations (whose product rule is the vertical reflection of \eqref{eq_avoid}.
% \begin{center}
%  	 \includegraphics[scale=0.3]{figs/allAvoiding.jpeg}
% \end{center}
\end{itemize}

\end{family}
\tempnp

%=============================================
%  Staircase polygons
%=============================================

%>>>>>>>>>>>>>>>>>>>>>>>>>>>>>>>>>>>>>>>>>>>>>>>>>>>>>>>>>>>>>>>>>>>>>>>>>>
\begin{family}[Staircase polygons or parallelogram polyominoes \stan{57}]
%>>>>>>>>>>>>>>>>>>>>>>>>>>>>>>>>>>>>>>>>>>>>>>>>>>>>>>>>>>>>>>>>>>>>>>>>>>
\label{fam_sp}

\addcontentsline{toca}{subsection}{\fref{fam_sp}: Staircase polygons or parallelogram polyominoes}

A pair of $n$ step binomial paths
% \footnote{Called binomial paths as they are enumerated by binomial coefficients}
(a sequence of vertical and horizontal steps) which: 1) do not intersect  (ie.\ have no  vertices in common) except for the first and last vertex and 2) start and end at the same position. 
\begin{center}
\def\gap{1cm}
\begin{tikzpicture}[scale=0.75]
\draw (0,0) rectangle (1,1) (0,1) rectangle (1,2) (1,1) rectangle (2,2);
\end{tikzpicture}
\hspace{\gap}
\begin{tikzpicture}[scale=0.75]
\draw (0,0) grid (3,1);
\end{tikzpicture}
\hspace{\gap}
\begin{tikzpicture}[scale=0.75]
\draw (0,0) grid (1,3);
\end{tikzpicture}
\hspace{\gap}
\begin{tikzpicture}[scale=0.75]
\draw (0,0) grid (2,2);
\end{tikzpicture}
\hspace{\gap}
\begin{tikzpicture}[scale=0.75]
\draw (0,0) rectangle (1,1) (1,1) rectangle (2,2) (1,0) rectangle (2,1);
\end{tikzpicture}
\hspace{\gap}
\end{center} 
% \begin{center}
%     \includegraphics[scale=\fscale]{figs/staircase_Ex_lr.pdf}
% \end{center}
 
\begin{itemize}

\item Generator $\eps =  \stairgen  $ (an edge).  The mark (half circle) is used to highlight the location of the generators in the polygon.

% The generators correspond to certain corners of the staircase polygon, but on occasion, in figures they are offset a small amount for clarity. 

\item Product:

\begin{equation} 
    \def\tsc{0.5}
    \def\bsl{2ex}
    \def\lwd{0.5pt}
    \begin{tikzpicture}[style_size3,baseline=\bsl]
      \path [fill,color=red!30,draw=black]  
      (0,0) -- (0,1) -- (1,2) -- (2,2) --(2,1) --(1,0)--++(-1,0);
    \end{tikzpicture}
   \quad \star\quad
    \begin{tikzpicture}[style_size3,baseline=\bsl]
    % \draw [help lines, gray!50] (0,0) grid(3,3);
      \path [fill,color_right,draw=black ] 
      (0,0) -- (0,1) -- (1,2) -- (2.5,2) --(2.5,1) --(1.5,0)--++(-1.5,0) ;
    \end{tikzpicture}
   \quad = \quad
    \begin{tikzpicture}[style_size3,baseline=\bsl]
    % \draw [help lines, gray!50] (0,0) grid(5,5);    
      \path [fill,color=red!30,draw=black ] 
        (0,0) -- (0,1) -- (1,2) -- (2,2) --(2,1) --(1,0) --++(-1,0);
      \path (2,1) [ draw=orange, line width=1pt]  -- ++(1.5,0) --++(0,1);
        % \path[fill=orange]  (3.5,1) pic{circle_product} ;
      \path [fill,color_right,draw=black ]  
        (2,2) -- ++(0,1) -- ++(1,1) -- ++(1.5,0) --++(0,-1) -- ++(-1 ,-1) --++(-1.5,0);
    \end{tikzpicture}
\end{equation}

with the convention that if the left factor is empty then the generator edge is rotated vertically.  The orange edges signify that a single cell is added below each column of the right factor.
% \begin{center}
%  	 \includegraphics[scale=\fscale]{figs/staircase_All.pdf}
% \end{center}
% \filbreak 

\item $\text{\Valns}=\text{Number of steps in either binomial path} $ with $||\gensp||=1$. 
 
\item Examples:

\begin{align*}
 \gensp\quad\star\quad\gensp 
 &\quad=\quad 
 \begin{tikzpicture}[scale=0.75]
 \path pic[rotate=90,transform shape]{gensp} (0,1) pic[transform shape]{gensp} ;
 \draw[orange,line width=1pt] (0,0) -- (1,0) -- (1,1) ;
 \end{tikzpicture}
 %===========================
 \\
  & \\
  \begin{tikzpicture}[scale=0.75]
 \path pic[rotate=90,transform shape]{gensp} (0,1) pic[transform shape]{gensp} ;
 \draw  (0,0) -- (1,0) -- (1,1) ;
 \end{tikzpicture} \quad\star\quad\gensp 
 &\quad=\quad
  \begin{tikzpicture}[scale=0.75]
 \path pic[rotate=90,transform shape]{gensp} (0,1) pic[transform shape]{gensp} (1,1) pic[transform shape]{gensp} ;
\draw  (0,0)-- (1,0)--(1,1);
 \draw[orange,line width=1pt] (1,0) -- (2,0) -- (2,1) ;
 \end{tikzpicture}
 \\
  & \\
 %===========================
\gensp  \quad\star\quad   \begin{tikzpicture}[scale=0.75]
 \path pic[rotate=90,transform shape]{gensp} (0,1) pic[transform shape]{gensp} ;
 \draw  (0,0) -- (1,0) -- (1,1) ;
 \end{tikzpicture} 
 &\quad=\quad
  \begin{tikzpicture}[scale=0.75]
 \path pic[rotate=90,transform shape]{gensp} (0,1) pic[rotate=90,transform shape]{gensp} (0,2) pic[transform shape]{gensp} ;
\draw  (0,1)-- (1,1)--(1,2);
 \draw[orange,line width=1pt] (0,0) -- (1,0) -- (1,1) ;
 \end{tikzpicture}
 \\
 & \\
 %===========================
\begin{tikzpicture}[scale=0.75]
 \path pic[rotate=90,transform shape]{gensp} (0,1) pic[transform shape]{gensp} ;
 \draw  (0,0) -- (1,0) -- (1,1) ;
 \end{tikzpicture} 
\quad\star\quad  
 \begin{tikzpicture}[scale=0.75]
 \path pic[rotate=90,transform shape]{gensp} (0,1) pic[transform shape]{gensp} ;
 \draw  (0,0) -- (1,0) -- (1,1) ;
 \end{tikzpicture} 
 &\quad=\quad
\begin{tikzpicture}[scale=0.75]
\path pic[rotate=90,transform shape]{gensp} 
(0,1) pic[transform shape] {gensp} 
(1,1) pic[rotate=90,transform shape]{gensp}  
(1,2) pic[transform shape] {gensp};
\draw  (0,0)-- (1,0)--(1,1) -- (2,1) -- (2,2);
\draw[orange,line width=1pt] (1,0) -- (2,0) -- (2,1) ;
\end{tikzpicture}
\end{align*}

% \begin{center}
% \includegraphics[scale=\fscale]{figs/staircase_Examples_lr.pdf}
% \end{center}

% \item Narayana parameter:  number of columns.

% \item Embedded Complete Binary Tree:
% \begin{itemize}
%     \item[] $\rho : \text{root}\to \text{rightmost (then) bottom-most corner} $
%     \item[] $\lambda : \text{leaf}\to \text{generator mark} $
% \end{itemize}
% \begin{center}
% 	\includegraphics[width=5cm]{figs/staircase_CBT.pdf}
% \end{center}

\end{itemize}
 
% %=============================================
% Non-nested matchings
%=============================================
\end{family}

\tempnp
%>>>>>>>>>>>>>>>>>>>>>>>>>>>>>>>>>>>>>>>>>>>>>>>>>>>>>>>>>>>>>>>>>>>>>>>>>>
\begin{family}[Two row standard tableaux \stan{168}]
%>>>>>>>>>>>>>>>>>>>>>>>>>>>>>>>>>>>>>>>>>>>>>>>>>>>>>>>>>>>>>>>>>>>>>>>>>>
\label{fam_st}

 \addcontentsline{toca}{subsection}{\fref{fam_st}: Two row standard tableaux}

Two rows of $n$ square cells. The cells contain the integers $1$ to $2n$. 
Each integer occurs exactly once and the integers increase left to right along each row and increase down each column.
\begin{center}
\begin{tabular}{ |c|c|c| } 
 \hline
 1 & 3 & 5 \\ 
 \hline
2 & 4& 6 \\ 
 \hline
\end{tabular} \qquad
\begin{tabular}{ |c|c|c| } 
 \hline
 1 & 2 & 5 \\ 
 \hline
3 & 4& 6 \\ 
 \hline
\end{tabular}\qquad
\begin{tabular}{ |c|c|c| } 
 \hline
 1 & 2 & 3 \\ 
 \hline
4 & 5& 6 \\ 
 \hline
\end{tabular}\qquad
\begin{tabular}{ |c|c|c| } 
 \hline
 1 & 2 & 4 \\ 
 \hline
3 & 5 & 6 \\ 
 \hline
\end{tabular}\qquad
\begin{tabular}{ |c|c|c| } 
 \hline
 1 & 3 & 4 \\ 
 \hline
2 & 5 & 6 \\ 
 \hline
\end{tabular}
\end{center}
%  \begin{center}
% \includegraphics[scale=\figscale]{figs/magma-trt-ex-Layer_1_lr}
%  \end{center}
\begin{itemize}

\item Generator:  $\eps = \emptyset$ (the empty tableau: no column).

\item Product: 

\begin{equation*}
\begin{tikzpicture}[scale=0.75]
\fill[color_left] (0,0) rectangle (3,2);
\draw (0,0) grid (3,2);
\path (0.5,1.5) node {$1$} (2.5,0.5) node {$k$} ;
\end{tikzpicture}
\quad \star \quad 
\begin{tikzpicture}[scale=0.75]
\draw (0,0) grid (4,2);
\fill[color_right] (0,0) rectangle (4,2);
\path (0.5,1.5) node {$1$} (3.5,0.5) node {$\ell$} ;
\end{tikzpicture}
\quad = \quad 
\begin{tikzpicture}[baseline=0cm,scale=0.75]
\fill[color_left] (0,0) rectangle (3,2);
\fill[color_right] (3,0) rectangle (8,2);
\fill[orange] (3,1) rectangle (4,2);
\fill[orange] (7,0) rectangle (8,1);
\path (0.5,1.5) node {$1$} (2.5,0.5) node {$k$} ;
\path (4.5,1.5) node {$1$} (6.5,0.5) node {$\ell$} ;
\path (3.5,1.5) node {$0$} (4.5,1.5) node {$1$} 
(7.5,0.5) node {$\ell\!\!+\!\!1$} ;
\draw[decoration={brace,mirror},decorate]
  (3,-0.25) -- node[anchor = north] {add $k+1$ to each cell} (8,-0.25);
\draw (0,0) grid (8,2);
\end{tikzpicture}
 \end{equation*}
 If the left tableau is empty take $k=0$ and if the right tableau is empty take $\ell=0$.
% \begin{center}
%     \includegraphics[scale=\fscale]{figs/twoRowTableau_lr.pdf}
% \end{center}

\item $\text{\Valns}=\text{(Number of columns) }+1 $

\item Examples:

\begin{align*}
\emptyset\quad\star\quad\emptyset  
&\quad=\quad
\begin{tikzpicture}[scale=0.75]
\fill[orange] (0,0) rectangle (1,2);
\draw (0,0) grid (1,2);
\path (0.5,0.5) node {$2$} (0.5,1.5) node {$1$} ;
\end{tikzpicture}
\\
&\\
\begin{tikzpicture}[scale=0.75]
\draw (0,0) grid (1,2);
\path (0.5,0.5) node {$2$} (0.5,1.5) node {$1$} ;
\end{tikzpicture}
\quad\star\quad
\emptyset 
&\quad=\quad
\begin{tikzpicture}[scale=0.75]
\fill[orange] (1,0) rectangle (2,2);
\draw (0,0) grid (2,2);
\path (0.5,0.5) node {$2$} (0.5,1.5) node {$1$} ;
\path (1.5,0.5) node {$4$} (1.5,1.5) node {$3$} ;
\end{tikzpicture}
\\
&\\
\emptyset \quad\star\quad
\begin{tikzpicture}[scale=0.75]
\draw (0,0) grid (1,2);
\path (0.5,0.5) node {$2$} (0.5,1.5) node {$1$} ;
\end{tikzpicture}
&\quad=\quad
\begin{tikzpicture}[scale=0.75]
\fill[orange] (0,1) rectangle (1,2);
\fill[orange] (1,0) rectangle (2,1);
\draw (0,0) grid (2,2);
\path (0.5,0.5) node {$3$} (0.5,1.5) node {$1$} ;
\path (1.5,0.5) node {$4$} (1.5,1.5) node {$2$} ;
\end{tikzpicture}
\\
&\\
\begin{tikzpicture}[scale=0.75]
\draw (0,0) grid (1,2);
\path (0.5,0.5) node {$2$} (0.5,1.5) node {$1$} ;
\end{tikzpicture}
\quad\star\quad
\begin{tikzpicture}[scale=0.75]
\draw (0,0) grid (1,2);
\path (0.5,0.5) node {$2$} (0.5,1.5) node {$1$} ;
\end{tikzpicture}
&\quad=\quad
\begin{tikzpicture}[scale=0.75]
\fill[orange] (1,1) rectangle (2,2);
\fill[orange] (2,0) rectangle (3,1);
\draw (0,0) grid (3,2);
\path (0.5,0.5) node {$2$} (0.5,1.5) node {$1$} ;
\path (1.5,0.5) node {$5$} (1.5,1.5) node {$3$} ;
\path (2.5,0.5) node {$6$} (2.5,1.5) node {$4$} ;
\end{tikzpicture}
\end{align*}

% \begin{center}
% 	\includegraphics[scale=\fscale]{figs/twoRowTableauxExamples_lr.pdf}
% \end{center}

% \item Narayana parameter: Number of occurrences where the number $i$ occurs in the top row and $i+1$ in the bottom row.

% \item Embedded Canonical word:
%   \begin{itemize}
%       \item Write the numbers $1$ to $2n$ from left to right.
%       \item Each number from the top row corresponds to product, ``$\st$''.
%       \item Each number from the bottom row corresponds to a right bracket, ``$)$''.
%       \item This uniquely fixes the left brackets and the generators.
%   \end{itemize}
%   \begin{center}
% 	\includegraphics[scale=\fscale]{figs/twoRowTableauCanWord.pdf}
% \end{center}
 
\end{itemize}

%=============================================
%  Lukacewitz paths
%=============================================

\end{family}

%=============================================

% \begin{family}[Lukacewitz paths]\label{fam_lp}

% Lattice paths in the upper half plane where the step set is $\set{(1,j)\suchthat  j=-1,0,1,2\dots}$. The value of $j$ is the `jump height' of the step. The paths start and end on the $x$-axis.
% \begin{center}
% 	\includegraphics[scale=\figscale]{figs/magma-lp-ex-Layer_1.pdf}
% %\includegraphics{figs/ex-lp}
% \end{center}

% \begin{itemize}
% \item Generator: 
% $\eps = \raisebox{-1ex}{\includegraphics[scale=\figscale]{figs/magma-lp-gen-Layer_1.pdf}}$
% \item Product:  For $k\ge 0$:
% \begin{center}
% $\raisebox{-4ex}{\includegraphics[scale=\figscale]{figs/magma-lp-prod-LHS-Layer_1} }=
% \begin{cases}
% \raisebox{-4ex}{\includegraphics[scale=\figscale]{figs/magma-lp-prodC1-Layer_1} }  & \text{if $l_2=\eps$}\\
% &\\
% \raisebox{-4ex}{\includegraphics[scale=\figscale]{figs/magma-lp-prodC2-Layer_1} }  & \text{otherwise.}
% \end{cases}
% $		
% \end{center}

% \item $\text{\valns}=\text{(number of steps)}+1 $ 
% \item Example:
% \begin{center}
% 	\includegraphics[scale=\figscale]{figs/magma-lp-ex2-Layer_1.pdf}
% \end{center}

% \item Narayana parameter:
% \end{itemize}

% \end{family}

% %=============================================
% % Floor plans
% %=============================================
%=============================================
% \filbreak 
\tempnp
%>>>>>>>>>>>>>>>>>>>>>>>>>>>>>>>>>>>>>>>>>>>>>>>>>>>>>>>>>>>>>>>>>>>>>>>>>>
\begin{family}[Catalan floor plans  \cite{Beaton:2015aa}]
%>>>>>>>>>>>>>>>>>>>>>>>>>>>>>>>>>>>>>>>>>>>>>>>>>>>>>>>>>>>>>>>>>>>>>>>>>>
\label{fam_fp}

 \addcontentsline{toca}{subsection}{\fref{fam_fp}: Catalan Floor Plans}

Catalan floor plans are equivalence classes of rectangles containing rectangles. A size $n$ floor plan rectangle contains $n$ non-nested  rectangles -- called the `rooms' of the plan. The rooms are constrained by the junctions between the interior walls. Junctions of the two forms
 
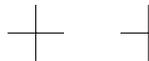
\begin{center}
\begin{tikzpicture}[scale=0.75]
\draw (0 ,0.5) -- (1,0.5) (0.5,0) -- (0.5,1) ;
\draw (2,0.5) --(2.5,0.5) (2.5,0) --( 2.5,1);
\end{tikzpicture}
 \end{center}
%
% \begin{center}
%     \includegraphics[scale=\fscale]{figs/forbidJunc_lr.pdf}
% \end{center}
are forbidden. For example, the following are forbidden floor plans
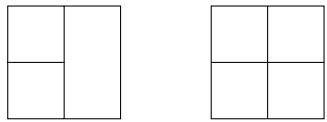
\begin{center}
\begin{tikzpicture}[scale=0.75]
\draw (0,0) rectangle (2,2) (0,1)--(1,1) (1,0) --(1,2);
\end{tikzpicture}
\hspace{1cm}
\begin{tikzpicture}[scale=0.75]
\draw (0,0) rectangle (2,2) (0,1)--(2,1) (1,0) --(1,2);
\end{tikzpicture}    
\end{center}
%
%  \begin{center}
%     \includegraphics[scale=\fscale]{figs/floorPlans_forbid_lr.pdf}
% \end{center}
 Any two floor plans are equivalent if one can be obtained from another by  sequences of wall ``slidings'': a horizontal  (respec.\ vertical)  wall can be moved up or down (respec.\ left or right) so long as it does not pass over or coincide with any other wall (or result in one of the forbidden junctions). For example the following are all equivalent. 
 
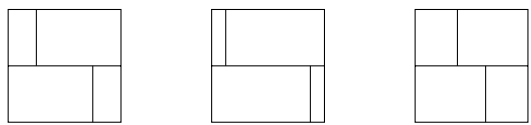
\begin{center}
\begin{tikzpicture}[scale=0.75]
 \draw (0,0) rectangle (2,2) (0,1)--(2,1) ;
 \draw (0.5,1) -- (0.5,2) (1.5,0) -- (1.5,1);
 \end{tikzpicture}
 \hspace{1cm}
 \begin{tikzpicture}[scale=0.75]
 \draw (0,0) rectangle (2,2) (0,1)--(2,1) ;
   \draw (0.25,1) -- (0.25,2) (1.75,0) -- (1.75,1);
 \end{tikzpicture}
  \hspace{1cm}
 \begin{tikzpicture}[scale=0.75]
 \draw (0,0) rectangle (2,2) (0,1)--(2,1) ;
   \draw (0.75,1) -- (0.75,2) (1.25,0) -- (1.25,1);
 \end{tikzpicture}
 \end{center}
% \begin{center}
%     \includegraphics[scale=\fscale]{figs/floorPlans_equiv_lr.pdf}
% \end{center}
The rectangular shape of the outer wall rectangle is irrelevant.
 
Thus the five Catalan floor plans with three rooms are: 
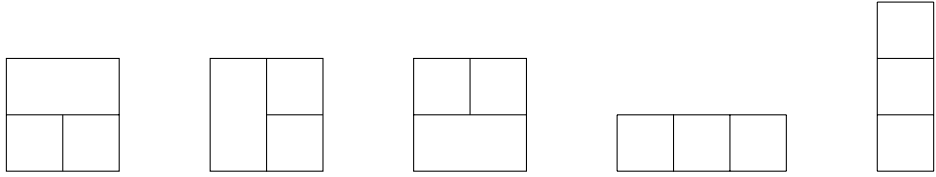
\begin{center}
\begin{tikzpicture}[scale=0.75]
\draw (0,0) rectangle (2,2) (0,1)--(2,1) ;
\draw (1,0) -- (1,1)  ;
\end{tikzpicture}
\hspace{1cm}
\begin{tikzpicture}[scale=0.75]
\draw (0,0) rectangle (2,2) (1,0)--(1,2) ;
\draw (1,1) -- (2,1)  ;
\end{tikzpicture}
\hspace{1cm}
\begin{tikzpicture}[scale=0.75]
\draw (0,0) rectangle (2,2) (0,1)--(2,1) ;
\draw (1,1) -- (1,2)  ;
\end{tikzpicture}
\hspace{1cm}
\begin{tikzpicture}[scale=0.75]
\draw (0,0) grid (3,1)   ;
\end{tikzpicture}
\hspace{1cm}
\begin{tikzpicture}[scale=0.75]
\draw (0,0) grid (1,3)   ;
\end{tikzpicture}
\end{center}
% \begin{center}
% \includegraphics[scale=\fscale]{figs/floorPlans_Examples_lr.pdf}
% \end{center}

\newcommand{\medge}{\text{\rule[2pt]{12pt}{1pt}}}

\begin{itemize}

\item Generator: $\eps =\genfp$ (an edge). The mark (half circle) is used to highlight the location of the generators in the floor plan.

\item Product:  

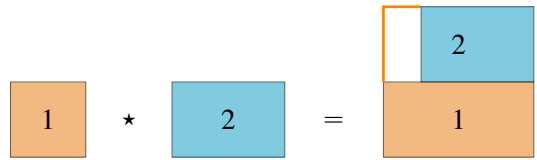
\begin{equation*}
\begin{tikzpicture}[baseline=0.4cm]
\fill[color_left,draw=black] (0,0) rectangle (1,1) ; 
\node at (0.5,0.5) {1};
\end{tikzpicture}
\quad\star\quad
\begin{tikzpicture}[baseline=0.4cm]
\fill[color_right,draw=black] (0,0) rectangle (1.5,1) ; 
\node at (0.75,0.5) {2};
\end{tikzpicture}   
\quad = \quad
\begin{tikzpicture}[baseline=0.4cm]
\fill[color_left,draw=black] (0,0) rectangle (2,1) ; 
\node at (1,0.5) {1};
\fill[color_right,draw=black] (0.5,1) rectangle (2,2); 
\node at (1 ,1.5) {2};
\draw[orange,line width=1pt] (0,1) --(0,2) -- (0.5,2); 
\end{tikzpicture}
\end{equation*}

% \begin{center}
%     \includegraphics[scale=\bfscale]{figs/floorPlans_prod_lr.pdf}
% \end{center}
Thus, first add a new column (the product geometry) spanning the left side of $f_2$, then scale $f_1$ to fit beneath. 
If $f_1$ or $f_2$ are generators then we use the convention that if $f_1=\eps$ then the edge is scaled to fit the width of $f_2$ above and if $f_2=\eps$ then the edge is rotated vertically.

% \raisebox{-9.5ex}{\includegraphics[scale=\figscale]{figs/magma-at-prod-Layer_1.pdf}} 
\item $\text{\Valns}=\text{(Number of rectangles)}+1 $  

\filbreak

\item Examples:
\begin{align*}
\genfp \quad\star\quad \genfp 
&\quad = \quad 
\begin{tikzpicture}[scale=0.75]
\draw[orange,line width=1pt] (0,0) -- (0,1) -- (1,1);
\path pic[transform shape]{genfp} (1,0) pic[rotate=90,transform shape]{genfp} ;
\end{tikzpicture}
\\
 &\\
\genfp \quad\star\quad
\begin{tikzpicture}[scale=0.75]
\draw  (0,0) -- (0,1) -- (1,1);
\path pic[transform shape]{genfp} (1,0) pic[rotate=90,transform shape]{genfp} ;
\end{tikzpicture}
& \quad = \quad 
\begin{tikzpicture}[scale=0.75]
\draw[orange,line width=1pt] (0,0) -- (0,1) -- (1,1);
\draw (1,0) rectangle (2,1);
\path pic[transform shape]{genfp} (1,0) pic[transform shape]{genfp}  (2,0) pic[rotate=90,transform shape]{genfp} ;
\end{tikzpicture}
\\
 &\\
\begin{tikzpicture}[scale=0.75]
\draw  (0,0) -- (0,1) -- (1,1);
\path pic[transform shape]{genfp} (1,0) pic[rotate=90,transform shape]{genfp} ;
\end{tikzpicture}
\quad\star\quad  \genfp 
& \quad = \quad 
\begin{tikzpicture}[scale=0.75]
\draw[orange,line width=1pt] (0,1) -- (0,2) -- (1,2);
\draw (0,0) rectangle (1,1);
\path pic[transform shape]{genfp} (1,0) pic[rotate=90,transform shape]{genfp}  (1,1) pic[rotate=90,transform shape]{genfp} ;
\end{tikzpicture}
\\
 &\\
\begin{tikzpicture}[scale=0.75]
\draw  (0,0) -- (0,1) -- (1,1);
\path pic[transform shape]{genfp} (1,0) pic[rotate=90,transform shape]{genfp} ;
\end{tikzpicture}
\quad\star\quad
\begin{tikzpicture}[scale=0.75]
\draw  (0,0) -- (0,1) -- (1,1);
\path pic[transform shape]{genfp} (1,0) pic[rotate=90,transform shape]{genfp} ;
\end{tikzpicture}
& \quad = \quad 
\begin{tikzpicture}[scale=0.75]
\draw[orange,line width=1pt] (0,1) -- (0,2) -- (1,2);
\draw (0,0) rectangle (2,1);
\fill[white,draw=black ]  (0.9,0) arc[start angle=180,end angle=360,radius=0.1cm];  
\path  (2,0) pic[rotate=90,transform shape]{genfp} ;
\draw  (1,1) -- (1,2) -- (2,2);
\path (1,1) pic[transform shape]{genfp} (2,1) pic[rotate=90,transform shape]{genfp} ;
%
% \path pic{genfp} (1,0) pic[rotate=90]{genfp}  (1,1) pic[rotate=90]{genfp} ;
\end{tikzpicture}
\end{align*}
% \begin{center}
% 	\includegraphics[scale=\bfscale]{figs/floorPlans_Prod_Examples_lr.pdf}
% \end{center}

% \item Narayana parameter: Number of rows on the right side of the diagram.
\end{itemize}

%=============================================
%  Frieze Patterns
%=============================================

\end{family}

\tempnp
%>>>>>>>>>>>>>>>>>>>>>>>>>>>>>>>>>>>>>>>>>>>>>>>>>>>>>>>>>>>>>>>>>>>>>>>>>>
\begin{family}[Frieze patterns \stan{197}: \cite{Conway:1973aa,Moon:1963aa}]
%>>>>>>>>>>>>>>>>>>>>>>>>>>>>>>>>>>>>>>>>>>>>>>>>>>>>>>>>>>>>>>>>>>>>>>>>>>
\label{fam_fz}

 \addcontentsline{toca}{subsection}{\fref{fam_fz}: Frieze Patterns}

%Stanley Prob:197
\newcommand{\wh}{\!\!\!\!\!}
Positive integer sequences $a_1, a_2,\dots, a_{n}$ which generate  an  infinite  array by periodically translating the $n-1$ row rhombus (first and last  row all $1$'s),
 \begin{equation}
   \begin{array}{*{22}c}
\cdots & 1   &  \wh  & 1     & \wh   & \cdots     & \wh    & 1    &   \wh    & \cdots      &   \wh &   \wh    & \wh   & \wh     & \wh   & \wh        & \wh   & \wh      & \wh       & \wh      & \wh     & \wh    
 \\
\wh &  \cdots & a_1  &\wh    &  a_2  & \wh   & \cdots & \wh  & a_{n}    & \wh   &   \cdots  &   \wh & \wh  &   \wh &   \wh & \wh  & \wh   &   \wh & \wh & \wh   & \wh  & \wh 
 \\
\wh  & \wh &\cdots   &   b_1   & \wh   & b_2     & \wh    & \cdots    &   \wh         & b_{n}     &   \wh &   \cdots    & \wh   & \wh     & \wh   & \wh       & \wh   & \wh    & \wh       & \wh      & \wh     &\wh  
 \\
\wh &\wh   &  \wh  &\wh   &  \ddots     & \wh      & \wh    & \wh     &   \wh         & \wh      &   \ddots  &   \wh    & \wh   & \wh      & \wh   & \wh        & \wh   & \wh      & \wh       & \wh      & \wh     & \wh   
  \\
\wh &\wh    &  \wh  & \wh      & \cdots  & r_1     & \wh    & r_2    &   \wh     &   \cdots     & \wh      & r_{n } &  \wh   & \cdots & \wh       & \wh          & \wh   & \wh    & \wh       & \wh     & \wh     & \wh  
 \\
\wh & & \wh      &  \wh  &  \wh    & \cdots   &  1    & \wh    & 1    &   \wh         & \cdots      &   \wh &   1   & \wh   & \cdots     & \wh    &  \wh       & \wh   &  \wh     & \wh       &  \wh     & \wh  
 \\
%     & b_1  & b_2      & 
\end{array} 
%   \begin{array}{*{21}c}
% 1   &  \wh  & 1     & \wh   & 1     & \wh    & 1    &   \wh         & 1     &   \wh &   1   & \wh   & 1     & \wh   & 1       & \wh   & 1     & \wh       & 1     & \wh     & 1   
%  \\
% \wh &  a_1  &\wh    &  a_2  & \wh   & \cdots & \wh  & a_{n}    & \wh   &   a_1 &   \wh &  a_2  &   \wh &   \wh & \cdots  & \wh   &   \wh &a_{n-2} & \wh   & a_{n+1}  & \wh 
%  \\
% \wh &  \wh  & b_1   & \wh   & b_2     & \wh    & \cdots    &   \wh         & b_{n}     &   \wh &   b_1   & \wh   & b_2     & \wh   & \cdots       & \wh   & b_{n-2}     & \wh       & \wh      & \wh     &\wh  
%  \\
% \wh   &  \wh  &\wh   &  \ddots     & \wh      & \wh    & \wh     &   \wh         & \wh      &   \wh &   \wh    & \wh   & \wh      & \wh   & \wh        & \wh   & \wh      & \wh       & \wh      & \wh     & \wh   
%   \\
% \wh    &  \wh  & \wh      & \wh   & r_1     & \wh    & r_2    &   \wh         & \wh      &   \cdots &  \wh   & \wh   & r_{n }     & \wh   & r_1       & \wh   & \wh    & \wh       & \wh     & \wh     & \wh  
%  \\
%  & \wh      &  \wh  &  \wh    & \wh   &  1    & \wh    & 1    &   \wh         & \cdots      &   \wh &   1   & \wh   & 1     & \wh   &  \wh       & \wh   &  \wh     & \wh       &  \wh     & \wh  
%  \\
% %     & b_1  & b_2      & 
% \end{array} 
\end{equation}
to the left and right . The values of the sequences in rows  three to the second last row are  fixed by requiring 
 that   any quadrangle 
\begin{equation}
    \begin{array}{ccc}
       \wh  & r & \\
        s & \wh & t\\
         \wh  & u & 
    \end{array}
\end{equation}
of four neighbouring entries satisfies the `unimodular' property: $st-ru=1$. The unimodular quadrangles can always be iterated downwards, the constraint on the $a_i$'s is  that  row $n-1$ must all be   $1$'s.

 For $n=5$ there are only five sequences which satisfy the above constraint: 
\begin{equation}
    12213,\quad 22131, \quad 21312,\quad 13122,\quad 31221\,.
\end{equation}
The  sequence $12213$ defines the four row frieze pattern:
\begin{equation}
   \cdots   \begin{array}{*{14}c}
    1   & \wh   & 1     & \wh   & 1     & \wh   & 1     & \wh   & 1  & \wh      &   & \wh       &     &    \\
    \wh & 1     & \wh   & 2     & \wh   & 2     & \wh   & 1     & \wh   & 3     & \wh   &       & \wh        &    \\
    \wh &\wh    & 1     & \wh   & 3     & \wh   & 1     & \wh   & 2     & \wh   &   2    & \wh   & \wh        &       \\
    \wh &\wh    & \wh   & 1     & \wh   & 1     & \wh   & 1     & \wh   & 1   & \wh   &   1 & \wh        & 
\end{array}\cdots
\end{equation}

\begin{itemize}
\item Generator: $\eps = 00$.

\item Product\footnote{This result appears in \cite{Conway:1973aa} as a way of generating new patterns from old and is a consequence of the bijection to triangulations of $n$-gons --  \fref{fam_pt}.}
\[
a_1, a_2,\dots, a_{n}\,\st\, b_1, b_2,\dots, b_{m} = c_1, c_2,\dots, c_{n+m-1}
\]
where
\begin{equation}
    c_i=\begin{cases}
    a_1+1 & i=1\\
    a_i & 1<i<n\\ 
    a_{n}+ b_1+1  & i = n \\
    b_{i} &  n<i<n+m-1\\
    b_{m}+1 & i=n+m-1
    \end{cases}
\end{equation}

\item $\text{\Valns}=(\text{Length of sequence})-1$  

\item Examples:
\begin{align*}
    00\st 00 & = 111\\
    00\st 111 & = 1212\\
    111\st 00 & = 2121\\
    111\st 111 & = 21312\\
\end{align*}

 \end{itemize}
\end{family}

\newpage
\bibliographystyle{unsrt}  
\bibliography{refs}

\begin{thebibliography}{10}

\bibitem{A000108:aa}
N.~J.~A. Sloane.
\newblock The {O}n-{L}ine {E}ncyclopedia of {I}nteger {S}equences, {C}atalan
  numbers: A000108.
\newblock https://oeis.org/A000108.

\bibitem{stanley:1999vw2}
R~Stanley.
\newblock {\em Enumerative Combinatorics}, volume~2 of {\em Cambridge Studies
  in Advanced Mathematics 62}.
\newblock Cambridge University Press, 1999.

\bibitem{Koshy:2009aa}
T.~Koshy.
\newblock {\em Catalan Numbers with Applications}.
\newblock Oxford University Press, 2009.

\bibitem{Stanley:2015aa}
Richard Stanley.
\newblock {\em Catalan numbers}.
\newblock Cambridge University Press, 2015.

\bibitem{Gould:aa}
H.~W. Gould.
\newblock {\em Bell and {C}atalan Numbers: A Research Bibliography of Two
  Special Number Sequences}.
\newblock 2011.

\bibitem{Beaton:2015aa}
N.~R. {Beaton}, M.~{Bouvel}, V.~{Guerrini}, and S.~{Rinaldi}.
\newblock Slicings of parallelogram polyominoes: {C}atalan, {S}chr\"oder,
  {B}axter, and other sequences.
\newblock {\em To appear in Electronic Journal of Combinatorics}, 2019.

\bibitem{Segner:1761aa}
J.~A. Segner.
\newblock Enumeratio modorum quibus figurae planae.
\newblock {\em Novi Comment. Acad. Sci. Imp. Petropol.}, 7:203--210, 1761.

\bibitem{Brak:2018aa}
R.~Brak and N.~Mahony.
\newblock Free magmas for {M}otzkin and {S}chr\"oder families, 2019.

\bibitem{motzkin:1948lr}
T.~Motzkin.
\newblock Relations between hypersurface cross ratios, and a combinatorial
  formula for partitions of a polygon, for permanent preponderance, and for
  non-associative products.
\newblock {\em Bull. Amer. Math. Soc.}, 54:352--360, 1948.

\bibitem{sulanke:1998ix}
R.~A. Sulanke.
\newblock Bijective recurrences concerning {S}chr\"oder paths.
\newblock {\em Electronic Journal of Combinatorics}, 5(R47), 1998.

\bibitem{Bourbaki:1981qy}
N.~Bourbaki.
\newblock {\em Algebra}.
\newblock Springer- Verlag Berlin Heidelberg, {E}nglish translation by {C}ohn
  and {H}owie edition, 1981.

\bibitem{Loday:2012aa}
L.~Loday and B.~Vallette.
\newblock {\em Algebraic Operads}.
\newblock Springer- Verlag Berlin Heidelberg, 2012.

\bibitem{A02522X:aa}
N.~J.~A. Sloane.
\newblock The {O}n-{L}ine {E}ncyclopedia of {I}nteger {S}equences, {C}atalan
  numbers: {A}025225 and {A}025226.
\newblock https://oeis.org/A000108.

\bibitem{Conway:1973aa}
J.~H. Conway and H.~S.~M. Coxeter.
\newblock Triangulated polygons and frieze patterns.
\newblock {\em The Mathematical Gazette}, 57(400):87--9, 1973.

\bibitem{Brown:1964aa}
W.G. Brown.
\newblock Enumeration of triangulations of the disk.
\newblock {\em Proc. Lond. Math. Soc.}, 14:746--768, 1964.

\bibitem{Knuth:2005aa}
D.~Knuth.
\newblock Three catalan bijections.
\newblock Mittag Leffler Institute and Stanford Univeristy Report, 2005.

\bibitem{delest84}
M.~Delest and G.~Viennot.
\newblock Algebraic languages and polyominoe enumeration.
\newblock {\em Theo. Comp. Sc.}, 34:169--2068, 1984.

\bibitem{Comtet:1974aa}
L.~Comtet.
\newblock {\em Advanced Combinatorics}.
\newblock Kluwer Academic Publishers-Plenum Publishers, 1974.

\bibitem{Rotem:1975aa}
D.~Rotem.
\newblock On a correspondence between binary trees and a certain type of
  permutations.
\newblock {\em Information processing letters}, 4:58--61, 1975.

\bibitem{narayana:1959ix}
T.~V. Narayana.
\newblock A partial order and its application to probability theory.
\newblock {\em Sankhya}, 21:91--98, 1959.

\bibitem{sulanke:1998mz}
R.~A. Sulanke.
\newblock Catalan path statistics having the {N}arayana distribution.
\newblock {\em Discrete Mathematics}, 180:369--389, 1998.

\bibitem{sulanke:2002lc}
R.~A. Sulanke.
\newblock The {N}arayana distribution.
\newblock {\em J. Statist. Plann. Inference}, 101:311--326, 2002.

\bibitem{flajolet:2001ys}
P.~Flajolet and R.~Sedgewick.
\newblock Analytic combinatorics: Symbolic combinatorics, 2001.

\bibitem{goulden:1983yg}
P.~Goulden and D.~M. Jackson.
\newblock {\em Combinatorial Enumeration}.
\newblock John Wiley and Sons, 1983.

\bibitem{simion:2000lr}
R.~Simion.
\newblock Non-crossing partitions.
\newblock {\em Discrete Mathematics}, 217:367--409, 2000.

\bibitem{rothe:1800dj}
H.~A. Rothe.
\newblock {\em Sammlung Combinatorisch-Analytischer Abhandlungen}, volume~2,
  pages 263--305.
\newblock 1800.

\bibitem{Moon:1963aa}
J.~W. Moon and L.~Moser.
\newblock Triangular dissections of n-gons.
\newblock {\em Can. math. Bull}, 6:175--178, 1963.

\end{thebibliography}

\end{document}